\documentclass[a4paper,10pt]{article}
\usepackage[utf8]{inputenc}

 \usepackage{geometry}
\usepackage{amsmath}
\usepackage{comment}
\usepackage{leftidx}
\usepackage{multicol}
\usepackage{amssymb}
\usepackage{stmaryrd}
\usepackage{accents}
\usepackage{mathrsfs}
\usepackage{xfrac}
\usepackage{amsthm}
\usepackage{xcolor}
\usepackage{sectsty}
\usepackage{relsize}
\usepackage{float}

\usepackage{indentfirst}
\usepackage{tikz}
\tikzset{shorten <>/.style={shorten >=#1,shorten <=#1}}

\usepackage{tikz-cd}
\newcounter{nodemaker}
\tikzset{Rightarrow/.style={double equal sign distance,>={Implies},->},
triple/.style={-,preaction={draw,Rightarrow}},
quadruple/.style={preaction={draw,Rightarrow,shorten >=0pt},shorten >=1pt,-,double,double
distance=0.2pt}}

\tikzset{%
    symbol/.style={%
        draw=none,
        every to/.append style={%
            edge node={node [sloped, allow upside down, auto=false]{$#1$}}}
    }
}
 \usepackage{quiver}
\makeatletter
\newcommand{\bigdoublevee}{\big@doubleop{\bigvee}}
\newcommand{\bigdoublewedge}{\big@doubleop{\bigwedge}}
\newcommand{\big@doubleop}[1]{%
  \DOTSB\mathop{\mathpalette\big@doubleop@aux{#1}}\slimits@
}

\newcommand\big@doubleop@aux[2]{%
  \sbox\z@{$\m@th#1#2$}%
  \makebox[1.35\wd\z@][s]{$\m@th#1#2\hss#2$}%
}
\makeatother

\makeatletter
\newcommand*{\doublerightarrow}[2]{\mathrel{
  \settowidth{\@tempdima}{$\scriptstyle#1$}
  \settowidth{\@tempdimb}{$\scriptstyle#2$}
  \ifdim\@tempdimb>\@tempdima \@tempdima=\@tempdimb\fi
  \mathop{\vcenter{
    \offinterlineskip\ialign{\hbox to\dimexpr\@tempdima+1em{##}\cr
    \rightarrowfill\cr\noalign{\kern.5ex}
    \rightarrowfill\cr}}}\limits^{\!#1}_{\!#2}}}
\newcommand*{\triplerightarrow}[1]{\mathrel{
  \settowidth{\@tempdima}{$\scriptstyle#1$}
  \mathop{\vcenter{
    \offinterlineskip\ialign{\hbox to\dimexpr\@tempdima+1em{##}\cr
    \rightarrowfill\cr\noalign{\kern.5ex}
    \rightarrowfill\cr\noalign{\kern.5ex}
    \rightarrowfill\cr}}}\limits^{\!#1}}}
\makeatother

\makeatletter
\newcommand{\xRrightarrow}[2][]{\ext@arrow 0359\Rrightarrowfill@{#1}{#2}}
\newcommand{\Rrightarrowfill@}{\arrowfill@\equiv\equiv\Rrightarrow}
\newcommand{\xLleftarrow}[2][]{\ext@arrow 3095\Lleftarrowfill@{#1}{#2}}
\newcommand{\Lleftarrowfill@}{\arrowfill@\Lleftarrow\equiv\equiv}
\makeatother

\newcommand{\hirayo}{\text{\usefont{U}{min}{m}{n}\symbol{'210}}}
\DeclareFontFamily{U}{min}{}
\DeclareFontShape{U}{min}{m}{n}{<-> udmj30}{}

\DeclareFontFamily{U}{min}{}
\DeclareFontShape{U}{min}{m}{n}{<-> udmj30}{}

\newcommand{\katayo}{\text{\usefont{U}{min}{m}{n}\symbol{'350}}}
\DeclareFontFamily{U}{min}{}
\DeclareFontShape{U}{min}{m}{n}{<-> udmj30}{}

\DeclareUnicodeCharacter{3088}{\yo}

\newcommand{\Spec}% spectrum
 {{\bf Spec}}
 
 \newcommand{\Alg}% alg
 {{\bf Alg}}
 
  \newcommand{\psAlg}% psalg
 {{\bf psAlg}}
 
 \newcommand{\Lex}% Lex
 {{\bf Lex}}

  \newcommand{\Mon}% Mon
 {{\bf Mon}}
 
 \newcommand{\biLex}% biLex
 {{\bf biLex}}
 
 \newcommand{\Et}% etale map
 {{\bf Et}}
 
\newcommand{\Loc}% local map
 {{\bf Loc}}
 
  \newcommand{\Diag}% diagonally universal morphisms
 {{\bf Diag}}
 
 \newcommand{\GTop}% Topos
 {{\bf GTop}}
 
 \newcommand{\Geom}% Topos
 {{\bf Geom}}
 
 \newcommand{\Cat}% cat
 {{\bf Cat}}

\newcommand{\cod}% codomain
 {{\rm cod}}

\newcommand{\comp}% composition
 {\circ}

\newcommand{\Cont}% category of continuous G-sets
 {{\bf Cont}}

\newcommand{\dom}% domain
 {{\rm dom}}
 
\newcommand{\fp}% fp
{{\rm fp }}

\newcommand{\li}% limit
{{\textup{lim} }}

\newcommand{\oplaxlim}% limit
{{\textup{oplaxlim} }}

\newcommand{\oplaxcolim}% limit
{{\textup{oplaxcolim} }}

\newcommand{\laxcolim}% limit
{{\textup{laxcolim} }}

\newcommand{\bicolim}% bicolimit
{{\textup{bicolim} }}

\newcommand{\bilim}% bilimit
{{\textup{bilim} }}

\newcommand{\pscolim}% pscolimit
{{\textup{pscolim} }}

\newcommand{\pslim}% pslimit
{{\textup{pslim} }}

\newcommand{\lan}% Lan
{{\textup{lan} }}

\newcommand{\ran}% limit
{{\textup{ran} }}

\newcommand{\colim}% colimit
{{\textup{colim}}}

\newcommand{\coeq}% coeq
{{\textup{coeq}}}

\newcommand{\eq}% eq
{{\textup{eq}}}

\newcommand{\comma}[2]% comma object
{\mbox{$(#1\!\downarrow\!#2)$}}

\newcommand{\empstg}% empty string
 {[\,]}

\newcommand{\epi}% epimorphism
 {\twoheadrightarrow}

\newcommand{\hy}% hyphen (in math mode)
 {\mbox{-}}

\newcommand{\im}% image
 {{\rm im}}

\newcommand{\imp}% implication
 {\!\Rightarrow\!}

\newcommand{\Ind}% ind
 {{\rm Ind}}
 
 \newcommand{\Pro}% ind
 {{\rm Pro}}

\newcommand{\mono}% monomorphism 
 {\rightarrowtail}

\newcommand{\ob}% class of objects
 {{\rm ob}}
 
 \newcommand{\can}% canonical topology
 {{\rm can}}
 
 \newcommand{\Hom}% class of objects
 {{\rm Hom}}

\newcommand{\op}% opposite category
 {{\rm op}}
 
 \newcommand{\pt}% points
 {{\bf pt}}

\newcommand{\Set}% category of sets
 {{\bf Set }}

\newcommand{\Sh}% category of sheaves
 {{\bf Sh}}

\newcommand{\sh}% category of sheaves
 {{\bf sh}}

\newcommand{\Sub}% subobject lattice
 {{\rm Sub}}

\newcommand{\Flat}% flat
 {{\bf Flat}}
\newcommand{\biIns}% biInserter
{{\textup{biIns} }}

\newcommand{\biInv}% biInverter
{{\textup{biInv} }}

\newcommand{\biIsoIns}% biIsoInserter
{{\textup{biIsoIns} }}

\newcommand{\biend}% biend
{{\textup{bi}\int }}

\newcommand{\biLan}% biLan
{{\textup{biLan} }}

\newcommand{\biRan}% biRan
{{\textup{biRan} }}

 \newcommand{\Cart}% cartesian map
 {{\bf Cart}}

  \newcommand{\ps}% pseudonat
 {{\bf ps}}

 \newcommand{\St}% bicategory of stacks
 {{\bf St}}

\usepackage{hyperref}
\hypersetup{
    colorlinks = true,
    linkbordercolor = {red},
   	linkcolor ={teal},
	anchorcolor = {pink},
	citecolor =  {orange},
	filecolor = {teal},
	menucolor = {teal},
	runcolor =  {teal},
	urlcolor = {teal},
}

\usepackage{cleveref}
\newtheorem{theorem}{Theorem}[subsection]
\newtheorem*{theorem*}{Theorem}
\newtheorem{proposition}[theorem]{Proposition}
\newtheorem{corollary}[theorem]{Corollary}
\newtheorem{corollary'}[theorem]{Corollary}
\newtheorem{lemma}[theorem]{Lemma}
\theoremstyle{definition}
\newtheorem{definition}[theorem]{Definition}
\newtheorem{example}[theorem]{Example}
\theoremstyle{definition}
\newtheorem{remark}[theorem]{Remark}
\theoremstyle{definition}
\newtheorem{division}[theorem]{}
\theoremstyle{conjecture}

\usepackage{imakeidx}
\usepackage{authblk}

\title{Bi-accessible and bipresentable 2-categories}

\author{Ivan Di Liberti\footnote{Corresponding Author.} \hspace{1.75cm} Axel Osmond}

\date{}

\begin{document}

\maketitle

\vspace{-1cm}

\begin{abstract}
We develop a 2-dimensional version of accessibility and presentability compatible with the formalism of flat pseudofunctors. First we give prerequisites on the different notions of 2-dimensional colimits, filteredness and cofinality; in particular we show that \emph{$\sigma$-filteredness} and \emph{bifilteredness} are actually equivalent in practice for our purposes. Then, we define bi-accessible and bipresentable 2-categories in terms of \emph{bicompact} objects and \emph{bifiltered} bicolimits. We then characterize them as categories of \emph{flat pseudofunctors}. We also prove a bi-accessible right bi-adjoint functor theorem and deduce a 2-dimensional Gabriel-Ulmer duality relating small \emph{bilex} 2-categories and finitely bipresentable 2-categories.  Finally, we show that 2-categories of pseudo-algebras of finitary 2-monads on $\Cat$ are finitely bipresentable, which in particular captures the case of $\Lex$, the 2-category of small lex categories. Invoking the technology of \emph{lex-colimits}, we prove further that several $2$-categories arising in categorical logic (\textbf{Reg, Ex, Coh, Ext, Adh, Pretop}) are also finitely bipresentable. 

\end{abstract}

  {
   \hypersetup{linkcolor=black}
   \tableofcontents
  }

\section*{Introduction}

The original motivation of this work was to find a 2-dimensional notion of presentability encompassing the different finitary \emph{first order doctrines} corresponding to different fragments of first order logic. Those are 2-categories whose objects are small categories endowed with a certain structure allowing to see them as syntactic categories for first order theories, where one can interpret connectors and inference rules. Prominent examples of doctrines are $ \textbf{Lex}$, for \emph{left exact} categories, corresponding to cartesian logic; $ \textbf{Reg}$, for \emph{regular} categories, corresponding to regular logic; $ \textbf{Coh}$ for \emph{coherent categories}, corresponding to coherent logic; but also the 2-categories $ \textbf{Ext}$ of \emph{extensive} categories, $ \textbf{Adh}$ of \emph{adhesive} categories, $\textbf{Ex}$ of \emph{exact} categories, $ \textbf{Pretop}_\omega$ of \emph{finitary pretopoi}, or $ \textbf{BoolPretop}$ of \emph{boolean} finitary pretopoi.\\ 

Those doctrines can also be seen as higher-dimensional versions of the different categories of propositional algebras, as $ \wedge \hy \textbf{Slat}$, the category of meet-semilattices, $ \textbf{DLat}$, the category of bounded distributive lattices, $ \textbf{Bool}$ the category of boolean algebras and so on. A common feature to most of those categories of propositional algebras is that they are \emph{finitely presentable}: they are cocomplete and generated under filtered colimits by an essentially small subcategory of compact objects. Finitely presentable categories are known to enjoy a lot of excellent properties and provide a framework generalizing universal algebra, a reason for which the 1-categorical theory of presentability, as well as the more general theory of \emph{accessibility}, have become classical topics at the intersection of category theory and model theory since \cite{makkai1989accessible} and \cite{adamek1994locally}. Already in \cite{makkai1995gabbay}, Makkai envisions the relevance of this project, prospecting that a general theory of $2$-dimensional accessibility and presentability should exist and should extend to the $2$-categories of interest in this paper\footnote{In \cite{makkai1995gabbay}, Makkai claims that much of this work should be straightforward, yet we will discuss a few paragraphs below that \textit{the devil is in the details}.}. \\

Previous proposal for 2-dimensional accessibility and presentability can be found in \cite{kelly1982structures} and \cite{bourke2020accessible} in the stricter context of enriched categories. However, capturing first order doctrines as examples requires a more relaxed version involving weaker notions of filteredness and colimits: for instance, $\textbf{Lex}$ is not 2-presentable in the sense of \cite{kelly1982structures} because it has only bicolimits and not all strict ones, beside issues about its expected rank of 2-accessibility in the sense of \cite{bourke2020accessible}. Here, we will work in the environment of strict 2-categories together with pseudofunctors and pseudonatural transformation (or equivalently, strict 2-functors and pseudonatural transformations), as our notions will involve weak inversal properties relying on equivalences of categories rather than isomorphisms; a good reference for this framework is provided by \cite{johnson20212}. \\

%for instance $ \Lex$ lacks some strict 2-colimits which would be expected in those formalism, while it has bicolimits. %for those versions still relied on 1-dimensional filtered colimits and involved too strict notions of 2-(co)limits, first order doctrines failed to be finitely 2-presentable in their sense - for instance, $\textbf{Lex}$ is not 2-presentable in the sense of \cite{kelly1982structures} because it has only bicolimits and not all strict ones, beside issues about its expected rank of 2-accessibility in the sense of \cite{bourke2020accessible}.  \\

We introduce here relaxed notions of \emph{bi-accessible} (\Cref{sigma accessible cat}) and \emph{bipresentable} (\Cref{bipres}) 2-categories and connect them to the recent advance of \cite{descotte2018sigma} on the theory of \emph{flat pseudofunctors}. Our notion relies on \cite{kennison1992fundamental} notion of \emph{bifilteredness}, together with a convenient notion of \emph{bicompact objects} (\Cref{bicompact}) enjoying the analogous property of compact objects against bifiltered bicolimits; we define finitely bi-accessible categories as those having bifiltered bicolimits and an essentially small subcategory of bicompact objects generating them under bifiltered bicolimits; finitely bipresentable 2-categories are as those that are moreover bicocomplete - but similarly to the one dimensional case, this amounts to having weighted bilimits. In particular, a handy criterion to detect finite bipresentability of a $2$-category is offered in \Cref{1.11}, which is a $2$-dimensional analog of a classical characterization through strong generators.  We then prove that categories of flat pseudofunctors are bi-accessible (\Cref{2-cat of flat pseudofunctor are accessible}) and bipresentable (\Cref{Flat pseudofunctors on bilex are bipres}) if their domain admits finite weighted bilimits, the latter result being part of a categorification of the well known \emph{Gabriel-Ulmer duality} established in \Cref{GU}. \\

A subtle part of this work is to chose first the appropriate shape of 2-dimensional filteredness and which associated classes of 2-dimensional (co)-limits should be used to express accessibility. There are two competing possible choices: \begin{itemize}
    \item one either could rely on (conical) bifiltered bicolimits as mentioned above à la \cite{kennison1992fundamental}, which looks as the most naive possible generalization; the problem with this approach is that $\Cat$-valuated pseudofunctors do not express in general as conical bicolimits of representables, which a priori is an obstruction to a theory of flat pseudofunctors; 
    \item or one could either use the formalism of \emph{$\sigma$-filtered $\sigma$-colimits}, a class intermediate between bi and lax. The point with this formulation is that it goes with an already established, nicely behaved theory of flat pseudofunctors, as developed in \cite{descotte2018sigma}. This theoretical advantage qualifies this formalism as a \textit{conceptually} correct notion of 2-dimensional filteredness.  
    
    %though most pseudofunctors cannot be expressed as conical bicolimits of representables (requiring either $\sigma$-bicolimits or non trivial weights to encode 2-dimensional data), 

\end{itemize} 

Luckily, those two approaches actually merge into a single one thanks to the key observation that any $ \sigma$-filtered 2-category admits a \emph{$\sigma$-cofinal} bifiltered 2-category (see \Cref{key lemma}) which simplifies the corresponding $ \sigma$-bicolimit into an ordinary bicolimit. Those simplifications can also be observed by the fact that all non invertible oplax inclusions of a $\sigma$-filtered $\sigma$-colimit can be ``swallowed" by the diagram and replaced by invertible inclusion 2-cells (\Cref{triangle lemma}). \\

In particular this explains why, though defined with bifiltered bicolimits, bi-accessible 2-category corresponds to 2-categories of flat pseudofunctors which where first defined in \cite{descotte2018sigma} in terms of $\sigma$-filteredness and $\sigma$-colimits. This phenomenon happens to be related also to the 2-dimensional part of the lifting property of bicompact objects, which allows to transfer non-invertible 2-cells of a $ \sigma$-cocone into the underlying diagram. It also ensures that the 2-dimensional analog of the canonical cocone over an object in a bi-accessible 2-category can be chosen either as the $\sigma$-filtered oplax cocone (relative to the cocartesian 1-cells, see \Cref{the oplaxcocone is the category of element of the binerve}) or more simply as a bifiltered bicocone over it (see \Cref{the pseudococone is bifiltered}). All of this ensures that any attempt to do a theory of ``$\sigma$-accessible 2-categories" uniquely stated in term of $ \sigma$-filtered $\sigma$-bicolimits will automatically reduce to our theory of bi-accessibility. \\

Our last section is devoted to the motivating examples. After the preliminary observation that $\Cat$ really is the ur-example of finitely bipresentable 2-categories, we prove in \Cref{Algebras of a finitary pseudomonad are sigma-presentable} that the 2-category of pseudo-algebras and pseudomorphisms for a bifinitary pseudomonad on a finitely bipresentable 2-category is itself finitely bipresentable. This captures in particular the example of $\Lex$ (see \Cref{Lex is finitary}), for its bifiltered bicolimits can be shown to be computed in $\Cat$. Incidentally, this argument applies also to the two category of monoidal categories and strong monoidal functors (see \Cref{monoidalcats}).\\

Proving the remaining finitary first order doctrines to be finitely bipresentable could be achieved in several ways. One could directly try to manipulate their axioms as kind of ``2-dimensional cartesian predicates" (involving only finitely weighted bilimits or left and right Kan extensions) in a manner reminiscent of Makkai paradigm of \emph{injectivity classes}. While we believe such a work to be interesting by itself as an occasion to conscientiously perform 2-dimensional logic, such a process would have represented an enormous amount of specific considerations and lemmas for each example. Such a tedious work appeared superfluous thanks to the powerful paradigm of \emph{lex colimits} introduced by \cite{GarnerLacklexcolimits}, which was specifically developed to capture a large class of examples of doctrines defined through different flavours of exactness conditions. Here they are axiomatized as classes of pseudo-algebras for some 2-monad on $\Lex$ defined as a free cocompletion under specific colimits weighted in some class of finite weights - this has been described as ``cocompleteness in the lex world". Having proven $\Lex$ to be finitely bipresentable, it suffices then to prove the embedding of the category of those pseudo-algebras (which form the corresponding 2-categories of exact categories relative to those weights) to be bifinitary, which is done in \Cref{The phiexactmonad is finitary}. This captures all the remaining doctrines defined from exactness properties as $ \textbf{Reg}$, $ \textbf{Ex}$, $ \textbf{Coh}$, $ \textbf{Adh}$, $ \textbf{Ext}$ and $ \textbf{Pretop}_\omega$.

\section{Prerequisites} \label{generalities}

\subsection{Notions of strictness and universality}

In this subsection we shall recall some general definitions and facts about $2$-dimensional category theory, mainly addressing the strictness nuances that will arise later in the paper. We briefly recall the notion of weighted (bi)limit. We refer to   \cite{johnson20212} for an introduction to the topic, but we address the reader to \cite[Sec. 3.4]{bourke2020accessible} and \cite[Sec. 1 and 2]{bird1989flexible}, for a very detailed discussion and comparison with the literature.

\begin{division}[$2$-Categories of $2$-functors]
For any 2-categories $\mathcal{A}, \, \mathcal{B}$ we have the following inclusions of 2-categories of strict 2-functors together with respectively \emph{strict}, \emph{pseudo} and \emph{lax} natural transformations:
\[ [\mathcal{A}, \mathcal{B}] \hookrightarrow [\mathcal{A}, \mathcal{B}]_p \hookrightarrow [\mathcal{A}, \mathcal{B}]_l. \]
%Beware not to confuse in general $ [\mathcal{A}, \mathcal{B}]_p $, the 2-category of strict 2-functors and pseudonatural transformations, with $ \ps[\mathcal{A}, \mathcal{B}]$, the 2-category of \emph{pseudofunctors} and pseudonatural transformations. However, this caution should not be taken too seriously in the case of (pseudo)functors with values in $\Cat$, as in practice the following result tells us those are almost the same.
\end{division}

However, we are going to work mostly with \emph{pseudofunctors} between 2-categories: whence the importance of the following observation, from \cite{johnson20212}, ensuring that considering pseudofunctors between strict 2-categories returns actually again a strict 2-category rather than the weaker notion of bicategory - which we are not going to consider in this work:

\begin{proposition}[{\cite{johnson20212}[Corollary 4.4.13]}] 
Let be $\mathcal{A}, \mathcal{B}$ two 2-categories; then taking pseudofunctors, pseudonatural transformations between them and natural modifications between those yields a 2-category $ \ps[\mathcal{A}, \mathcal{B}]$. 
\end{proposition}

\begin{proposition}[Strictification of pseudofunctors, originally shown in {\cite{power1989general}}]\label{strictification}
If $ \mathcal{A}$ is a strict 2-category, then any pseudofunctor $ F: \mathcal{A} \rightarrow \Cat$ admits a \emph{strictification}, that is there exists a strict 2-functor $ \overline{F}: \mathcal{A} \rightarrow \Cat$ and a pseudonatural equivalence $ \eta_F : F \Rightarrow \overline{F}$. Moreover for any strict 2-functor $ G : \mathcal{A} \rightarrow \Cat$ we have an equivalence of categories, 
\[ \ps[\mathcal{A}, \Cat][F, \iota(G)] \simeq [\mathcal{A}, \Cat][\overline{F}, G].\] 
\end{proposition}

\begin{division}[Notions of $2$-limits]
Let $ I$ be a 2-category. For a weight, that is a pseudofunctor $W : I \rightarrow \Cat$, and a pseudofunctor $ F : I \rightarrow \mathcal{B}$ we denote respectively as $ \pslim^{\!W} F$ and $ \bilim^{W} F $ the  \emph{$W$-weighted} pseudolimit and {bilimit}, whose universal properties are respectively that for any $B$ in $\mathcal{B}$ we have a natural isomorphism, resp. a natural equivalence of categories

\[ \mathcal{B}[B, \pslim^{W}F] \cong \ps[I, \Cat]\big[ W, \mathcal{B}[B, F] \big]  \]
\[ \mathcal{B}[B, \bilim^{W}F] \simeq \ps[I, \Cat]\big[ W, \mathcal{B}[B, F] \big]  \]

Similarly for a contravariant weight $ W: I^{\op} \rightarrow \Cat $ and $F :I \rightarrow \mathcal{B}$ a pseudofunctor we define respectively the $W$-weighted \emph{pseudocolimit}\index{pseudocolimit} $ \pscolim^{W} F$ and \emph{bicolimit}\index{bicolimit} $ \bicolim^{W}F$ from the universal property that for any $B$ we have an isomorphism, resp. an equivalence of categories

\[ \mathcal{B}[\pscolim^{W}F, B] \cong \ps[I^{\op}, \Cat]\big[W, \mathcal{B}[F,B] \big] \]
\[ \mathcal{B}[\bicolim^{W}F, B] \simeq \ps[I^{\op}, \Cat]\big[W, \mathcal{B}[F,B] \big] \]

In particular, one recovers the corresponding \emph{conical} kind of (co)limits by considering the terminal weight $ \Delta_1 : I \rightarrow Cat$ sending any object of $I$ to the one object category 1.
\end{division}

\begin{definition}[Finite weighted bilimits] \label{finiteweibili}
A \emph{finite weight}\index{finite weight} (in the sense of \cite[Def. 3.2.2]{descotte2018sigma}) is a weight $ W : \mathcal{A} \rightarrow \Cat$ such that $ \mathcal{A}$ is an essentially finite category and each $ W(A)$ is an essentially finite category. A \emph{finite limit}\index{finite limit} is a weighted bilimit with finite weight.
\end{definition}

\begin{example}\label{sufficient all finite bilimits}
It immediately follows from the definition that finite biproducts, biequalizers and bicotensors with the arrow category 2 are all finite weighted bilimits. It is observed in \cite[Rem. 3.2.5]{descotte2018sigma} that these suffice to construct all finite weighted bilimits.
\end{example}

\begin{remark}[On finite $2$-limits]
Depending on the flavor of mathematics, the intentions of the paper and the historical moment, people have used different notions of \textit{finiteness} for $2$-limits.  Let us review those that are most related to ours.
\begin{itemize}
    \item In \cite[p. 149 and (more importantly) 150]{street1976limits}, Street provides one of the first historical discussions on the notion of finite $2$-limit. There a $2$-category is said to be \textit{finitely complete} when it has \textit{products, equalizers, and cotensors with the arrow category}. The notion of finite limit is based on the notion of \textit{finite presentable category}.  Of course, Street shows that a category is finitely complete if and only if it has all finite $2$-limits. This notion of $2$-limit will be later used 
     \cite{street1982two} to characterize Street's notion of $2$-dimensional topos. Street's theory is very \textit{strict}, and thus can be only considered inspirational for us.
    This theory may lead to a strict version of Diaconescu theorem, while \cite{descotte2018sigma} result on extension of flat functors is stated in terms of finite bilimits. All in all, the notions of finite limit in \cite{street1982two} and \cite{street1976limits} are strict and based on \textit{finite presentable categories}, this makes it incomparable to our framework.
    \item \cite{street1982characterization} clarifies how to technically manipulate the theory of \cite{street1982two} and \cite{street1976limits} to characterize bitopoi and thus meets our framework in a comparable way. His notion of finite (bi)limits is still based on \cite{street1976limits}, with the exception of the strictness of the universal property. Thus, finite bilimits in the sense of \cite{street1982characterization} are more than finite bilimits in our sense. Yet, because finite bilimits in the sense of Street are generated by \textit{biproducts, biequalizers, and bicotensors with the arrow category} \cite[analog of Cor. 8]{street1976limits}, his notion of lex functor is equivalent to ours via \Cref{sufficient all finite bilimits}.
    \item The last notion we would like to mention is in \cite[Def. 4.1]{kelly1982structures}. \cite[4.6]{kelly1982structures} makes this theory collapse to \cite{street1976limits} in the case of $\Cat$-enriched categories, at least for what concerns the notion of lexity.
\end{itemize}
\end{remark}

\subsection{$\sigma$-limits}

Now we turn to an intermediate class of limits between the pseudo and lax, identified first in \cite{descotte2018sigma}: it is the universal (op)lax cone where one requires only some 2-cells inclusions to be invertible as in a bilimit. In particular, this definition is always relative to a choice of a class of maps in the indexing 2-category. In this section, we carefully recall the main ideas of the theory of $\sigma$-limits from \cite{descotte2018sigma} and \cite{descotte2016exactness}. More recently, this notion has been revisited from another perspective in \cite{gagna2021bilimits} under the name of \emph{marked limits}. 

\begin{remark} \label{morethandubuc}
Our definition of $\sigma$-limit is slightly more general at first sight from that of \cite{descotte2016exactness}, as they only define the notion of $\sigma$-limit for a strict $2$-functor, while we will need it for pseudofunctors too for we will have to consider $\sigma$-limits of composite of strict 2-functors along pseudofunctors. Such a level of generality does not change the theory in expressivity, and we shall provide a lemma to prove it (see \Cref{sigma colim and strictification}). 
\end{remark}

\begin{definition}[$\sigma$-natural transformations]
Let $I$ be a 2-category and $ \Sigma$ a class of maps in $I$ containing equivalences and closed under composition and invertible 2-cells; let $ \mathcal{C}$ be a category and $ F,G : I \rightarrow \mathcal{C}$ a pair of pseudofunctors. A $\sigma$-\emph{natural transformation} relatively to $\Sigma$ is a lax natural transformation $ f: F \Rightarrow G$ whose lax naturality squares% https://q.uiver.app/?q=WzAsNCxbMCwwLCJGKGkpIl0sWzEsMCwiRihqKSJdLFswLDEsIkcoaSkiXSxbMSwxLCJHKGopIl0sWzAsMSwiRihzKSJdLFswLDIsImZfaSIsMl0sWzEsMywiZl9qIl0sWzIsMywiRyhzKSIsMl0sWzcsNiwiZl9zIiwwLHsiY3VydmUiOi0xLCJzaG9ydGVuIjp7InNvdXJjZSI6MjAsInRhcmdldCI6MjB9fV1d
\[\begin{tikzcd}
	{F(i)} & {F(j)} \\
	{G(i)} & {G(j)}
	\arrow["{F(s)}", from=1-1, to=1-2]
	\arrow["{f_i}"', from=1-1, to=2-1]
	\arrow[""{name=0, anchor=center, inner sep=0}, "{f_j}", from=1-2, to=2-2]
	\arrow[""{name=1, anchor=center, inner sep=0}, "{G(s)}"', from=2-1, to=2-2]
	\arrow["{f_s}", curve={height=-6pt}, shorten <=4pt, shorten >=4pt, Rightarrow, from=1, to=0]
\end{tikzcd}\]
at an arrow $ s$ in $\Sigma$ are invertible 2-cells of $\mathcal{C}$. Similarly an \emph{op}$\sigma$-\emph{natural transformation} is an oplax natural transformation whose oplax naturality squares over maps in $\Sigma$ are invertible. We denote as $\ps[I, \mathcal{C}]_\Sigma $ the 2-category of pseudofunctors and $ \sigma$-natural transformations relatively to $ \Sigma$, with no restriction on 2-cells, and $\ps[I, \mathcal{C}]_{\op\Sigma} $ for op$\sigma$-natural transformations.  
\end{definition}

\begin{definition}
For any $I$ and any class of map $\Sigma$ in $I$ (without assumption about $\Sigma$), one can consider the closure $\overline\Sigma$ of $\Sigma$ defined as the smallest sub-2-category containing $\Sigma$ and closed under invertible 2-cells.
\end{definition}

\begin{division}
It is clear that any $ \overline{\Sigma}$-natural transformation is in particular $\Sigma$-natural, so we have for each pairs of pseudofunctors $ F,G : I \rightarrow \Sigma$ a full inclusion $ \ps[I, \mathcal{C}]_{\overline{\Sigma}} \hookrightarrow\ps[I, \mathcal{C}]_\Sigma[F,G]  $, and similarly for op-$\sigma$-natural transformations. But in fact, the coherence conditions of (op)lax transformations ensure that any (op-)$\Sigma$-natural transformation is automatically (op-)$\overline{\Sigma}$-natural, and more formally:
\end{division}

\begin{lemma}
Let be $\Sigma$ a class of maps in a 2-category $I$. Then we have a biequivalence of 2-categories, pseudonatural in $C$:
\[ \ps[I, \mathcal{C}]_\Sigma \simeq \ps[I, \mathcal{C}]_{\overline{\Sigma}} \hskip0.5cm (\textrm{resp. } \ps[I, \mathcal{C}]_{\op\Sigma} \simeq \ps[I, \mathcal{C}]_{\op\overline{\Sigma}} )  \]
\end{lemma}

\begin{proof}
For there are no restrictions on pseudofunctors, those 2-categories have the same objects. It suffices then to prove for any two $F,G$ that the full inclusion $ \ps[I, \mathcal{C}]_{\overline{\Sigma}} \hookrightarrow\ps[I, \mathcal{C}]_\Sigma[F,G]  $ is actually surjective on objects: we show as promised that any $\Sigma$-natural transformation $ f: F \Rightarrow G$ is actually $ \overline{\Sigma}$-natural. First, take $ s: i \rightarrow j$ and $t : j \rightarrow k$ in $\Sigma$ and a pair: then lax naturality gives us the equality of 2-cells
% https://q.uiver.app/?q=WzAsNixbMCwxLCJGKGkpIl0sWzEsMCwiRihqKSJdLFsyLDEsIkYoaykiXSxbMCwyLCJHKGkpIl0sWzIsMiwiRyhrKSJdLFsxLDEsIkcoaikiXSxbMCwxLCJGKHMpIl0sWzEsMiwiRih0KSJdLFswLDMsInFfaSIsMl0sWzIsNCwicV9qIl0sWzEsNV0sWzMsNSwiRyhzKSIsMV0sWzUsNCwiRyh0KSIsMV0sWzMsNCwiRyh0cykiLDJdLFs1LDAsImZfcyBcXGF0b3AgXFxzaW1lcSIsMSx7InN0eWxlIjp7ImJvZHkiOnsibmFtZSI6Im5vbmUifSwiaGVhZCI6eyJuYW1lIjoibm9uZSJ9fX1dLFsyLDUsImZfdCBcXGF0b3AgXFxzaW1lcSIsMSx7InN0eWxlIjp7ImJvZHkiOnsibmFtZSI6Im5vbmUifSwiaGVhZCI6eyJuYW1lIjoibm9uZSJ9fX1dLFs1LDEzLCJHX3tzLHR9IFxcYXRvcCBcXHNpbWVxIiwxLHsic2hvcnRlbiI6eyJ0YXJnZXQiOjIwfSwic3R5bGUiOnsiYm9keSI6eyJuYW1lIjoibm9uZSJ9LCJoZWFkIjp7Im5hbWUiOiJub25lIn19fV1d
% https://q.uiver.app/?q=WzAsNixbMCwxLCJGKGkpIl0sWzEsMCwiRihqKSJdLFsyLDEsIkYoaykiXSxbMCwyLCJHKGkpIl0sWzIsMiwiRyhrKSJdLFsxLDEsIkcoaikiXSxbMCwxLCJGKHMpIl0sWzEsMiwiRih0KSJdLFswLDMsImZfaSIsMl0sWzIsNCwiZl9qIl0sWzEsNV0sWzMsNSwiRyhzKSIsMV0sWzUsNCwiRyh0KSIsMV0sWzMsNCwiRyh0cykiLDJdLFs1LDAsImZfcyBcXGF0b3AgXFxzaW1lcSIsMSx7InN0eWxlIjp7ImJvZHkiOnsibmFtZSI6Im5vbmUifSwiaGVhZCI6eyJuYW1lIjoibm9uZSJ9fX1dLFsyLDUsImZfdCBcXGF0b3AgXFxzaW1lcSIsMSx7InN0eWxlIjp7ImJvZHkiOnsibmFtZSI6Im5vbmUifSwiaGVhZCI6eyJuYW1lIjoibm9uZSJ9fX1dLFs1LDEzLCJHX3tzLHR9XnstMX0gXFxhdG9wIFxcc2ltZXEiLDEseyJzaG9ydGVuIjp7InRhcmdldCI6MjB9LCJzdHlsZSI6eyJib2R5Ijp7Im5hbWUiOiJub25lIn0sImhlYWQiOnsibmFtZSI6Im5vbmUifX19XV0=
\[\begin{tikzcd}
	& {F(j)} \\
	{F(i)} & {G(j)} & {F(k)} \\
	{G(i)} && {G(k)}
	\arrow["{F(s)}", from=2-1, to=1-2]
	\arrow["{F(t)}", from=1-2, to=2-3]
	\arrow["{f_i}"', from=2-1, to=3-1]
	\arrow["{f_k}", from=2-3, to=3-3]
	\arrow["f_j", from=1-2, to=2-2]
	\arrow["{G(s)}"{description}, from=3-1, to=2-2]
	\arrow["{G(t)}"{description}, from=2-2, to=3-3]
	\arrow[""{name=0, anchor=center, inner sep=0}, "{G(ts)}"', from=3-1, to=3-3]
	\arrow["{f_s \atop \simeq}"{description}, draw=none, from=2-2, to=2-1]
	\arrow["{f_t \atop \simeq}"{description}, draw=none, from=2-3, to=2-2]
	\arrow["{G_{s,t}^{-1} \atop \simeq}"{description}, draw=none, from=2-2, to=0]
\end{tikzcd} = 
% https://q.uiver.app/?q=WzAsNSxbMCwxLCJGKGkpIl0sWzEsMCwiRihqKSJdLFsyLDEsIkYoaykiXSxbMCwyLCJHKGkpIl0sWzIsMiwiRyhrKSJdLFswLDEsIkYocykiXSxbMSwyLCJGKHQpIl0sWzAsMywiZl9pIiwyXSxbMiw0LCJmX2oiXSxbMyw0LCJHKHRzKSIsMl0sWzAsMiwiRih0cykiLDFdLFsxLDEwLCJGX3tzLHR9XnstMX0gXFxhdG9wIFxcc2ltZXEiLDEseyJzaG9ydGVuIjp7InRhcmdldCI6MjB9LCJzdHlsZSI6eyJib2R5Ijp7Im5hbWUiOiJub25lIn0sImhlYWQiOnsibmFtZSI6Im5vbmUifX19XSxbOSw4LCJmX3tzdH0gIiwwLHsiY3VydmUiOi0xLCJzaG9ydGVuIjp7InNvdXJjZSI6MjAsInRhcmdldCI6MjB9fV1d
\begin{tikzcd}
	& {F(j)} \\
	{F(i)} && {F(k)} \\
	{G(i)} && {G(k)}
	\arrow["{F(s)}", from=2-1, to=1-2]
	\arrow["{F(t)}", from=1-2, to=2-3]
	\arrow["{f_i}"', from=2-1, to=3-1]
	\arrow[""{name=0, anchor=center, inner sep=0}, "{f_k}", from=2-3, to=3-3]
	\arrow[""{name=1, anchor=center, inner sep=0}, "{G(ts)}"', from=3-1, to=3-3]
	\arrow[""{name=2, anchor=center, inner sep=0}, "{F(ts)}"{description}, from=2-1, to=2-3]
	\arrow["{F_{s,t}^{-1} \atop \simeq}"{description}, draw=none, from=1-2, to=2]
	\arrow["{f_{ts} }", curve={height=-6pt}, shorten <=7pt, shorten >=7pt, Rightarrow, from=1, to=0]
\end{tikzcd}\]
Hence the composite $ f_k*F^{-1}_{s,t} f_{ts}  $ is invertible, and so is $f_{ts}$ by cancellation of invertible 2-cell. Similar argument for identity arrows. Finally suppose that $ \alpha : d \simeq s$ is an invertible 2-cell with $s $ in $\Sigma$, than in the same vein one has an equality of 2-cells
% https://q.uiver.app/?q=WzAsNCxbMCwwLCJGKGkpIl0sWzIsMCwiRihqKSJdLFswLDEsIkcoaSkiXSxbMiwxLCJHKGopIl0sWzAsMSwiRihzKSIsMCx7ImN1cnZlIjotMn1dLFswLDIsImZfaSIsMl0sWzEsMywiZl9qIl0sWzIsMywiRyhkKSIsMix7ImN1cnZlIjoyfV0sWzAsMSwiRihkKSIsMSx7ImN1cnZlIjoyfV0sWzQsOCwiRihcXGFscGhhKSBcXGF0b3AgXFxzaW1lcSIsMSx7InNob3J0ZW4iOnsic291cmNlIjoyMCwidGFyZ2V0IjoyMH0sInN0eWxlIjp7ImJvZHkiOnsibmFtZSI6Im5vbmUifSwiaGVhZCI6eyJuYW1lIjoibm9uZSJ9fX1dLFs3LDgsImZfZCIsMCx7InNob3J0ZW4iOnsic291cmNlIjoyMCwidGFyZ2V0IjoyMH19XV0=
\[\begin{tikzcd}
	{F(i)} && {F(j)} \\
	{G(i)} && {G(j)}
	\arrow[""{name=0, anchor=center, inner sep=0}, "{F(s)}", curve={height=-12pt}, from=1-1, to=1-3]
	\arrow["{f_i}"', from=1-1, to=2-1]
	\arrow["{f_j}", from=1-3, to=2-3]
	\arrow[""{name=1, anchor=center, inner sep=0}, "{G(d)}"', curve={height=12pt}, from=2-1, to=2-3]
	\arrow[""{name=2, anchor=center, inner sep=0}, "{F(d)}"{description}, curve={height=12pt}, from=1-1, to=1-3]
	\arrow["{F(\alpha) \atop \simeq}"{description}, draw=none, from=0, to=2]
	\arrow["{f_d}", shorten <=5pt, shorten >=5pt, Rightarrow, from=1, to=2]
\end{tikzcd} =
% https://q.uiver.app/?q=WzAsNCxbMCwwLCJGKGkpIl0sWzIsMCwiRihqKSJdLFswLDEsIkcoaSkiXSxbMiwxLCJHKGopIl0sWzAsMSwiRihzKSIsMCx7ImN1cnZlIjotMn1dLFswLDIsImZfaSIsMl0sWzEsMywiZl9qIl0sWzIsMywiRyhkKSIsMix7ImN1cnZlIjoyfV0sWzIsMywiIiwxLHsiY3VydmUiOi0yfV0sWzgsNywiRyhcXGFscGhhKSBcXGF0b3AgXFxzaW1lcSIsMSx7InNob3J0ZW4iOnsic291cmNlIjoyMCwidGFyZ2V0IjoyMH0sInN0eWxlIjp7ImJvZHkiOnsibmFtZSI6Im5vbmUifSwiaGVhZCI6eyJuYW1lIjoibm9uZSJ9fX1dLFs0LDgsImZfcyBcXGF0b3AgXFxzaW1lcSIsMSx7InNob3J0ZW4iOnsic291cmNlIjoyMCwidGFyZ2V0IjoyMH0sInN0eWxlIjp7ImJvZHkiOnsibmFtZSI6Im5vbmUifSwiaGVhZCI6eyJuYW1lIjoibm9uZSJ9fX1dXQ==
\begin{tikzcd}
	{F(i)} && {F(j)} \\
	{G(i)} && {G(j)}
	\arrow[""{name=0, anchor=center, inner sep=0}, "{F(s)}", curve={height=-12pt}, from=1-1, to=1-3]
	\arrow["{f_i}"', from=1-1, to=2-1]
	\arrow["{f_j}", from=1-3, to=2-3]
	\arrow[""{name=1, anchor=center, inner sep=0}, "{G(d)}"', curve={height=12pt}, from=2-1, to=2-3]
	\arrow[""{name=2, anchor=center, inner sep=0}, curve={height=-12pt}, from=2-1, to=2-3]
	\arrow["{G(\alpha) \atop \simeq}"{description}, draw=none, from=2, to=1]
	\arrow["{f_s \atop \simeq}"{description}, draw=none, from=0, to=2]
\end{tikzcd}\]
Then again cancellation of invertible 2-cells ensures that $ f_d$ has to be invertible.
\end{proof}

\begin{remark}
As a consequence, one can always stipulate $ \Sigma$ to be closed under the condition above; in this case $ \Sigma$ is in particular a full on 0-cells and 2-cells sub-2-category of $I$.
\end{remark}

\begin{definition}[$\sigma$-cone]
Let $ F : I \rightarrow \mathcal{C}$ be a pseudofunctor and $\Sigma$ a class in $I$. Then a $\sigma$-\emph{cone} on $F$ (relative to $\Sigma$) is a lax cone on $C$ in $\mathcal{C}$ is a lax natural transformation $ f: \Delta_B \Rightarrow F$ whose lax naturality triangles % https://q.uiver.app/?q=WzAsMyxbMCwxLCJGKGkpIl0sWzIsMSwiRihqKSJdLFsxLDAsIkMiXSxbMCwxLCJGKHMpIiwyXSxbMiwwLCJmX2kiLDJdLFsyLDEsImZfaiJdLFs0LDUsImZfcyIsMCx7Im9mZnNldCI6NCwic2hvcnRlbiI6eyJzb3VyY2UiOjIwLCJ0YXJnZXQiOjIwfX1dXQ==
\[\begin{tikzcd}
	& C \\
	{F(i)} && {F(j)}
	\arrow["{F(d)}"', from=2-1, to=2-3]
	\arrow[""{name=0, anchor=center, inner sep=0}, "{f_i}"', from=1-2, to=2-1]
	\arrow[""{name=1, anchor=center, inner sep=0}, "{f_j}", from=1-2, to=2-3]
	\arrow["{f_d}", shift right=4, shorten <=6pt, shorten >=6pt, Rightarrow, from=0, to=1]
\end{tikzcd}\]
which are in particular invertible whenever $ d$ is in $\Sigma$. Dually, a \emph{$\sigma$-cocone} is an \emph{op$\sigma$}-natural transformation $ q: F \Rightarrow \Delta_B$, with oplax naturality triangles invertible at maps in $\Sigma$
% https://q.uiver.app/?q=WzAsMyxbMCwwLCJGKGkpIl0sWzIsMCwiRihqKSJdLFsxLDEsIkIiXSxbMCwxLCJGKGQpIl0sWzEsMiwicV9qIl0sWzAsMiwicV9pIiwyXSxbNCw1LCJxX2QiLDIseyJzaG9ydGVuIjp7InNvdXJjZSI6MjAsInRhcmdldCI6MjB9fV1d
\[\begin{tikzcd}
	{F(i)} && {F(j)} \\
	& B
	\arrow["{F(d)}", from=1-1, to=1-3]
	\arrow[""{name=0, anchor=center, inner sep=0}, "{q_j}", from=1-3, to=2-2]
	\arrow[""{name=1, anchor=center, inner sep=0}, "{q_i}"', from=1-1, to=2-2]
	\arrow["{q_d}"', shorten <=6pt, shorten >=6pt, Rightarrow, from=0, to=1]
\end{tikzcd}\]
\end{definition}

\begin{remark}[Cones are lax, cocones are oplax]
The reader might be confused by the orientation of the $2$-dimensional data in the diagrams above: the $\sigma$-cones are made of lax cells below an object, while $\sigma$-cocones are made of \emph{oplax} cells above it.  %Beware that, though this last definition involves an opsigma natural transformation, we really speak of a $\sigma$-cocone. 
The opsigma naturality condition involved is explained at \cite[Remark 2.4.1]{descotte2018sigma}. We will also give some justifications at the definition of $\sigma$-colimit in \Cref{why}.
\end{remark}

\begin{remark}
Observe that any lax cone (resp. any pseudocone) is a $\sigma$-cone with $ \Sigma$ consisting only of isomorphisms in $I$ (resp. $ \Sigma$ containing all arrows of $I$). Dually, a pseudocone is a $\sigma$-cone for a $\Sigma$ consisting of all arrows. 
\end{remark}

\begin{definition}[Weighted $\sigma$-bilimit]
Let $I$ be a small 2-category and $ \Sigma$ a class of maps in $I$ containing equivalences and closed under composition and invertible 2-cells, $ W : I \rightarrow \Cat$ a pseudofunctor and $ F : I \rightarrow \mathcal{B}$ a pseudofunctor. Then the \emph{weighted $\sigma$-limit} (resp. \emph{$\sigma$-bilimit}) relatively to $ \Sigma$ 
such that for any $B$ in $\mathcal{B}$ we have an isomorphism (resp. an equivalence) of categories
\[ \mathcal{B}[B, \Sigma\lim \, F] \cong \ps[I, \Cat]_\Sigma \big[ W, \mathcal{B}[B, F] \big]  \]
\[ \mathcal{B}[B, \Sigma\bilim \, F] \simeq \ps[I, \Cat]_\Sigma \big[ W, \mathcal{B}[B, F] \big]  \]
\end{definition}

\begin{remark}[Conical $\sigma$-(bi)limits]
In particular we can consider conical $\sigma$-(bi)limits as those $\sigma$-cones \[  ( p_i : \Sigma\underset{i \in I}{\lim} \, F(i) \rightarrow F(i))_{i \in I} \hskip1cm \textup{( resp. }    ( p_i : \Sigma\underset{i \in I}{\bilim} \, F(i) \rightarrow F(i))_{i \in I}  \textup{)} \] 
such that any other $\sigma$-cone $ (f_i : B \rightarrow F(i))_{i \in I}$ induces a universal arrow $(f_i)_{i \in I} : B \rightarrow \Sigma\lim \, F $ commuting strictly, reps. up to a canonical invertible 2-cell, with the cone projections.
\end{remark}

\begin{division}[Weighted $\sigma$-(bi)colimits]\label{why}
We can also dually define weighted \emph{$\sigma$-(bi)colimit} for a weight $ I^{\op} \rightarrow \Cat$ with $(I,\Sigma)$ a $\Sigma$ pair and $F : I \rightarrow \mathcal{B}$ through the formula
\[ \mathcal{B}[\underset{i \in I}{\Sigma\bicolim^{W}} \, F(i), B] \simeq \ps[{I}^{\op}, \Cat]_{\Sigma^{\op}} \big{[} W, \mathcal{B} [F,B] \big{]}   \]
In particular, in the case of conical bicolimits, that is for the weight $ W= \Delta_1$, observe also that the later homcategory is equivalent to the category of $\sigma$-cocones over $F$ with tip $B$
\[  [{I}^{\op}, \Cat]_{\Sigma^{\op}} \big{[} \Delta_1, \mathcal{B} [F,B] \big{]} \simeq  \ps[I,\mathcal{B}]_{\op\Sigma} [F,\Delta_B] \]
In particular, this exhibits the $\sigma$-bicolimit as a universal $\sigma$-cocone; beware that its 2-cells inclusions at a morphism $ d : i \rightarrow j$ in $I$ are of the form
% https://q.uiver.app/?q=WzAsMyxbMCwwLCJGKGkpIl0sWzIsMCwiRihqKSJdLFsxLDEsIlxcdW5kZXJzZXR7aSBcXGluIEl9e1xcc2lnbWFfXFxTaWdtYVxcYmljb2xpbX0gXFwsIEYoaSkiXSxbMCwxLCJGKGQpIl0sWzEsMiwicV9qIl0sWzAsMiwicV9pIiwyXSxbNCw1LCJxX2QiLDIseyJzaG9ydGVuIjp7InNvdXJjZSI6MjAsInRhcmdldCI6MjB9fV1d
\[\begin{tikzcd}
	{F(i)} && {F(j)} \\
	& {\underset{i \in I}{\Sigma\bicolim} \, F(i)}
	\arrow["{F(d)}", from=1-1, to=1-3]
	\arrow[""{name=0, anchor=center, inner sep=0}, "{q_j}", from=1-3, to=2-2]
	\arrow[""{name=1, anchor=center, inner sep=0}, "{q_i}"', from=1-1, to=2-2]
	\arrow["{q_d}"', shorten <=11pt, shorten >=11pt, Rightarrow, from=0, to=1]
\end{tikzcd}\]
\end{division}

\begin{remark}[Avoiding weights] As one interest of $\sigma$-(co)limits is that they allow to turn weighted (co)limits into special conical lax (co)limits, for instance in the $\sigma$-colimit decomposition into representable we shall see below, we will only make use of conical $\sigma$-(co)limits in the following. We will mostly use also the bi(co)limits, yet in some contexts - in particular in $\Cat$ - we may innocently interchange pseudolimits with bilimits when the former have a canonical expression, using that, whenever they both exist, they are equivalent. 
\end{remark}

The next lemma ensures that there are no difference between our theory of sigma-bicolimits of pseudofunctors and \cite{descotte2018sigma} version for strict 2-functors:

%\textcolor{red}{Problem: that is not true that any pseudofunctor into a random 2-cat strictifies ! It normalizes, so we can use bicolimit of normal pseudofunctors, but there is still the problem of composition.}

\begin{lemma}\label{sigma colim and strictification}
Suppose that $ F : I \rightarrow \mathcal{A}$ is a pseudofunctor admitting a strictification $p_F : F \Rightarrow \overline{F}$. Take $\Sigma$ a class of maps in $I$. Then the strictification induces an equivalence:
\[  \Sigma\underset{I}{\bicolim} \, F \simeq  \Sigma\underset{I}{\bicolim} \, \overline{F} \]
\end{lemma}

\begin{proof}
This is simply because the strictification $p_F$ is a pseudonatural equivalence and hence induces a pseudonatural equivalence at the level of the hom-categories 
\[  \ps[I,\mathcal{B}]_{\op\Sigma} [F,\Delta_B] \simeq \ps[I,\mathcal{B}]_{\op\Sigma} [\overline{F},\Delta_B]  \]
Hence if those pseudofunctors are representables, the representing objects are equivalent in an essentially unique way: hence the $\sigma$-bicolimits of $F$ and of its strictification are equivalent.
\end{proof}

As the theory of $\sigma$-limits is very recent and lax limits are not as well documented as pseudo-limits or strict 2-limits, it is worth giving a few lemmas to ensure they can be manipulated as expected.

\begin{lemma}
Right bi-adjoints preserve weighted $\sigma$-bilimits. Respectively, left bi-adjoints preserve weighted $\sigma$-bicolimits.
\end{lemma}

\begin{proof}
Let $ L \dashv R $ be a biadjunction with $R : \mathcal{A} \rightarrow \mathcal{B}$, a 2-category $I$ with $\Sigma$ a class of maps in $I$, $ F : I \rightarrow \mathcal{A}$ and $W : I \rightarrow \Cat$ a weight. Suppose that $ \mathcal{A}$ and $\mathcal{B}$ have weighted $\sigma$-limits. Then one has 
\begin{align*}
    \mathcal{B}[B, R (\underset{i \in I}{\Sigma \bilim^{W}} \, F(i))] &\simeq  \mathcal{A}[LB, \underset{i \in I}{\Sigma \bilim^{W}} \, F(i)] \\
    &\simeq [I,\Cat]_\Sigma \big[ W, \mathcal{A}[LB, F(i)] \big] \\
    &\simeq [I,\Cat]_\Sigma \big[ W, \mathcal{B}[B, RF(i)] \big] \\
    &\simeq \mathcal{B}[B, \underset{i \in I}{\Sigma \bilim^{W}} \, RF(i)]
\end{align*}
\end{proof}

\begin{division}[$\sigma$-bicolimits of categories as fractions of oplax-colimits]\label{sigma-colimit of categories}
It is worth recalling the computation of conical $\sigma$-bicolimits of categories according to \cite[Subsec. 2.5]{descotte2018sigma}. For a $\sigma$-pair $ (I,\Sigma)$ and a pseudofunctor $ F : I \rightarrow \Cat $, the $\sigma$-colimit is obtained as the localization of the oplax colimit at cocartesian lifts of $\Sigma$-arrows. If one defines
\[ \Sigma_{(F,\Sigma)} = \textbf{coCart}_F \cap \pi_F^{-1}(\Sigma)  \]
with $ \pi_F : \oplaxcolim_{i \in I} F(i) \rightarrow I$ the associated fibration (recall that $\oplaxcolim_{i \in I} F(i)$ is the classifying category of the Grothendieck construction of $ F$), one has the equation below. Moreover this $\sigma$-bicolimit has actually the universal property of a $\sigma$-colimit. 
\[   \underset{i \in I}{\Sigma\bicolim} \, F(i) \simeq \underset{i \in I}{\oplaxcolim}\,  F(i)  [\Sigma_{(F,\Sigma)}^{-1}].  \]
\end{division}

% \begin{remark}
% This calculation is given for a 2-functor; in the case of a $\Cat$-valued pseudofunctor, we can consider by \cref{strictification} its strictification, whose $\sigma$-bicolimit will be computed as described above, and then we know by \cref{sigma colim and strictification} that this category will be canonically equivalent to the bicolimit of the pseudofunctor, and can hence be used as such in practice. 
% \end{remark}

\begin{remark}[Equalization of parallel vertical pairs] \label{local faithfulness of colimit inclusions}
Here we should give an observation about when two morphisms in a member of the $\sigma$-bicolimit are identified in the $\sigma$-bicolimit itself after applying the inclusion to them: this will be of use later. The class $ \Sigma_{F,\Sigma}$ admits a \emph{left calculus of fractions} where a free arrow is created for each cospan of the form 
% https://q.uiver.app/?q=WzAsMyxbMCwwLCIoaSxhKSJdLFsyLDAsIihpJyxhJykiXSxbMSwxLCIoaixGKHMpKGEnKSkiXSxbMCwyLCIoZixcXHBoaSkiLDJdLFsxLDIsIihzLCAxX3tGKHMpKGEnKX0pIl1d
\[\begin{tikzcd}
	{(i,a)} && {(i',a')} \\
	& {(j,F(s)(a'))}
	\arrow["{(f,\phi)}"', from=1-1, to=2-2]
	\arrow["{(s, 1_{F(s)(a')})}", from=1-3, to=2-2]
\end{tikzcd}\]
with $s$ in $\Sigma$. Moreover one can characterize when two such spans are equivalent in the localization, see the left dual of proposition \cite{borceux1994handbook}[Proposition 5.2.4 (3)] which stands for a right calculus of fractions. In particular if we apply this formula to the case of a parallel pair of the specific form $ (1_i, \phi), (1_i, \phi') : (i,a) \rightrightarrows (i,a')$ (that is, coming from a parallel pair in $F(i)$), then saying they are equivalent in the localization amounts to saying there exists $ s : i \rightarrow i' $ in $\Sigma$ such that the corresponding cocartesian morphism equalizes them 
% https://q.uiver.app/?q=WzAsMyxbMCwwLCIoaSxhKSJdLFsxLDAsIihpLGEnKSJdLFszLDAsIihqLEYocykoYScpKSJdLFswLDEsIigxX2ksXFxwaGkpIiwwLHsib2Zmc2V0IjotMX1dLFswLDEsIigxX2ksXFxwaGknKSIsMix7Im9mZnNldCI6MX1dLFsxLDIsIihzLDFfe0YocykoYScpfSkiXV0=
\[\begin{tikzcd}
	{(i,a)} & {(i,a')} && {(j,F(s)(a'))}
	\arrow["{(1_i,\phi)}", shift left=1, from=1-1, to=1-2]
	\arrow["{(1_i,\phi')}"', shift right=1, from=1-1, to=1-2]
	\arrow["{(s,1_{F(s)(a')})}", from=1-2, to=1-4]
\end{tikzcd}\]
But this latter condition exactly means that $F(s)(\phi) = F(s)(\phi)'$ in $F(i')$. This observation will have useful consequences in the study of the specific case of \emph{$\sigma$-filtered $\sigma$-bicolimits} of categories.
\end{remark}

\begin{example} 
In $\Cat$, pseudolimits can be taken for bilimits, and are computed as follows. For $I$ a small 2-category and $ F : I \rightarrow \Cat$ a pseudofunctor, the conical pseudolimit $ \pslim \, F$ has as objects pairs $ ((A_i)_{i \in I}, (\alpha_{d})_{d \in I^2})$ with $ A_i \in F(i)$ and for each $ d : i \rightarrow j$ in $I$, $ \alpha_d : F(i) \simeq F(j)$ an isomorphism, such that moreover we have the cocycle identities 
\[  \alpha_{1_i} = 1_{A_i} \quad \alpha_{d_1d_2} = \alpha_{d_1}\alpha_{d_2} \]
They are the case of a $\sigma$-bicolimit where all arrows are in $\Sigma$. 
\end{example}

\subsection{BiKan extensions}
In this subsection we recall biKan extensions. We follow the treatment of \cite[Sec. 4]{descotte2018sigma}, originally inspired by \cite{nunes2016biadjoint}.

\begin{definition}[Left biKan extension]
Let $ F: \mathcal{A} \rightarrow \mathcal{B}$ be and $ G : \mathcal{A} \rightarrow \mathcal{C}$ two pseudofunctors. Then the \emph{left biKan extension of $F$ along $G$}\index{left biKan extension} is the following 2-cell in $2$-$\Cat$  with the universal property that
% https://q.uiver.app/?q=WzAsMyxbMCwwLCJcXG1hdGhjYWx7QX0iXSxbMSwwLCJcXG1hdGhjYWx7Qn0iXSxbMCwxLCJcXG1hdGhjYWx7Q30iXSxbMCwyLCJHIiwyXSxbMCwxLCJGIl0sWzIsMSwiXFxiaUxhbl9HIEYiLDJdLFs0LDUsIlxcemV0YSIsMix7ImN1cnZlIjoxLCJzaG9ydGVuIjp7InNvdXJjZSI6MjAsInRhcmdldCI6MjB9fV1d
\[\begin{tikzcd}
	{\mathcal{A}} & {\mathcal{B}} \\
	{\mathcal{C}}
	\arrow["G"', from=1-1, to=2-1]
	\arrow[""{name=0, anchor=center, inner sep=0}, "F", from=1-1, to=1-2]
	\arrow[""{name=1, anchor=center, inner sep=0}, "{\biLan_{G} F}"', from=2-1, to=1-2]
	\arrow[" \lambda", shorten <=3pt, shorten >=3pt, Rightarrow, from=1-1, to=1]
\end{tikzcd}\]
 \begin{itemize}
    \item for any other pseudonatural transformation $ \zeta : F \Rightarrow HG$ there exists an essentially unique pseudonatural transformation $ \xi : \biLan_{G} F \Rightarrow H$ such that we have a canonical invertible modification $ \zeta \simeq \xi * G \lambda$
    % https://q.uiver.app/?q=WzAsMyxbMCwwLCJcXG1hdGhjYWx7QX0iXSxbMiwwLCJcXG1hdGhjYWx7Qn0iXSxbMCwxLCJcXG1hdGhjYWx7Q30iXSxbMCwyLCJHIiwyXSxbMCwxLCJGIl0sWzIsMSwiSCIsMix7ImN1cnZlIjoyfV0sWzQsNSwiXFx6ZXRhIiwyLHsiY3VydmUiOjEsInNob3J0ZW4iOnsic291cmNlIjoyMCwidGFyZ2V0IjoyMH19XV0=
\[% https://q.uiver.app/?q=WzAsMyxbMCwwLCJcXG1hdGhjYWx7QX0iXSxbMiwwLCJcXG1hdGhjYWx7Qn0iXSxbMCwxLCJcXG1hdGhjYWx7Q30iXSxbMCwxLCJGIl0sWzIsMSwiSCIsMix7ImN1cnZlIjozfV0sWzAsMiwiRyIsMl0sWzIsMSwiXFxiaUxhbl9HIEYiLDEseyJsYWJlbF9wb3NpdGlvbiI6MzB9XSxbMyw2LCJcXGxhbWJkYSIsMix7Im9mZnNldCI6MywiY3VydmUiOjEsInNob3J0ZW4iOnsic291cmNlIjoyMCwidGFyZ2V0IjoxMH19XSxbNiw0LCJcXHhpIiwwLHsib2Zmc2V0IjotMywic2hvcnRlbiI6eyJzb3VyY2UiOjIwLCJ0YXJnZXQiOjIwfX1dXQ==
\begin{tikzcd}[sep=large]
	{\mathcal{A}} && {\mathcal{B}} \\
	{\mathcal{C}}
	\arrow[""{name=0, anchor=center, inner sep=0}, "F", from=1-1, to=1-3]
	\arrow[""{name=1, anchor=center, inner sep=0}, "H"', curve={height=18pt}, from=2-1, to=1-3]
	\arrow["G"', from=1-1, to=2-1]
	\arrow[""{name=2, anchor=center, inner sep=0}, "{\biLan_{G} F}"{description, pos=0.3}, from=2-1, to=1-3]
	\arrow["\lambda", shorten <=5pt, shorten >=5pt, Rightarrow, from=1-1, to=2]
	\arrow["\xi", shift left=3, shorten <=3pt, shorten >=3pt, Rightarrow, from=2, to=1]
\end{tikzcd} \simeq \begin{tikzcd}[sep=large]
	{\mathcal{A}} && {\mathcal{B}} \\
	{\mathcal{C}}
	\arrow["G"', from=1-1, to=2-1]
	\arrow[""{name=0, anchor=center, inner sep=0}, "F", from=1-1, to=1-3]
	\arrow[""{name=1, anchor=center, inner sep=1}, "H"', from=2-1, to=1-3]
	\arrow["\zeta", shorten <=5pt, shorten >=5pt, Rightarrow, from=1-1, to=1]
\end{tikzcd}\]
\item for any $ H : \mathcal{C} \rightarrow \mathcal{B}$ and any natural modification $ \phi : \zeta \Rrightarrow \zeta'$ in the homcategory $ \ps[\mathcal{A}, \mathcal{B}][F, HG]$, there is a unique modification $ \psi$ such that $ \phi = (\psi * G)*\lambda$ as depicted below
% https://q.uiver.app/?q=WzAsMyxbMCwwLCJcXG1hdGhjYWx7QX0iXSxbMiwwLCJcXG1hdGhjYWx7Qn0iXSxbMCwyLCJcXG1hdGhjYWx7Q30iXSxbMCwyLCJHIiwyXSxbMCwxLCJGIl0sWzIsMSwiXFxiaUxhbl9HIEYiLDEseyJsYWJlbF9wb3NpdGlvbiI6MjB9XSxbMiwxLCJIIiwyLHsiY3VydmUiOjR9XSxbNSw2LCJcXHhpIiwyLHsib2Zmc2V0IjoyLCJzaG9ydGVuIjp7InNvdXJjZSI6MjAsInRhcmdldCI6MjB9fV0sWzQsNSwiXFxsYW1iZGEiLDIseyJvZmZzZXQiOjMsImN1cnZlIjoxLCJzaG9ydGVuIjp7InNvdXJjZSI6MTAsInRhcmdldCI6MTB9fV0sWzUsNiwiXFx4aSciLDAseyJvZmZzZXQiOi01LCJzaG9ydGVuIjp7InNvdXJjZSI6MjAsInRhcmdldCI6MjB9fV0sWzcsOSwiXFxwc2kiLDIseyJzaG9ydGVuIjp7InNvdXJjZSI6MjAsInRhcmdldCI6MjB9fV1d
\[\begin{tikzcd}[sep=large]
	{\mathcal{A}} && {\mathcal{B}} \\
	\\
	{\mathcal{C}}
	\arrow["G"', from=1-1, to=3-1]
	\arrow[""{name=0, anchor=center, inner sep=0}, "F", from=1-1, to=1-3]
	\arrow[""{name=1, anchor=center, inner sep=0}, "{\biLan_{G} F}"{description, pos=0.2}, from=3-1, to=1-3]
	\arrow[""{name=2, anchor=center, inner sep=0}, "H"', curve={height=35pt}, from=3-1, to=1-3]
	\arrow[""{name=3, anchor=center, inner sep=0}, "\xi"', shift right=2, shorten <=4pt, shorten >=4pt, Rightarrow, from=1, to=2]
	\arrow["\lambda"', shorten <=4pt, shorten >=4pt, Rightarrow, from=1-1, to=1]
	\arrow[""{name=4, anchor=center, inner sep=0}, "{\xi'}", shift left=5, shorten <=4pt, shorten >=4pt, Rightarrow, from=1, to=2]
	\arrow["\psi"', shorten <=2pt, shorten >=2pt, from=3, to=4, triple]
\end{tikzcd} = % https://q.uiver.app/?q=WzAsMyxbMCwwLCJcXG1hdGhjYWx7QX0iXSxbMiwwLCJcXG1hdGhjYWx7Qn0iXSxbMCwyLCJcXG1hdGhjYWx7Q30iXSxbMCwyLCJHIiwyXSxbMCwxLCJGIl0sWzIsMSwiSCIsMix7ImN1cnZlIjo0fV0sWzQsNSwiXFx6ZXRhJyIsMCx7Im9mZnNldCI6LTMsImN1cnZlIjoxLCJzaG9ydGVuIjp7InNvdXJjZSI6MjAsInRhcmdldCI6MjB9fV0sWzQsNSwiXFx6ZXRhIiwyLHsib2Zmc2V0Ijo1LCJjdXJ2ZSI6MSwic2hvcnRlbiI6eyJzb3VyY2UiOjIwLCJ0YXJnZXQiOjIwfX1dLFs3LDYsIlxccGhpIiwyLHsic2hvcnRlbiI6eyJzb3VyY2UiOjIwLCJ0YXJnZXQiOjIwfX1dXQ==
\begin{tikzcd}[column sep=large]
	{\mathcal{A}} && {\mathcal{B}} \\
	\\
	{\mathcal{C}}
	\arrow["G"', from=1-1, to=3-1]
	\arrow[""{name=0, anchor=center, inner sep=0}, "F", from=1-1, to=1-3]
	\arrow[""{name=1, anchor=center, inner sep=0}, "H"', curve={height=24pt}, from=3-1, to=1-3]
	\arrow[""{name=2, anchor=center, inner sep=0}, "{\zeta'}", shift left=4,  shorten <=7pt, shorten >=7pt, Rightarrow, from=1-1, to=1]
	\arrow[""{name=3, anchor=center, inner sep=0}, "\zeta"', shift right=4,  shorten <=7pt, shorten >=7pt, Rightarrow, from=1-1, to=1]
	\arrow["\phi"', shorten <=3pt, shorten >=3pt, from=3, to=2, triple]
\end{tikzcd}\]

% https://q.uiver.app/?q=WzAsMyxbMCwwLCJcXG1hdGhjYWx7QX0iXSxbMywwLCJcXG1hdGhjYWx7Qn0iXSxbMCwyLCJcXG1hdGhjYWx7Q30iXSxbMCwyLCJHIiwyXSxbMCwxLCJGIl0sWzIsMSwiXFxiaUxhbl9HIEYiLDEseyJsYWJlbF9wb3NpdGlvbiI6MjB9XSxbMiwxLCIiLDEseyJjdXJ2ZSI6M31dLFs1LDYsIlxceGkiLDIseyJvZmZzZXQiOjUsInNob3J0ZW4iOnsic291cmNlIjoyMCwidGFyZ2V0IjoyMH19XSxbNCw2LCIiLDEseyJjdXJ2ZSI6MSwic2hvcnRlbiI6eyJzb3VyY2UiOjIwLCJ0YXJnZXQiOjIwfX1dLFs0LDUsIlxcbGFtYmRhIiwyLHsib2Zmc2V0Ijo1LCJjdXJ2ZSI6MSwic2hvcnRlbiI6eyJzb3VyY2UiOjIwLCJ0YXJnZXQiOjIwfX1dLFs1LDYsIiIsMSx7Im9mZnNldCI6Miwic2hvcnRlbiI6eyJzb3VyY2UiOjIwLCJ0YXJnZXQiOjIwfX1dLFs0LDYsIlxcemV0YSIsMCx7Im9mZnNldCI6LTUsImN1cnZlIjoxLCJzaG9ydGVuIjp7InNvdXJjZSI6MjAsInRhcmdldCI6MjB9fV0sWzcsMTAsIlxccHNpIiwyLHsic2hvcnRlbiI6eyJzb3VyY2UiOjIwLCJ0YXJnZXQiOjIwfX1dLFs4LDExLCJcXHBoaSIsMCx7InNob3J0ZW4iOnsic291cmNlIjoyMCwidGFyZ2V0IjoyMH19XV0=

\end{itemize}
In other words, we have for each $ H : \mathcal{C} \rightarrow \mathcal{B}$ an equivalence between homcategories
\[ \ps[\mathcal{C}, \mathcal{B}][\biLan_{G} F, H] \simeq \ps[\mathcal{A}, \mathcal{B}][F, HG]  \]
We can also define the \emph{right biKan extension}, denoted $\textup{biRan}_G F$, which has the expected dual property. 
\end{definition}

Moreover, whenever they exist, left biKan extensions can be computed as weighted bicolimits: this is essentially the content of \cite[4.1.5,4.1.6]{descotte2018sigma}:

\begin{proposition}\label{Weighted colimit expression of pointwise biLan}
When $ \mathcal{B}$ has small bicolimits, we can compute the left biKan extension for any $ C$
 in $\mathcal{C}$ as the bicolimit
 \[  \biLan_{G} F (C) = \bicolim^{\mathcal{N}_G(C)} F   \]
 %= \biend^{A \in \mathcal{A}} \mathcal{C}[G(A), C] \times F(A)
where $ \mathcal{N}_G : \mathcal{C} \rightarrow \ps[\mathcal{A}^{\op}, \Cat]$ sends $ C$ on the contravariant weight $ \mathcal{C}[G, C] : \mathcal{A}^{\op} \rightarrow \Cat$.  
\end{proposition}

\begin{division}[Cancellation rule for left biKan extensions] \label{cancellation rule}
Now, let $\mathcal{A}$ be a small $2$-category, and consider an extension problem as above. Another way to express the formula above is given by the following diagram,

\[\begin{tikzcd}
	{\mathcal{A}} && {\mathcal{B}} \\
	& {\mathcal{C}} \\
	& {\ps[\mathcal{A}^\op, \Cat]}
	\arrow["G"{description}, from=1-1, to=2-2]
	\arrow[""{name=0, anchor=center, inner sep=0}, "F", from=1-1, to=1-3]
	\arrow["{\text{biLan}_GF}"{description}, from=2-2, to=1-3]
	\arrow[""{name=1, anchor=center, inner sep=0}, "\hirayo"', curve={height=18pt}, from=1-1, to=3-2]
	\arrow["{\mathcal{C}[G,-]}", from=2-2, to=3-2]
	\arrow[""{name=2, anchor=center, inner sep=0}, "{\text{biLan}_{\hirayo} F}"', curve={height=18pt}, from=3-2, to=1-3]
	\arrow["\lambda"', shorten <=3pt, Rightarrow, from=0, to=2-2]
	\arrow["\simeq", shorten >=7pt, Rightarrow, draw=none, from=2-2, to=2]
	\arrow["{G*}", shorten <=7pt, Rightarrow, from=1, to=2-2]
\end{tikzcd}\]
where the leftmost 2-cell $ G*$ is the pseudonatural transformation induced by $G$ between homsets, while the rightmost 2-cell happens to be a pseudonatural equivalence. Indeed the $\biLan_\hirayo F$, when applied to a weight $W$ gives us the bicolimit along the weight, and thus we have the following equation \[\biLan_\hirayo F (W) \simeq \bicolim^{W} F,\]
which thus gives us the formula \[\biLan_G F \simeq \biLan_\hirayo F \circ \mathcal{C}[G,-]\]
This observation suggests that a vast class of functors admits indeed the computation of a left Kan extenstion. Indeed, if the functor $\mathcal{C}[G,-]$ lands in the full subcategory of weights which are small bicolimits of representables, we can still use the formula above and compute the biKan extension, because we assume $\mathcal{B}$ to have small weighted bicolimits.
%The following is a well known generalization of the 1-categorical denseness criterion:
\end{division}

\begin{remark}\label{cancellation rule for biRan}
We also have the dual formula for the computation of the right biKan extension 
\[\biRan_G F \simeq \biRan_{\katayo^{\op}} F \circ \mathcal{C}[-,G]\]
with $ \katayo : \mathcal{A}^{\op} \hookrightarrow \ps[ \mathcal{A}, \Cat]$ being the contravariant 2-dimensional Yoneda embedding.
\end{remark}

\begin{proposition}[On the existence of biKan extensions] \label{cr}
In the notation of \Cref{cancellation rule}, if $\mathcal{A}$ is a (possibly large) $2$-category, $\mathcal{B}$ has small weighted bicolimits and the functor $\mathcal{C}[G,-]$ lands in the full subcategory of weights which are small bicolimits of representables, the $\biLan_GF$ exists and is computed as in \Cref{cancellation rule}.
\end{proposition}
\begin{proof}
Follows from the discussion in \Cref{cancellation rule}.
\end{proof}

\begin{proposition}\label{bikan ext along ff}
Suppose in the above context that $ G$ is pseudofully faithful; then the canonical pseudonatural transformation $ \lambda : F \Rightarrow \biLan_{G} F G $ is a point-wise equivalence.   
\end{proposition}
\begin{proof}
\cite[Prop. 4.1.10]{descotte2018sigma}
\end{proof}

\subsection{Bi-adjunction and bireflectiveness}
It is well known that reflective subcategories inherit limits. Here we want to prove the corresponding result for bilimits in the context of bireflective sub-bicategories. 

\begin{proposition}\label{bireflective subbicategories have bilimits}
Let $R : \mathcal{A} \hookrightarrow \mathcal{B}$ be a pseudo-fully faithful pseudofunctor with a bi-adjoint $ L \dashv R$. Then if $ \mathcal{B}$ has bilimit, $\mathcal{A}$ is closed in $\mathcal{B}$ under bilimits.
\end{proposition}

This proposition is general abstract nonsense for which we did not find a direct reference in the literature: we leave it as an exercise for the reader (hint: categorify \cite{borceux1994handbook}[Proposition 3.5.3].)

\begin{remark}\label{bicoreflective subbicat have bicolimits}
By duality, we also have that bicoreflective sub-bicategories, that are, those $ L : \mathcal{A} \hookrightarrow \mathcal{B}$ that have a right bi-adjoint, inherit bicolimits. We are going to use this form later. \end{remark}

We also know that reflective sub-2-categories inherit bicolimits - though they might not be preserved by the forgetful functor:

\begin{proposition}\label{reflective sub2cat have bicolimit}
Let $R: \mathcal{A} \hookrightarrow \mathcal{B}$  be a pseudo-full and faithful pseudofunctor with a bi-adjoint $ L \dashv R$. If $\mathcal{B}$ has bicolimits, then $\mathcal{A}$ also has bicolimits which are computed as 
\[ \bicolim^{W}\, F \simeq L(\bicolim^{W} \, RF)  \]
\end{proposition}

\begin{proof}
For $ F : I \rightarrow \mathcal{A}$ and $ W : I^{\op} \rightarrow \Cat$, we have a pseudonatural equivalence in each $A$ of $\mathcal{A}$
\begin{align*}
    [I^{\op}, \Cat] \big[ W, \mathcal{A}[F,A]   \big] &\simeq  [I^{\op}, \Cat] \big[ W, \mathcal{B}[RF,R(A)]   \big] \\
    &\simeq \mathcal{B}[ \bicolim^{W} F, R(A) ] \\
    &\simeq \mathcal{A}[ L(\bicolim^{W} F), A ]
\end{align*}
\end{proof}

\section{Notions of $2$-dimensional filteredness}

This section deals with categorifications of the notion of filteredness: namely \emph{bifilteredness} by \cite{kennison1992fundamental} and \cite{descotte2020theory} and \emph{$ \sigma$-filteredness} by \cite{descotte2018sigma}, together with suited notion of 2-dimensional cofinality. Although those notions were developed separately, we show here they are actually almost the same in the sense that the seemingly more complicated $\sigma$-filtered $\sigma$-bicolimits automatically collapse on ordinary bifiltered bicolimits, which was seemingly unnoticed before. 

\subsection{Bifilteredness and $\sigma$-filteredness} \label{sectionsigma}

%{\color{red} In the definition below, should it be $\sigma$-filtered or $\sigma$-filtered ?! Also, $\mathcal{A}$ should be replaced by $\mathcal{I}$.}
\begin{definition}[$\sigma$-filtered pairs, \cite{descotte2018sigma} $\sigma$-filtered]
Let $ I$ be a 2-category and $ \Sigma $ a class of maps in $I$. Then $ I$ is said to be \emph{$\sigma$-filtered relatively to $\Sigma$}, or that $(I,\Sigma)$ is a \emph{$\sigma$-filtered pair} if it satisfies the following conditions: \begin{enumerate}
    \item for any $i, \; i'$ in $I$ there exists a span in $\Sigma$ % https://q.uiver.app/?q=WzAsMyxbMCwwLCJBIl0sWzAsMiwiQSciXSxbMSwxLCJBJyciXSxbMCwyLCJhIl0sWzEsMiwiYSciLDJdXQ==
\[\begin{tikzcd}[row sep=tiny]
	i \\
	& {i''} \\
	{i'}
	\arrow["s \in \Sigma", from=1-1, to=2-2]
	\arrow["{s' \in \Sigma}"', from=3-1, to=2-2]
\end{tikzcd}\]
   \item for a parallel pair $ d, s : i \rightrightarrows i'$ with $ s$ in $ \Sigma$, there exists $ t : i' \rightarrow i''$ in $\Sigma$ and a  2-cell
% https://q.uiver.app/?q=WzAsNCxbMCwxLCJpIl0sWzEsMCwiaSciXSxbMSwyLCJpJyJdLFsyLDEsImknJyJdLFswLDEsImQiXSxbMSwzLCJ0XFxpbiBcXFNpZ21hIl0sWzAsMiwicyBcXGluIFxcU2lnbWEiLDJdLFsyLDMsInRcXGluIFxcU2lnbWEiLDJdLFsxLDIsIlxccGhpIiwyLHsic2hvcnRlbiI6eyJzb3VyY2UiOjMwLCJ0YXJnZXQiOjMwfSwibGV2ZWwiOjJ9XV0=
\[\begin{tikzcd}[row sep=small]
	& {i'} \\
	i && {i''} \\
	& {i'}
	\arrow["d", from=2-1, to=1-2]
	\arrow["{t\in \Sigma}", from=1-2, to=2-3]
	\arrow["{s \in \Sigma}"', from=2-1, to=3-2]
	\arrow["{t\in \Sigma}"', from=3-2, to=2-3]
	\arrow["\alpha"', shorten <=10pt, shorten >=10pt, Rightarrow, from=1-2, to=3-2]
\end{tikzcd}\]
Moreover the 2-cell $ \alpha$ can be chosen to be invertible whenever $d$ is also in $\Sigma$;
   \item for a pair of parallel 2-cells with codomain in $\Sigma$
   % https://q.uiver.app/?q=WzAsMixbMCwwLCJBIl0sWzIsMCwiQScnIl0sWzAsMSwiZiIsMCx7ImN1cnZlIjotMn1dLFswLDEsImYnIiwyLHsiY3VydmUiOjJ9XSxbMiwzLCIiLDEseyJzaG9ydGVuIjp7InNvdXJjZSI6MjAsInRhcmdldCI6MjB9LCJzdHlsZSI6eyJib2R5Ijp7Im5hbWUiOiJub25lIn0sImhlYWQiOnsibmFtZSI6Im5vbmUifX19XSxbMiwzLCJcXGFscGhhIiwyLHsib2Zmc2V0Ijo0LCJzaG9ydGVuIjp7InNvdXJjZSI6MjAsInRhcmdldCI6MjB9fV0sWzIsMywiXFxhbHBoYSciLDAseyJvZmZzZXQiOi00LCJzaG9ydGVuIjp7InNvdXJjZSI6MjAsInRhcmdldCI6MjB9fV1d
\[\begin{tikzcd}
	i && {i'}
	\arrow[""{name=0, anchor=center, inner sep=0}, "d", curve={height=-12pt}, from=1-1, to=1-3]
	\arrow[""{name=1, anchor=center, inner sep=0}, "{s \in \Sigma}"', curve={height=12pt}, from=1-1, to=1-3]
	\arrow[Rightarrow, draw=none, from=0, to=1]
	\arrow["\alpha"', shift right=3, shorten <=3pt, shorten >=3pt, Rightarrow, from=0, to=1]
	\arrow["{\alpha'}", shift left=3, shorten <=3pt, shorten >=3pt, Rightarrow, from=0, to=1]
\end{tikzcd}\]
there exists $ f : i' \rightarrow i''$ also in $\Sigma$ such that $ f*\alpha = f *\alpha'$.
\end{enumerate}
\end{definition}

\begin{comment}
The statement in the second condition that the inserted 2-cell can be chosen invertible in the case where both upper and lower 1-cell are in $\Sigma$ follows from the three conditions altogether: if $ s,s' : i \rightarrow j$ are a parallel pair in $\Sigma$, then one can successively find $ t : j \rightarrow k,t' : j \rightarrow k'$ in $\Sigma$ inserting respectively $ \alpha : ts \Rightarrow ts'$ and $ \alpha' : t's' \Rightarrow t's$; then one can chose a span $ u : k \rightarrow k' $
\end{comment}

In particular let us emphasize the following specific case into an autonomous definition:

\begin{definition}
A 2-category $ I$ is said to be \emph{bifiltered} if it satisfies the following conditions: \begin{enumerate}
    \item for any $i, \; i'$ in $I$ there exists a span % https://q.uiver.app/?q=WzAsMyxbMCwwLCJBIl0sWzAsMiwiQSciXSxbMSwxLCJBJyciXSxbMCwyLCJhIl0sWzEsMiwiYSciLDJdXQ==
\[\begin{tikzcd}[row sep=tiny]
	i \\
	& {i''} \\
	{i'}
	\arrow["d", from=1-1, to=2-2]
	\arrow["{d'}"', from=3-1, to=2-2]
\end{tikzcd}\]
   \item for a parallel pair $ d, d' : i \rightrightarrows i'$, there exists $ f : i' \rightarrow i''$ together with an invertible 2-cell
% https://q.uiver.app/?q=WzAsNCxbMCwxLCJpIl0sWzEsMCwiaSciXSxbMSwyLCJpJyJdLFsyLDEsImknJyJdLFswLDEsImQiXSxbMSwzLCJkJyciXSxbMCwyLCJkJyIsMl0sWzIsMywiZCcnIiwyXSxbMSwyLCJcXGFscGhhIFxcYXRvcCBcXHNpbWVxIiwxLHsic2hvcnRlbiI6eyJzb3VyY2UiOjMwLCJ0YXJnZXQiOjMwfSwibGV2ZWwiOjIsInN0eWxlIjp7ImJvZHkiOnsibmFtZSI6Im5vbmUifSwiaGVhZCI6eyJuYW1lIjoibm9uZSJ9fX1dXQ==
\[\begin{tikzcd}[row sep=small]
	& {i'} \\
	i && {i''} \\
	& {i'}
	\arrow["d", from=2-1, to=1-2]
	\arrow["{d''}", from=1-2, to=2-3]
	\arrow["{d'}"', from=2-1, to=3-2]
	\arrow["{d''}"', from=3-2, to=2-3]
	\arrow["{\alpha \atop \simeq}"{description}, Rightarrow, draw=none, from=1-2, to=3-2]
\end{tikzcd}\]
   \item for a pair of parallel 2-cells
   % https://q.uiver.app/?q=WzAsMixbMCwwLCJBIl0sWzIsMCwiQScnIl0sWzAsMSwiZiIsMCx7ImN1cnZlIjotMn1dLFswLDEsImYnIiwyLHsiY3VydmUiOjJ9XSxbMiwzLCIiLDEseyJzaG9ydGVuIjp7InNvdXJjZSI6MjAsInRhcmdldCI6MjB9LCJzdHlsZSI6eyJib2R5Ijp7Im5hbWUiOiJub25lIn0sImhlYWQiOnsibmFtZSI6Im5vbmUifX19XSxbMiwzLCJcXGFscGhhIiwyLHsib2Zmc2V0Ijo0LCJzaG9ydGVuIjp7InNvdXJjZSI6MjAsInRhcmdldCI6MjB9fV0sWzIsMywiXFxhbHBoYSciLDAseyJvZmZzZXQiOi00LCJzaG9ydGVuIjp7InNvdXJjZSI6MjAsInRhcmdldCI6MjB9fV1d
\[\begin{tikzcd}
	i && {i'}
	\arrow[""{name=0, anchor=center, inner sep=0}, "d", curve={height=-12pt}, from=1-1, to=1-3]
	\arrow[""{name=1, anchor=center, inner sep=0}, "{d'}"', curve={height=12pt}, from=1-1, to=1-3]
	\arrow[Rightarrow, draw=none, from=0, to=1]
	\arrow["\alpha"', shift right=3, shorten <=3pt, shorten >=3pt, Rightarrow, from=0, to=1]
	\arrow["{\alpha'}", shift left=3, shorten <=3pt, shorten >=3pt, Rightarrow, from=0, to=1]
\end{tikzcd}\]
there exists $ f : i' \rightarrow i''$ such that $ f*\alpha = f *\alpha'$.
\end{enumerate}
\end{definition}

\begin{remark}
Of course a $\sigma$-filtered pair $ (I,\Sigma)$ is bifiltered exactly when $\Sigma$ contains all arrows. 
\end{remark}

\begin{remark}[From finite to infinite filteredness] \label{infinite filteredness}
Traditionally there is a gradient of notions of filteredness, indexed by a regular cardinals $\lambda$, so that the definition above should be called $\omega,\sigma$-filteredness. In this paper, we shall mainly stick to this \textit{finite} notion, but in a couple of occurrences, we will need its infinite version. Thus, let us say that, in the notation of the definition above, $I$ is said to be $\lambda, \sigma$-filtered if the condition (1) of the definition holds for any $\lambda$-small family of objects in $I$.
\end{remark}

\begin{remark}[Trivial cases] \label{trivial}
Some reflections on the notion of $\sigma$-filteredness are needed. Let us analyze the trivial situations.
\begin{itemize}
    \item $\Sigma$ cannot be empty, unless $I$ is biequivalent to the terminal $2$-category;
    \item if the pair $(I, \Sigma)$ is $\sigma$-filtered, then the underlying category $I_0$ of $I$ is directed;
    \item a filtered $1$-category can be equipped with a locally discrete $2$-dimensional structure such that the resulting $2$-category is $\sigma$-filtered for every choice of $\Sigma$.
\end{itemize}
\end{remark}

%The following useful way to see $\sigma$-filteredness is \cite[Proposition 3.1.5]{descotte2018sigma}:

\begin{proposition}[{\cite[Proposition 3.1.5]{descotte2018sigma}}]\label{sigmacone in sigmafiltered}
A $2$-category $I$ is $\sigma$-filtered with respect to $\Sigma$ if and only if any finite $2$-subcategory admits a $\sigma$-cone above it with arrows in $\Sigma$.
\end{proposition}

\begin{division}[A concrete description of $\sigma$-bicolimits of in $\Cat$, {\cite[Def 2.1]{descotte2016exactness}}]\label{sigma-filtered sigma-colimit of categories} Here we should give a few words on $\sigma$-filtered $\sigma$-colimit of categories, which is a generalization of the construction of pseudocolimits from oplax colimits. The localization formula of \Cref{sigma-colimit of categories} exhibits the $\sigma$-colimit as a category of fraction of the oplax colimit at the cartesian lifts of $\Sigma$-arrows, which can be shown to enjoy a right calculus of fractions. From \cite[Definition 2.1]{descotte2016exactness}, we can give the following more concrete description of $ {\Sigma\bicolim}_{i \in I} \, F(i) $ when $(I,\Sigma)$ is a $\sigma$-filtered pair: its objects are pairs $ (i, a)$ with $ a $ an object of $F(i)$, and a morphism $(i_1, a_1) \rightarrow (i_2, a_2)$ is the data of a span $ s : i_1 \rightarrow i_3$, $ d : i_2 \rightarrow i_3$ in $I$ with $s$ in $\Sigma$ and a morphism $ \phi : F(s)(a_1) \Rightarrow F(d)(a_2)$, which can be visualized in the following diagram
% https://q.uiver.app/?q=WzAsNCxbMCwwLCJGKGlfMSkiXSxbMCwyLCJGKGlfMikiXSxbMywxLCJcXHVuZGVyc2V0e2kgXFxpbiBJfXtcXHNpZ21hX1xcU2lnbWFcXGJpY29saW19IFxcLCBGKGkpIl0sWzEsMSwiRihpXzMpIl0sWzEsMiwicV97aV8yfSIsMix7ImN1cnZlIjoyfV0sWzAsMiwicV97aV8xfSIsMCx7ImN1cnZlIjotMn1dLFswLDMsIkYocykiLDJdLFsxLDMsIkYoZCkiXSxbMywyLCJxX3tpXzN9IiwxXSxbNSwzLCJxX3MgXFxhdG9wIFxcc2ltZXEiLDAseyJzaG9ydGVuIjp7InNvdXJjZSI6MjB9LCJzdHlsZSI6eyJib2R5Ijp7Im5hbWUiOiJub25lIn0sImhlYWQiOnsibmFtZSI6Im5vbmUifX19XSxbMyw0LCJxX2QiLDAseyJzaG9ydGVuIjp7InRhcmdldCI6MjB9fV1d
% https://q.uiver.app/?q=WzAsNSxbMSwwLCJGKGlfMSkiXSxbMSwyLCJGKGlfMikiXSxbNCwxLCJcXHVuZGVyc2V0e2kgXFxpbiBJfXtcXHNpZ21hX1xcU2lnbWFcXGJpY29saW19IFxcLCBGKGkpIl0sWzIsMSwiRihpXzMpIl0sWzAsMSwiKiJdLFsxLDIsInFfe2lfMn0iLDIseyJjdXJ2ZSI6Mn1dLFswLDIsInFfe2lfMX0iLDAseyJjdXJ2ZSI6LTJ9XSxbMCwzLCJGKHMpIiwxXSxbMSwzLCJGKGQpIiwxXSxbMywyLCJxX3tpXzN9IiwxXSxbNCwwLCJhXzEiXSxbNCwxLCJhXzIiLDJdLFswLDEsIlxccGhpIiwyLHsic2hvcnRlbiI6eyJzb3VyY2UiOjMwLCJ0YXJnZXQiOjMwfSwibGV2ZWwiOjJ9XSxbNiwzLCJxX3MgXFxhdG9wIFxcc2ltZXEiLDAseyJzaG9ydGVuIjp7InNvdXJjZSI6MjB9LCJzdHlsZSI6eyJib2R5Ijp7Im5hbWUiOiJub25lIn0sImhlYWQiOnsibmFtZSI6Im5vbmUifX19XSxbMyw1LCJxX2QiLDAseyJzaG9ydGVuIjp7InRhcmdldCI6MjB9fV1d
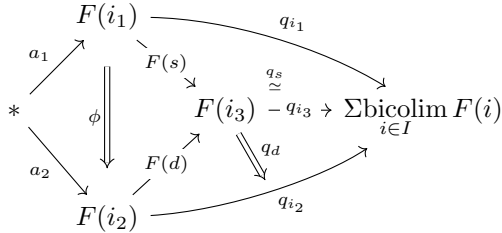
\[\begin{tikzcd}[column sep=small]
	& {F(i_1)} \\
	{*} && {F(i_3)} && {\underset{i \in I}{\Sigma\bicolim} \, F(i)} \\
	& {F(i_2)}
	\arrow[""{name=0, anchor=center, inner sep=0}, "{q_{i_2}}"', curve={height=12pt}, from=3-2, to=2-5]
	\arrow[""{name=1, anchor=center, inner sep=0}, "{q_{i_1}}", curve={height=-12pt}, from=1-2, to=2-5]
	\arrow["{F(s)}"{description}, from=1-2, to=2-3]
	\arrow["{F(d)}"{description}, from=3-2, to=2-3]
	\arrow["{q_{i_3}}"{description}, from=2-3, to=2-5]
	\arrow["{a_1}", from=2-1, to=1-2]
	\arrow["{a_2}"', from=2-1, to=3-2]
	\arrow["\phi"', shorten <=10pt, shorten >=10pt, Rightarrow, from=1-2, to=3-2]
	\arrow["{q_s \atop \simeq}", Rightarrow, draw=none, from=1, to=2-3]
	\arrow["{q_d}", shorten >=4pt, Rightarrow, from=2-3, to=0]
\end{tikzcd}\]
In particular, we end with a composite arrow in the $\sigma$-bicolimit
% https://q.uiver.app/?q=WzAsMyxbMCwwLCJxX3tpXzF9RihzKShhXzEpIl0sWzEsMCwicV97aV8zfUYoZCkoYV8yKSJdLFsyLDAsInFfe2lfMn0oYV8yKSJdLFswLDEsInFfe2lfM30oXFxwaGkpIl0sWzEsMiwicV9kIl1d
\[\begin{tikzcd}
	{q_{i_1}F(s)(a_1)} & {q_{i_3}F(d)(a_2)} & {q_{i_2}(a_2)}
	\arrow["{q_{i_3}(\phi)}", from=1-1, to=1-2]
	\arrow["{(q_d)_a}", from=1-2, to=1-3]
\end{tikzcd}\]
which corresponds to the \emph{(vertical, cocartesian)} factorization in the oplax colimit $ \int F$.\\

Two 1-cells $ (s,d,j, \phi),(s',d',j',\phi') : (i_1,a_1) \rightrightarrows (i_2,a_2) $ of the oplax colimit are identified in the $\sigma$-bicolimit if one has the following equality of 2-cells (this relation being called \emph{homotopies of premorphisms} in \cite{descotte2016exactness}): 
% https://q.uiver.app/?q=WzAsNSxbMSwwLCJGKGlfMSkiXSxbMSwyLCJGKGlfMikiXSxbNCwxLCJcXHVuZGVyc2V0e2kgXFxpbiBJfXtcXHNpZ21hX1xcU2lnbWFcXGJpY29saW19IFxcLCBGKGkpIl0sWzIsMSwiRihqKSJdLFswLDEsIioiXSxbMSwyLCJxX3tpXzJ9IiwyLHsiY3VydmUiOjJ9XSxbMCwyLCJxX3tpXzF9IiwwLHsiY3VydmUiOi0yfV0sWzAsMywiRihzKSIsMV0sWzEsMywiRihkKSIsMV0sWzMsMiwicV97an0iLDFdLFs0LDAsImFfMSJdLFs0LDEsImFfMiIsMl0sWzAsMSwiXFxwaGkiLDIseyJzaG9ydGVuIjp7InNvdXJjZSI6MzAsInRhcmdldCI6MzB9LCJsZXZlbCI6Mn1dLFs2LDMsInFfcyBcXGF0b3AgXFxzaW1lcSIsMCx7InNob3J0ZW4iOnsic291cmNlIjoyMH0sInN0eWxlIjp7ImJvZHkiOnsibmFtZSI6Im5vbmUifSwiaGVhZCI6eyJuYW1lIjoibm9uZSJ9fX1dLFszLDUsInFfZCIsMCx7InNob3J0ZW4iOnsidGFyZ2V0IjoyMH19XV0=
\[\begin{tikzcd}[column sep=small]
	& {F(i_1)} \\
	{*} && {F(j)} && {\underset{i \in I}{\Sigma\bicolim} \, F(i)} \\
	& {F(i_2)}
	\arrow[""{name=0, anchor=center, inner sep=0}, "{q_{i_2}}"', curve={height=12pt}, from=3-2, to=2-5]
	\arrow[""{name=1, anchor=center, inner sep=0}, "{q_{i_1}}", curve={height=-12pt}, from=1-2, to=2-5]
	\arrow["{F(s)}"{description}, from=1-2, to=2-3]
	\arrow["{F(d)}"{description}, from=3-2, to=2-3]
	\arrow["{q_{j}}"{description}, from=2-3, to=2-5]
	\arrow["{a_1}", from=2-1, to=1-2]
	\arrow["{a_2}"', from=2-1, to=3-2]
	\arrow["\phi"', shorten <=10pt, shorten >=10pt, Rightarrow, from=1-2, to=3-2]
	\arrow["{q_s \atop \simeq}", Rightarrow, draw=none, from=1, to=2-3]
	\arrow["{q_d}", shorten >=4pt, Rightarrow, from=2-3, to=0]
\end{tikzcd} = \begin{tikzcd}[column sep=small]
	& {F(i_1)} \\
	{*} && {F(j')} && {\underset{i \in I}{\Sigma\bicolim} \, F(i)} \\
	& {F(i_2)}
	\arrow[""{name=0, anchor=center, inner sep=0}, "{q_{i_2}}"', curve={height=12pt}, from=3-2, to=2-5]
	\arrow[""{name=1, anchor=center, inner sep=0}, "{q_{i_1}}", curve={height=-12pt}, from=1-2, to=2-5]
	\arrow["{F(s')}"{description}, from=1-2, to=2-3]
	\arrow["{F(d')}"{description}, from=3-2, to=2-3]
	\arrow["{q_{j'}}"{description}, from=2-3, to=2-5]
	\arrow["{a_1}", from=2-1, to=1-2]
	\arrow["{a_2}"', from=2-1, to=3-2]
	\arrow["\phi'"', shorten <=10pt, shorten >=10pt, Rightarrow, from=1-2, to=3-2]
	\arrow["{q_{s'} \atop \simeq}", Rightarrow, draw=none, from=1, to=2-3]
	\arrow["{q_{d'}}", shorten >=4pt, Rightarrow, from=2-3, to=0]
\end{tikzcd} \]
\end{division}

\begin{remark}
In particular, if $I$ is a bifiltered 2-category, then a morphism in the bicolimit $ \bicolim_{I} F$ is of the form 
\[\begin{tikzcd}[column sep=small]
	& {F(i_1)} \\
	{*} && {F(i_3)} && {\underset{i \in I}{\bicolim} \, F(i)} \\
	& {F(i_2)}
	\arrow[""{name=0, anchor=center, inner sep=0}, "{q_{i_2}}"', curve={height=12pt}, from=3-2, to=2-5]
	\arrow[""{name=1, anchor=center, inner sep=0}, "{q_{i_1}}", curve={height=-12pt}, from=1-2, to=2-5]
	\arrow["{F(d)}"{description}, from=1-2, to=2-3]
	\arrow["{F(d')}"{description}, from=3-2, to=2-3]
	\arrow["{q_{i_3}}"{description}, from=2-3, to=2-5]
	\arrow["{a_1}", from=2-1, to=1-2]
	\arrow["{a_2}"', from=2-1, to=3-2]
	\arrow["\phi"', shorten <=10pt, shorten >=10pt, Rightarrow, from=1-2, to=3-2]
	\arrow["{q_d \atop \simeq}", Rightarrow, draw=none, from=1, to=2-3]
	\arrow["{q_{d'} \atop \simeq}", shorten >=4pt, Rightarrow, draw=none, from=2-3, to=0]
\end{tikzcd}\]
with both the upper and lower inclusion 2-cells invertible, so that the resulting 2-cell $ (i_1,a_1) \Rightarrow (i_2, a_2)$ is obtained directly (up to invertible 2-cell) as the whiskering of $ \phi$ along $ q_{i_3}$ - differently from the general $ \sigma$-filtered case where one has to paste this whiskering further with a non invertible lower inclusion 2-cell.
\end{remark}

\begin{remark}[$\sigma$-filtered diagrams are not too many]
$\sigma$-filteredness puts actually some constrain on the relation between the shape of the category and the laxness of the 2-cells in the $\sigma$-colimit inclusion. While any oplax colimit is a case of $\sigma$-colimit with trivial choice of $\Sigma$ (that is, containing only isomorphisms), being $\sigma$-filtered for such a choice of $\Sigma$ trivializes the underlying 2-category as any two objects must be isomorphic. This prevents $\sigma$-filtered $\sigma$-bicolimit to become too ``wild".  
\end{remark}

%The notion of $\sigma$-filteredness is also tightly related to the notion of $\sigma$-cofinality, which is introduced in \cite[Definition 3.3.1]{descotte2018sigma}.

The following lemma, though being an easy consequence of the axioms of $\sigma$-filteredness, actually simplifies a lot the whole theory:

\begin{lemma}[Triangle lemma]\label{triangle lemma}
Let $ (I,\Sigma)$ be a $\sigma$-filtered pair. Then any arrow $d : i \rightarrow i'$ in $I$ can be completed into a triangle as below with $s,s'$ in $\Sigma$:% https://q.uiver.app/?q=WzAsMyxbMCwxLCJpIl0sWzIsMSwiaScnIl0sWzEsMCwiaSciXSxbMCwxLCJzIiwyXSxbMCwyLCJkIl0sWzIsMSwicyciXSxbMiwzLCJcXHBoaSIsMCx7InNob3J0ZW4iOnsidGFyZ2V0IjoyMH19XV0=
\[\begin{tikzcd}
	& {i'} \\
	i && {j}
	\arrow[""{name=0, anchor=center, inner sep=0}, "s"', from=2-1, to=2-3]
	\arrow["d", from=2-1, to=1-2]
	\arrow["{s'}", from=1-2, to=2-3]
	\arrow["\phi", shorten >=3pt, Rightarrow, from=1-2, to=0]
\end{tikzcd}\]
\end{lemma}

\begin{proof}
First choose a span $ t : i \rightarrow i''$, $ t' : i' \rightarrow i''$ in $\Sigma$ provided by the first axiom; one gets a parallel pair $ t'd, t : i \rightrightarrows i''$ in $I $ whose bottom arrow is in $\Sigma$. Hence there exists $ t'' : i'' \rightarrow j$ in $\Sigma$ together with the desired 2-cell $ \phi : t''t'd \Rightarrow t''t$ and both $t''t$ and $ t''t'$ are in $\Sigma$. 
\end{proof}

\begin{division}
A consequence of \Cref{triangle lemma} is that the non-invertible transition 2-cells in the $\sigma$-bicolimiting cocone under a $\sigma$-filtered diagram can always be ``swallowed" by a member of the diagram: for any $ d : i \rightarrow i'$, any choice of $s,s', \phi$ as above provides the following decomposition 
% https://q.uiver.app/?q=WzAsMyxbMCwwLCJGKGkpIl0sWzIsMCwiRihpJykiXSxbMSwxLCJ7XFxTaWdtYSBcXGJpY29saW19IEYiXSxbMCwxLCJGKGQpIl0sWzAsMiwicV9pIiwyXSxbMSwyLCJxX3tpJ30iXSxbNSw0LCJcXHRoZXRhX2QiLDIseyJzaG9ydGVuIjp7InNvdXJjZSI6MjAsInRhcmdldCI6MjB9fV1d
\[\begin{tikzcd}[column sep=small, row sep=large]
	{F(i)} && {F(i')} \\
	& {\underset{I}{\Sigma \bicolim}\; F}
	\arrow["{F(d)}", from=1-1, to=1-3]
	\arrow[""{name=0, anchor=center, inner sep=0}, "{q_i}"', from=1-1, to=2-2]
	\arrow[""{name=1, anchor=center, inner sep=0}, "{q_{i'}}", from=1-3, to=2-2]
	\arrow["{\theta_d}"', shorten <=9pt, shorten >=9pt, Rightarrow, from=1, to=0]
\end{tikzcd} = % https://q.uiver.app/?q=WzAsNCxbMCwwLCJGKGkpIl0sWzIsMCwiRihpJykiXSxbMSwyLCJ7XFxTaWdtYSBcXGJpY29saW19IEYiXSxbMSwxLCJGKGopIl0sWzAsMSwiRihkKSJdLFswLDIsInFfaSIsMix7ImN1cnZlIjozfV0sWzEsMiwicV97aSd9IiwwLHsiY3VydmUiOi0zfV0sWzAsMywiRihzKSIsMl0sWzEsMywiRihzJykiXSxbMywyXSxbOCw3LCJGKFxccGhpKSIsMix7InNob3J0ZW4iOnsic291cmNlIjoyMCwidGFyZ2V0IjoyMH19XSxbMyw1LCJcXHRoZXRhX3MgXFxhdG9wIFxcc2ltZXEiLDAseyJzaG9ydGVuIjp7InRhcmdldCI6MjB9LCJzdHlsZSI6eyJib2R5Ijp7Im5hbWUiOiJub25lIn0sImhlYWQiOnsibmFtZSI6Im5vbmUifX19XSxbMyw2LCJcXHRoZXRhX3tzJ30gXFxhdG9wIFxcc2ltZXEiLDIseyJzaG9ydGVuIjp7InRhcmdldCI6MjB9LCJzdHlsZSI6eyJib2R5Ijp7Im5hbWUiOiJub25lIn0sImhlYWQiOnsibmFtZSI6Im5vbmUifX19XV0=
\begin{tikzcd}[column sep=small, row sep=large]
	{F(i)} && {F(i')} \\
	& {F(j)} \\
	& {\underset{I}{\Sigma \bicolim}\; F}
	\arrow["{F(d)}", from=1-1, to=1-3]
	\arrow[""{name=0, anchor=center, inner sep=0}, "{q_i}"', curve={height=18pt}, from=1-1, to=3-2]
	\arrow[""{name=1, anchor=center, inner sep=0}, "{q_{i'}}", curve={height=-18pt}, from=1-3, to=3-2]
	\arrow[""{name=2, anchor=center, inner sep=0}, "{F(s)}"', from=1-1, to=2-2]
	\arrow[""{name=3, anchor=center, inner sep=0}, "{F(s')}", from=1-3, to=2-2]
	\arrow[from=2-2, to=3-2]
	\arrow["{F(\phi)}"', shorten <=9pt, shorten >=9pt, Rightarrow, from=3, to=2]
	\arrow["{\theta_s \atop \simeq}", Rightarrow, draw=none, from=2-2, to=0]
	\arrow["{\theta_{s'} \atop \simeq}"', Rightarrow, draw=none, from=2-2, to=1]
\end{tikzcd}\]
Hence any non invertible transition 2-cell in the cocone can be replaced by a choice of a transition 2-cell in the diagram together with invertible transition 2-cells.  
\end{division}

\begin{division}[Morphisms in the $\sigma$-bicolimit comes from morphisms of the bicolimit]\label{morphisms of sigma-colimit are morphisms of the bicolimit}
In the case of a $\sigma$-filtered $\sigma$-colimit in $\Cat$, this means that cocartesian morphisms in the oplax colimit $ \int F$ acquire vertical mates when passing to the $\sigma$-bicolimit. Concretely, this means that the representation of morphisms in $ \Sigma \bicolim_{I} F$ as given in \Cref{sigma-filtered sigma-colimit of categories} can be simplified as if we worked in an ordinarily bifiltered bicolimit: for a morphism $ (i_1, a_1) \rightarrow (i_2,a_2)$ in $\Sigma \bicolim_{I} F$ represented by a triple $ s : i_1 \rightarrow i_3$, $ d : i_2 \rightarrow i_3$ and $ \phi : F(s)a_1 \Rightarrow F(d)a_2$ as above, the lower transition 2-cell $ \theta_d$, which is by itself not invertible, can be replaced thanks to some choice $ s_2 : i_2 \rightarrow i_4$, $ s_3 : i_3 \rightarrow i_4$ in $ \Sigma$ together with $ \psi : s_3 d \Rightarrow s_2$, which provides a decomposition of the diagram representing the morphism in the colimit as 
% https://q.uiver.app/?q=WzAsNyxbMSwwLCJGKGlfMSkiXSxbMSwyLCJGKGlfMikiXSxbNCwyLCJcXHVuZGVyc2V0e2kgXFxpbiBJfXtcXHNpZ21hX1xcU2lnbWFcXGJpY29saW19IFxcLCBGKGkpIl0sWzIsMSwiRihpXzMpIl0sWzAsMSwiKiJdLFszLDIsIkYoaV80KSJdLFsyLDMsIkYoaV8yKSJdLFswLDIsInFfe2lfMX0iLDAseyJjdXJ2ZSI6LTN9XSxbMCwzLCJGKHMpIiwxXSxbMSwzLCJGKGQpIiwxXSxbMywyLCJxX3tpXzN9IiwxLHsiY3VydmUiOi0xfV0sWzQsMCwiYV8xIl0sWzQsMSwiYV8yIiwyXSxbMyw1LCJGKHNfMykiLDFdLFs1LDIsInFfe2lfNH0iXSxbNiwyLCIiLDEseyJjdXJ2ZSI6Mn1dLFs2LDUsIkYoc18yKSIsMV0sWzEsNiwiIiwxLHsibGV2ZWwiOjIsInN0eWxlIjp7ImhlYWQiOnsibmFtZSI6Im5vbmUifX19XSxbMyw2LCJGKFxccHNpKSIsMSx7InNob3J0ZW4iOnsic291cmNlIjoyMCwidGFyZ2V0IjoyMH0sImxldmVsIjoyfV0sWzAsMSwiXFxwaGkiLDIseyJzaG9ydGVuIjp7InNvdXJjZSI6MzAsInRhcmdldCI6MzB9LCJsZXZlbCI6Mn1dLFs3LDMsInFfcyBcXGF0b3AgXFxzaW1lcSIsMSx7InNob3J0ZW4iOnsic291cmNlIjoyMH0sInN0eWxlIjp7ImJvZHkiOnsibmFtZSI6Im5vbmUifSwiaGVhZCI6eyJuYW1lIjoibm9uZSJ9fX1dLFsxMCw1LCJcXHRoZXRhX3tzXzN9IFxcYXRvcCBcXHNpbWVxIiwxLHsibGFiZWxfcG9zaXRpb24iOjcwLCJvZmZzZXQiOjIsInNob3J0ZW4iOnsic291cmNlIjoyMH0sInN0eWxlIjp7ImJvZHkiOnsibmFtZSI6Im5vbmUifSwiaGVhZCI6eyJuYW1lIjoibm9uZSJ9fX1dLFs1LDE1LCJcXHRoZXRhX3tzXzJ9IFxcYXRvcCBcXHNpbWVxIiwxLHsic2hvcnRlbiI6eyJ0YXJnZXQiOjIwfSwic3R5bGUiOnsiYm9keSI6eyJuYW1lIjoibm9uZSJ9LCJoZWFkIjp7Im5hbWUiOiJub25lIn19fV1d
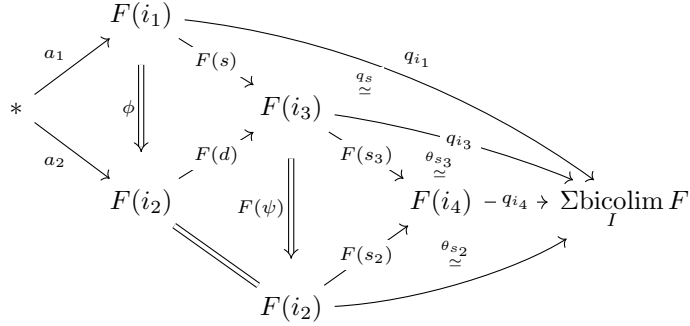
\[\begin{tikzcd}
	& {F(i_1)} \\
	{*} && {F(i_3)} \\
	& {F(i_2)} && {F(i_4)} & {\underset{ I}{\Sigma\bicolim} \,F }  \\
	&& {F(i_2)}
	\arrow[""{name=0, anchor=center, inner sep=0}, "{q_{i_1}}", curve={height=-18pt}, from=1-2, to=3-5]
	\arrow["{F(s)}"{description}, from=1-2, to=2-3]
	\arrow["{F(d)}"{description}, from=3-2, to=2-3]
	\arrow[""{name=1, anchor=center, inner sep=0}, "{q_{i_3}}"{description}, curve={height=-6pt}, from=2-3, to=3-5]
	\arrow["{a_1}", from=2-1, to=1-2]
	\arrow["{a_2}"', from=2-1, to=3-2]
	\arrow["{F(s_3)}"{description}, from=2-3, to=3-4]
	\arrow["{q_{i_4}}"{description}, from=3-4, to=3-5]
	\arrow[""{name=2, anchor=center, inner sep=0}, curve={height=12pt}, from=4-3, to=3-5]
	\arrow["{F(s_2)}"{description}, from=4-3, to=3-4]
	\arrow[Rightarrow, no head, from=3-2, to=4-3]
	\arrow["{F(\psi)}"', shorten <=10pt, shorten >=10pt, Rightarrow, from=2-3, to=4-3]
	\arrow["\phi"', shorten <=10pt, shorten >=10pt, Rightarrow, from=1-2, to=3-2]
	\arrow["{q_s \atop \simeq}"{description}, Rightarrow, draw=none, from=0, to=2-3]
	\arrow["{\theta_{s_3} \atop \simeq}"{description, pos=0.7}, shift right=2, Rightarrow, draw=none, from=1, to=3-4]
	\arrow["{\theta_{s_2} \atop \simeq}"{description}, Rightarrow, draw=none, from=3-4, to=2]
\end{tikzcd}\]
where all the colimiting transition 2-cells are invertible as indexed by morphisms in $\Sigma$. Moreover, any two such decompositions are identified in the $\sigma$-bicolimit under the relation described at \Cref{sigma-filtered sigma-colimit of categories}. This means that the corresponding arrow $(i_1,a_1) \rightarrow (i_2,a_2)$ in the $\sigma$-bicolimit comes in an essentially unique way from some arrow already existing in some member of the bicolimit. 
\end{division}

The condition above is somewhat reminiscent of a condition of fullness: it says that any arrows in the $\sigma$-bicolimit comes from somewhere in the diagram. We should give here a complementary faithfulness statement. %We must first ensure some simplification in the way one can lifts parallels morphisms. 

\begin{lemma}\label{coequification lemma}
Let $ (I,\Sigma)$ be a $\sigma$-filtered pair and $ I \rightarrow \Cat$ a pseudofunctor. For any parallel pair $ (i,a) \rightrightarrows (i',a')$ in $ \Sigma \bicolim_{I} F$ \begin{itemize}
    \item there exists a parallel pair of 2-cells as below, with $ w,w'$ in $\Sigma$
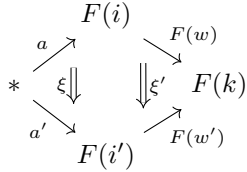
\[\begin{tikzcd}[sep=small]
	& {F(i)} \\
	{*} && {F(k)} \\
	& {F(i')}
	\arrow[""{name=0, anchor=center, inner sep=0}, "a", from=2-1, to=1-2]
	\arrow[""{name=1, anchor=center, inner sep=0}, "{F(w)}", from=1-2, to=2-3]
	\arrow[""{name=2, anchor=center, inner sep=0}, "{a'}"', from=2-1, to=3-2]
	\arrow[""{name=3, anchor=center, inner sep=0}, "{F(w')}"', from=3-2, to=2-3]
	\arrow["\xi"', shift left=3, shorten <=4pt, shorten >=4pt, Rightarrow, from=0, to=2]
	\arrow["{\xi'}", shift right=3, shorten <=4pt, shorten >=4pt, Rightarrow, from=1, to=3]
\end{tikzcd}\]
providing representants of those two morphisms;
\item if moreover those 2-morphisms are identified in $ \Sigma \bicolim_{I} F$, then there exists $ v : k \rightarrow l$ in $\Sigma$ such that $ F(v)(\xi) = F(v)(\xi')$. 
\end{itemize}   
%represented as $ (j,s,t,\phi)$ and $(j',s',t',\phi')$ which are identified in $ \Sigma\bicolim_I F$. Then there exist $ j $
\end{lemma}

\begin{proof}
We can pick two representants of this parallel pair, corresponding to 2-cells 
% https://q.uiver.app/?q=WzAsNCxbMCwxLCIqIl0sWzEsMCwiRihpKSJdLFsyLDEsIkYoaikiXSxbMSwyLCJGKGknKSJdLFswLDEsImEiXSxbMSwyLCJGKHMpIl0sWzAsMywiYSciLDJdLFszLDIsIkYocycpIiwyXSxbMSwzLCJcXHBoaSIsMix7InNob3J0ZW4iOnsic291cmNlIjoyMCwidGFyZ2V0IjoyMH0sImxldmVsIjoyfV1d
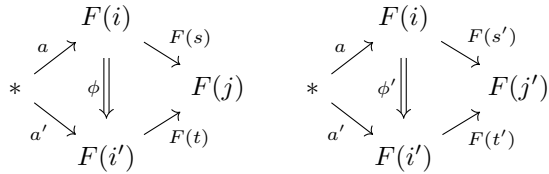
\[\begin{tikzcd}[sep=small]
	& {F(i)} \\
	{*} && {F(j)} \\
	& {F(i')}
	\arrow["a", from=2-1, to=1-2]
	\arrow["{F(s)}", from=1-2, to=2-3]
	\arrow["{a'}"', from=2-1, to=3-2]
	\arrow["{F(t)}"', from=3-2, to=2-3]
	\arrow["\phi"', shorten <=6pt, shorten >=6pt, Rightarrow, from=1-2, to=3-2]
\end{tikzcd} \hskip0.5cm \begin{tikzcd}[sep=small]
	& {F(i)} \\
	{*} && {F(j')} \\
	& {F(i')}
	\arrow["a", from=2-1, to=1-2]
	\arrow["{F(s')}", from=1-2, to=2-3]
	\arrow["{a'}"', from=2-1, to=3-2]
	\arrow["{F(t')}"', from=3-2, to=2-3]
	\arrow["\phi'"', shorten <=6pt, shorten >=6pt, Rightarrow, from=1-2, to=3-2]
\end{tikzcd}\]
Now combining the first and second axioms of $\sigma$-filteredness one can find $ u : j \rightarrow k$ and $ u' : j' \rightarrow k$ together with 2-cells 
% https://q.uiver.app/?q=WzAsNCxbMCwxLCJpIl0sWzEsMCwiaiJdLFsyLDEsImsiXSxbMSwyLCJqJyJdLFswLDEsInMiXSxbMSwyLCJ1Il0sWzAsMywidCIsMl0sWzMsMiwidSciLDJdLFsxLDMsIlxccHNpIiwyLHsic2hvcnRlbiI6eyJzb3VyY2UiOjIwLCJ0YXJnZXQiOjIwfSwibGV2ZWwiOjJ9XV0=
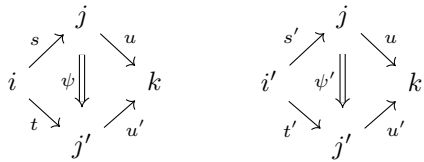
\[\begin{tikzcd}[sep=small]
	& j \\
	i && k \\
	& {j'}
	\arrow["s", from=2-1, to=1-2]
	\arrow["u", from=1-2, to=2-3]
	\arrow["t"', from=2-1, to=3-2]
	\arrow["{u'}"', from=3-2, to=2-3]
	\arrow["\psi"', shorten <=6pt, shorten >=6pt, Rightarrow, from=1-2, to=3-2]
\end{tikzcd}% https://q.uiver.app/?q=WzAsNCxbMCwxLCJpJyJdLFsxLDAsImoiXSxbMiwxLCJrIl0sWzEsMiwiaiciXSxbMCwxLCJzJyJdLFsxLDIsInUiXSxbMCwzLCJ0JyIsMl0sWzMsMiwidSciLDJdLFsxLDMsIlxccHNpJyIsMix7InNob3J0ZW4iOnsic291cmNlIjoyMCwidGFyZ2V0IjoyMH0sImxldmVsIjoyfV1d
 \hskip1cm  \begin{tikzcd}[sep=small]
	& j \\
	{i'} && k \\
	& {j'}
	\arrow["{s'}", from=2-1, to=1-2]
	\arrow["u", from=1-2, to=2-3]
	\arrow["{t'}"', from=2-1, to=3-2]
	\arrow["{u'}"', from=3-2, to=2-3]
	\arrow["{\psi'}"', shorten <=6pt, shorten >=6pt, Rightarrow, from=1-2, to=3-2]
\end{tikzcd}\]
which provides two alternative decompositions of the same morphism
% https://q.uiver.app/?q=WzAsNyxbMSwwLCJGKGkpIl0sWzEsMiwiRihpJykiXSxbNCwyLCJcXFNpZ21hXFx1bmRlcnNldHtJfXtcXGJpY29saW19IFxcOyBGIl0sWzIsMSwiRihqKSJdLFswLDEsIioiXSxbMywyLCJGKGspIl0sWzIsMywiRihqJykiXSxbMCwyLCJxX3tpXzF9IiwwLHsiY3VydmUiOi0zfV0sWzAsMywiRihzKSIsMV0sWzEsMywiRihzJykiLDFdLFszLDIsInFfe2lfM30iLDEseyJjdXJ2ZSI6LTF9XSxbNCwwLCJhXzEiXSxbNCwxLCJhXzIiLDJdLFszLDUsIkYodSkiLDFdLFs1LDIsInFfe2t9Il0sWzYsMiwicV97aid9IiwyLHsiY3VydmUiOjJ9XSxbNiw1LCJGKHUnKSIsMV0sWzEsNiwiRih0JykiLDJdLFszLDYsIkYoXFxwc2knKSIsMSx7InNob3J0ZW4iOnsic291cmNlIjoyMCwidGFyZ2V0IjoyMH0sImxldmVsIjoyfV0sWzAsMSwiXFxwaGkiLDIseyJzaG9ydGVuIjp7InNvdXJjZSI6MzAsInRhcmdldCI6MzB9LCJsZXZlbCI6Mn1dLFs3LDMsInFfcyBcXGF0b3AgXFxzaW1lcSIsMSx7InNob3J0ZW4iOnsic291cmNlIjoyMH0sInN0eWxlIjp7ImJvZHkiOnsibmFtZSI6Im5vbmUifSwiaGVhZCI6eyJuYW1lIjoibm9uZSJ9fX1dLFsxMCw1LCJcXHRoZXRhX3t1fSBcXGF0b3AgXFxzaW1lcSIsMSx7ImxhYmVsX3Bvc2l0aW9uIjo3MCwib2Zmc2V0IjoyLCJzaG9ydGVuIjp7InNvdXJjZSI6MjB9LCJzdHlsZSI6eyJib2R5Ijp7Im5hbWUiOiJub25lIn0sImhlYWQiOnsibmFtZSI6Im5vbmUifX19XSxbNSwxNSwiXFx0aGV0YV97dSd9IFxcYXRvcCBcXHNpbWVxIiwxLHsic2hvcnRlbiI6eyJ0YXJnZXQiOjIwfSwic3R5bGUiOnsiYm9keSI6eyJuYW1lIjoibm9uZSJ9LCJoZWFkIjp7Im5hbWUiOiJub25lIn19fV1d
\[\begin{tikzcd}[column sep=small]
	& {F(i)} \\
	{*} && {F(j)} \\
	& {F(i')} && {F(k)} & {\Sigma\underset{I}{\bicolim} \; F} \\
	&& {F(j')}
	\arrow[""{name=0, anchor=center, inner sep=0}, "{q_{i}}", curve={height=-18pt}, from=1-2, to=3-5]
	\arrow["{F(s)}"{description}, from=1-2, to=2-3]
	\arrow["{F(s')}"{description}, from=3-2, to=2-3]
	\arrow[""{name=1, anchor=center, inner sep=0}, "{q_{j}}"{description}, curve={height=-6pt}, from=2-3, to=3-5]
	\arrow["{a}", from=2-1, to=1-2]
	\arrow["{a'}"', from=2-1, to=3-2]
	\arrow["{F(u)}"{description}, from=2-3, to=3-4]
	\arrow["{q_{k}}", from=3-4, to=3-5]
	\arrow[""{name=2, anchor=center, inner sep=0}, "{q_{j'}}"', curve={height=12pt}, from=4-3, to=3-5]
	\arrow["{F(u')}"{description}, from=4-3, to=3-4]
	\arrow["{F(t')}"', from=3-2, to=4-3]
	\arrow["{F(\psi')}"{description}, shorten <=6pt, shorten >=6pt, Rightarrow, from=2-3, to=4-3]
	\arrow["\phi"', shorten <=10pt, shorten >=10pt, Rightarrow, from=1-2, to=3-2]
	\arrow["{\theta_s \atop \simeq}"{description}, Rightarrow, draw=none, from=0, to=2-3]
	\arrow["{\theta_{u} \atop \simeq}"{description, pos=0.7}, shift right=2, Rightarrow, draw=none, from=1, to=3-4]
	\arrow["{\theta_{u'} \atop \simeq}"{description}, Rightarrow, draw=none, from=3-4, to=2]
\end{tikzcd} = % https://q.uiver.app/?q=WzAsNyxbMSwzLCJGKGknKSJdLFsxLDEsIkYoaSkiXSxbNCwxLCJcXFNpZ21hXFx1bmRlcnNldHtJfXtcXGJpY29saW19IFxcOyBGIl0sWzIsMCwiRihqKSJdLFswLDIsIioiXSxbMywxLCJGKGspIl0sWzIsMiwiRihqJykiXSxbMCwyLCJxX3tpJ30iLDIseyJjdXJ2ZSI6NX1dLFsxLDMsIkYocykiXSxbMywyLCJxX3tqfSIsMCx7ImN1cnZlIjotMX1dLFs0LDAsImFfMSIsMl0sWzQsMSwiYV8yIl0sWzUsMiwicV97a30iXSxbNiwyLCJxX3tqJ30iLDEseyJjdXJ2ZSI6Mn1dLFs2LDUsIkYodScpIiwxXSxbMSw2LCJGKHQpIiwxXSxbMyw2LCJGKFxccHNpKSIsMSx7InNob3J0ZW4iOnsic291cmNlIjoyMCwidGFyZ2V0IjoyMH0sImxldmVsIjoyfV0sWzEsMCwiXFxwaGknIiwwLHsic2hvcnRlbiI6eyJzb3VyY2UiOjMwLCJ0YXJnZXQiOjMwfSwibGV2ZWwiOjJ9XSxbMCw2LCJGKHMnKSIsMV0sWzMsNSwiRih1KSIsMV0sWzksNSwiXFx0aGV0YV97dX0gXFxhdG9wIFxcc2ltZXEiLDEseyJsYWJlbF9wb3NpdGlvbiI6NzAsIm9mZnNldCI6Miwic2hvcnRlbiI6eyJzb3VyY2UiOjIwfSwic3R5bGUiOnsiYm9keSI6eyJuYW1lIjoibm9uZSJ9LCJoZWFkIjp7Im5hbWUiOiJub25lIn19fV0sWzUsMTMsIlxcdGhldGFfe3UnfSBcXGF0b3AgXFxzaW1lcSIsMSx7InNob3J0ZW4iOnsidGFyZ2V0IjoyMH0sInN0eWxlIjp7ImJvZHkiOnsibmFtZSI6Im5vbmUifSwiaGVhZCI6eyJuYW1lIjoibm9uZSJ9fX1dLFs3LDYsInFfcyBcXGF0b3AgXFxzaW1lcSIsMSx7InNob3J0ZW4iOnsic291cmNlIjoyMH0sInN0eWxlIjp7ImJvZHkiOnsibmFtZSI6Im5vbmUifSwiaGVhZCI6eyJuYW1lIjoibm9uZSJ9fX1dXQ==
\begin{tikzcd}[column sep=small]
	&& {F(j)} \\
	& {F(i)} && {F(k)} & {\Sigma\underset{I}{\bicolim} \; F} \\
	{*} && {F(j')} \\
	& {F(i')}
	\arrow[""{name=0, anchor=center, inner sep=0}, "{q_{i'}}"', curve={height=30pt}, from=4-2, to=2-5]
	\arrow["{F(s)}", from=2-2, to=1-3]
	\arrow[""{name=1, anchor=center, inner sep=0}, "{q_{j}}", curve={height=-6pt}, from=1-3, to=2-5]
	\arrow["{a'}"', from=3-1, to=4-2]
	\arrow["{a}", from=3-1, to=2-2]
	\arrow["{q_{k}}", from=2-4, to=2-5]
	\arrow[""{name=2, anchor=center, inner sep=0}, "{q_{j'}}"{description}, curve={height=12pt}, from=3-3, to=2-5]
	\arrow["{F(u')}"{description}, from=3-3, to=2-4]
	\arrow["{F(t)}"{description}, from=2-2, to=3-3]
	\arrow["{F(\psi)}"{description}, shorten <=6pt, shorten >=6pt, Rightarrow, from=1-3, to=3-3]
	\arrow["{\phi'}", shorten <=10pt, shorten >=10pt, Rightarrow, from=2-2, to=4-2]
	\arrow["{F(t')}"{description}, from=4-2, to=3-3]
	\arrow["{F(u)}"{description}, from=1-3, to=2-4]
	\arrow["{\theta_{u} \atop \simeq}"{description, pos=0.7}, shift right=2, Rightarrow, draw=none, from=1, to=2-4]
	\arrow["{\theta_{u'} \atop \simeq}"{description}, Rightarrow, draw=none, from=2-4, to=2]
	\arrow["{\theta_s \atop \simeq}"{description}, Rightarrow, draw=none, from=0, to=3-3]
\end{tikzcd}\]
We thus end with a parallel pair \emph{in the same component} $F(k)$
% https://q.uiver.app/?q=WzAsNixbMSwwLCJGKGkpIl0sWzIsMCwiRihqKSJdLFswLDEsIioiXSxbMywxLCJGKGspIl0sWzIsMiwiRihqJykiXSxbMSwyLCJGKGknKSJdLFswLDEsIkYocykiXSxbMiwwLCJhIl0sWzQsMywiRih1JykiLDJdLFsxLDMsIkYodSkiXSxbMiw1LCJhJyIsMl0sWzUsNCwiRih0JykiLDJdLFswLDUsIkYoXFxwc2kpKmFGKHUnKSpcXHBoaSciLDEseyJzaG9ydGVuIjp7InNvdXJjZSI6MjAsInRhcmdldCI6MjB9LCJsZXZlbCI6Mn1dLFsxLDQsIkYoXFxwc2knKSphJ0YodSlcXHBoaSIsMSx7InNob3J0ZW4iOnsic291cmNlIjoyMCwidGFyZ2V0IjoyMH0sImxldmVsIjoyfV1d
\[\begin{tikzcd}[column sep=large, row sep=small]
	& {F(i)} & {F(j)} \\
	{*} &&& {F(k)} \\
	& {F(i')} & {F(j')}
	\arrow["{F(s)}", from=1-2, to=1-3]
	\arrow["a", from=2-1, to=1-2]
	\arrow["{F(u')}"', from=3-3, to=2-4]
	\arrow["{F(u)}", from=1-3, to=2-4]
	\arrow["{a'}"', from=2-1, to=3-2]
	\arrow["{F(t')}"', from=3-2, to=3-3]
	\arrow["{F(\psi)*aF(u')*\phi'}"{description}, shorten <=6pt, shorten >=6pt, Rightarrow, from=1-2, to=3-2]
	\arrow["{F(\psi')*a'F(u)*\phi}"{description}, shorten <=6pt, shorten >=6pt, Rightarrow, from=1-3, to=3-3]
\end{tikzcd}\]
If now this parallel pair happens to be identified in $\Sigma \bicolim_{I} F$ (that is, coequified by whiskering with $q_k)$: hence, applying \Cref{local faithfulness of colimit inclusions}, there exists a further morphism $ v : k \rightarrow l$ in $\Sigma$ such that $F(v)$ coequifies this parallel pair 
\[  F(v)( F(\psi)*aF(u')*\phi' ) = F(v)(F(\psi')*a'F(u)*\phi) \]
\end{proof}

\begin{remark}
Although the last argument in \Cref{coequification lemma} relies on a general consideration about the calculus of fractions involved in arbitrary $ \sigma$-bicolimits, the crucial simplifications used here are strictly specific to $\sigma$-filteredness and by no mean can be inferred to arbitrary $ \sigma$-bicolimits.
\end{remark}

\subsection{2-dimensional cofinality and trivialization of $ \sigma$-filtered $\sigma$-bicolimit}

A consequence of \Cref{coequification lemma} and \Cref{triangle lemma} is that $\sigma$-filtered $ \sigma$-bicolimits are actually no more complicated nor expressive than bifiltered bicolimits, in the sense that for a $\sigma$-filtered pair $ (I,\Sigma)$, the arrows out of $\Sigma$ are somewhat ``useless" when it comes to compute an $I$-indexed $ \sigma$-bicolimit relatively to $\Sigma$. This observation will be made formal thanks to the following notions:

\begin{definition}[Cofinality]
Let $ I$ be a small 2-category with $ \Sigma$ a class of maps such that $I$ is $ \Sigma$-filtered, $J$ another 2-category with $\Sigma'$ a class of maps in $J$ and $ F : I \rightarrow J$ a pseudofunctor. Then $F$ is \emph{$\sigma$-cofinal}\index{$\sigma$-cofinal} with respect to $\Sigma$ and $\Sigma'$ if \begin{itemize}
    \item for any $ j $ in $J$ there is some arrow $ s : j \rightarrow F(i)$ in $\Sigma'$;
    \item for any parallel pair $d, t : j \rightrightarrows F(i)$ with $t$ in $\Sigma'$ there is $ s: i \rightarrow i' $ in $\Sigma$ and a 2-cell% https://q.uiver.app/?q=WzAsNCxbMCwxLCJqIl0sWzEsMCwiRihpKSJdLFsyLDEsIkYoaScpIl0sWzEsMiwiRihpKSJdLFswLDEsImRfMSJdLFsxLDIsIkYoZCkiXSxbMCwzLCJkXzIiLDJdLFszLDIsIkYoZCcpIiwyXSxbMSwzLCJcXGFscGhhIFxcYXRvcCBcXHNpbWVxIiwxLHsic3R5bGUiOnsiYm9keSI6eyJuYW1lIjoibm9uZSJ9LCJoZWFkIjp7Im5hbWUiOiJub25lIn19fV1d
\[\begin{tikzcd}[row sep=small]
	& {F(i)} \\
	j && {F(i')} \\
	& {F(i)}
	\arrow["{d}", from=2-1, to=1-2]
	\arrow["{F(s)}", from=1-2, to=2-3]
	\arrow["{t}"', from=2-1, to=3-2]
	\arrow["{F(s)}"', from=3-2, to=2-3]
	\arrow["{\alpha}"', Rightarrow, from=1-2, to=3-2]
\end{tikzcd}\]
Moreover $ \alpha$ can be chosen to be invertible whenever $d$ also is in $\Sigma$.
    \item For any parallel 2-cells \[\begin{tikzcd}
	j && {F(i)}
	\arrow[""{name=0, anchor=center, inner sep=0}, start anchor=40, "d", bend left=20, from=1-1, to=1-3]
	\arrow[""{name=1, anchor=center, inner sep=0}, start anchor=-40, "{t}"', bend right=20, from=1-1, to=1-3]
	\arrow[Rightarrow, draw=none, from=0, to=1]
	\arrow["\alpha"', shift right=4, shorten <=3pt, shorten >=3pt, Rightarrow, from=0, to=1]
	\arrow["{\alpha'}", shift left=4, shorten <=3pt, shorten >=3pt, Rightarrow, from=0, to=1]
\end{tikzcd}\]
with $t$ in $\Sigma'$, there exists $ s :i \rightarrow i'$ in $\Sigma$ such that $ F(s)*\alpha = F(s) *\alpha'$.
\end{itemize}
\end{definition}

\begin{remark}
This is a specific form of cofinality in the context of $\sigma$-filteredness, while a more general definition of cofinality may exist for non $\sigma$-filtered diagrams; however this one is sufficient for our purpose. Also we should point out that \cite[Definition 3.3.1]{descotte2018sigma} definition of $\sigma$-cofinality is given for a strict 2-functor, yet this does not modify anything in practice, for there is no relevant interaction between the data of $\sigma$-cofinality and the unit and composition data associated to the pseudofunctor. Moreover, for the classes of maps in $\sigma$-pairs can always be assumed to be closed under invertible 2-cells, one can check that a pseudofunctor admitting a strictification is $\sigma$-cofinal if and only if its strictification $\overline{F}$ is so. 
\end{remark}

\begin{lemma}[Cofinality preserve filteredness]\label{cofinal functors transfer sigma-filteredness}
Let $ (I,\Sigma)$ and $ (J,\Sigma')$ be two $\sigma$-pairs and $ F : I \rightarrow J$ such that \begin{itemize}
    \item $ I$ $\sigma$-filtered,
    \item $F$ is $\sigma$-cofinal for $ \Sigma, \Sigma'$,
    \item and $ F(\Sigma) \subseteq \Sigma'$.
\end{itemize} Then $ (J,\Sigma')$ is $\sigma$-filtered.
\end{lemma}

\begin{proof}
Let us prove the different conditions of $\sigma$-filteredness. First, for any $j,j'$ in $J$, there are, by cofinality, both  $i,i'$ in $I$ with $ s: j \rightarrow i$ and $s' : j' \rightarrow i'$ in $\Sigma'$ and then by $\Sigma$-filteredness of $I$ a span $t : i \rightarrow i'' $ and $ t' : i' \rightarrow i''$ in $\Sigma$, and the composite $ F(t)s$, $F(t')s'$ are in $\Sigma'$ and provide a desired span. Now for a parallel pair $ d,s : j \rightarrow F(i)$ with $s$ in $\Sigma'$, cofinalness entails the existence of $t : i \rightarrow i'$ in $\Sigma$ such that $ F(t) $ inserts some $ \alpha : F(t)d \Rightarrow F(t)s$ (which can be made invertible if $ d $ is in $\Sigma'$, but $F(t)$ is in $\Sigma'$. Same argument for equalization of parallel 2-cells. 
\end{proof}

We also have the following converse property:

\begin{lemma}[\cite{descotte2018sigma} Proposition 3.3.2]\label{reflection of filteredness}
Let $ (I,\Sigma)$ and $ (J,\Sigma')$ be two $\sigma$-pairs and $ F : I \rightarrow J$ a pseudofunctor such that \begin{itemize}
    \item $(J,\Sigma')$ is $\sigma$-filtered,
    \item $F$ is pseudo-fully-faithful,
    \item for each $j$ in $J$ there is some $ s : i \rightarrow F(i)$ in $\Sigma'$.
\end{itemize}
Then $ F$ is $\sigma$-cofinal for $ \Sigma$, $\Sigma'$ and $ (I,\Sigma) $ is $\sigma$-filtered.
\end{lemma}

The following proposition is a $\sigma$-version of \cite{descotte2020theory}[Theorem 1.3.9] (which was originally stated for \emph{bifilteredness} and \emph{bicofinality}) confirms that $\sigma$-cofinal functors have the expected behavior regarding $\sigma$-colimits - at least for $\sigma$-filtered ones:

\begin{proposition}\label{cofinal functors and bicolimits}
Let $ (I,\Sigma)$ and $ (J,\Sigma')$ be two $\sigma$-filtered pairs and $ F : I \rightarrow J$ a $\sigma$-cofinal pseudofunctor relatively to $ \Sigma$, $ \Sigma'$ such that $ F(\Sigma) \subseteq \Sigma'$, and $ G : J \rightarrow \mathcal{B}$ a pseudofunctor. Then one has
\[ \underset{j \in J}{{\Sigma'}\bicolim} \, G(j) \simeq \underset{i \in I}{\Sigma\bicolim} \, GF(i)   \]
\end{proposition}

\begin{proof}

For $ F(\Sigma) \subseteq \Sigma'$, the restriction of the colimiting $\sigma$-cocone for $G$ to objects in the range of $F$ defines a $\Sigma$-cocone $(q_{F(i)} : GF(i) \rightarrow {\Sigma'}\bicolim_{j \in J} G(j))_{i \in I}$ as its 2-cells $ q_{F(s)}$ for $s$ in $\Sigma$ are invertible. Then we have a uniquely induced arrow $\langle q_{F(i)} \rangle_{i \in I}$ with a factorization of the colimit inclusion at each $i$ 
% https://q.uiver.app/?q=WzAsMyxbMCwwLCJcXHVuZGVyc2V0e2kgXFxpbiBJfXtcXHNpZ21hX1xcU2lnbWFcXGJpY29saW19IFxcLCBHRihpKSJdLFsxLDAsIlxcdW5kZXJzZXR7aiBcXGluIEp9e1xcc2lnbWFfe1xcU2lnbWEnfVxcYmljb2xpbX0gXFwsIEcoaikiXSxbMCwxLCJHRihpKSJdLFswLDEsIlxcbGFuZ2xlIHFfe0YoaSl9IFxccmFuZ2xlX3tpIFxcaW4gSX0iXSxbMiwwLCJxJ19pIl0sWzIsMSwicV97RihpKX0iLDJdLFswLDUsIlxcc2ltZXEiLDEseyJzaG9ydGVuIjp7InRhcmdldCI6MjB9LCJzdHlsZSI6eyJib2R5Ijp7Im5hbWUiOiJub25lIn0sImhlYWQiOnsibmFtZSI6Im5vbmUifX19XV0=
\[\begin{tikzcd}
	{\underset{i \in I}{\Sigma\bicolim} \, GF(i)} & {\underset{j \in J}{{\Sigma'}\bicolim} \, G(j)} \\
	{GF(i)}
	\arrow["{\langle q_{F(i)} \rangle_{i \in I}}", from=1-1, to=1-2]
	\arrow["{q'_i}", from=2-1, to=1-1]
	\arrow[""{name=0, anchor=center, inner sep=0}, "{q_{F(i)}}"', from=2-1, to=1-2]
	\arrow["\simeq"{description}, Rightarrow, draw=none, from=1-1, to=0]
\end{tikzcd}\]

By cofinality, any $ j$ in $J$ admits a $\Sigma'$ arrow $ t: j \rightarrow F(i)$, so that we have a factorization of the corresponding $j$-inclusion through the following invertible 2-cells
% https://q.uiver.app/?q=WzAsNCxbMywwLCJcXHVuZGVyc2V0e2kgXFxpbiBJfXtcXHNpZ21hX1xcU2lnbWFcXGJpY29saW19IFxcLCBHRihpKSJdLFsxLDEsIlxcdW5kZXJzZXR7aiBcXGluIEp9e1xcc2lnbWFfe1xcU2lnbWEnfVxcYmljb2xpbX0gXFwsIEcoaikiXSxbMSwwLCJHRihpKSJdLFswLDAsIkcoaikiXSxbMCwxLCJcXGxhbmdsZSBxX3tGKGkpfSBcXHJhbmdsZV97aSBcXGluIEl9Il0sWzIsMCwicSdfaSJdLFsyLDEsInFfe0YoaSl9IiwxXSxbMywyLCJHKHQpIl0sWzMsMSwicV9qIiwyXSxbMCw2LCJcXHNpbWVxIiwxLHsic2hvcnRlbiI6eyJ0YXJnZXQiOjIwfSwic3R5bGUiOnsiYm9keSI6eyJuYW1lIjoibm9uZSJ9LCJoZWFkIjp7Im5hbWUiOiJub25lIn19fV0sWzIsOCwicV90IFxcYXRvcCBcXHNpbWVxIiwxLHsic2hvcnRlbiI6eyJ0YXJnZXQiOjIwfSwic3R5bGUiOnsiYm9keSI6eyJuYW1lIjoibm9uZSJ9LCJoZWFkIjp7Im5hbWUiOiJub25lIn19fV1d
\[\begin{tikzcd}
	{G(j)} & {GF(i)} && {\underset{i \in I}{\Sigma\bicolim} \, GF(i)} \\
	& {\underset{j \in J}{{\Sigma'}\bicolim} \, G(j)}
	\arrow["{\langle q_{F(i)} \rangle_{i \in I}}", from=1-4, to=2-2]
	\arrow["{q'_i}", from=1-2, to=1-4]
	\arrow[""{name=0, anchor=center, inner sep=0}, "{q_{F(i)}}"{description}, from=1-2, to=2-2]
	\arrow["{G(t)}", from=1-1, to=1-2]
	\arrow[""{name=1, anchor=center, inner sep=0}, "{q_j}"', from=1-1, to=2-2]
	\arrow["\simeq"{description}, Rightarrow, draw=none, from=1-4, to=0]
	\arrow["{q_t \atop \simeq}"{description}, Rightarrow, draw=none, from=1-2, to=1]
\end{tikzcd}\]
Moreover such a factorization is actually uniquely defined: by cofinality, if one as two distinct arrows $ t_1 : j \rightarrow F(i_1)$, $t_2 : j \rightarrow F(i_2)$ in $\Sigma'$, one has a $\Sigma$ span $ s_1 : i_1 \rightarrow i$, $ s_2 : i_2 \rightarrow i$ in $I$ as $( I,\Sigma)$ is $\sigma$-filtered, and moreover one can choose this $i$ together with an invertible 2-cell $ \alpha : F(s_1)t_1 \simeq F(s_2)t_2$. This ensures that one has actually a unique $g_j : G(j) \rightarrow {{\Sigma}\bicolim}_{i \in J} \, GF(i)$ factorizing $ q_j$ through $ \langle q_{F(i)} \rangle_{i \in I}$: hence we have a canonical invertible 2-cell $ 1_{{\Sigma'}\bicolim_{j \in J} G(j) } \simeq \langle q_{F(i)} \rangle_{i \in I} \langle g_j \rangle_{j \in J}$
% https://q.uiver.app/?q=WzAsMyxbMCwxLCJ7XFx1bmRlcnNldHtqIFxcaW4gSn17e1xcU2lnbWEnfVxcYmljb2xpbX0gXFwsIEcoail9Il0sWzAsMCwie1xcdW5kZXJzZXR7aiBcXGluIEp9e3tcXFNpZ21hJ31cXGJpY29saW19IFxcLCBHKGopfSJdLFsxLDAsIntcXHVuZGVyc2V0e2kgXFxpbiBJfXtcXFNpZ21hXFxiaWNvbGltfSBcXCwgR0YoaSl9Il0sWzEsMiwiXFxsYW5nbGUgZ19qIFxccmFuZ2xlX3tqIFxcaW4gSn0iXSxbMiwwLCJcXGxhbmdsZSBxX3tGKGkpfSBcXHJhbmdsZV97aSBcXGluIEl9Il0sWzEsMCwiIiwyLHsibGV2ZWwiOjIsInN0eWxlIjp7ImhlYWQiOnsibmFtZSI6Im5vbmUifX19XV0=
% https://q.uiver.app/?q=WzAsMyxbMCwxLCJ7XFx1bmRlcnNldHtqIFxcaW4gSn17e1xcU2lnbWEnfVxcYmljb2xpbX0gXFwsIEcoail9Il0sWzAsMCwie1xcdW5kZXJzZXR7aiBcXGluIEp9e3tcXFNpZ21hJ31cXGJpY29saW19IFxcLCBHKGopfSJdLFsxLDAsIntcXHVuZGVyc2V0e2kgXFxpbiBJfXtcXFNpZ21hXFxiaWNvbGltfSBcXCwgR0YoaSl9Il0sWzEsMiwiXFxsYW5nbGUgZ19qIFxccmFuZ2xlX3tqIFxcaW4gSn0iXSxbMiwwLCJcXGxhbmdsZSBxX3tGKGkpfSBcXHJhbmdsZV97aSBcXGluIEl9Il0sWzEsMCwiIiwyLHsibGV2ZWwiOjIsInN0eWxlIjp7ImhlYWQiOnsibmFtZSI6Im5vbmUifX19XSxbMSw0LCJcXHNpbWVxIiwxLHsibGFiZWxfcG9zaXRpb24iOjQwLCJzaG9ydGVuIjp7InRhcmdldCI6MjB9LCJzdHlsZSI6eyJib2R5Ijp7Im5hbWUiOiJub25lIn0sImhlYWQiOnsibmFtZSI6Im5vbmUifX19XV0=
\[\begin{tikzcd}
	{{\underset{j \in J}{{\Sigma'}\bicolim} \, G(j)}} & {{\underset{i \in I}{\Sigma\bicolim} \, GF(i)}} \\
	{{\underset{j \in J}{{\Sigma'}\bicolim} \, G(j)}}
	\arrow["{\langle g_j \rangle_{j \in J}}", from=1-1, to=1-2]
	\arrow[""{name=0, anchor=center, inner sep=0}, "{\langle q_{F(i)} \rangle_{i \in I}}", from=1-2, to=2-1]
	\arrow[Rightarrow, no head, from=1-1, to=2-1]
	\arrow["\simeq"{description, pos=0.4}, Rightarrow, draw=none, from=1-1, to=0]
\end{tikzcd}\]
Moreover one has a canonical invertible 2-cell $ g_j \simeq \langle g_j \rangle_{j \in J} q_j$, so in particular one has $ g_{F(i)} \simeq \langle g_j \rangle_{j \in J} q_{F(i)}$: but clearly $ g_F(i) \simeq q'_i$ as one can take $ 1_{F(i)}$ as the $t$ in the construction of the $g_{F(i)}$ above: this gives a converse factorization of the bicolimit inclusions 
% https://q.uiver.app/?q=WzAsMyxbMSwwLCJ7XFx1bmRlcnNldHtqIFxcaW4gSn17e1xcU2lnbWEnfVxcYmljb2xpbX0gXFwsIEcoail9Il0sWzEsMSwie1xcdW5kZXJzZXR7aSBcXGluIEl9e1xcU2lnbWFcXGJpY29saW19IFxcLCBHRihpKX0iXSxbMCwwLCJHRihpKSJdLFswLDEsIlxcbGFuZ2xlIGdfaiBcXHJhbmdsZV97aiBcXGluIEp9Il0sWzIsMCwicV97RihpKX0iXSxbMiwxLCJxJ19pIiwyXSxbMCw1LCJcXHNpbWVxIiwxLHsic2hvcnRlbiI6eyJ0YXJnZXQiOjIwfSwic3R5bGUiOnsiYm9keSI6eyJuYW1lIjoibm9uZSJ9fX1dXQ==
\[\begin{tikzcd}
	{GF(i)} & {{\underset{j \in J}{{\Sigma'}\bicolim} \, G(j)}} \\
	& {{\underset{i \in I}{\Sigma\bicolim} \, GF(i)}}
	\arrow["{\langle g_j \rangle_{j \in J}}", from=1-2, to=2-2]
	\arrow["{q_{F(i)}}", from=1-1, to=1-2]
	\arrow[""{name=0, anchor=center, inner sep=0}, "{q'_i}"', from=1-1, to=2-2]
	\arrow["\simeq"{description, pos=0.4}, shorten >=2pt, Rightarrow, no body, from=1-2, to=0]
\end{tikzcd}\]
which in turn induces an invertible 2-cell 
\[ 1_{{\Sigma}\bicolim_{i \in I} GF(i) } \simeq \langle g_j \rangle_{j \in J}\langle q_{F(i)} \rangle_{i \in I}  \]
This exhibits the desired equivalence between those $\sigma$-bicolimits.
\end{proof}

%\textcolor{red}{The key lemma is actually almost trivial in its general form; yet I put it.}
%The following observation, though it sounds tautological and its hypothesis ad-hoc, is actually so crucial we chose to put it as a lemma: \textcolor{red}{This seems both trivial and surprizingly strong to me because it trivializes Dubuc:}

Now, we come to a central observation, essential to the next sections of this work, which seems to have been unnoticed hitherto; yet it reduces the theory of $\sigma$-filteredness to the theory of bifilteredness thanks to an almost tautological argument of $\sigma$-cofinality: 

\begin{lemma}[Trivialization lemma]\label{key lemma}
Let $ (I,\Sigma)$ be a $\sigma$-pair. Then $ (I,\Sigma)$ is $\sigma$-filtered if and only if the full on 0-cells and 2-cells subcategory $ \Sigma \hookrightarrow I$ (i.e., full on objects and $2$ cells and containing only $\Sigma$ as $1$-cells) is bifiltered and $\sigma$-cofinal relatively to $\Sigma$.
\end{lemma}

\begin{proof}
This is just putting altogether the definition of $\sigma$-filteredness and $\sigma$-cofinality, having in mind that moreover a $\sigma$-pair $ (I, \Sigma)$ is bifiltered exactly when $ \Sigma$ contains all arrows of $I$. Although this is tautological, this lemma is important enough for us to state a carefull proof. Suppose that $ (I,\Sigma) $ is a $\sigma$-pair such that the subcategory $ \Sigma$ is bifiltered and the inclusion $ \iota_\Sigma : \Sigma \hookrightarrow I$ is $\sigma$-cofinal for $\Sigma$: then we are in the condition of \Cref{cofinal functors transfer sigma-filteredness}, which entails that $(I,\Sigma)$ is $\sigma$-filtered. Conversely, suppose that $ (I,\Sigma)$ is $ \sigma$-filtered. Then for $ i,i'$ in $I$ the span provided by $\sigma$-filteredness is in $\Sigma$. %As $\Sigma$, seen as a subcategory of $I$, contains all object of $ I$ and also their identities, any object of $I$ admit an arrow in $\Sigma$ toward an object of $\Sigma$. 
Now take a parallel pair $ s,t : i \rightarrow j$ with both $s,s'$ in $\Sigma$: then there exists $ t $ in $\Sigma$ together with an invertible 2-cell $\alpha : ts \simeq ts'$; similar argument for parallel pairs of 2-cells: the $\sigma$-filteredness of $(I,\Sigma)$ automatically entails the bifilteredness of $\Sigma$.   
\end{proof}

\begin{remark}
As a corollary, this proves that, any $\sigma$-filtered pair being actually controlled by the bifiltered subcategory made of its $\sigma$-filtration arrows, any $\sigma$-bicolimit over a $\sigma$-filtered pair is actually equivalent to the bicolimit over this $\sigma$-cofinal bifiltered subcategory. This will explain why the $\sigma$-bicolimit decomposition of pseudofunctors into $\Cat$ can always be simplified to an ordinary bicolimit for flat pseudofunctors, and why
bifiltered bicolimits are sufficient in the definition of bi-accessible categories. 
\end{remark}

\begin{corollary}[$\sigma$-filtered bicolimits are bifiltered bicolimits] \label{sigmafiltered colimits are bifiltered colimits}
For any $\sigma$-filtered pair $(I, \Sigma)$, with $ \iota_\Sigma : \Sigma \hookrightarrow I$ being the corresponding inclusion, and any $F: I \rightarrow \Cat$ one has an equivalence of categories
\[\Sigma \underset{ I}\bicolim F  \simeq  \underset{ \Sigma}\bicolim \; F\iota_\Sigma\]
\end{corollary}

\begin{division}[A direct proof without trivialization lemma]
It is worth the detailing of the construction to convince oneself of the fact above, for it might sound surprisingly strong. For this reason, we choose to provide here a concrete proof of the corollary above without invoking the trivialization lemma. Indeed, one could ask how the data encoded in the morphisms out of $\Sigma$ are managed in the restricted filtered bicolimit over $\Sigma$ if it is the same as the $\sigma$-bicolimit, knowing that the cocartesian morphisms are all invertible in the restricted bicolimit while those indexed by morphisms out of $\Sigma$ are not in the $\sigma$-bicolimit. Let us examine the case of the $\sigma$-bicolimit of a diagram of small categories $ F : I \rightarrow \Cat$ over a $\sigma$-filtered pair $ (I, \Sigma)$. Denote $ \iota_\Sigma : \Sigma \rightarrow I$ the corresponding inclusion: we have an induced functor 
% https://q.uiver.app/?q=WzAsMixbMCwwLCJcXGJpY29saW1fXFxTaWdtYSBGXFxpb3RhX1xcU2lnbWEiXSxbMSwwLCJcXFNpZ21hIFxcYmljb2xpbV9JIEYiXSxbMCwxLCJxIl1d
\[\begin{tikzcd}
	{\underset{ \Sigma}\bicolim \; F\iota_\Sigma} & {\Sigma \underset{ I}\bicolim \; F}
	\arrow["q", from=1-1, to=1-2]
\end{tikzcd}\]
induced by the inclusions $ q_i$ over the restricted cocone over $ \Sigma$. Let us prove that this functor is essentially surjective on objects, full and faithful by using successively the different axioms of cofinality. Essential surjectivity makes no doubt as $ \iota_\Sigma$ is essentially surjective on objects. Now take a morphism $ (i,a) \rightarrow (i', a') $ in $ \Sigma \bicolim_{I} F$: it can be presented by data $s : i \rightarrow i'',d : i' \rightarrow i''$ and $ \phi : F(s)a \Rightarrow F(d)a'$ with $s$ in $\Sigma$ as in \Cref{sigma-filtered sigma-colimit of categories}, but from \Cref{morphisms of sigma-colimit are morphisms of the bicolimit} this presentation can be itself replaced by a better presentation where $ d$ is substituted with a 2-cell $ \psi : t d \Rightarrow t'$ with both $ t,t'$ in $\Sigma$: then the following pasting
% https://q.uiver.app/?q=WzAsNCxbMCwxLCIqIl0sWzEsMCwiRihpKSJdLFsyLDEsIkYoaScnKSJdLFsxLDIsIkYoaScpIl0sWzAsMSwiYSJdLFsxLDIsIkYodHMpIiwxXSxbMCwzLCJhJyIsMl0sWzMsMiwiRih0JykiLDFdLFsxLDMsIkYoXFxwc2kpKmEnIEYodCkqXFxwaGkiLDEseyJzaG9ydGVuIjp7InNvdXJjZSI6MjAsInRhcmdldCI6MjB9LCJsZXZlbCI6Mn1dXQ==
\[\begin{tikzcd}
	& {F(i)} \\
	{*} && {F(i'')} \\
	& {F(i')}
	\arrow["a", from=2-1, to=1-2]
	\arrow["{F(ts)}"{description}, from=1-2, to=2-3]
	\arrow["{a'}"', from=2-1, to=3-2]
	\arrow["{F(t')}"{description}, from=3-2, to=2-3]
	\arrow["{F(\psi)*a' F(t)*\phi}"{description}, shorten <=6pt, shorten >=6pt, Rightarrow, from=1-2, to=3-2]
\end{tikzcd}\]
can be used to represent a morphism $ (i,a) \rightarrow (i',a')$ in the restricted bicolimit $ \bicolim_{\Sigma} F\iota_\Sigma$ which returns the original morphism in $ \Sigma \bicolim_{I} F$ through the following whiskering, proving $q$ to be full:
% https://q.uiver.app/?q=WzAsNixbMCwxLCIqIl0sWzEsMCwiRihpKSJdLFsyLDEsIkYoaScnKSJdLFsxLDIsIkYoaScpIl0sWzMsMSwiXFxiaWNvbGltX1xcU2lnbWEgRlxcaW90YV9cXFNpZ21hIl0sWzQsMSwiXFxTaWdtYVxcYmljb2xpbV9JIEYiXSxbMCwxLCJhIl0sWzEsMiwiRih0cykiLDFdLFswLDMsImEnIiwyXSxbMywyLCJGKHQnKSIsMV0sWzEsMywiRihcXHBzaSkqYScgRih0KSpcXHBoaSIsMSx7InNob3J0ZW4iOnsic291cmNlIjoyMCwidGFyZ2V0IjoyMH0sImxldmVsIjoyfV0sWzIsNCwicV97aScnfSIsMl0sWzEsNCwicV9pIiwwLHsiY3VydmUiOi0yfV0sWzMsNCwicV97aSd9IiwyLHsiY3VydmUiOjJ9XSxbNCw1LCJxIl0sWzIsMTIsIlxcdGhldGFfe3RzfSBcXGF0b3AgXFxzaW1lcSIsMSx7InNob3J0ZW4iOnsidGFyZ2V0IjoyMH0sInN0eWxlIjp7ImJvZHkiOnsibmFtZSI6Im5vbmUifSwiaGVhZCI6eyJuYW1lIjoibm9uZSJ9fX1dLFsyLDEzLCJcXHRoZXRhX3t0J30gXFxhdG9wIFxcc2ltZXEiLDEseyJzaG9ydGVuIjp7InRhcmdldCI6MjB9LCJzdHlsZSI6eyJib2R5Ijp7Im5hbWUiOiJub25lIn0sImhlYWQiOnsibmFtZSI6Im5vbmUifX19XV0=
\[\begin{tikzcd}[column sep=small]
	& {F(i)} \\
	{*} && {F(i'')} & {\bicolim_{\Sigma} F\iota_\Sigma} & {\Sigma\bicolim_I F} \\
	& {F(i')}
	\arrow["a", from=2-1, to=1-2]
	\arrow["{F(ts)}"{description}, from=1-2, to=2-3]
	\arrow["{a'}"', from=2-1, to=3-2]
	\arrow["{F(t')}"{description}, from=3-2, to=2-3]
	\arrow["{F(\psi)*a' F(t)*\phi}"{description}, shorten <=6pt, shorten >=6pt, Rightarrow, from=1-2, to=3-2]
	\arrow["{q_{i''}}"', from=2-3, to=2-4]
	\arrow[""{name=0, anchor=center, inner sep=0}, "{q_i}", curve={height=-12pt}, from=1-2, to=2-4]
	\arrow[""{name=1, anchor=center, inner sep=0}, "{q_{i'}}"', curve={height=12pt}, from=3-2, to=2-4]
	\arrow["q", from=2-4, to=2-5]
	\arrow["{\theta_{ts} \atop \simeq}"{description}, Rightarrow, draw=none, from=2-3, to=0]
	\arrow["{\theta_{t'} \atop \simeq}"{description}, Rightarrow, draw=none, from=2-3, to=1]
\end{tikzcd}\]
% https://q.uiver.app/?q=WzAsNixbMCwxLCIqIl0sWzEsMCwiRihpKSJdLFsyLDEsIkYoaScnKSJdLFsxLDIsIkYoaScpIl0sWzMsMSwiXFxTaWdtYVxcYmljb2xpbV9JIEYiXSxbNCwxXSxbMCwxLCJhIl0sWzEsMiwiRih0cykiLDFdLFswLDMsImEnIiwyXSxbMywyLCJGKHQnKSIsMV0sWzEsMywiRihcXHBzaSkqYScgRih0KSpcXHBoaSIsMSx7InNob3J0ZW4iOnsic291cmNlIjoyMCwidGFyZ2V0IjoyMH0sImxldmVsIjoyfV0sWzIsNCwicV97aScnfSIsMl0sWzEsNCwicV9pIiwwLHsiY3VydmUiOi0yfV0sWzMsNCwicV97aSd9IiwyLHsiY3VydmUiOjJ9XSxbMiwxMiwiXFx0aGV0YV97dHN9IFxcYXRvcCBcXHNpbWVxIiwxLHsic2hvcnRlbiI6eyJ0YXJnZXQiOjIwfSwic3R5bGUiOnsiYm9keSI6eyJuYW1lIjoibm9uZSJ9LCJoZWFkIjp7Im5hbWUiOiJub25lIn19fV0sWzIsMTMsIlxcdGhldGFfe3QnfSBcXGF0b3AgXFxzaW1lcSIsMSx7InNob3J0ZW4iOnsidGFyZ2V0IjoyMH0sInN0eWxlIjp7ImJvZHkiOnsibmFtZSI6Im5vbmUifSwiaGVhZCI6eyJuYW1lIjoibm9uZSJ9fX1dXQ==
Finally, to prove $q$ to be faithful: suppose we have two morphisms $ (i,a) \rightrightarrows (i',a')$ in the bicolimit that are identified in the $\Sigma$-bicolimit after whiskering with $q$. Then by \Cref{coequification lemma}, there is a parallel pair of 2-cells as below
% https://q.uiver.app/?q=WzAsNSxbMCwxLCIqIl0sWzEsMCwiRihpKSJdLFsxLDIsIkYoaScpIl0sWzIsMSwiRihrKSJdLFszLDEsIlxcU2lnbWFcXHVuZGVyc2V0e0l9e1xcYmljb2xpbX0gXFw7IEYiXSxbMCwxLCJhIl0sWzEsMywiRihzKSJdLFswLDIsImEnIiwyXSxbMiwzLCJGKHMnKSIsMl0sWzMsNCwicV9rIl0sWzUsNywiXFxwaGkiLDIseyJvZmZzZXQiOi0zLCJzaG9ydGVuIjp7InNvdXJjZSI6MjAsInRhcmdldCI6MjB9fV0sWzYsOCwiXFxwaGknIiwwLHsib2Zmc2V0IjozLCJzaG9ydGVuIjp7InNvdXJjZSI6MjAsInRhcmdldCI6MjB9fV1d
% https://q.uiver.app/?q=WzAsNCxbMCwxLCIqIl0sWzEsMCwiRihpKSJdLFsxLDIsIkYoaScpIl0sWzIsMSwiRihrKSJdLFswLDEsImEiXSxbMSwzLCJGKHMpIl0sWzAsMiwiYSciLDJdLFsyLDMsIkYocycpIiwyXSxbNCw2LCJcXHBoaSIsMix7Im9mZnNldCI6LTMsInNob3J0ZW4iOnsic291cmNlIjoyMCwidGFyZ2V0IjoyMH19XSxbNSw3LCJcXHBoaSciLDAseyJvZmZzZXQiOjMsInNob3J0ZW4iOnsic291cmNlIjoyMCwidGFyZ2V0IjoyMH19XV0=
\[\begin{tikzcd}[sep=small]
	& {F(i)} \\
	{*} && {F(k)} \\
	& {F(i')}
	\arrow[""{name=0, anchor=center, inner sep=0}, "a", from=2-1, to=1-2]
	\arrow[""{name=1, anchor=center, inner sep=0}, "{F(s)}", from=1-2, to=2-3]
	\arrow[""{name=2, anchor=center, inner sep=0}, "{a'}"', from=2-1, to=3-2]
	\arrow[""{name=3, anchor=center, inner sep=0}, "{F(s')}"', from=3-2, to=2-3]
	\arrow["\phi"', shift left=5, shorten <=4pt, shorten >=4pt, Rightarrow, from=0, to=2]
	\arrow["{\phi'}", shift right=5, shorten <=4pt, shorten >=4pt, Rightarrow, from=1, to=3]
\end{tikzcd}\]
such that $ (s,s',\phi)$ and $ (s,s',\phi')$ provides representants for the two morphisms above in $\bicolim_{\Sigma} F \iota_\Sigma$; but now, they are moreover coequified by $ q_k : F(k) \rightarrow \Sigma \bicolim_{I} F $ by hypothesis. Hence by \Cref{coequification lemma} there is a $ t : k \rightarrow l$ in $\Sigma$ such that $ F(t)(\phi) = F(t)(\phi')$. But this implies in particular that $ (s,s',\phi)$ and $ (s,s',\phi')$ are already identified in the bicolimit $ \bicolim_{\Sigma} F \iota_\Sigma$. Hence $q$ is faithful. This achieves to prove the equivalence of categories between the $\sigma$-bicolimit and the corresponding restricted bicolimit, and we hope at least the reader is convinced. 
\end{division} 

\begin{corollary}\label{bifiltered bicocompleteness entails sigma filtered sigma cocompleteness} 
A 2-category $\mathcal{C}$ has bifiltered bicolimits if and only if it has $\sigma$-filtered $\sigma$-bicolimits. Moreover a pseudofunctor preserves bifiltered bicolimits if and only if it preserves $ \sigma$-filtered $\sigma$-bicolimits. 
\end{corollary}

\begin{proof}
It is clear that cocompleteness under $\sigma$-filtered $\sigma$-bicolimits entails cocompleteness under bifiltered bicolimits as the latter are instances of the first; same argument for preservation. But now, suppose one has a $\sigma$-filtered pair $ (I,\Sigma)$ and $ F : I \rightarrow \mathcal{C}$. Then for $ \iota_\Sigma : \Sigma \hookrightarrow I$ is $\sigma$-cofinal relatively to $\Sigma$ and $ \Sigma$ is bifiltered, we can compute the restricted bifiltered bicolimit $ \bicolim F \iota_\Sigma $ in $\mathcal{C}$, and it provides a $\sigma$-bicolimit over $ I$. Now if $ G : \mathcal{C} \rightarrow \mathcal{D}$ preserves bifiltered bicolimits:  
\end{proof}

\begin{comment}
Let $(I,\Sigma)$ be a $\sigma$-filtered pair. Suppose moreover that $I$ has finitely weighted bicolimits and \begin{itemize}
    \item bicoproduct inclusions are always in $\Sigma$;
    \item for a parallel pair $ d,s : i \rightrightarrows j$ with $s \in \Sigma$, then the bicoinserter inclusion $ q_{(d,s)} : j \rightarrow \textbf{biCoins}(d,s)$ also is in $ \Sigma$;
    \item for a 2-cell $ \alpha : s \Rightarrow t$ with $ s,t : i \rightrightarrows j$ such that $s, t \in \Sigma$, then the bicoinverter inclusion $ q_{\alpha} : j \rightarrow \textbf{biCoinv}(\alpha)$ also is in $ \Sigma$;
    \item for parallel 2-cells $ \alpha, \alpha' : d \Rightarrow s$ with $ d,s : i \rightarrow j$ and $s \in \Sigma$, the bicoequifier inclusion $ q_{(\alpha,\alpha')} : j \rightarrow \textbf{biCoeq}(\alpha, \alpha') $ also is in $ \Sigma$;
\end{itemize}
Then the full on 0-cells and 2-cells subcategory $ I\!\! \mid_\Sigma $ whose objects are the same of $I$ and 1-cells are those in $\Sigma$, is bifiltered and the inclusion $ I\!\! \mid_\Sigma \, \hookrightarrow I$ is $\sigma$-cofinal relatively to $ \Sigma$ and $(I\!\! \mid_\Sigma)^{2} $. \begin{remark}
As a consequence, $\sigma$-filtered pairs having this property are controlled by a bifiltered part, and any $\sigma$-bicolimit over them reduce to an ordinary conical bicolimit over this bifiltered part. 
\end{remark}
\end{comment}

We should end this section by recalling the following crucial property, which can be found at \cite{descotte2018sigma}[2.7.3] and also \cite{descotte2016exactness}[Theorem 3.2] - again, it is stated there for 2-functors yet it works for pseudofunctors:

\begin{proposition}
$\sigma$-filtered $\sigma$-colimits commute with finitely weighted bilimits in $\Cat$: for $ (I,\Sigma)$ a $\sigma$-filtered pair and $ J$ a finite category with $ W : J \rightarrow \Cat$ a finite weight, if $ F : I \times J \rightarrow \Cat$ is a pseudofunctor, then one has
\[  \underset{i \in I}{\Sigma\bicolim}\, \underset{j \in J}{\bilim^{W}} \, F(i,j) \simeq  \underset{j \in J}{\bilim^{W}} \underset{i \in I}{\Sigma\bicolim} \, F(i,j)   \]
In particular, bifiltered bicolimits commute with finitely weighted bilimits in $\Cat$. 
\end{proposition}

\section{Bi-accessible and bipresentable 2-categories} \label{sectionaccessible}

This core section introduces 2-dimensional analogs of accessible and locally presentable categories. Our notions are closely related to \cite{kelly1982structures} definition, up to the difference that the latter is stated in an enriched context with strict constructions, while ours is suited for a weaker context; in particular, while \cite{kelly1982structures} directly use 1-dimensional filtered colimit in 2-categories together with strict 2-limits or colimits, we use bifiltered bicolimits for our definition of presentability.

%We should also emphasize that bifilteredness, rather than $\sigma$-filteredness, can be chosen to develop the $2$-dimensional theory of accessibility thanks to \Cref{the pseudococone is bifiltered}.

As opposed to the very 1-dimensional and strict version of \cite{kelly1982structures}, we build here a $2$-dimensional theory of accessibility from the more genuinely 2-dimensional approach of \cite{descotte2018sigma} in order to connect our definition with their theory of flat pseudofunctors. However, we should emphasize the fact that, though this later formalism relies on $\sigma$-filteredness, we can foresee from \Cref{key lemma} that we can reduce to only consider bifiltered bicolimits, as we will see in several instances as \Cref{sigmacompact are bicompact} and \Cref{the pseudococone is bifiltered}. 

\subsection{Bicompact objects}

This subsection is devoted to the properties of our 2-dimensional analog of finitely presented objects \cite{adamek1994locally} - also known as \emph{compact} objects, a terminology we prefer in order to avoid confusion with bipresentability - for we consider \emph{bicompact categories} we will avoid to confuse with presentable categories in the ordinary sense. As well as ordinary compact objects, bicompact objects will be defined through a lifting property relatively to bifiltered bicolimits.

\begin{definition}[Bicompactness] \label{bicompact}
An object $ K$ in a 2-category $\mathcal{B}$ is said to be (finitely) \emph{bicompact} if for any bifiltered 2-category $I$ and any pseudofunctor $F : I \rightarrow \mathcal{B}$, the functor induced from composing with bicolimits inclusions provides us with an equivalence of categories 
\[ \underset{i \in I}{\bicolim}\, \mathcal{B}[K,F(i)] \simeq \mathcal{B}[K, {\underset{I}{\bicolim} \; F}]   \]
\end{definition}

\begin{remark}
There is no difference with defining the lifting property of the bicompact against 2-functors rather than pseudofunctors: one can check that any object having the property above against 2-functors satisfies it automatically for pseudofunctors. This is because the conditions in bicompactness do not interact in a particular manner with unit and composition of pseudofunctors. 
\end{remark}

\begin{remark}\label{infinitary compactnes}
As in \Cref{infinite filteredness}, we can define $\lambda,\sigma$-compactness by preservation of $\lambda$-bifiltered bicolimits.
\end{remark}

\begin{division}[A concrete description of bicompactness: $1$-cells]\label{lifting property of sigma-compact}
Unravelling this definition gives us the explicit definition of a bicompact object. For any $ a : K \rightarrow \bicolim_{i \in I} F(i)$ we can pick some $ b : K \rightarrow F(i) $ and an invertible 2-cell  
% https://q.uiver.app/?q=WzAsMyxbMCwxLCJLIl0sWzEsMCwiRihpKSJdLFsxLDEsIlxcdW5kZXJzZXR7aSBcXGluIEl9e1xccHNjb2xpbX0gXFwsIEYoaSkiXSxbMCwxLCJiIl0sWzEsMiwicV9pIl0sWzAsMiwiYSIsMl0sWzMsMiwiXFxiZXRhIFxcYXRvcCBcXHNpbWVxIiwxLHsic2hvcnRlbiI6eyJzb3VyY2UiOjIwfSwic3R5bGUiOnsiYm9keSI6eyJuYW1lIjoibm9uZSJ9LCJoZWFkIjp7Im5hbWUiOiJub25lIn19fV1d
\[\begin{tikzcd}
	& {F(i)} \\
	K & {\underset{I}{\bicolim} \;  F}
	\arrow[""{name=0, anchor=center, inner sep=0}, "b", from=2-1, to=1-2]
	\arrow["{q_i}", from=1-2, to=2-2]
	\arrow["a"', from=2-1, to=2-2]
	\arrow["{\beta \atop \simeq}"{description, pos=0.65}, Rightarrow, draw=none, from=0, to=2-2]
\end{tikzcd}\]
and moreover, any such two choices $ (b, \beta)$ and $(b', \beta')$ of lifts over $a$ must be isomorphic in the bicolimit, which means that there is some $ i'' \in I$ and $ d : i \rightarrow i''$ and $ d': i' \rightarrow i''$ together with an invertible 2-cell 
% https://q.uiver.app/?q=WzAsNCxbMCwxLCJLIl0sWzEsMCwiRihpKSJdLFsxLDIsIkYoaScpIl0sWzIsMSwiRihpJycpIl0sWzAsMSwiYiJdLFswLDIsImInIiwyXSxbMSwyLCIiLDEseyJzdHlsZSI6eyJib2R5Ijp7Im5hbWUiOiJub25lIn0sImhlYWQiOnsibmFtZSI6Im5vbmUifX19XSxbMSwzLCJGKGQpIiwxXSxbMiwzLCJGKGQnKSIsMV0sWzAsMywiXFxnYW1tYSBcXGF0b3AgXFxzaW1lcSIsMSx7InN0eWxlIjp7ImJvZHkiOnsibmFtZSI6Im5vbmUifSwiaGVhZCI6eyJuYW1lIjoibm9uZSJ9fX1dXQ==
\[\begin{tikzcd}
	& {F(i)} \\
	K && {F(i'')} \\
	& {F(i')}
	\arrow["b", from=2-1, to=1-2]
	\arrow["{b'}"', from=2-1, to=3-2]
	\arrow[draw=none, from=1-2, to=3-2]
	\arrow["{F(d)}", from=1-2, to=2-3]
	\arrow["{F(d')}"', from=3-2, to=2-3]
	\arrow["{\gamma \atop \simeq}"{description}, draw=none, from=2-1, to=2-3]
\end{tikzcd}\]
whose pasting with the canonical 2-cells of the bicolimit provides with a 2-cell
% https://q.uiver.app/?q=WzAsNSxbMCwxLCJLIl0sWzEsMCwiRihpKSJdLFszLDEsIlxcdW5kZXJzZXR7aSBcXGluIEl9e1xccHNjb2xpbX0gXFwsIEYoaSkiXSxbMSwyLCJGKGknKSJdLFsyLDEsIkYoaScnKSJdLFswLDEsImIiXSxbMSwyLCJxX2kiLDAseyJjdXJ2ZSI6LTJ9XSxbMCwzLCJiJyIsMl0sWzMsMiwicV97aSd9IiwyLHsiY3VydmUiOjN9XSxbMSwzLCIiLDEseyJzdHlsZSI6eyJib2R5Ijp7Im5hbWUiOiJub25lIn0sImhlYWQiOnsibmFtZSI6Im5vbmUifX19XSxbNCwyLCJxX3tpJyd9IiwxXSxbMSw0LCJGKGQpIiwxXSxbMyw0LCJGKGQnKSIsMV0sWzAsNCwiXFxnYW1tYSBcXGF0b3AgXFxzaW1lcSIsMSx7InN0eWxlIjp7ImJvZHkiOnsibmFtZSI6Im5vbmUifSwiaGVhZCI6eyJuYW1lIjoibm9uZSJ9fX1dLFs0LDYsIlxcdGhldGFfZCBcXGF0b3AgXFxzaW1lcSIsMSx7InNob3J0ZW4iOnsidGFyZ2V0IjoyMH0sInN0eWxlIjp7ImJvZHkiOnsibmFtZSI6Im5vbmUifSwiaGVhZCI6eyJuYW1lIjoibm9uZSJ9fX1dLFs0LDgsIlxcdGhldGFfe2QnfSBcXGF0b3AgXFxzaW1lcSIsMSx7InNob3J0ZW4iOnsidGFyZ2V0IjoyMH0sInN0eWxlIjp7ImJvZHkiOnsibmFtZSI6Im5vbmUifSwiaGVhZCI6eyJuYW1lIjoibm9uZSJ9fX1dXQ==
\[\begin{tikzcd}
	& {F(i)} \\
	K && {F(i'')} & {\underset{I}{\bicolim} \, F} \\
	& {F(i')}
	\arrow["b", from=2-1, to=1-2]
	\arrow[""{name=0, anchor=center, inner sep=0}, "{q_i}", curve={height=-18pt}, from=1-2, to=2-4]
	\arrow["{b'}"', from=2-1, to=3-2]
	\arrow[""{name=1, anchor=center, inner sep=0}, "{q_{i'}}"', curve={height=18pt}, from=3-2, to=2-4]
	\arrow[draw=none, from=1-2, to=3-2]
	\arrow["{q_{i''}}"{description}, from=2-3, to=2-4]
	\arrow["{F(d)}"{description}, from=1-2, to=2-3]
	\arrow["{F(d')}"{description}, from=3-2, to=2-3]
	\arrow["{\gamma \atop \simeq}"{description}, draw=none, from=2-1, to=2-3]
	\arrow["{\theta_d \atop \simeq}"{description}, Rightarrow, draw=none, from=2-3, to=0]
	\arrow["{\theta_{d'} \atop \simeq}"{description}, Rightarrow, draw=none, from=2-3, to=1]
\end{tikzcd} \; = \; 
% https://q.uiver.app/?q=WzAsNCxbMCwxLCJLIl0sWzEsMCwiRihpKSJdLFsyLDEsIlxcdW5kZXJzZXR7aSBcXGluIEl9e1xccHNjb2xpbX0gXFwsIEYoaSkiXSxbMSwyLCJGKGknKSJdLFswLDEsImIiXSxbMSwyLCJxX2kiXSxbMCwyLCJhIiwxXSxbMCwzLCJiJyIsMl0sWzMsMiwicV97aSd9IiwyXSxbMSw2LCJcXGJldGEgXFxhdG9wIFxcc2ltZXEiLDEseyJzaG9ydGVuIjp7InNvdXJjZSI6MjB9LCJzdHlsZSI6eyJib2R5Ijp7Im5hbWUiOiJub25lIn0sImhlYWQiOnsibmFtZSI6Im5vbmUifX19XSxbMyw2LCJcXGJldGEnXFxhdG9wIFxcc2ltZXEiLDEseyJzaG9ydGVuIjp7InRhcmdldCI6MjB9LCJzdHlsZSI6eyJib2R5Ijp7Im5hbWUiOiJub25lIn0sImhlYWQiOnsibmFtZSI6Im5vbmUifX19XV0=
\begin{tikzcd}
	& {F(i)} \\
	K && {\underset{I}{\bicolim} \, F} \\
	& {F(i')}
	\arrow["b", from=2-1, to=1-2]
	\arrow["{q_i}", from=1-2, to=2-3]
	\arrow[""{name=0, anchor=center, inner sep=0}, "a"{description}, from=2-1, to=2-3]
	\arrow["{b'}"', from=2-1, to=3-2]
	\arrow["{q_{i'}}"', from=3-2, to=2-3]
	\arrow["{\beta \atop \simeq}"{description}, Rightarrow, draw=none, from=1-2, to=0]
	\arrow["{\beta'\atop \simeq}"{description}, Rightarrow, draw=none, from=3-2, to=0]
\end{tikzcd}   % https://q.uiver.app/?q=WzAsNCxbMCwxLCJLIl0sWzEsMCwiRihpKSJdLFsyLDEsIlxcdW5kZXJzZXR7aSBcXGluIEl9e1xccHNjb2xpbX0gXFwsIEYoaSkiXSxbMSwyLCJGKGknKSJdLFswLDEsImIiXSxbMSwyLCJxX2kiXSxbMCwzLCJiJyIsMl0sWzMsMiwicV97aSd9IiwyXSxbMSwzLCJcXGJldGEnIFxcYmV0YV57LTF9IFxcYXRvcCBcXHNpbWVxIiwxLHsic3R5bGUiOnsiYm9keSI6eyJuYW1lIjoibm9uZSJ9LCJoZWFkIjp7Im5hbWUiOiJub25lIn19fV1d
\]
\end{division}

\begin{division}[A concrete description of bicompactness: 2-cells]\label{A concrete description of sigma-compactness: 2-cells}

Moreover, for any 2-cell 
% https://q.uiver.app/?q=WzAsMixbMCwwLCJLIl0sWzIsMCwie1xcdW5kZXJzZXR7aSBcXGluIEl9e1xccHNjb2xpbX0gXFwsIEYoaSl9Il0sWzAsMSwiYSIsMCx7ImN1cnZlIjotM31dLFswLDEsImEnIiwyLHsiY3VydmUiOjN9XSxbMiwzLCJcXHBoaSIsMix7InNob3J0ZW4iOnsic291cmNlIjoyMCwidGFyZ2V0IjoyMH19XV0=
\[\begin{tikzcd}
	K \arrow[""{name=0, anchor=center, inner sep=0}, rr, bend left=20, "a", start anchor=45, end anchor=165]
	\arrow[""{name=1, anchor=center, inner sep=0, pos=0.51}, rr, "{a'}"', bend right=20, start anchor=-42, end anchor=190] && {\underset{I}{\bicolim} \, F}
	\arrow["\phi"', shorten <=3pt, shorten >=3pt, Rightarrow, from=0, to=1]
\end{tikzcd}\]
it suffices to paste it with the invertible 2-cells at two lifts as below to get a morphism between lifts
% https://q.uiver.app/?q=WzAsNCxbMCwxLCJLIl0sWzIsMSwie1xcdW5kZXJzZXR7aSBcXGluIEl9e1xccHNjb2xpbX0gXFwsIEYoaSl9Il0sWzEsMCwiRihpKSJdLFsxLDIsIkYoaScpIl0sWzAsMSwiYSIsMCx7ImN1cnZlIjotMn1dLFswLDEsImEnIiwyLHsiY3VydmUiOjJ9XSxbMCwyLCJiIl0sWzIsMSwicV9pIiwwLHsiY3VydmUiOi0yfV0sWzAsMywiYiciLDJdLFszLDEsInFfe2knfSIsMix7ImN1cnZlIjoyfV0sWzQsNSwiXFxwaGkiLDIseyJzaG9ydGVuIjp7InNvdXJjZSI6MjAsInRhcmdldCI6MjB9fV0sWzcsNCwiXFxiZXRhIFxcYXRvcCBcXHNpbWVxIiwxLHsic2hvcnRlbiI6eyJzb3VyY2UiOjIwLCJ0YXJnZXQiOjIwfSwic3R5bGUiOnsiYm9keSI6eyJuYW1lIjoibm9uZSJ9LCJoZWFkIjp7Im5hbWUiOiJub25lIn19fV0sWzksNSwiXFxiZXRhJyBcXGF0b3AgXFxzaW1lcSIsMSx7InNob3J0ZW4iOnsic291cmNlIjoyMCwidGFyZ2V0IjoyMH0sInN0eWxlIjp7ImJvZHkiOnsibmFtZSI6Im5vbmUifSwiaGVhZCI6eyJuYW1lIjoibm9uZSJ9fX1dXQ==
\[\begin{tikzcd}[row sep=large]
	& {F(i)} \\
	K \arrow[""{name=0, anchor=center, inner sep=0}, rr, bend left=20, "a", start anchor=45, end anchor=165]
	\arrow[""{name=1, anchor=center, inner sep=0}, rr, "{a'}"', bend right=20, start anchor=-45, end anchor=190] && {\underset{I}{\bicolim} \, F} \\
	& {F(i')}
	\arrow["b", shift left=1, from=2-1, to=1-2, start anchor=50]
	\arrow[""{name=2, anchor=center, inner sep=0}, "{q_i}", curve={height=-12pt}, from=1-2, to=2-3]
	\arrow["{b'}"', shift right=1, from=2-1, to=3-2]
	\arrow[""{name=3, anchor=center, inner sep=0}, "{q_{i'}}"', curve={height=12pt}, from=3-2, to=2-3]
	\arrow["\phi"', shorten <=3pt, shorten >=3pt, Rightarrow, from=0, to=1]
	\arrow["{\beta \atop \simeq}"{description}, Rightarrow, draw=none, from=2, to=0]
	\arrow["{\beta' \atop \simeq}"{description}, Rightarrow, draw=none, from=3, to=1]
\end{tikzcd} =
% https://q.uiver.app/?q=WzAsNCxbMCwxLCJLIl0sWzIsMSwie1xcdW5kZXJzZXR7aSBcXGluIEl9e1xccHNjb2xpbX0gXFwsIEYoaSl9Il0sWzEsMCwiRihpKSJdLFsxLDIsIkYoaScpIl0sWzAsMiwiYiJdLFsyLDEsInFfaSJdLFswLDMsImInIiwyXSxbMywxLCJxX3tpJ30iLDJdLFsyLDMsIlxcYmV0YSdcXHBoaVxcYmV0YV57LTF9IiwxLHsic2hvcnRlbiI6eyJzb3VyY2UiOjIwLCJ0YXJnZXQiOjIwfSwibGV2ZWwiOjJ9XV0=
\begin{tikzcd}
	& {F(i)} \\
	K && {\underset{I}{\bicolim} \, F} \\
	& {F(i')}
	\arrow["b", from=2-1, to=1-2]
	\arrow["{q_i}", from=1-2, to=2-3]
	\arrow["{b'}"', from=2-1, to=3-2]
	\arrow["{q_{i'}}"', from=3-2, to=2-3]
	\arrow["{\beta'\phi\beta^{-1}}"{description}, shorten <=6pt, shorten >=6pt, Rightarrow, from=1-2, to=3-2]
\end{tikzcd}\]
But now from the expression of the bifiltered bicolimit in $\Cat$ at \Cref{sigma-filtered sigma-colimit of categories}, we know that the functoriality of the bicompactness condition tells us that there exists some $ d : i \rightarrow i''$ and $ d' : i' \rightarrow i''$ in $I$ together with a 2-cell 
% https://q.uiver.app/?q=WzAsNCxbMCwxLCJLIl0sWzEsMCwiRihpKSJdLFsxLDIsIkYoaScpIl0sWzIsMSwiRihpJycpIl0sWzAsMSwiYiJdLFswLDIsImInIiwyXSxbMSwyLCJcXHBzaSIsMix7InNob3J0ZW4iOnsic291cmNlIjoyMCwidGFyZ2V0IjoyMH0sImxldmVsIjoyfV0sWzEsMywiRihkKSIsMV0sWzIsMywiRihkJykiLDFdXQ==
\[\begin{tikzcd}
	& {F(i)} \\
	K && {F(i'')} \\
	& {F(i')}
	\arrow["b", from=2-1, to=1-2]
	\arrow["{b'}"', from=2-1, to=3-2]
	\arrow["\psi"', shorten <=6pt, shorten >=6pt, Rightarrow, from=1-2, to=3-2]
	\arrow["{F(d)}", from=1-2, to=2-3]
	\arrow["{F(d')}"', from=3-2, to=2-3]
\end{tikzcd}\]
such that we have an equality between the following pasting 
% https://q.uiver.app/?q=WzAsNSxbMCwxLCJLIl0sWzMsMSwie1xcdW5kZXJzZXR7aSBcXGluIEl9e1xccHNjb2xpbX0gXFwsIEYoaSl9Il0sWzEsMCwiRihpKSJdLFsxLDIsIkYoaScpIl0sWzIsMSwiRihpJycpIl0sWzAsMiwiYiJdLFsyLDEsInFfaSIsMCx7ImN1cnZlIjotMn1dLFswLDMsImInIiwyXSxbMywxLCJxX3tpJ30iLDIseyJjdXJ2ZSI6Mn1dLFsyLDMsIlxccHNpIFxcRG93bmFycm93IiwxLHsic2hvcnRlbiI6eyJzb3VyY2UiOjIwLCJ0YXJnZXQiOjIwfSwibGV2ZWwiOjIsInN0eWxlIjp7ImJvZHkiOnsibmFtZSI6Im5vbmUifSwiaGVhZCI6eyJuYW1lIjoibm9uZSJ9fX1dLFsyLDQsIkYoZCkiLDFdLFszLDQsIkYoZCcpIiwxXSxbNCwxLCJxX3tpJyd9IiwxXSxbOCw0LCJxX3tkJ30gXFxhdG9wIFxcc2ltZXEiLDEseyJzaG9ydGVuIjp7InNvdXJjZSI6MjB9LCJzdHlsZSI6eyJib2R5Ijp7Im5hbWUiOiJub25lIn0sImhlYWQiOnsibmFtZSI6Im5vbmUifX19XSxbNiw0LCJxX2QgXFxhdG9wIFxcc2ltZXEiLDEseyJzaG9ydGVuIjp7InNvdXJjZSI6MjB9LCJzdHlsZSI6eyJib2R5Ijp7Im5hbWUiOiJub25lIn0sImhlYWQiOnsibmFtZSI6Im5vbmUifX19XV0=
\[\begin{tikzcd}
	& {F(i)} \\
	K && {F(i'')} & {\underset{I}{\bicolim} \, F} \\
	& {F(i')}
	\arrow["b", from=2-1, to=1-2]
	\arrow[""{name=0, anchor=center, inner sep=0}, "{q_i}", curve={height=-18pt}, from=1-2, to=2-4]
	\arrow["{b'}"', from=2-1, to=3-2]
	\arrow[""{name=1, anchor=center, inner sep=0}, "{q_{i'}}"', curve={height=18pt}, from=3-2, to=2-4]
	\arrow["{\psi}"', Rightarrow, shorten <=6pt, shorten >=6pt, from=1-2, to=3-2]
	\arrow["{F(d)}"{description}, from=1-2, to=2-3]
	\arrow["{F(d')}"{description}, from=3-2, to=2-3]
	\arrow["{q_{i''}}"{description}, from=2-3, to=2-4]
	\arrow["{q_{d'} \atop \simeq}"{description}, Rightarrow, draw=none, from=1, to=2-3]
	\arrow["{q_d \atop \simeq}"{description}, Rightarrow, draw=none, from=0, to=2-3]
\end{tikzcd}
=
\begin{tikzcd}
	& {F(i)} \\
	K && {\underset{I}{\bicolim} \, F} \\
	& {F(i')}
	\arrow["b", from=2-1, to=1-2]
	\arrow["{q_i}", from=1-2, to=2-3]
	\arrow["{b'}"', from=2-1, to=3-2]
	\arrow["{q_{i'}}"', from=3-2, to=2-3]
	\arrow["{\beta'\phi\beta^{-1} }"{description}, shorten <=6pt, shorten >=6pt, Rightarrow, from=1-2, to=3-2]
\end{tikzcd} \]
\end{division}

% \begin{remark}
% It is interesting to remark that those lifting properties do not actually involve 2-cells of $I$; this suggests that bicompact object do not distinguish diagrams indexed by 1 or 2-categories as long as they are bifiltered.
% \end{remark}

\begin{lemma}[Lifts of parallel 2-cells]\label{Lifts of parallel 2-cells}
Similarly, one can lift parallels pairs of 2-cells into a parallel pair between the sames lifts: for any parallel pair of 2-cells of the form %Let $K$ be bicompact and $ (I,\Sigma)$ a $\sigma$-filtered pair; 
% https://q.uiver.app/?q=WzAsMixbMCwwLCJLIl0sWzIsMCwiXFx1bmRlcnNldHtpIFxcaW4gSX17XFxzaWdtYV9cXFNpZ21hXFxiaWNvbGltfSBcXCwgRihpKSJdLFswLDEsImIiLDIseyJjdXJ2ZSI6Mn1dLFswLDEsImEiLDAseyJjdXJ2ZSI6LTJ9XSxbMywyLCJcXHBoaSIsMix7Im9mZnNldCI6Mywic2hvcnRlbiI6eyJzb3VyY2UiOjIwLCJ0YXJnZXQiOjIwfX1dLFszLDIsIlxccHNpIiwwLHsib2Zmc2V0IjotMywic2hvcnRlbiI6eyJzb3VyY2UiOjIwLCJ0YXJnZXQiOjIwfX1dXQ==
\[\begin{tikzcd}
	K && {\underset{i \in I}{\bicolim} \, F(i)}
	\arrow[""{name=0, anchor=center, inner sep=0, pos=0.51}, "a'"', start anchor=-40, end anchor=190, bend right=25, from=1-1, to=1-3]
	\arrow[""{name=1, anchor=center, inner sep=0}, "a", start anchor=42, end anchor=168, bend left=25, from=1-1, to=1-3]
	\arrow["\phi"', shift right=3, shorten <=3pt, shorten >=3pt, Rightarrow, from=1, to=0]
	\arrow["\psi", shift left=3, shorten <=3pt, shorten >=3pt, Rightarrow, from=1, to=0]
\end{tikzcd}\]
there exists a span $ d : i \rightarrow j$, $d' : i' \rightarrow j $ together with $ b : K \rightarrow F(i)$, $b' : K \rightarrow F(i')$ and invertible 2-cells $ \beta : a \simeq F(d) b $, $ \beta' : a' \simeq F(d')b'$, and parallel 2-cells $ \zeta, \xi : F(d)b \Rightarrow F(d')b' $ such that $ q_d^{-1} q_j*\zeta q_{d'} = \beta'\phi \beta^{-1}$ and $ q_d^{-1} q_j*\xi q_{d'} = \beta'\psi \beta^{-1}$.
\end{lemma}
\begin{proof}
These $ \zeta$ and $\xi$ can be constructed as follows: take lifts $(b_\phi, b_\phi', i_\phi, i_\phi', j_\phi, d_\phi, d'_\phi, \alpha_\phi) $ and $(b_\psi, b_\psi', i_\psi, i_\psi', j_\psi, d_\psi, d'_\psi, \alpha_\psi) $ as in \Cref{A concrete description of sigma-compactness: 2-cells}: then one has a common refinement $ t_\phi : j_\phi \rightarrow j$, $t_\psi : j_\psi \rightarrow j$, and moreover this common refinement can be chosen such that there is also invertible 2-cells $ \gamma' : F(t_\phi) F(d'_\phi) b'_\phi \simeq F(t_\psi) F(d'_\psi) b'_\psi$, which altogether provides the following two parallel composites 2-cells one can choose as the desired $ \zeta, \xi$
% https://q.uiver.app/?q=WzAsNyxbMCwxLCJLIl0sWzEsMCwiRihpX1xccGhpKSJdLFsxLDEsIkYoaV9cXHBoaScpIl0sWzIsMSwiRihqX1xccGhpKSJdLFszLDEsIkYoaikiXSxbMSwyLCJGKGknX1xccHNpKSJdLFsyLDIsIkYoal9cXHBzaSkiXSxbMCwxLCJiX1xccGhpIl0sWzAsMiwiYidfXFxwaGkiLDFdLFsxLDIsIlxcYWxwaGFfXFxwaGkiLDIseyJzaG9ydGVuIjp7InNvdXJjZSI6MjAsInRhcmdldCI6MjB9LCJsZXZlbCI6Mn1dLFsxLDMsIkYoc19cXHBoaSkiLDFdLFszLDQsIkYodF9cXHBoaSkiXSxbNiw0LCJGKHRfXFxwc2kpIiwxXSxbNSw2LCJGKGRfXFxwc2kpIiwyXSxbMCw1LCJiJ19cXHBzaSIsMl0sWzIsMywiRihkX1xccGhpKSIsMV0sWzE1LDEzLCJcXGdhbW1hJyBcXGF0b3AgXFxzaW1lcSIsMSx7InNob3J0ZW4iOnsic291cmNlIjoyMCwidGFyZ2V0IjoyMH0sInN0eWxlIjp7ImJvZHkiOnsibmFtZSI6Im5vbmUifSwiaGVhZCI6eyJuYW1lIjoibm9uZSJ9fX1dXQ==
\[\begin{tikzcd}
	& {F(i_\phi)} \\
	K & {F(i_\phi')} &[10pt] {F(j_\phi)} & {F(j)} \\
	& {F(i'_\psi)} & {F(j_\psi)}
	\arrow["{b_\phi}", from=2-1, to=1-2]
	\arrow["{b'_\phi}"{description}, from=2-1, to=2-2]
	\arrow["{\alpha_\phi}"', shorten <=2pt, shorten >=2pt, Rightarrow, from=1-2, to=2-2]
	\arrow["{F(d_\phi)}", from=1-2, to=2-3]
	\arrow["{F(t_\phi)}", from=2-3, to=2-4]
	\arrow["{F(t_\psi)}"{description}, from=3-3, to=2-4]
	\arrow[""{name=0, anchor=center, inner sep=0}, "{F(d'_\psi)}"', from=3-2, to=3-3]
	\arrow["{b'_\psi}"', from=2-1, to=3-2]
	\arrow[""{name=1, anchor=center, inner sep=0}, "{F(d'_\phi)}"{description}, from=2-2, to=2-3]
	\arrow["{\gamma' \atop \simeq}"{description}, Rightarrow, draw=none, from=1, to=0]
\end{tikzcd} \hskip0.7cm % https://q.uiver.app/?q=WzAsNyxbMCwxLCJLIl0sWzEsMCwiRihpX1xccGhpKSJdLFsyLDAsIkYoal9cXHBoaSkiXSxbMywxLCJGKGopIl0sWzEsMiwiRihpJ19cXHBzaSkiXSxbMiwxLCJGKGpfXFxwc2kpIl0sWzEsMSwiRihpX1xccHNpKSJdLFswLDEsImJfXFxwaGkiLDFdLFsyLDMsIkYodF9cXHBoaSkiXSxbNSwzLCJGKHRfXFxwc2kpIiwyXSxbNCw1LCJGKGRfXFxwc2kpIiwyXSxbMCw0LCJiJ19cXHBzaSIsMl0sWzEsMiwiRihzX1xccGhpKSJdLFsxLDUsIlxcZ2FtbWEnIFxcYXRvcCBcXHNpbWVxIiwxLHsic2hvcnRlbiI6eyJzb3VyY2UiOjIwLCJ0YXJnZXQiOjIwfSwibGV2ZWwiOjIsInN0eWxlIjp7ImJvZHkiOnsibmFtZSI6Im5vbmUifSwiaGVhZCI6eyJuYW1lIjoibm9uZSJ9fX1dLFswLDYsImJfXFxwc2kiLDFdLFs2LDUsIkYoc19cXHBzaSkiLDFdLFs2LDQsIlxcYWxwaGFfXFxwc2kiLDIseyJzaG9ydGVuIjp7InNvdXJjZSI6MjAsInRhcmdldCI6MTB9LCJsZXZlbCI6Mn1dXQ==
\begin{tikzcd}
	& {F(i_\phi)} &[10pt] {F(j_\phi)} \\
	K & {F(i_\psi)} & {F(j_\psi)} & {F(j)} \\
	& {F(i'_\psi)}
	\arrow["{b_\phi}"{description}, from=2-1, to=1-2]
	\arrow["{F(t_\phi)}", from=1-3, to=2-4]
	\arrow["{F(t_\psi)}"', from=2-3, to=2-4]
	\arrow["{F(d'_\psi)}"', from=3-2, to=2-3]
	\arrow["{b'_\psi}"', from=2-1, to=3-2]
	\arrow["{F(d_\phi)}", from=1-2, to=1-3]
	\arrow["{\gamma' \atop \simeq}"{description}, Rightarrow, draw=none, from=1-2, to=2-3]
	\arrow["{b_\psi}"{description}, from=2-1, to=2-2]
	\arrow["{F(d_\psi)}"{description}, from=2-2, to=2-3]
	\arrow["{\alpha_\psi}"', shorten <=2pt, shorten >=1pt, Rightarrow, from=2-2, to=3-2]
\end{tikzcd}\]
\end{proof}

\begin{remark}
One could also be tempted to define a notion of \emph{$\sigma$-compact} objects as having the analogous property relatively to $\sigma$-filtered colimits. From \Cref{sigmafiltered colimits are bifiltered colimits} such a notion would be redundant:
\end{remark}

\begin{proposition}\label{sigmacompact are bicompact}
Bicompact have the lifting property relative to $\sigma$-filtered $\sigma$-bicolimits: an object $ K$ of a 2-category $ \mathcal{B}$ is bicompact if and only if for any $\sigma$-filtered pair $(I, \Sigma)$ and any $F: I \rightarrow \mathcal{B} $ one has 
\[ \mathcal{B}[K, {\underset{I}{\Sigma\bicolim} \;  F}] \simeq \underset{i \in I}{ \Sigma \bicolim}\; \mathcal{B}[K,F(i)] \]
\end{proposition}

\begin{proof}
It is clear that such a property implies bicompactness; if now $ K$ is bicompact, then combining \Cref{sigmafiltered colimits are bifiltered colimits} with \Cref{key lemma} one has 
 \begin{align*}
     \mathcal{B}[K, {\underset{I}{\Sigma\bicolim} \; F}] &\simeq \mathcal{B}[K, {\underset{\Sigma}{\bicolim} \; F \iota_\Sigma}] \\ 
     &\simeq \underset{i \in \Sigma}{\bicolim}\; \mathcal{B}[K,F(i)] \\
     &\simeq \underset{i \in I}{ \Sigma \bicolim}\; \mathcal{B}[K,F(i)]
 \end{align*}  
\end{proof}

\begin{remark}[A comparison with the literature] \label{literaturefinpres}
Two notions of finitely presentable categories had historically been introduced:
\begin{itemize}
    \item In \cite{street1976limits}, they are defined via the notion of computad, see \cite[between Thm. 3 and Prop. 4]{street1976limits}. Notably, finite categories are finitely presentable in the sense of Street, and they are also in our sense (see \Cref{finite cat are sigma-compact}). It would be very interesting to have a similar presentation of our notion of compactness, especially in the direction of providing a characterization as in \cite[Prop. 4]{street1976limits}
    \item In \cite{kelly1982structures}, Kelly defined finitely presentable categories in terms of filtered colimits, to be precise, as to this source an object $X$ is finitely presentable if the functor $\Cat[X,-]$ preserves \textit{conical} filtered colimits (see \cite[1.1 and 2.1]{kelly1982structures}). This notion is too strict to be compared with ours, but its bi-version would be comparable, because of \Cref{trivial}. 
\end{itemize}
\end{remark}

\begin{lemma}\label{finite cat are sigma-compact}
In $\Cat$, finite categories are bicompact. 
\end{lemma}

\begin{proof}
For any finite category $C$ and $ F : I \rightarrow \Cat$ (which can be chosen as a strict 2-functor) with $ I$ bifiltered, a functor $ G : I \rightarrow \pscolim \, F$ can be described as picking for each $ c \in C$ an object $ (i,x)$ with $ i \in I$ and $ x \in F(i)$ in $\oplaxcolim \, F$ such that any two such choices for the same $c$ have to live in the same isomorphism class in the localization $ \pscolim \, F = \Pi_1 \oplaxcolim \, F[\Cart_F^{-1}]$; but as $ C$ is finite, if one picks $ i_c$ for each $c$ in $\mathcal{C}$ there exists some $ i_0 $ and a family of arrows $ d_c : i_c \rightarrow i_0 $ in $I$. Moreover, for any $ g: c \rightarrow c'$ in $\mathcal{C}$, we can pick a representant $(f_g, \phi_g)$ with $ f_g : i_c \rightarrow i_{c'}$ and $ \phi_g : F(f)(x_c) \rightarrow x_{c'}$ for a convenient choice of $ (i_c, x_c)$ and $(i_{c'}, x_{c'})$. Then one can define an assignment $ H : C \rightarrow F(i_0) $ sending $ c$ on $(i_0,F(d_c)(x_c))$. In fact $H$ may not yet be a functor, but one can use bifilteredness of $I$ to correct it into a functor: for a map $ g : c \rightarrow c'$ with a representant $ (f_g, \phi_g)$ we have in $ i_0$ a map 
% https://q.uiver.app/?q=WzAsMixbMCwwLCJGKGRfe2MnfSlGKGZfZykoeF97Y30pIl0sWzIsMCwiRihkX3tjJ30pKHhfe2MnfSkiXSxbMCwxLCJGKGRfe2MnfSkoXFxwaGlfZykiXV0=
\[\begin{tikzcd}
	{F(d_{c'})F(f_g)(x_{c})} && {F(d_{c'})(x_{c'})}
	\arrow["{F(d_{c'})(\phi_g)}", from=1-1, to=1-3]
\end{tikzcd}\]
Now we have a parallel pair in $\oplaxcolim\, F$
% https://q.uiver.app/?q=WzAsMyxbMCwwLCIoaV9jLCB4X2MpIl0sWzIsMCwiKGlfe2MnfSwgRihmX2cpKHhfe2N9KSkiXSxbMSwxLCIoaV8wLCBGKGRfe2MnfSlGKGZfZykoeF97Y30pKSJdLFswLDIsIihkX2MsRihkX3tjJ30pKFxccGhpX2cpKSIsMix7ImxhYmVsX3Bvc2l0aW9uIjoyMH1dLFswLDEsIihmX2csIDFfe0YoZl9nKSh4X3tjfSl9KSJdLFsxLDIsIihkX3tjJ30sIDFfe0YoZF97Yyd9KUYoZl9nKSh4X3tjfSl9KSIsMCx7ImxhYmVsX3Bvc2l0aW9uIjoxMH1dXQ==
\[\begin{tikzcd}
	{(i_c, x_c)} && {(i_{c'}, F(f_g)(x_{c}))} \\
	& {(i_0, F(d_{c'}f_g)(x_{c}))}
	\arrow["{(d_c,F(d_{c'})(\phi_g))}"', from=1-1, to=2-2]
	\arrow["{(f_g, 1_{F(f_g)(x_{c})})}", from=1-1, to=1-3]
	\arrow["{(d_{c'}, 1_{F(d_{c'}f_g)(x_{c})})}", from=1-3, to=2-2]
\end{tikzcd}\]
with the composite $ (d_{c'}, 1_{F(d_{c'})F(f_g)(x_{c})}) (f_g, 1_{F(f_g)(x_{c})})$ cocartesian. Hence there exists a 2-cell 
% https://q.uiver.app/?q=WzAsNCxbMCwwLCIoaV9jLCB4X2MpIl0sWzIsMCwiKGlfe2MnfSwgRihmX2cpKHhfe2N9KSkiXSxbMSwxLCIoaV8wLCBGKGRfe2MnfWZfZykoeF97Y30pKSJdLFsxLDMsIihpXzEsIEYoZl8xZF97Yyd9Zl9nKSh4X3tjfSkpIl0sWzAsMiwiKGRfYyxGKGRfe2MnfSkoXFxwaGlfZykpIiwxXSxbMCwxLCIoZl9nLCAxX3tGKGZfZykoeF97Y30pfSkiXSxbMSwyLCIoZF97Yyd9LCAxX3tGKGRfe2MnfWZfZykoeF97Y30pfSkiLDFdLFsyLDMsIihmXzEsMV97RihmXzFkX3tjJ31mX2cpKHhfe2N9KX0pIiwxLHsibGFiZWxfcG9zaXRpb24iOjcwfV0sWzEsMywiKGZfMWRfe2MnfSwgMV97RihmXzFkX3tjJ31mX2cpKHhfe2N9KX0pIiwwLHsiY3VydmUiOi00fV0sWzAsMywiKGZfMWRfYywgRihmXzFkX3tjJ30pKFxccGhpX2cpKSIsMix7ImN1cnZlIjo0fV0sWzksOCwicyIsMCx7ImxhYmVsX3Bvc2l0aW9uIjozMCwic2hvcnRlbiI6eyJzb3VyY2UiOjEwLCJ0YXJnZXQiOjEwfX1dXQ==
\[\begin{tikzcd}
	{(i_c, x_c)} && {(i_{c'}, F(f_g)(x_{c}))} \\
	& {(i_0, F(d_{c'}f_g)(x_{c}))} \\
	\\
	& {(i_1, F(f_1d_{c'}f_g)(x_{c}))}
	\arrow["{(d_c,F(d_{c'})(\phi_g))}"{description}, from=1-1, to=2-2]
	\arrow["{(f_g, 1_{F(f_g)(x_{c})})}", from=1-1, to=1-3]
	\arrow["{(d_{c'}, 1_{F(d_{c'}f_g)(x_{c})})}"{description}, from=1-3, to=2-2]
	\arrow["{(f_1,1_{F(f_1d_{c'}f_g)(x_{c})})}"{description, pos=0.7}, from=2-2, to=4-2]
	\arrow[""{name=0, anchor=center, inner sep=0}, "{(f_1d_{c'}, 1_{F(f_1d_{c'}f_g)(x_{c})})}", curve={height=-30pt}, from=1-3, to=4-2]
	\arrow[""{name=1, anchor=center, inner sep=0}, "{(f_1d_c, F(f_1d_{c'})(\phi_g))}"', curve={height=30pt}, from=1-1, to=4-2]
	\arrow["s"{pos=0.3}, shorten <=11pt, shorten >=11pt, Rightarrow, from=1, to=0, crossing over]
\end{tikzcd}\]
which provides an intermediate map
\[ F(f_1d_c)(x_c) \stackrel{s}{\rightarrow} F(f_1d_{c'}f_g)(x_{c}) \]
one can compose with $ F(d_{c'})(\phi_g)$ to get a map as desired
% https://q.uiver.app/?q=WzAsMyxbMCwwLCJGKGZfMWRfe2MnfWZfZykoeF97Y30pIl0sWzAsMSwiIEYoZl8xZF9jKSh4X2MpIl0sWzEsMCwiRihmXzFkX3tjJ30pKHhfe2MnfSkiXSxbMSwwLCJzIl0sWzAsMiwiRihkX3tjJ30pKFxccGhpX2cpIl0sWzEsMiwiIiwyLHsic3R5bGUiOnsiYm9keSI6eyJuYW1lIjoiZGFzaGVkIn19fV1d
\[\begin{tikzcd}
	{F(f_1d_{c'}f_g)(x_{c})} & {F(f_1d_{c'})(x_{c'})} \\
	{ F(f_1d_c)(x_c)}
	\arrow["s", from=2-1, to=1-1]
	\arrow["{F(d_{c'})(\phi_g)}", from=1-1, to=1-2]
	\arrow[dashed, from=2-1, to=1-2]
\end{tikzcd}\]
Hence it is mostly a technical effort to show that this defines an invertible 2-cell
% https://q.uiver.app/?q=WzAsMyxbMCwxLCJDIl0sWzEsMSwiXFxwc2NvbGltXFwsIEYiXSxbMSwwLCJGKGlfMSkiXSxbMCwyLCJIIl0sWzIsMSwicV97aV8xfSJdLFswLDEsIkciLDJdLFszLDEsIlxcYWxwaGEgXFxhdG9wIFxcc2ltZXEiLDEseyJzaG9ydGVuIjp7InNvdXJjZSI6MjB9LCJzdHlsZSI6eyJib2R5Ijp7Im5hbWUiOiJub25lIn0sImhlYWQiOnsibmFtZSI6Im5vbmUifX19XV0=
\[\begin{tikzcd}
	& {F(i_1)} \\
	C & {\pscolim\, F}
	\arrow[""{name=0, anchor=center, inner sep=0}, "H", from=2-1, to=1-2]
	\arrow["{q_{i_1}}", from=1-2, to=2-2]
	\arrow["G"', from=2-1, to=2-2]
	\arrow["{\alpha \atop \simeq}"{description, pos=0.6}, Rightarrow, draw=none, from=0, to=2-2]
\end{tikzcd}\]
showing the bicompactness of $C$. 
\end{proof}

\begin{remark}\label{Bicompact categories might not be finite}[Bicompact categories might not be finite]
We do not assert that all bicompact objects in $\Cat$ are finite. For example, the monoid $\mathbb{N}$ - seen as a $1$-object category - is the coinserter of the diagram below and thus is bicompact, being a finite bicolimit of bicompact objects by \Cref{finite bicolimit of sigma-compact are sigma-compact}.

% https://q.uiver.app/?q=WzAsMyxbMCwwLCJcXHtcXGJ1bGxldFxcfSJdLFsyLDAsIlxce1xcYnVsbGV0XFx9Il0sWzQsMCwiXFxtYXRoYmJ7Tn0iXSxbMCwxLCJcXHRleHR7aWR9IiwwLHsib2Zmc2V0IjotNX1dLFswLDEsIlxcdGV4dHtpZH0iLDIseyJvZmZzZXQiOjV9XSxbMSwyXV0=
\[\begin{tikzcd}
	{\{*\}} && {\{*\}} && {\mathbb{N}}
	\arrow["{\text{id}}", shift left=2, from=1-1, to=1-3]
	\arrow["{\text{id}}"', shift right=2, from=1-1, to=1-3]
	\arrow[from=1-3, to=1-5]
\end{tikzcd}\]

\end{remark}

\begin{comment}
A preliminary lecture of \cite{descotte2018sigma} theorem of extensions of flat functors, involving $\sigma$-filtered $\sigma$-colimits of representables, could lead to think one may need actually a notion of ``lax-compact" object preserving $\sigma$-filtered $ \sigma$-bicolimits. As pointed in the introduction, such a notion does not capture enough objects: for instance in $ \Cat$, even the walking arrow 2 fails to be lax compact in this sense - the reader may convince himself by transposing the concrete lifting properties above but against $\sigma$-filtered $\sigma$-bicolimits and see how one of the 2-cell involved in the pasting as above, being no longer invertible, exactly says that the property does not works with 2. If this seems at first sight to be a discrepancy between a bi-version of accessibility and Dubuc theory of flat pseudofunctors, we will actually manage to conciliate them later thanks to our key \Cref{key lemma}. 
\end{comment}

Bicompactness behaves as nicely as it can. We will now list some theorems that show how the theory of bicompact objects perfectly mirrors that of compact objects in \cite{adamek1994locally}. The following two results below deal with the closure properties of bicompact objects; the first one is the $2$-dimensional analog of \cite[1.3]{adamek1994locally}.

\begin{proposition}[Bicompact objects are closed under finite weighed bicolimits]\label{finite bicolimit of sigma-compact are sigma-compact}
Let $ \mathcal{B}$ be a 2-category, $ W: I^{\op} \rightarrow \Cat$ a finite weight and $G : I \rightarrow \mathcal{B}  $ a pseudofunctor such that each $ G(i)$ is bicompact. Then $ \bicolim^{W} \, G$ still is bicompact. 
\end{proposition}

\begin{proof}
Let $ F : J \rightarrow \mathcal{B}$ be a pseudofunctor with $J$ bifiltered. We have an equivalence of categories
\begin{align*}
     \mathcal{B}[\underset{I}\bicolim^{W} \; G, \underset{J}\bicolim\; F] &\simeq \ps[I^{\op}, \Cat]\big[ W, \mathcal{B}[G, \underset{J}\bicolim \; F]\big] \\
     &\simeq \ps[I^{\op}, \Cat]\big[ W, \underset{j \in J}{\bicolim}\;\mathcal{B}[G, F(j)]\big]
\end{align*}
where the last equivalence comes from the fact that each $ G(i)$ is bicompact. But now for any pseudonatural transformation $ \phi : W \Rightarrow {\bicolim}_{j \in J}\,\mathcal{B}[G, F(j)]$, each $ W(i)$ is a finite category, hence by \Cref{finite cat are sigma-compact}, is bicompact in $\Cat$, so we have a pseudonatural equivalence
\[ \Cat\big[ W(i), \underset{j \in J}{\bicolim}\,\mathcal{B}[G(i), F(j)]\big] \simeq \underset{j \in J}{\bicolim}\,\Cat\big[ W(i),
\mathcal{B}[G(i), F(j)]\big] \]
Now as pointwise natural equivalences are equivalences in $ \ps[I^{\op}, \Cat]$ this means we have an equivalence in $ \ps[I^{\op}, \Cat]$, which achieves to prove that $ \bicolim^{W}\, G$ is bicompact in $\mathcal{B}$.
\begin{align*}
    \ps[I^{\op}, \Cat]\big[ W, \underset{j \in J}{\bicolim}\,\mathcal{B}[G, F(j)]\big] &\simeq \underset{j \in J}{\bicolim}\,\ps[I^{\op}, \Cat]\big[ W, 
\mathcal{B}[G, F(j)]\big] \\
&\simeq \underset{j \in J}{\bicolim}\, \mathcal{B}[\underset{I}\bicolim^{W} \, G, F(j)]
\end{align*}

\end{proof}

\begin{comment}\label{compactly weighted bicolim of bicompacts}
The previous lemma admits actually a more general form. Let be $ I $ a finite category and $ W : I^{\op} \rightarrow \Cat$ such that each $W(i)$ is a bicompact category - not necessarily finite - and $ F : I \rightarrow \mathcal{B}  $ 2-functor such that each $ F(i)$ is bicompact in $\mathcal{B}$:  then $ \bicolim^W F$ still is bicompact. 
\end{comment}

\begin{definition}[Pseudoretracts]
 We recall that a \emph{pseudoretract}\index{pseudoretract} of an object $ A$ is an object $B$ equipped with an invertible 2-cell
% https://q.uiver.app/?q=WzAsMyxbMCwwLCJCIl0sWzIsMCwiQiJdLFsxLDEsIkEiXSxbMCwxLCIiLDAseyJsZXZlbCI6Miwic3R5bGUiOnsiaGVhZCI6eyJuYW1lIjoibm9uZSJ9fX1dLFswLDIsInIiLDJdLFsyLDEsInMiLDJdLFszLDIsIlxcYWxwaGEgXFxhdG9wIFxcc2ltZXEiLDEseyJzaG9ydGVuIjp7InNvdXJjZSI6MjB9LCJzdHlsZSI6eyJib2R5Ijp7Im5hbWUiOiJub25lIn0sImhlYWQiOnsibmFtZSI6Im5vbmUifX19XV0=
\[\begin{tikzcd}
	B && B \\
	& A
	\arrow[""{name=0, anchor=center, inner sep=0}, Rightarrow, no head, from=1-1, to=1-3]
	\arrow["r"', from=1-1, to=2-2]
	\arrow["s"', from=2-2, to=1-3]
	\arrow["{\alpha \atop \simeq}"{description}, Rightarrow, draw=none, from=0, to=2-2]
\end{tikzcd}\]
\end{definition}

\begin{corollary}
A pseudoretract of a bicompact is bicompact. 
\end{corollary}

Now we get back briefly to $\sigma$-bicolimits which are involved in the following auxiliary notion, we shall make use of when embedding bi-accessible 2-categories in 2-categories of pseudofunctors.  

\begin{definition}[$\sigma$-tiny]
An object $ K$ in a 2-category $\mathcal{B}$ is said to be \emph{$\sigma$-tiny}\index{$\sigma$-tiny} if for any $\sigma$-pair $(I,\Sigma)$ and $F : I \rightarrow \mathcal{B}$, we have an equivalence of category 
\[ \mathcal{B}[K, {\underset{I}{\Sigma\bicolim} \, F}] \simeq \underset{i \in I}{\Sigma\bicolim}\, \mathcal{B}[K,F(i)] \]
\end{definition}

\begin{remark}
In other words, a $\sigma$-tiny object preserves any small $\sigma$-bicolimits. Any $\sigma$-tiny object is in particular bicompact. As $\sigma$-tiny are tiny, the only $\sigma$-tiny object in $\Cat$ is the point category. 
\end{remark}

\subsection{Bi-accessible 2-categories}

\begin{definition}[Bi-accessible $2$-category] \label{sigma accessible cat}
A 2-category $\mathcal{B}$ will be said \emph{finitely bi-accessible} if \begin{itemize}
    \item $\mathcal{B}$ has bifiltered bicolimits,
    \item there is an essentially small full on $1$-cells and $2$-cells sub-2-category $ \mathcal{B}_0 \hookrightarrow \mathcal{B}$ consisting of bicompact objects such that for any $B$ in $\mathcal{B}$ is a bifiltered bicolimit of object in $\mathcal{B}_0$.
\end{itemize}
We refer to such a choice of $\mathcal{B}_0$ as \emph{a generator of bicompact objects}. As in \Cref{infinitary compactnes} and \Cref{infinite filteredness}, a  $\lambda$-bi-accessible $2$-category can be defined as above by replacing all the occurrences of compactness and filteredness with their $\lambda$-generalization.
\end{definition}

\begin{lemma}
If $ \mathcal{B}$ is finitely bi-accessible, then the full on $1$-cells and $2$-cells sub-2-category consisting of all bicompact objects is essentially small. 
\end{lemma}
\begin{proof}
Consider a bicompact object $X$. By definition of accessibility, this a bifiltered bicolimit of objects in $\mathcal{B}_0$, as above. Thus, the identity $X \to X$ has to factor through the bifiltered bicolimit. As a result, we obtain that $X$ is the pseudoequalizer of a pair of parallel morphisms on a bicompact object in $\mathcal{B}_0$. Hence bicompact objects form a set up to isomorphisms.
\end{proof}

\begin{division}[The canonical $\sigma$-cocone and the bidenseness condition]\label{the oplaxcocone is the category of element of the binerve}
If $\mathcal{B}$ is finitely bi-accessible, then one can consider the \emph{canonical lax-cone} of $B$ relatively to the full subcategory of all bicompact objects $\mathcal{B}_\omega$ which is defined by the \emph{binerve} $ \nu_B$ of $\mathcal{B}_\omega$ at $B$, which is the composite
% https://q.uiver.app/?q=WzAsMyxbMCwwLCJcXG1hdGhjYWx7Qn1fMCBee1xcb3B9Il0sWzEsMCwiXFxtYXRoY2Fse0J9XntcXG9wfSJdLFsyLDAsIlxcQ2F0Il0sWzAsMSwiXFxpb3RhXzAiLDAseyJzdHlsZSI6eyJ0YWlsIjp7Im5hbWUiOiJob29rIiwic2lkZSI6InRvcCJ9fX1dLFsxLDIsIlxcbWF0aGNhbHtCfVstLEJdIl1d
\[\begin{tikzcd}
	{\mathcal{B}_\omega ^{\op}} & {\mathcal{B}^{\op}} & \Cat
	\arrow["{\iota^{\op}}", hook, from=1-1, to=1-2]
	\arrow["{\mathcal{B}[-,B]}", from=1-2, to=1-3]
\end{tikzcd}\]
This defines a fibration $ \int \nu_B $ over $ \mathcal{B}_\omega$ whose objects are pairs $ (K, a)$ with $ a : K \rightarrow B$, and a morphism $(K_1,a_1) \rightarrow (K_2, a_2)$ is a pair $(k, \phi)$ coding for a 2-cell
% https://q.uiver.app/?q=WzAsMyxbMCwwLCJLXzEiXSxbMiwwLCJLXzIiXSxbMSwxLCJCIl0sWzAsMiwiYV8xIiwyXSxbMSwyLCJhXzIiXSxbMCwxLCJrIl0sWzQsMywiXFxwaGkiLDIseyJzaG9ydGVuIjp7InNvdXJjZSI6MjAsInRhcmdldCI6MjB9fV1d
\[\begin{tikzcd}
	{K_1} && {K_2} \\
	& B
	\arrow[""{name=0, anchor=center, inner sep=0}, "{a_1}"', from=1-1, to=2-2]
	\arrow[""{name=1, anchor=center, inner sep=0}, "{a_2}", from=1-3, to=2-2]
	\arrow["k", from=1-1, to=1-3]
	\arrow["\phi"', shorten <=6pt, shorten >=6pt, Rightarrow, from=1, to=0]
\end{tikzcd}\]
In particular, cartesian morphisms are exactly those $(k,\phi)$ where $ \phi$ is an isomorphism $ a_1 \simeq a_2 k$. They will be denoted as $ \Cart_B$. Observe that $ \int \nu_B$ is equivalent to the oplax-slice $ \mathcal{B}_\omega \Downarrow B$.
\end{division}

\begin{division}[The canonical pseudococone]
Although the canonical $\sigma$-cocone $ \mathcal{B}_\omega \Downarrow B$ contains all the necessary information about $B$ relatively to $ \mathcal{B}_\omega$, we are going to see below that this information is actually condensed into a canonical \emph{pseudococone} consisting of the restriction to cartesian morphisms over $B$, which corresponds to the restriction of the pseudoslice over $B$ to the bicompact objects $ \mathcal{B}_\omega \downarrow B$. 
\end{division}

\begin{lemma}[The canonical pseudococone is bifiltered]\label{the pseudococone is bifiltered}
If $\mathcal{B}$ is finitely bi-accessible with $ \mathcal{B}_0$ a generator of bicompact objects, then for any object $B$, if $ F : I_B \rightarrow \mathcal{B}_0$ with $I_B $ bifiltered is such that $ B \simeq \bicolim\, F$, then the inclusion 
\[ I_B \hookrightarrow \mathcal{B}_\omega\downarrow B \]
sending $i $ on $ q_i : F(i) \rightarrow B$ is bicofinal, whence $\mathcal{B}_\omega\downarrow B$ is bifiltered and $ B \simeq \bicolim \; \mathcal{B}_\omega\downarrow B$.  
\end{lemma}

\begin{proof}
By assumption of finite bi-accessibility, one has for each $B$ a $\sigma$-filtered $\sigma$-bicolimit decomposition
\[  B \simeq \underset{i \in I_B}{{\Sigma_B}\bicolim} \, F(i)  \]
with the $F : I_B \rightarrow \mathcal{B}_\omega$ returning bicompact objects. In particular this diagram admits an embedding into the pseudoslice of bicompact objects over $B$  
% https://q.uiver.app/?q=WzAsMixbMCwwLCJJX0IiXSxbMSwwLCJcXG1hdGhjYWx7Qn1fXFxvbWVnYVxcRG93bmFycm93IEIiXSxbMCwxLCIiLDAseyJzdHlsZSI6eyJ0YWlsIjp7Im5hbWUiOiJob29rIiwic2lkZSI6InRvcCJ9fX1dXQ==
\[\begin{tikzcd}
	{I_B} & {\mathcal{B}_\omega\downarrow B}
	\arrow[hook, from=1-1, to=1-2]
\end{tikzcd}\]
sending $i$ on the pair $(F(i), q_i) $, $ d :i \rightarrow j$ on the invertible 2-cell $(F(d), q_d) : (F(i), q_i) \rightarrow (F(j), q_j)$ and the 2-cells as expected. %In particular it sends arrows $s$ in $\Sigma$ to cartesian arrows as the corresponding 2-cell inclusion $q_s$ are invertible. 
We must show this embedding to be bicofinal to deduce bifilteredness of $\mathcal{B}_\omega \downarrow B$ from the bifilteredness of $I_B$. 

For $ K $ in $\mathcal{B}_\omega$ any $ a : K \rightarrow B$ factorizes through some $ q_i : F(i) \rightarrow B$, so we have a triangle in $\mathcal{B}_\omega \Downarrow B$
% https://q.uiver.app/?q=WzAsMyxbMCwwLCJLIl0sWzEsMSwiQiJdLFsyLDAsIkYoaSkiXSxbMCwyLCJiIl0sWzAsMSwiYSIsMl0sWzIsMSwicV9pIl0sWzMsMSwiXFxhbHBoYSBcXGF0b3AgXFxzaW1lcSIsMSx7InNob3J0ZW4iOnsic291cmNlIjoyMH0sInN0eWxlIjp7ImJvZHkiOnsibmFtZSI6Im5vbmUifSwiaGVhZCI6eyJuYW1lIjoibm9uZSJ9fX1dXQ==
\[\begin{tikzcd}
	K && {F(i)} \\
	& B
	\arrow[""{name=0, anchor=center, inner sep=0}, "b", from=1-1, to=1-3]
	\arrow["a"', from=1-1, to=2-2]
	\arrow["{q_i}", from=1-3, to=2-2]
	\arrow["{\sigma \atop \simeq}"{description}, Rightarrow, draw=none, from=0, to=2-2]
\end{tikzcd}\]
This is in particular the name of a cartesian arrow $(b,\sigma) : (K,a) \rightarrow (F_i, q_i)$ in . This ensures the first condition of cofinality.

For the second condition, suppose one has a parallel pair $(b,\beta), (b',\beta') : (K,a) \rightrightarrows (F(i),q_i)$ in $ \mathcal{B}_\omega \downarrow B$: then this provides two invertible 2-cells $ \beta : a \simeq q_i b $ and $ \beta' : a \simeq q_i b'$, so that $ (b, \beta),(b', \beta')$ are two lifts of the same arrow, so by \Cref{lifting property of sigma-compact} one has a span $ d : i \rightarrow j$, $ d' : i \rightarrow j$ in $I_B$ together with some invertible two-cell $\gamma$ as below
\[\begin{tikzcd}
	& {F(i)} \\
	K && {F(j)} & B \\
	& {F(i)}
	\arrow["b", from=2-1, to=1-2]
	\arrow[""{name=0, anchor=center, inner sep=0}, "{q_i}", curve={height=-18pt}, from=1-2, to=2-4]
	\arrow["{b'}"', from=2-1, to=3-2]
	\arrow[""{name=1, anchor=center, inner sep=0}, "{q_{i}}"', curve={height=18pt}, from=3-2, to=2-4]
	\arrow[draw=none, from=1-2, to=3-2]
	\arrow["{q_{j}}"{description}, from=2-3, to=2-4]
	\arrow["{F(d)}"{description}, from=1-2, to=2-3]
	\arrow["{F(d')}"{description}, from=3-2, to=2-3]
	\arrow["{\gamma \atop \simeq }"{description}, Rightarrow, draw=none, shorten <=6pt, shorten >=6pt, from=1-2, to=3-2]
	\arrow["{\theta_d \atop \simeq}"{description}, Rightarrow, draw=none, from=2-3, to=0]
	\arrow["{\theta_{d'} \atop \simeq}"{description}, Rightarrow, draw=none, from=2-3, to=1]
\end{tikzcd} \]
As the lifts $b,b'$ were through the same index $i$, the span provided by bicompactness is actually a parallel pair in $I_B$: hence by bifilteredness there exists a further arrow $ d'' : j \rightarrow k$ in $I_B$ together with an invertible 2-cell $ \alpha : d''d \simeq d'' d'$: then the composite $ d''d : i \rightarrow k$ together with the pasting $ \alpha d'' \gamma$ provides the desired inserted 2-cell.

To conclude, if now one has two parallel 2-cells $ \beta, \beta' : (b,\sigma) \Rightarrow (b',\sigma')$ (with both $ \sigma$, $\sigma'$ invertible), that is two simultaneous factorizations of $\sigma : q_i b \Rightarrow a$ as below
% https://q.uiver.app/?q=WzAsMyxbMCwwLCJLIl0sWzEsMSwiQiJdLFsyLDAsIkYoaSkiXSxbMCwyLCJiIiwwLHsib2Zmc2V0IjotMSwiY3VydmUiOi0yfV0sWzAsMSwiYSIsMl0sWzIsMSwicV9pIl0sWzAsMiwiYiciLDEseyJvZmZzZXQiOjF9XSxbNiwxLCJcXGFscGhhJyBcXGF0b3AgXFxzaW1lcSIsMSx7InNob3J0ZW4iOnsic291cmNlIjoyMH0sInN0eWxlIjp7ImJvZHkiOnsibmFtZSI6Im5vbmUifSwiaGVhZCI6eyJuYW1lIjoibm9uZSJ9fX1dLFszLDYsIlxcYmV0YSciLDAseyJvZmZzZXQiOi00LCJzaG9ydGVuIjp7InNvdXJjZSI6MjAsInRhcmdldCI6MjB9fV0sWzMsNiwiXFxiZXRhIiwyLHsib2Zmc2V0Ijo0LCJzaG9ydGVuIjp7InNvdXJjZSI6MjAsInRhcmdldCI6MjB9fV1d
\[\begin{tikzcd}
	K && {F(i)} \\
	& B
	\arrow[""{name=0, anchor=center, inner sep=0}, "b", shift left=1, curve={height=-12pt}, from=1-1, to=1-3]
	\arrow["a"', from=1-1, to=2-2]
	\arrow["{q_i}", from=1-3, to=2-2]
	\arrow[""{name=1, anchor=center, inner sep=0}, "{b'}"{description}, shift right=1, from=1-1, to=1-3]
	\arrow["{\sigma' \atop \simeq}"{description}, Rightarrow, draw=none, from=1, to=2-2]
	\arrow["{\beta'}", shift left=4, shorten <=2pt, shorten >=2pt, Rightarrow, from=0, to=1]
	\arrow["\beta"', shift right=4, shorten <=2pt, shorten >=2pt, Rightarrow, from=0, to=1]
\end{tikzcd}\]
expressing the equalities $ \sigma' q_i * \beta = \sigma = \sigma' q_i*\beta' $, then for $K$ is bicompact we have common lift $(j,j',s,d',k, \psi)$ together with invertible 2-cells $ \gamma, \gamma'$ as below providing a decomposition of $\sigma \sigma'^{-1}$
% https://q.uiver.app/?q=WzAsNyxbMCwyLCJLIl0sWzEsMCwiRihpKSJdLFsxLDEsIkYoaikiXSxbMSwzLCJGKGonKSJdLFsxLDQsIkYoaSkiXSxbMiwyLCJGKGspIl0sWzMsMiwiQiJdLFswLDEsImIiLDAseyJjdXJ2ZSI6LTJ9XSxbMCwyLCJhIiwxXSxbMCwzLCJhJyIsMV0sWzAsNCwiYiciLDIseyJjdXJ2ZSI6Mn1dLFsyLDUsIkYocykiLDFdLFszLDUsIkYoZCkiLDFdLFsyLDMsIlxccHNpIiwxLHsic2hvcnRlbiI6eyJzb3VyY2UiOjIwLCJ0YXJnZXQiOjIwfSwibGV2ZWwiOjJ9XSxbNSw2LCJxX2siLDFdLFszLDYsInFfe2onfSIsMSx7ImN1cnZlIjoyfV0sWzEsNiwicV9pIiwwLHsiY3VydmUiOi0yfV0sWzQsNiwicV9pIiwyLHsiY3VydmUiOjJ9XSxbMyw0LCJcXGdhbW1hJyBcXGF0b3AgXFxzaW1lcSIsMSx7InN0eWxlIjp7ImJvZHkiOnsibmFtZSI6Im5vbmUifSwiaGVhZCI6eyJuYW1lIjoibm9uZSJ9fX1dLFsxLDIsIlxcZ2FtbWEgXFxhdG9wIFxcc2ltZXEiLDEseyJzdHlsZSI6eyJib2R5Ijp7Im5hbWUiOiJub25lIn0sImhlYWQiOnsibmFtZSI6Im5vbmUifX19XSxbNSwxNSwiXFx0aGV0YV9kIiwwLHsic2hvcnRlbiI6eyJ0YXJnZXQiOjIwfX1dXQ==
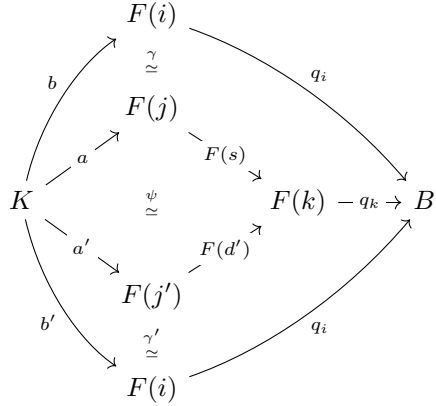
\[\begin{tikzcd}
	& {F(i)} \\
	& {F(j)} \\
	K && {F(k)} & B \\
	& {F(j')} \\
	& {F(i)}
	\arrow["b", curve={height=-12pt}, from=3-1, to=1-2]
	\arrow["a"{description}, from=3-1, to=2-2]
	\arrow["{a'}"{description}, from=3-1, to=4-2]
	\arrow["{b'}"', curve={height=12pt}, from=3-1, to=5-2]
	\arrow["{F(s)}"{description}, from=2-2, to=3-3]
	\arrow["{F(d')}"{description}, from=4-2, to=3-3]
	\arrow["\psi \atop \simeq"{description}, shorten <=6pt, shorten >=6pt, Rightarrow, draw=none, from=2-2, to=4-2]
	\arrow["{q_k}"{description}, from=3-3, to=3-4]
	%\arrow[""{name=0, anchor=center, inner sep=0}, "{q_{j'}}"{description}, curve={height=12pt}, from=4-2, to=3-4]
	\arrow["{q_i}", curve={height=-12pt}, from=1-2, to=3-4]
	\arrow["{q_i}"', curve={height=12pt}, from=5-2, to=3-4]
	\arrow["{\gamma' \atop \simeq}"{description}, draw=none, from=4-2, to=5-2]
	\arrow["{\gamma \atop \simeq}"{description}, draw=none, from=1-2, to=2-2]
	%\arrow["{\theta_{d'} \atop \simeq}"{description, pos=0.4}, shorten >=3pt, Rightarrow, draw=none, from=3-3, to=0]
\end{tikzcd}\]
% https://q.uiver.app/?q=WzAsNyxbMCwxLCJLIl0sWzEsMCwiRihpKSJdLFsxLDEsIkYoaikiXSxbMSwyLCJGKGonKSJdLFsxLDMsIkYoaSkiXSxbMiwxLCJGKGspIl0sWzMsMSwiQiJdLFswLDEsImIiXSxbMCwyLCJhIiwxXSxbMCwzLCJhJyIsMV0sWzAsNCwiYiciLDIseyJjdXJ2ZSI6Mn1dLFsyLDUsIkYocykiLDFdLFszLDUsIkYoZCkiLDFdLFsyLDMsIlxccHNpIiwxLHsic2hvcnRlbiI6eyJzb3VyY2UiOjIwLCJ0YXJnZXQiOjIwfSwibGV2ZWwiOjJ9XSxbNSw2LCJxX2siLDFdLFszLDYsInFfe2onfSIsMSx7ImN1cnZlIjoyfV0sWzEsNiwicV9pIl0sWzQsNiwicV9pIiwyLHsiY3VydmUiOjJ9XSxbMyw0LCJcXGdhbW1hJyBcXGF0b3AgXFxzaW1lcSIsMSx7InN0eWxlIjp7ImJvZHkiOnsibmFtZSI6Im5vbmUifSwiaGVhZCI6eyJuYW1lIjoibm9uZSJ9fX1dLFsxLDIsIlxcZ2FtbWEgXFxhdG9wIFxcc2ltZXEiLDEseyJzdHlsZSI6eyJib2R5Ijp7Im5hbWUiOiJub25lIn0sImhlYWQiOnsibmFtZSI6Im5vbmUifX19XSxbNSwxNSwiXFx0aGV0YV9kIiwwLHsic2hvcnRlbiI6eyJ0YXJnZXQiOjIwfX1dXQ==
But then, combining bicompactness of $K$ and bifilteredness of $I_B$, one can find a $l $ in $I_B$ together with $t : i \rightarrow l$ and $r: k \rightarrow l$ together with decompositions of $\gamma$ and $\gamma'$ as pasting of invertible 2-cells $\delta$ and $\delta'$ such that this same 2-cell factorizes as
% https://q.uiver.app/?q=WzAsOCxbMCwyLCJLIl0sWzEsMCwiRihpKSJdLFsxLDEsIkYoaikiXSxbMSwzLCJGKGonKSJdLFsxLDQsIkYoaSkiXSxbMiwyLCJGKGspIl0sWzQsMiwiQiJdLFszLDIsIkYobCkiXSxbMCwxLCJiIiwwLHsiY3VydmUiOi0yfV0sWzAsMiwiYSIsMV0sWzAsMywiYSciLDFdLFswLDQsImInIiwyLHsiY3VydmUiOjJ9XSxbMiw1LCJGKHMpIiwxXSxbMyw1LCJGKGQpIiwxXSxbMiwzLCJcXHBzaSIsMix7InNob3J0ZW4iOnsic291cmNlIjoyMCwidGFyZ2V0IjoyMH0sImxldmVsIjoyfV0sWzEsNiwicV9pIiwwLHsiY3VydmUiOi0zfV0sWzQsNiwicV9pIiwyLHsiY3VydmUiOjN9XSxbNyw2LCJxX2wiLDFdLFs1LDcsIkYocikiLDFdLFs0LDcsIkYodCkiLDFdLFsxLDcsIkYodCkiLDFdLFszLDQsIlxcemV0YSIsMCx7InNob3J0ZW4iOnsidGFyZ2V0IjoyMH0sImxldmVsIjoyfV0sWzEsMiwiXFxkZWx0YSBcXGF0b3AgXFxzaW1lcSIsMSx7InN0eWxlIjp7ImJvZHkiOnsibmFtZSI6Im5vbmUifSwiaGVhZCI6eyJuYW1lIjoibm9uZSJ9fX1dLFsxNSw3LCJcXG51IFxcYXRvcCBcXHNpbWVxIiwxLHsic2hvcnRlbiI6eyJzb3VyY2UiOjIwfSwic3R5bGUiOnsiYm9keSI6eyJuYW1lIjoibm9uZSJ9LCJoZWFkIjp7Im5hbWUiOiJub25lIn19fV0sWzE5LDE2LCJcXG51JyBcXGF0b3AgXFxzaW1lcSIsMCx7ImxhYmVsX3Bvc2l0aW9uIjo2MCwic2hvcnRlbiI6eyJzb3VyY2UiOjMwLCJ0YXJnZXQiOjIwfSwic3R5bGUiOnsiYm9keSI6eyJuYW1lIjoibm9uZSJ9LCJoZWFkIjp7Im5hbWUiOiJub25lIn19fV1d
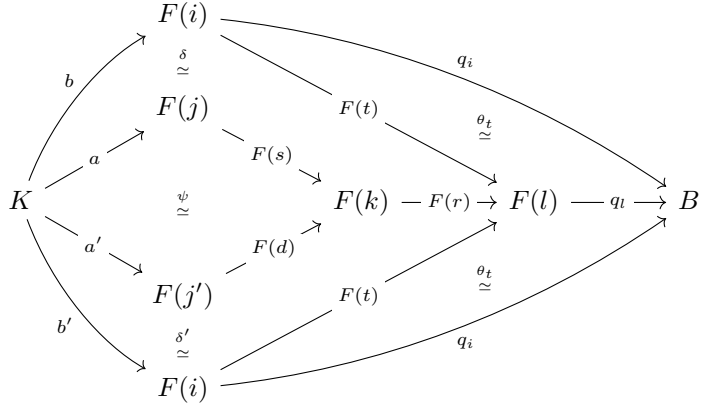
\[\begin{tikzcd}[column sep=large]
	& {F(i)} \\
	& {F(j)} \\
	K && {F(k)} & {F(l)} & B \\
	& {F(j')} \\
	& {F(i)}
	\arrow["b", curve={height=-12pt}, from=3-1, to=1-2]
	\arrow["a"{description}, from=3-1, to=2-2]
	\arrow["{a'}"{description}, from=3-1, to=4-2]
	\arrow["{b'}"', curve={height=12pt}, from=3-1, to=5-2]
	\arrow["{F(s)}"{description}, from=2-2, to=3-3]
	\arrow["{F(d)}"{description}, from=4-2, to=3-3]
	\arrow["\psi \atop \simeq"{description}, shorten <=6pt, shorten >=6pt, Rightarrow, draw=none, from=2-2, to=4-2]
	\arrow[""{name=0, anchor=center, inner sep=0}, "{q_i}", curve={height=-18pt}, from=1-2, to=3-5]
	\arrow[""{name=2, anchor=center, inner sep=0}, ""{name=1, anchor=center, inner sep=0}, "{q_i}"', curve={height=18pt}, from=5-2, to=3-5]
	\arrow["{q_l}"{description}, from=3-4, to=3-5]
	\arrow["{F(r)}"{description}, from=3-3, to=3-4]
	\arrow["{F(t)}"{description}, from=5-2, to=3-4]
	\arrow["{F(t)}"{description}, from=1-2, to=3-4]
	\arrow["\delta' \atop \simeq"{description}, shorten >=2pt, Rightarrow, draw=none, from=4-2, to=5-2]
	\arrow["{\delta \atop \simeq}"{description}, draw=none, from=1-2, to=2-2]
	\arrow["{\theta_t \atop \simeq}"{description}, Rightarrow, draw=none, from=0, to=3-4]
	\arrow["{\theta_t \atop \simeq}"{description}, Rightarrow, draw=none, from=2, to=3-4]
\end{tikzcd}\]
% https://q.uiver.app/?q=WzAsOCxbMCwxLCJLIl0sWzIsMCwiRihpKSJdLFsxLDEsIkYoaikiXSxbMSwyLCJGKGonKSJdLFsxLDMsIkYoaSkiXSxbMiwyLCJGKGspIl0sWzQsMywiQiJdLFszLDMsIkYobCkiXSxbMCwxLCJiIiwwLHsiY3VydmUiOi0yfV0sWzAsMiwiYSIsMV0sWzAsMywiYSciLDFdLFswLDQsImInIiwyLHsiY3VydmUiOjJ9XSxbMiw1LCJGKHMpIiwxXSxbMyw1LCJGKGQpIiwxXSxbMiwzLCJcXHBzaSIsMix7InNob3J0ZW4iOnsic291cmNlIjoyMCwidGFyZ2V0IjoyMH0sImxldmVsIjoyfV0sWzEsNiwicV9pIiwwLHsiY3VydmUiOi0zfV0sWzQsNiwicV9pIiwyLHsiY3VydmUiOjN9XSxbNyw2LCJxX2wiLDFdLFs1LDcsIkYocikiLDFdLFs0LDcsIkYodCkiLDFdLFsxLDcsIkYodCkiLDFdLFszLDQsIlxcemV0YSIsMCx7InNob3J0ZW4iOnsidGFyZ2V0IjoyMH0sImxldmVsIjoyfV0sWzE1LDcsIlxcbnUgXFxhdG9wIFxcc2ltZXEiLDEseyJzaG9ydGVuIjp7InNvdXJjZSI6MjB9LCJzdHlsZSI6eyJib2R5Ijp7Im5hbWUiOiJub25lIn0sImhlYWQiOnsibmFtZSI6Im5vbmUifX19XSxbMTksMTYsIlxcbnUnIFxcYXRvcCBcXHNpbWVxIiwwLHsibGFiZWxfcG9zaXRpb24iOjYwLCJzaG9ydGVuIjp7InNvdXJjZSI6MzAsInRhcmdldCI6MjB9LCJzdHlsZSI6eyJib2R5Ijp7Im5hbWUiOiJub25lIn0sImhlYWQiOnsibmFtZSI6Im5vbmUifX19XSxbMSwxMiwiXFxkZWx0YSBcXGF0b3AgXFxzaW1lcSIsMSx7ImxldmVsIjoxLCJzdHlsZSI6eyJib2R5Ijp7Im5hbWUiOiJub25lIn0sImhlYWQiOnsibmFtZSI6Im5vbmUifX19XV0=
Then this exactly says that $ F(t)$ coequalizes $ \beta,\beta'$ into the same invertible 2-cell $ \delta'^{-1} F(r)*\psi \delta : F(t)b \simeq F(t)b'$, which manifests itself as a coequalizing 2-cell in $\mathcal{B}_\omega\downarrow B$:
% https://q.uiver.app/?q=WzAsMyxbMCwwLCIoSyxhKSJdLFsyLDAsIihGKGkpLHFfaSkiXSxbNCwwLCIoRihsKSxxX2wpIl0sWzAsMSwiKGInLFxcc2lnbWEpIiwyLHsiY3VydmUiOjN9XSxbMCwxLCIoYixcXHBoaSkiLDAseyJjdXJ2ZSI6LTN9XSxbMSwyLCIoRih0KSwgXFx0aGV0YV90KSJdLFs0LDMsIlxcYmV0YSIsMix7Im9mZnNldCI6Mywic2hvcnRlbiI6eyJzb3VyY2UiOjIwLCJ0YXJnZXQiOjIwfX1dLFs0LDMsIlxcYmV0YSciLDAseyJvZmZzZXQiOi0zLCJzaG9ydGVuIjp7InNvdXJjZSI6MjAsInRhcmdldCI6MjB9fV1d
\[\begin{tikzcd}
	{(K,a)} && {(F(i),q_i)} && {(F(l),q_l)}
	\arrow[""{name=0, anchor=center, inner sep=0}, "{(b',\sigma')}"', curve={height=18pt}, from=1-1, to=1-3]
	\arrow[""{name=1, anchor=center, inner sep=0}, "{(b,\sigma)}", curve={height=-18pt}, from=1-1, to=1-3]
	\arrow["{(F(t), \theta_t)}", from=1-3, to=1-5]
	\arrow["\beta"', shift right=3, shorten <=5pt, shorten >=5pt, Rightarrow, from=1, to=0]
	\arrow["{\beta'}", shift left=3, shorten <=5pt, shorten >=5pt, Rightarrow, from=1, to=0]
\end{tikzcd}\]
This achieves to prove the bicofinalness of the inclusion of $I_B$ in $\mathcal{B}_\omega\downarrow B $
\end{proof}

\begin{corollary}\label{the canonical cone is sigma-filtered}
If $ \mathcal{B}$ is finitely bi-accessible, then for any $B$ the canonical pseudocone $ \mathcal{B}_\omega \downarrow B$ is bifiltered.  
\end{corollary}
\begin{proof}
Combine the lemma above with \Cref{cofinal functors transfer sigma-filteredness}.
\end{proof}

\begin{remark}
As a consequence, we know also from \Cref{key lemma} that the oplax cocone $ \mathcal{B}_\omega \Downarrow B$ is $\sigma$-filtered for $\Cart_B$, as the pseudococone coincides with the restriction of the oplax cocone to cartesian arrows. Moreover by $\sigma$-cofinality, we also have a $\sigma$-filtered $\sigma$-bicolimit decomposition $ B \simeq \Cart_B \bicolim \; \mathcal{B}_\omega \Downarrow B$.
\end{remark}

As in the $1$-dimensional case, we can provide an equivalent characterization of accessibility, providing a canonical candidate for the choice of the generator in an accessible category. The corollary below is the $2$-dimensional analog of \cite[page 19]{makkai1989accessible}.

\begin{corollary}
\label{characterization of biacc with pseudococone}
A 2-category $\mathcal{B}$ is {finitely bi-accessible} if and only if \begin{itemize}
    \item $\mathcal{B}$ has bifiltered bicolimits
    \item the full subcategory $ \mathcal{B}_\omega$ of bicompact objects is essentially small
    \item for any object $B$ the canonical pseudococone provides an equivalence
    \[ B \simeq \bicolim\; \mathcal{B}_\omega\downarrow B \]
\end{itemize}
\end{corollary}

Moreover, from what precedes, we know this characterization can be rephrased in term of $\sigma$-filtered $\sigma$-bicolimits - observing that cocompleteness under bifiltered bicolimits is sufficient to entail cocompleteness under $\sigma$-filtered $\sigma$-bicolimit as stated in \Cref{bifiltered bicocompleteness entails sigma filtered sigma cocompleteness}:

\begin{corollary}
A 2-category $\mathcal{B}$ is {finitely bi-accessible} if and only if \begin{itemize}
    \item $\mathcal{B}$ has $\sigma$-filtered $\sigma$-bicolimits
    \item the full subcategory $ \mathcal{B}_\omega$ of bicompact objects is essentially small
    \item for any object $B$, the oplax cocone together with the class of cartesian arrows provides us with a $\sigma$-bicolimit:
    \[ B \simeq \Cart_B\bicolim\; \mathcal{B}_\omega\Downarrow B \]
\end{itemize}
\end{corollary}

This latter characterization, though seemingly redundant given the previous one, is of interest because it more directly relates with the notion of \emph{bidenseness} which is a condition on the \emph{binerve}.  

\begin{definition}
For $ F : \mathcal{C} \rightarrow \mathcal{D}$ a pseudofunctor with $ \mathcal{C}$ a small 2-category, the \emph{binerve} of $ F$ is the pseudofunctor $ \nu_F : \mathcal{D} \rightarrow \ps[\mathcal{C}^{\op}, \Cat]$ sending $ D$ on the pseudofunctor $ [F,D] : \mathcal{C}^{\op} \rightarrow \Cat$. A pseudofunctor $ F$ is said to be \emph{bidense} if its binerve $ \nu_F$ is pseudo-fully-faithful.
\end{definition}

As expected, bidenseness of a functor equates to the possibility to decompose any object as a bicolimit over a convenient notion of cocone from the functor to this object; while this bicolimit decomposition is usually stated in term of weights, it can be de-weighted into a $\sigma$-bicolimit:

\begin{lemma}
For a small full sub-2-category $\mathcal{C}$ of a 2-category $\mathcal{B}$, the following are equivalent:\begin{itemize}
    \item Any object $B$ in $ \mathcal{B}$ decomposes as the $\sigma$-bicolimit: $ B \simeq {\Sigma_B}\bicolim \; \mathcal{C} \Downarrow B$ for the $\sigma$-pair $ (\mathcal{C} \Downarrow B, \Sigma_B)$ where $\Sigma_B$ consists of all oplax cells above $B$ whose underlying 2-cell is invertible;
    \item The binerve 2-functor $ \nu_{\mathcal{C}} : \mathcal{B} \rightarrow \ps[\mathcal{C}^{\op}, \Cat]$ is pseudo-fully faithful, that is, $\mathcal{C}$ is a bidense subcategory of $\mathcal{B}$.
\end{itemize}
\end{lemma}

\begin{proof}
The binerve functor sends $B$ on the pseudofunctor $ \mathcal{B}[\iota_\mathcal{C}, B]$ where $ \iota_\mathcal{C}$ is the full inclusion of $\mathcal{C}$. In one direction, for $ B_1,B_2$, a pseudonatural transformation $ \phi : \mathcal{B}[\iota_\mathcal{C}, B_1] \rightarrow \mathcal{B}[\iota_\mathcal{C}, B_2]  $ defines uniquely a $ \sigma$-cocone indexed by the $\sigma$-pair $(\mathcal{C} \Downarrow B_2, \Sigma_B)$ as follows: take a pair $ (C, a) $ with $ a : C \rightarrow B_1$ to $\phi_C(a)$ and a 2-cell $(u,\alpha) : (C_1,a_1) \rightarrow (C_2,a_2)$ to the 2-cell % https://q.uiver.app/?q=WzAsMyxbMCwwLCIgdV4qXFxwaGlfe0NfMn0oYV8yKSAiXSxbMSwwLCJcXHBoaV97Q18xfShhXzJ1KSJdLFsyLDAsIlxccGhpX3tDXzF9KGFfMSkiXSxbMCwxLCJcXHBoaV91IFxcYXRvcCBcXHNpbWVxIiwwLHsibGV2ZWwiOjJ9XSxbMSwyLCJcXHBoaV97Q18xfShcXGFscGhhKSIsMCx7ImxldmVsIjoyfV1d
\[\begin{tikzcd}
	{ \phi_{C_2}(a_2)u } & {\phi_{C_1}(a_2u)} & {\phi_{C_1}(a_1)}
	\arrow["{\phi_u \atop \simeq}", Rightarrow, from=1-1, to=1-2]
	\arrow["{\phi_{C_1}(\alpha)}", Rightarrow, from=1-2, to=1-3]
\end{tikzcd}\]
where $ \phi_u$ is invertible pseudonaturality component of $\phi$ at $u$. This 2-cell is invertible whenever $\alpha$ is, that is, whenever $ (u,\alpha)$ is in $\Sigma_{B_1}$. This defines a $ \sigma$-cocone over $B_2$, and $ B_1 \simeq \sigma_{\Sigma_{B_1}}\!\bicolim \mathcal{C} \Downarrow B_1$ ensures that we end up with a unique 1-cell $\phi : B_1 \rightarrow B_2$ in $\mathcal{B}$; this moreover can be done functorially relative to natural modifications $ \phi \Rrightarrow \psi $, which we let as an exercise. The converse direction is obvious. 
\end{proof}

\begin{remark}\label{biacc as category of flat pseudofunctors}
In particular, it follows from the previous lemma and the discussion above that for any bi-accessible 2-category $ \mathcal{B}$, the pseudofully faithful inclusion $ \iota : \mathcal{B}_\omega \hookrightarrow \mathcal{B}$ is bidense, which amounts to saying that the corresponding binerve sending any $ B$ on the strict 2-functor $ \mathcal{B}[\iota, B] $, is a pseudo-fully faithful pseudofunctor,
% https://q.uiver.app/?q=WzAsMixbMCwwLCJcXG1hdGhjYWx7Qn0gIl0sWzEsMCwiXFxwc1soXFxtYXRoY2Fse0J9X1xcb21lZ2EpXntcXG9wfSwgXFxDYXRdIl0sWzAsMSwiXFxudV9cXG1hdGhjYWx7Qn0iLDAseyJzdHlsZSI6eyJ0YWlsIjp7Im5hbWUiOiJob29rIiwic2lkZSI6InRvcCJ9fX1dXQ==
\[\begin{tikzcd}
	{\mathcal{B} } & {\ps[(\mathcal{B}_\omega)^{\op}, \Cat]}
	\arrow["{\nu}", hook, from=1-1, to=1-2]
\end{tikzcd}\]
But the point of this observation is less to provide a further restatement of the definition of bi-accessibility in term of bidenseness than leading us to consider the binerve relative to bicompact objects. Instead, as in the 1-dimensional case, crucial properties of bi-accessible (and later bipresentable) 2-categories will be retrieved from the possibility to embed them into the category of pseudofunctors over their bicompact objects thanks to this binerve.   
\end{remark}

\begin{lemma}\label{binerve preserve bifiltered bicolimits}
The binerve pseudofunctor $ \nu: \mathcal{B} \hookrightarrow \ps[(\mathcal{B}_\omega)^{\op}, \Cat] $ preserves bifiltered bicolimits (and thus $\sigma$-filtered $\sigma$-bicolimits). 
\end{lemma}

\begin{proof}
Take $ F :I \rightarrow \mathcal{B}$ with $ I$ bifiltered. In each bicompact object $K$ and  we have an equivalence
\[ \mathcal{B}[\iota(K), \bicolim \; F] \stackrel{e_K}{\simeq} \underset{i \in I}{\bicolim} \;\mathcal{B}[\iota(K), F(i)] \]
We must show this equivalence is pseudonatural in $ K$. Take $ f : K_1 \rightarrow K_2$ in $\mathcal{B}_\omega$: we must prove that any precomposition by $f$ of a lift is isomorphic to a lift of the precomposition by $f$. For any $ a : K_2 \rightarrow \bicolim\,F$ there is for some $i \in I$ an invertible 2-cell
% https://q.uiver.app/?q=WzAsMyxbMCwwLCJLXzIiXSxbMSwwLCJcXGJpY29saW0gXFwsIEYiXSxbMSwxLCJGKGkpIl0sWzAsMSwiYSJdLFsyLDEsInFfaSIsMl0sWzAsMiwieCIsMl0sWzEsNSwiXFx4aSBcXGF0b3AgXFxzaW1lcSIsMSx7InNob3J0ZW4iOnsidGFyZ2V0IjoyMH0sInN0eWxlIjp7ImJvZHkiOnsibmFtZSI6Im5vbmUifSwiaGVhZCI6eyJuYW1lIjoibm9uZSJ9fX1dXQ==
\[\begin{tikzcd}
	{K_2} & {\bicolim \, F} \\
	& {F(i)}
	\arrow["a", from=1-1, to=1-2]
	\arrow["{q_i}"', from=2-2, to=1-2]
	\arrow[""{name=0, anchor=center, inner sep=0}, "x"', from=1-1, to=2-2]
	\arrow["{\xi \atop \simeq}"{description}, Rightarrow, draw=none, from=1-2, to=0]
\end{tikzcd}\]
Then whiskering with $f$ produces a lift $(xf, \xi*f)$ of $ af$. Now take another lift $(x', \xi')$ of $ af$ 
% https://q.uiver.app/?q=WzAsMyxbMSwxLCJcXGJpY29saW0gXFwsIEYiXSxbMSwwLCJGKGkpIl0sWzAsMCwiS18xIl0sWzEsMCwicV9pIl0sWzIsMCwiYWYiLDJdLFsyLDEsIngnIl0sWzAsNSwiXFx4aScgXFxhdG9wIFxcc2ltZXEiLDEseyJzaG9ydGVuIjp7InRhcmdldCI6MjB9LCJzdHlsZSI6eyJib2R5Ijp7Im5hbWUiOiJub25lIn0sImhlYWQiOnsibmFtZSI6Im5vbmUifX19XV0=
\[\begin{tikzcd}
	{K_1} & {F(i')} \\
	& {\bicolim \, F}
	\arrow["{q_{i'}}", from=1-2, to=2-2]
	\arrow["af"', from=1-1, to=2-2]
	\arrow[""{name=0, anchor=center, inner sep=0}, "{x'}", from=1-1, to=1-2]
	\arrow["{\xi' \atop \simeq}"{description, pos=0.6}, Rightarrow, draw=none, from=2-2, to=0]
\end{tikzcd}\]
Since $ K_1$ is bicompact and $I$ is bifiltered, there exists $ d : i \rightarrow j $ and $ d' : i' \rightarrow j$ and an invertible 2-cell $ \gamma $ such that 
% https://q.uiver.app/?q=WzAsNSxbMCwxLCJLIl0sWzEsMCwiRihpKSJdLFszLDEsIlxcYmljb2xpbVxcLCBGIl0sWzEsMiwiRihpJykiXSxbMiwxLCJGKGknJykiXSxbMCwxLCJ4ZiJdLFsxLDIsInFfaSIsMCx7ImN1cnZlIjotMn1dLFswLDMsIngnIiwyXSxbMywyLCJxX3tpJ30iLDIseyJjdXJ2ZSI6M31dLFsxLDMsIiIsMSx7InN0eWxlIjp7ImJvZHkiOnsibmFtZSI6Im5vbmUifSwiaGVhZCI6eyJuYW1lIjoibm9uZSJ9fX1dLFs0LDIsInFfe2knJ30iLDFdLFsxLDQsIkYoZCkiLDFdLFszLDQsIkYoZCcpIiwxXSxbMCw0LCJcXGdhbW1hIFxcYXRvcCBcXHNpbWVxIiwxLHsic3R5bGUiOnsiYm9keSI6eyJuYW1lIjoibm9uZSJ9LCJoZWFkIjp7Im5hbWUiOiJub25lIn19fV0sWzQsNiwiXFx0aGV0YV9kIFxcYXRvcCBcXHNpbWVxIiwxLHsic2hvcnRlbiI6eyJ0YXJnZXQiOjIwfSwic3R5bGUiOnsiYm9keSI6eyJuYW1lIjoibm9uZSJ9LCJoZWFkIjp7Im5hbWUiOiJub25lIn19fV0sWzQsOCwiXFx0aGV0YV97ZCd9IFxcYXRvcCBcXHNpbWVxIiwxLHsic2hvcnRlbiI6eyJ0YXJnZXQiOjIwfSwic3R5bGUiOnsiYm9keSI6eyJuYW1lIjoibm9uZSJ9LCJoZWFkIjp7Im5hbWUiOiJub25lIn19fV1d
\[\begin{tikzcd}
	& {F(i)} \\
	K && {F(i'')} & {\bicolim\, F} \\
	& {F(i')}
	\arrow["xf", from=2-1, to=1-2]
	\arrow[""{name=0, anchor=center, inner sep=0}, "{q_i}", curve={height=-18pt}, from=1-2, to=2-4]
	\arrow["{x'}"', from=2-1, to=3-2]
	\arrow[""{name=1, anchor=center, inner sep=0}, "{q_{i'}}"', curve={height=18pt}, from=3-2, to=2-4]
	\arrow[draw=none, from=1-2, to=3-2]
	\arrow["{q_{i''}}"{description}, from=2-3, to=2-4]
	\arrow["{F(d)}"{description}, from=1-2, to=2-3]
	\arrow["{F(d')}"{description}, from=3-2, to=2-3]
	\arrow["{\gamma \atop \simeq}"{description}, draw=none, from=2-1, to=2-3]
	\arrow["{\theta_d \atop \simeq}"{description}, Rightarrow, draw=none, from=2-3, to=0]
	\arrow["{\theta_{d'} \atop \simeq}"{description}, Rightarrow, draw=none, from=2-3, to=1]
\end{tikzcd} = \begin{tikzcd}
	& {F(i)} \\
	{K_1} && {\bicolim \, F} \\
	& {F(i')}
	\arrow["xf", from=2-1, to=1-2]
	\arrow["{q_i}", from=1-2, to=2-3]
	\arrow[""{name=0, anchor=center, inner sep=0}, "af"{description}, from=2-1, to=2-3]
	\arrow["{x'}"', from=2-1, to=3-2]
	\arrow["{q_{i'}}"', from=3-2, to=2-3]
	\arrow["{\xi*f \atop \simeq}"{description}, Rightarrow, draw=none, from=1-2, to=0]
	\arrow["{\xi'\atop \simeq}"{description}, Rightarrow, draw=none, from=3-2, to=0]
\end{tikzcd}\]
This means that $ xf$ and $ x'$ are isomorphic in $ {\bicolim}_{i \in I}\mathcal{B}[K, F(i)]$. Denote as $ (e_f)_a$ this isomorphism. Then we claim this produces a pseudocommutative square 
% https://q.uiver.app/?q=WzAsNCxbMCwwLCJcXG1hdGhjYWx7Qn1bS18yLCBcXGJpY29saW0gRl0iXSxbMCwxLCJcXG1hdGhjYWx7Qn1bS18xLCBcXGJpY29saW0gRl0iXSxbMSwwLCJcXHVuZGVyc2V0e2kgXFxpbiBJfXtcXGJpY29saW19W0tfMiwgRihpKV0iXSxbMSwxLCJcXHVuZGVyc2V0e2kgXFxpbiBJfXtcXGJpY29saW19W0tfMSwgRihpKV0iXSxbMCwxLCJcXG1hdGhjYWx7Qn1bZiwgXFxiaWNvbGltIEZdIiwyXSxbMCwyLCJlX3tLXzJ9IFxcYXRvcCBcXHNpbWVxIl0sWzIsMywiXFx1bmRlcnNldHtpIFxcaW4gSX17XFxiaWNvbGltfVtmLCBGKGkpXSJdLFsxLDMsImVfe0tfMX0gXFxhdG9wIFxcc2ltZXEiLDJdLFs1LDcsImVfZiBcXGF0b3AgXFxzaW1lcSIsMSx7InNob3J0ZW4iOnsic291cmNlIjoyMCwidGFyZ2V0IjoyMH0sInN0eWxlIjp7ImJvZHkiOnsibmFtZSI6Im5vbmUifSwiaGVhZCI6eyJuYW1lIjoibm9uZSJ9fX1dXQ==
\[\begin{tikzcd}
	{\mathcal{B}[\iota(K_2), \bicolim  \;F]} & {\underset{i \in I}{\bicolim} \;\mathcal{B}[\iota(K_2), F(i)]} \\
	{\mathcal{B}[\iota(K_1), \bicolim \; F]} & {\underset{i \in I}{\bicolim}\;\mathcal{B}[\iota(K_1), F(i)]}
	\arrow["{\mathcal{B}[f, \bicolim \;F]}"', from=1-1, to=2-1]
	\arrow[""{name=0, anchor=center, inner sep=0}, "{e_{K_2} \atop \simeq}", from=1-1, to=1-2]
	\arrow["{\underset{i \in I}{\bicolim}\;[f, F(i)]}", from=1-2, to=2-2]
	\arrow[""{name=1, anchor=center, inner sep=0}, "{e_{K_1} \atop \simeq}"', from=2-1, to=2-2]
	\arrow["{e_f \atop \simeq}"{description}, Rightarrow, draw=none, from=0, to=1]
\end{tikzcd}\]
Now consider $ \phi : a_1 \Rightarrow a_2$ in $\mathcal{B}[K_2, \bicolim F]$. It lifts as a 2-cell $ \psi : F(d_1)x_1 \Rightarrow F(d_2)x_2$ with $ d_1 : i_1 \rightarrow j$ and $ d_2 : i_2 \rightarrow j$ with $ (x_1, \xi_1)$ cartesian and $ (x_2, \xi_2)$ lifts of $a_1$ and $ a_2$ respectively, and whiskering with $ f$ provides lift of $ f^*\phi$; now if we choose an alternative lift $ \psi' : F(i_1')x_1' \Rightarrow F(i_2')x_2'$ of $ f*\phi$, we know that there are isomorphisms $ (e^f)_{a_1} : x_1f \simeq x_1'$ and $ (e^f)_{a_2} : x_2f \simeq x_2'$ in $ {\bicolim}_{i \in I}\mathcal{B}[K, F(i)]$, and we have a commutation of 2-cells
% https://q.uiver.app/?q=WzAsNCxbMCwwLCJxX3tpXzF9eF8xZiJdLFswLDEsInFfe2lfMn14XzJmIl0sWzEsMCwicV97aV8xJ314XzEnIl0sWzEsMSwicV97aV8yJ314XzInIl0sWzAsMSwiXFxwc2kqIGYiLDIseyJsZXZlbCI6Mn1dLFswLDIsIihlXmYpX3thXzF9IFxcYXRvcCBcXHNpbWVxIiwwLHsibGV2ZWwiOjJ9XSxbMSwzLCIoZV5mKV97YV8yfSBcXGF0b3AgXFxzaW1lcSIsMix7ImxldmVsIjoyfV0sWzIsMywiXFxwc2knIiwwLHsibGV2ZWwiOjJ9XV0=
\[\begin{tikzcd}
	{q_{i_1}x_1f} & {q_{i_1'}x_1'} \\
	{q_{i_2}x_2f} & {q_{i_2'}x_2'}
	\arrow["{\psi* f}"', Rightarrow, from=1-1, to=2-1]
	\arrow["{(e_f)_{a_1} \atop \simeq}", Rightarrow, from=1-1, to=1-2]
	\arrow["{(e_f)_{a_2} \atop \simeq}"', Rightarrow, from=2-1, to=2-2]
	\arrow["{\psi'}", Rightarrow, from=1-2, to=2-2]
\end{tikzcd}\]
which expresses the naturality of $e_f$. Finally, hence the $e_f$, for $f$ ranging over all 1-cell between bicompact objects, altogether define a pseudonatural equivalence 
\[  \mathcal{B}[\iota, \bicolim  \; F] \stackrel{e}{\simeq} \underset{i \in I}{\bicolim} \; \mathcal{B}[\iota, F(i)] \]
\end{proof}

We are going to see later (\Cref{sigma-accessible 2-cat are categories of flat functors}) that this binerve pseudofunctor identifies any bi-accessible 2-category with the 2-category of \emph{flat pseudofunctors} in the sense of \cite{descotte2018sigma} over its generator of bicompact objects. But we defer the study of flatness to a further section, and turn now to bipresentable 2-categories where further properties will be extracted from the binerve.

\subsection{Bipresentable 2-categories}

\begin{definition}[Bipresentable 2-categories] \label{bipres}
A 2-category is said to be \emph{finitely bipresentable} if it is finitely bi-accessible and has all small weighted bicolimits. We define $ \lambda$-bipresentable $2$-categories as $\lambda$-bi-accessible $2$-categories with all small weighted bicolimits.
\end{definition}

It is worth detailing how one can see directly in the case of a finitely bipresentable 2-category why the canonical diagram $ (\mathcal{B}_\omega \Downarrow B, \Cart_B)$ is a $\sigma$-filtered pair and, for this very reason, the pseudoslice $ \mathcal{B}_\omega \downarrow B $ is lax-cofinal in $\mathcal{B}_\omega \Downarrow B$ relatively to cartesian morphisms. From \Cref{finite bicolimit of sigma-compact are sigma-compact}, we know that finitely weighted bicolimits of bicompacts are bicompacts. This encompasses in particular bicoproducts, bicoinserter and bicoequifiers. Altogether, these arguments will ensure that the pair made of the oplax-slice together with is cartesian morphisms $(\mathcal{B}_\omega \Downarrow B, \Cart_B)$ is $\sigma$-filtered - from which we automatically deduce bifilteredness and cofinality of the pseudoslice.

\begin{division}[Discrete cones via bicoproducts]\label{Discrete cones via bicoproducts}
Hence in $\mathcal{B}_\omega \Downarrow B$ if one has two objects $ (K_1, a_1) $ and $ (K_2, a_2)$, then there is a common factorization through the bicoproduct
% https://q.uiver.app/?q=WzAsNCxbMSwxLCJLXzErS18yIl0sWzAsMiwiS18yIl0sWzAsMCwiS18xIl0sWzMsMSwiQiJdLFsyLDAsInFfMSIsMl0sWzEsMCwicV8yIl0sWzAsMywiXFxsYW5nbGUgYV8xLCBhXzIgXFxyYW5nbGUiLDFdLFsyLDMsImFfMSIsMCx7ImN1cnZlIjotMn1dLFsxLDMsImFfMiIsMix7ImN1cnZlIjoyfV0sWzAsNywiXFxhbHBoYV8xIFxcYXRvcCBcXHNpbWVxIiwxLHsic2hvcnRlbiI6eyJ0YXJnZXQiOjIwfSwic3R5bGUiOnsiYm9keSI6eyJuYW1lIjoibm9uZSJ9LCJoZWFkIjp7Im5hbWUiOiJub25lIn19fV0sWzAsOCwiXFxhbHBoYV8yIFxcYXRvcCBcXHNpbWVxIiwxLHsic2hvcnRlbiI6eyJ0YXJnZXQiOjIwfSwic3R5bGUiOnsiYm9keSI6eyJuYW1lIjoibm9uZSJ9LCJoZWFkIjp7Im5hbWUiOiJub25lIn19fV1d
\[\begin{tikzcd}
	{K_1} \\
	& {K_1+K_2} && B \\
	{K_2}
	\arrow["{q_1}"', from=1-1, to=2-2]
	\arrow["{q_2}", from=3-1, to=2-2]
	\arrow["{\langle a_1, a_2 \rangle}"{description}, from=2-2, to=2-4]
	\arrow[""{name=0, anchor=center, inner sep=0}, "{a_1}", curve={height=-12pt}, from=1-1, to=2-4]
	\arrow[""{name=1, anchor=center, inner sep=0}, "{a_2}"', curve={height=12pt}, from=3-1, to=2-4]
	\arrow["{\alpha_1 \atop \simeq}"{description}, Rightarrow, draw=none, from=2-2, to=0]
	\arrow["{\alpha_2 \atop \simeq}"{description}, Rightarrow, draw=none, from=2-2, to=1]
\end{tikzcd}\]
and from the universal property of the bicolimit, one has invertible 2-cells $ \alpha_1 : a_1 \simeq \langle a_1, a_2 \rangle q_1$ and $ \alpha_2: a_2 \simeq \langle a_1, a_2 \rangle q_2$ exhibiting the pairs $(q_1, \alpha_1) : (K_1, a_1) \rightarrow (K_1 + K_2, \langle a_1, a_2 \rangle )$, $(q_2, \alpha_2) : (K_2, a_2) \rightarrow (K_1 + K_2, \langle a_1, a_2 \rangle )$ as a span of cartesian arrows in $\mathcal{B}_\omega \Downarrow B$.
\end{division}

\begin{division}[Insertion of 2-cells]\label{Insertion of 2-cells}

For a parallel pair $(k_1,\phi), (k_2, \sigma) : (K_1, a_1) \rightrightarrows (K_2, a_2)$ with $ (k_2,\sigma)$ cartesian, that is with $ \sigma : a_2k_2\simeq a_1$ invertible, then $ a_2$ inserts a 2-cell 
% https://q.uiver.app/?q=WzAsNCxbMCwxLCJLXzEiXSxbMiwxLCJCIl0sWzEsMCwiS18yIl0sWzEsMiwiS18yIl0sWzAsMiwia18xIiwwLHsib2Zmc2V0IjotMX1dLFsyLDEsImFfMiJdLFswLDEsImFfMSIsMV0sWzAsMywia18yIiwyXSxbMywxLCJhXzIiLDJdLFs2LDMsIlxcc2lnbWFeey0xfSBcXGF0b3AgXFxzaW1lcSIsMSx7InNob3J0ZW4iOnsic291cmNlIjoyMH0sInN0eWxlIjp7ImJvZHkiOnsibmFtZSI6Im5vbmUifX19XSxbMiw2LCJcXHBoaSIsMCx7InNob3J0ZW4iOnsidGFyZ2V0IjoyMH19XV0=
\[\begin{tikzcd}
	& {K_2} \\
	{K_1} && B \\
	& {K_2}
	\arrow["{k_1}", shift left=1, from=2-1, to=1-2]
	\arrow["{a_2}", from=1-2, to=2-3]
	\arrow[""{name=0, anchor=center, inner sep=0}, "{a_1}"{description}, from=2-1, to=2-3]
	\arrow["{k_2}"', from=2-1, to=3-2]
	\arrow["{a_2}"', from=3-2, to=2-3]
	\arrow["{\sigma^{-1} \atop \simeq}"{description}, Rightarrow,draw=none, from=0, to=3-2]
	\arrow["\phi", shorten >=3pt, Rightarrow, from=1-2, to=0]
\end{tikzcd}\]
Hence we have a factorizations of $a_2$ through the bicoinserter of $k_1,k_2$ (which is bicompact):
% https://q.uiver.app/?q=WzAsMyxbMCwwLCJLXzIiXSxbMSwxLCJCIl0sWzIsMCwiXFx0ZXh0YmZ7Y29JbnN9KGtfMSxrXzIpIl0sWzIsMSwiXFxsYW5nbGUgXFxzaWdtYV57LTF9XFxwaGkgXFxyYW5nbGUiXSxbMCwyLCJxX3soa18xLGtfMil9Il0sWzAsMSwiYV8yIiwyXSxbNCwxLCJcXGFscGhhIFxcYXRvcCBcXHNpbWVxIiwxLHsic2hvcnRlbiI6eyJzb3VyY2UiOjIwfSwic3R5bGUiOnsiYm9keSI6eyJuYW1lIjoibm9uZSJ9LCJoZWFkIjp7Im5hbWUiOiJub25lIn19fV1d
\[\begin{tikzcd}
	{K_2} && {\textbf{coIns}(k_1,k_2)} \\
	& B
	\arrow["{\langle \sigma^{-1}\phi \rangle}", from=1-3, to=2-2]
	\arrow[""{name=0, anchor=center, inner sep=0}, "{q_{(k_1,k_2)}}", from=1-1, to=1-3]
	\arrow["{a_2}"', from=1-1, to=2-2]
	\arrow["{\alpha \atop \simeq}"{description}, Rightarrow, draw=none, from=0, to=2-2]
\end{tikzcd}\]
which provides in particular a cartesian 1-cell in $\mathcal{B}_\omega \Downarrow B$. This inserts a 2-cell % https://q.uiver.app/?q=WzAsNCxbMCwxLCIoS18xLGFfMSkiXSxbMSwwLCIoS18yLGFfMikiXSxbMiwxLCIoXFx0ZXh0YmZ7Y29JbnN9KGtfMSxrXzIpLFxcbGFuZ2xlIFxcc2lnbWFeey0xfVxccGhpIFxccmFuZ2xlKSJdLFsxLDIsIihLXzIsYV8yKSJdLFswLDEsIihrXzEsIFxccGhpKSIsMCx7Im9mZnNldCI6LTF9XSxbMSwyLCIocV97KGtfMSxrXzIpfSwgXFxhbHBoYSApIl0sWzAsMywiKGtfMiwgXFxzaWdtYSkiLDIseyJvZmZzZXQiOjF9XSxbMywyLCIocV97KGtfMSxrXzIpfSwgXFxhbHBoYSApIiwyXSxbMSwzLCJcXHBoaV97KGtfMSxrXzIpfSIsMSx7InNob3J0ZW4iOnsic291cmNlIjoxMCwidGFyZ2V0IjoyMH0sImxldmVsIjoyfV1d
\[\begin{tikzcd}
	& {(K_2,a_2)} \\
	{(K_1,a_1)} && {(\textbf{coIns}(k_1,k_2),\langle \sigma^{-1}\phi \rangle)} \\
	& {(K_2,a_2)}
	\arrow["{(k_1, \phi)}", shift left=1, from=2-1, to=1-2]
	\arrow["{(q_{(k_1,k_2)}, \alpha )}", from=1-2, to=2-3]
	\arrow["{(k_2, \sigma)}"', shift right=1, from=2-1, to=3-2]
	\arrow["{(q_{(k_1,k_2)}, \alpha )}"', from=3-2, to=2-3]
	\arrow["{\phi_{(k_1,k_2)}}"{description}, shorten <=3pt, shorten >=6pt, Rightarrow, from=1-2, to=3-2]
\end{tikzcd}\]
where $\phi_{(k_1,k_2)}$ is the universal 2-cell inserted by the bicoinserter.
\end{division}

\begin{division}[Equification of parallel 2-cells]\label{Equification of parallel 2-cells}
If now one has parallel 2-cells 
% https://q.uiver.app/?q=WzAsMixbMCwwLCIoS18xLGFfMSkiXSxbMiwwLCIoS18yLGFfMikiXSxbMCwxLCIoa18xLCBcXHBoaSkiLDAseyJjdXJ2ZSI6LTN9XSxbMCwxLCIoa18yLCBcXHNpZ21hKSIsMix7ImN1cnZlIjozfV0sWzIsMywiXFxwaGlfMiIsMCx7Im9mZnNldCI6LTQsInNob3J0ZW4iOnsic291cmNlIjoyMCwidGFyZ2V0IjoyMH19XSxbMiwzLCJcXHBoaV8xIiwyLHsib2Zmc2V0Ijo0LCJzaG9ydGVuIjp7InNvdXJjZSI6MjAsInRhcmdldCI6MjB9fV1d
\[\begin{tikzcd}
	{(K_1,a_1)} && {(K_2,a_2)}
	\arrow[""{name=0, anchor=center, inner sep=0}, "{(k_1, \phi)}", curve={height=-18pt}, from=1-1, to=1-3]
	\arrow[""{name=1, anchor=center, inner sep=0}, "{(k_2, \sigma)}"', curve={height=18pt}, from=1-1, to=1-3]
	\arrow["{\theta_2}", shift left=4, shorten <=5pt, shorten >=5pt, Rightarrow, from=0, to=1]
	\arrow["{\theta_1}"', shift right=4, shorten <=5pt, shorten >=5pt, Rightarrow, from=0, to=1]
\end{tikzcd}\]
then one has an equality of 2-cell in the oplax slice
% https://q.uiver.app/?q=WzAsMyxbMCwwLCJLXzEiXSxbMiwwLCJLXzIiXSxbMSwxLCJCIl0sWzAsMSwia18xIl0sWzEsMiwiYV8yIl0sWzAsMiwiYV8xIiwyXSxbNCw1LCJcXHBoaSIsMix7InNob3J0ZW4iOnsic291cmNlIjoyMCwidGFyZ2V0IjoyMH19XV0=
\[ \begin{tikzcd}
	{K_1} && {K_2} \\
	& B
	\arrow[""{name=0, anchor=center, inner sep=0}, "{k_2}"{description}, from=1-1, to=1-3]
	\arrow["{a_2}", from=1-3, to=2-2]
	\arrow["{a_1}"', from=1-1, to=2-2]
	\arrow[""{name=1, anchor=center, inner sep=0}, "{k_1}", curve={height=-18pt}, from=1-1, to=1-3]
	\arrow["{\sigma \atop \simeq}"{description}, Rightarrow, draw=none, from=0, to=2-2]
	\arrow["{\theta_1}", shorten <=2pt, shorten >=2pt, Rightarrow, from=1, to=0]
\end{tikzcd} =
\begin{tikzcd}
	{K_1} && {K_2} \\
	& B
	\arrow["{k_1}", from=1-1, to=1-3]
	\arrow[""{name=0, anchor=center, inner sep=0}, "{a_2}", from=1-3, to=2-2]
	\arrow[""{name=1, anchor=center, inner sep=0}, "{a_1}"', from=1-1, to=2-2]
	\arrow["\phi"', shorten <=6pt, shorten >=6pt, Rightarrow, from=0, to=1]
\end{tikzcd} = 
\begin{tikzcd}
	{K_1} && {K_2} \\
	& B
	\arrow[""{name=0, anchor=center, inner sep=0}, "{k_2}"{description}, from=1-1, to=1-3]
	\arrow["{a_2}", from=1-3, to=2-2]
	\arrow["{a_1}"', from=1-1, to=2-2]
	\arrow[""{name=1, anchor=center, inner sep=0}, "{k_1}", curve={height=-18pt}, from=1-1, to=1-3]
	\arrow["{\sigma \atop \simeq}"{description}, Rightarrow, draw=none, from=0, to=2-2]
	\arrow["{\theta_2}", shorten <=2pt, shorten >=2pt, Rightarrow, from=1, to=0]
\end{tikzcd}\]
Hence $ a_2 $ coequifies $ \theta_1$ and $ \theta_2$, hence we have as above a factorizations through the bicoequifier (which is bicompact)
\[\begin{tikzcd}
	{K_2} && {\textbf{coEq}(\theta_1,\theta_2)} \\
	& B
	\arrow["{\langle \sigma^{-1}\phi \rangle}", from=1-3, to=2-2]
	\arrow[""{name=0, anchor=center, inner sep=0}, "{q_{(\theta_1,\theta_2)}}", from=1-1, to=1-3]
	\arrow["{a_2}"', from=1-1, to=2-2]
	\arrow["{\alpha \atop \simeq}"{description}, Rightarrow, draw=none, from=0, to=2-2]
\end{tikzcd}\] 
and this provides again a coequifying diagram in the following with the coequifying 1-cell being cartesian.
% https://q.uiver.app/?q=WzAsMyxbMCwwLCIoS18xLGFfMSkiXSxbMiwwLCIoS18yLGFfMikiXSxbMywwLCIoXFx0ZXh0YmZ7Y29FcX0oXFx0aGV0YV8xLFxcdGhldGFfMiksXFxsYW5nbGUgXFxzaWdtYV57LTF9XFxwaGkgXFxyYW5nbGUpIl0sWzAsMSwiKGtfMSwgXFxwaGkpIiwwLHsiY3VydmUiOi0zfV0sWzEsMiwiKHFfeyhcXHRoZXRhXzEsXFx0aGV0YV8yKX0sIFxcYWxwaGEgKSJdLFswLDEsIihrXzIsIFxcc2lnbWEpIiwyLHsiY3VydmUiOjN9XSxbMyw1LCJcXHRoZXRhXzIiLDAseyJvZmZzZXQiOi00LCJzaG9ydGVuIjp7InNvdXJjZSI6MjAsInRhcmdldCI6MjB9fV0sWzMsNSwiXFx0aGV0YV8xIiwyLHsib2Zmc2V0Ijo0LCJzaG9ydGVuIjp7InNvdXJjZSI6MjAsInRhcmdldCI6MjB9fV1d
\[\begin{tikzcd}
	{(K_1,a_1)} && {(K_2,a_2)} && {(\textbf{coEq}(\theta_1,\theta_2),\langle \sigma^{-1}\phi \rangle)}
	\arrow[""{name=0, anchor=center, inner sep=0}, "{(k_1, \phi)}", curve={height=-18pt}, from=1-1, to=1-3]
	\arrow["{(q_{(\theta_1,\theta_2)}, \alpha )}", from=1-3, to=1-5]
	\arrow[""{name=1, anchor=center, inner sep=0}, "{(k_2, \sigma)}"', curve={height=18pt}, from=1-1, to=1-3]
	\arrow["{\theta_2}", shift left=4, shorten <=5pt, shorten >=5pt, Rightarrow, from=0, to=1]
	\arrow["{\theta_1}"', shift right=4, shorten <=5pt, shorten >=5pt, Rightarrow, from=0, to=1]
\end{tikzcd}\]
From this together with \Cref{key lemma}, we see even more directly than in the bi-accessible case why the pseudoslice over the bicompacts is bifiltered and bicofinal thanks to internalization of finite bicolimits; this condenses to the following result: 
\end{division}

\begin{lemma} 
If $ \mathcal{B}$ is a finitely bipresentable 2-category, then for any $B$ the $\sigma$-pair $(\mathcal{B}_\omega\Downarrow B, \Cart_B)$ is $\sigma$-filtered; as a consequence, the pseudococone $ \mathcal{B}_\omega\downarrow B$ is bifiltered and cofinal in $\mathcal{B}_\omega\Downarrow B$ relatively to $\Cart_B $, and we have 
\[  B \simeq \bicolim \; \mathcal{B}_\omega \downarrow B \]
\end{lemma}

Now we are going to prove that, as like as in the 1-dimensional case, existence of bicolimits in a finitely accessible category is actually equivalent to existence of bilimits. This makes use of the binerve functor.

\begin{lemma} \label{reflectivity}
For a finitely bipresentable 2-category $ \mathcal{B}$ we have a biadjunction
% https://q.uiver.app/?q=WzAsMixbMCwwLCJcXG1hdGhjYWx7Qn0iXSxbMiwwLCJcXHBzW1xcbWF0aGNhbHtCfSwgXFxDYXRdX3AiXSxbMCwxLCJcXGlvdGEiLDIseyJzdHlsZSI6eyJ0YWlsIjp7Im5hbWUiOiJob29rIiwic2lkZSI6InRvcCJ9fX1dLFsxLDAsIkwiLDIseyJjdXJ2ZSI6M31dLFsxLDAsIiIsMSx7ImN1cnZlIjotMywic3R5bGUiOnsiYm9keSI6eyJuYW1lIjoibm9uZSJ9LCJoZWFkIjp7Im5hbWUiOiJub25lIn19fV0sWzMsNCwiXFxwZXJwIiwxLHsibGFiZWxfcG9zaXRpb24iOjMwLCJsZXZlbCI6MSwic3R5bGUiOnsiYm9keSI6eyJuYW1lIjoibm9uZSJ9LCJoZWFkIjp7Im5hbWUiOiJub25lIn19fV1d
\[\begin{tikzcd}
	\mathcal{B}	\arrow["\nu"', rr, bend right=20, start anchor=-25, end anchor=190, hook] & \perp & \arrow[ll, "L "', bend right=20, start anchor=170, end anchor=25] \ps[(\mathcal{B}_\omega)^{\op}, \Cat]
	\end{tikzcd}\]
where $ L$ denotes the left biKan extension $ L = \textup{biLan}_\hirayo \iota$. In particular, we can exhibit $ \mathcal{B}$ as a bireflective sub-2-category of $\ps[(\mathcal{B}_\omega)^{\op}, \Cat]$. 
\end{lemma}

\begin{proof}
By hypothesis, $ \mathcal{B}$ is bicocomplete, so we can compute the pointwise left biKan extension
% https://q.uiver.app/?q=WzAsMyxbMCwwLCJcXG1hdGhjYWx7Qn1fIFxcb21lZ2EiXSxbMCwxLCJcXHBzWyhcXG1hdGhjYWx7Qn1fXFxvbWVnYSlee1xcb3B9LCBcXENhdF0iXSxbMSwwLCJcXG1hdGhjYWx7Qn0iXSxbMCwyLCJcXGlvdGEiLDAseyJzdHlsZSI6eyJ0YWlsIjp7Im5hbWUiOiJob29rIiwic2lkZSI6InRvcCJ9fX1dLFswLDEsIlxcaGlyYXlvIiwyXSxbMSwyLCJcXGJpTGFuX1xcaGlyYXlvIFxcaW90YSIsMl0sWzMsNSwiXFxzaW1lcSIsMSx7InNob3J0ZW4iOnsic291cmNlIjoyMCwidGFyZ2V0IjoyMH0sInN0eWxlIjp7ImJvZHkiOnsibmFtZSI6Im5vbmUifSwiaGVhZCI6eyJuYW1lIjoibm9uZSJ9fX1dXQ==
\[\begin{tikzcd}
	{\mathcal{B}_ \omega} & {\mathcal{B}} \\
	{\ps[(\mathcal{B}_\omega)^{\op}, \Cat]}
	\arrow[""{name=0, anchor=center, inner sep=0}, "\iota", hook, from=1-1, to=1-2]
	\arrow["\hirayo"', from=1-1, to=2-1, hook]
	\arrow[""{name=1, anchor=center, inner sep=0}, "{\biLan_\hirayo \iota}"', from=2-1, to=1-2]
	\arrow["\simeq"{description}, Rightarrow, draw=none, from=0, to=1]
\end{tikzcd}\]
and this pointwise extension expresses as the weighted bicolimit as discussed in \Cref{Weighted colimit expression of pointwise biLan} and \Cref{cancellation rule},
\[ \biLan_\hirayo \iota(F) = \bicolim^{N_\hirayo(F)} \iota  \]
% where $ \mathcal{N}_\hirayo$ is the left biKan extension of $ \hirayo$ along itself, with $ F \simeq N_\hirayo(F)$ (see lemma on the binerve). 
% https://q.uiver.app/?q=WzAsMyxbMCwwLCJcXG1hdGhjYWx7Qn1fXFxvbWVnYSJdLFsxLDAsIlxccHNbXFxtYXRoY2Fse0J9X1xcb21lZ2Fee29wfSwgXFxDYXRdIl0sWzAsMSwiXFxwc1tcXG1hdGhjYWx7Qn1fXFxvbWVnYV57b3B9LCBcXENhdF0iXSxbMCwyLCJcXGhpcmF5byIsMix7InN0eWxlIjp7InRhaWwiOnsibmFtZSI6Imhvb2siLCJzaWRlIjoidG9wIn19fV0sWzAsMSwiXFxoaXJheW8iLDAseyJzdHlsZSI6eyJ0YWlsIjp7Im5hbWUiOiJob29rIiwic2lkZSI6InRvcCJ9fX1dLFsyLDEsIk5fXFxoaXJheW8gPSBcXGJpTGFuX1xcaGlyYXlvIFxcaGlyYXlvIl0sWzQsNSwiXFxzaW1lcSIsMSx7InNob3J0ZW4iOnsic291cmNlIjoyMCwidGFyZ2V0IjoyMH0sInN0eWxlIjp7ImJvZHkiOnsibmFtZSI6Im5vbmUifSwiaGVhZCI6eyJuYW1lIjoibm9uZSJ9fX1dXQ==
%and is hence a biequivalence by bidenseness, so that for any $ F$ in $\ps[(\mathcal{B}_\omega)^{\op}, \Cat]$ we have a pseudonatural equivalence $ F \simeq N_\hirayo(F)$. 
Though such a result is rather standard, let us check that this defines a left biadjoint to $\nu$, for the sake of completeness. For $ B $ in $ \mathcal{B}$ and $ F $ in $ \ps[(\mathcal{B}_\omega)^{\op}, \Cat]$ we have 
\begin{align*}
    \mathcal{B}[\biLan_\hirayo \iota(F), B] &\simeq \mathcal{B}[\bicolim^{N_\hirayo(F)} \iota, B] \\ 
    &\simeq \ps[(\mathcal{B}_\omega)^{\op}, \Cat]\big[ N_\hirayo(F), \mathcal{B}[\iota, B] \big] \\
    &\simeq \ps[(\mathcal{B}_\omega)^{\op}, \Cat]\big[ F, \nu_B\big]
\end{align*}
\end{proof}

%This says that the binerve exhibits $ \mathcal{B}$ as a sub 2-category of the pseudofunctor category which is moreover closed under bilimits. 

\begin{corollary}\label{bipres are bicomplete}
Any finitely bipresentable 2-category $ \mathcal{B}$ is bicomplete, and moreover, finite weighted bilimits commute with bifiltered bicolimits. 
\end{corollary}

\begin{proof}
The first item comes from the general fact that bireflective sub-2-categories are closed under bilimits as seen at \Cref{bireflective subbicategories have bilimits}. In particular, if we see a finitely birepresentable 2-category $ \mathcal{B}$ as a category of pseudofunctor from its binerve, bilimits are computed pointwisely as in the pseudofunctor category. 

Moreover we saw at \Cref{binerve preserve bifiltered bicolimits} that the binerve preserves bifiltered bicolimits. Hence both bilimits and bifiltered bicolimits of $ \mathcal{B}$ are computed in $\ps[(\mathcal{B}_\omega)^{\op}, \Cat]$, where both are pointwise. But finite weighted bilimits commute with bifiltered bicolimit in $\Cat$, hence so do they in the pseudofunctor category, and in $\mathcal{B}$. 
\end{proof}

\begin{theorem}\label{Biacc is bipres iff bicomplete} The following are equivalent:
\begin{itemize}
    \item $\mathcal{B}$ is finitely bipresentable,
    \item $\mathcal{B}$ is finitely bi-accessible and bicomplete.
\end{itemize}
\end{theorem}
\begin{proof}
The implication $(1) \Rightarrow (2)$ was already proven, thus we concentrate on the other one. Consider the nerve pseudofunctor % https://q.uiver.app/?q=WzAsMixbMCwwLCJcXG1hdGhjYWx7Qn0iXSxbMSwwLCJcXHBzW1xcbWF0aGNhbHtCfV9cXG9tZWdhXntcXG9wfSwgXFxDYXRdIl0sWzAsMSwiXFxudSAiLDAseyJzdHlsZSI6eyJ0YWlsIjp7Im5hbWUiOiJob29rIiwic2lkZSI6InRvcCJ9fX1dXQ==
\[\begin{tikzcd}
	{\mathcal{B}} & {\ps[(\mathcal{B}_\omega)^{\op}, \Cat]}
	\arrow["{\nu }", hook, from=1-1, to=1-2]
\end{tikzcd}\]
we know that this is fully faithful, because $\mathcal{B}$ is finitely bi-accessible, and it preserve all weighted bilimits, as all nerves do. Thus, we can apply \Cref{LAFT}, which we will prove later in the paper, and infer that the nerve has a left adjoint. As a result, $\mathcal{B}$ is reflective and thus has all bicolimits.
\end{proof}

\begin{remark}
Beware however that arbitrary small bicolimits are not preserved by the binerve pseudofunctor. In particular, even when seeing objects of $B$ as pseudofunctors $ (\mathcal{B}_\omega)^{\op} \rightarrow \Cat$, arbitrary bicolimits cannot be calculated pointwisely as bicolimits in the pseudofunctor 2-category - though bifiltered ones are so. \end{remark}

To conclude this section, we give a process to construct finitely bipresentable 2-categories from other ones. It is well known that, for a locally finitely presentable category, a reflective subcategory whose embedding is accessible is locally finitely presentable itself. We give here the corresponding statement; but before that, the following lemma, with weaker condition, gives some insight to bicompact objects in this context:

\begin{lemma}\label{left biadjoint of finitary 2-functor preserve sigma-compact}
Let $ \mathcal{B}$ be a finitely bipresentable 2-category, and $ R :  \mathcal{A} \hookrightarrow \mathcal{B}$ preserving bifiltered bicolimits with a left bi-adjoint $L$. Then $L$ sends bicompacts on bicompacts. 
\end{lemma}

\begin{proof}
For $ F : I \rightarrow \mathcal{A}$ with $ (I,\Sigma)$ bifiltered, one has 
\begin{align*}
    \mathcal{A}[L(K), \underset{i \in I}{\bicolim} \, F(i) ] &\simeq \mathcal{B}[K, R(\underset{i \in I}{\bicolim} \, F(i))] \\
    &\simeq \mathcal{B}[K, \underset{i \in I}{\bicolim} \, RF(i)] \\
    &\simeq \underset{i \in I}{\bicolim}\, \mathcal{B}[K,  RF(i)] \\
    &\simeq  \underset{i \in I}{\bicolim}\, \mathcal{A}[L(K), F(i)] 
\end{align*}
Hence each $L(K)$ is bicompact in $\mathcal{A}$. 
\end{proof}

\subsection{A criterion to recognize finitely bipresentable $2$-categories}

\begin{definition}
A a small sub $2$-category $\iota: \mathcal{G} \hookrightarrow \mathcal{B}$ is a \textit{strong generator} if its associated nerve% https://q.uiver.app/?q=WzAsMixbMCwwLCJcXG1hdGhjYWx7Qn0iXSxbMSwwLCJcXHBzW1xcbWF0aGNhbHtHfV57XFxvcH0sIFxcQ2F0XSJdLFswLDEsIlxcbnUiXV0=
\[\begin{tikzcd}
	{\mathcal{B}} & {\ps[\mathcal{G}^{\op}, \Cat]}
	\arrow["\nu", from=1-1, to=1-2]
\end{tikzcd}\] 
is biconservative, that is, reflects equivalences.
\end{definition}

\begin{remark}
Dense generators are strong, of course. Under very weak and natural assumptions on the ambient $2$-category $\mathcal{B}$, eso-generators in the sense of \cite{street1982two,street1982characterization} are also strong. In practise, we do not know a reasonable $2$-category where eso-generators are not strong.
\end{remark}

\begin{theorem}[Recognition theorem for finitely bipresentable $2$-categories]\label{1.11}
Let $\mathcal{B}$ be locally small $2$-category with weigthed bicolimits. Then the following are equivalent:
\begin{enumerate}
    \item $\mathcal{B}$ is finitely bipresentable,
    \item $\mathcal{B}$ has a strong generator $\mathcal{G} \hookrightarrow \mathcal{B}$ made of bicompact objects.
\end{enumerate}
\end{theorem}
\begin{proof}
Because dense generators are in particular strong, the implication $(1) \Rightarrow (2)$ is trivial. We focus on the other implication. Our proof is inspired by \cite[7.2 (i)]{kelly1982structures} and \cite[1.11]{adamek1994locally}. Call $\bar{\mathcal{G}}$ the closure of $\mathcal{G}$ under finite weighted bicolimits and consider the obvious inclusions,
% https://q.uiver.app/?q=WzAsMyxbMCwwLCJcXG1hdGhjYWx7R30iXSxbMSwwLCJcXGJhcntcXG1hdGhjYWx7R319Il0sWzIsMCwiXFxtYXRoY2Fse0J9Il0sWzAsMSwiaSIsMCx7InN0eWxlIjp7InRhaWwiOnsibmFtZSI6Imhvb2siLCJzaWRlIjoidG9wIn19fV0sWzEsMiwiaiIsMCx7InN0eWxlIjp7InRhaWwiOnsibmFtZSI6Imhvb2siLCJzaWRlIjoidG9wIn19fV1d
\[\begin{tikzcd}
	{\mathcal{G}} & {\bar{\mathcal{G}}} & {\mathcal{B}}
	\arrow["i", hook, from=1-1, to=1-2]
	\arrow["j", hook, from=1-2, to=1-3]
\end{tikzcd}\]
Of course $\bar{\mathcal{G}}$ consists of compact objects, because of \Cref{finite bicolimit of sigma-compact are sigma-compact} and duly, $j$ preserves finite weighted bicolimits by construction. Also, $\bar{\mathcal{G}}$ is a strong generator too. In order to finish the proof it is enough to show that every object in $\mathcal{B}$ is a $\sigma$-filtered bicolimit of objects in $\bar{\mathcal{G}}$, indeed by the \Cref{key lemma} this entails that every objects is a bifiltered bicolimit of objects in $\bar{\mathcal{G}}$. Let $B$ be an object in $\mathcal{B}$ and consider the canonical diagram \[  (\bar{\mathcal{G}} \Downarrow B) \stackrel{\pi_B}{\longrightarrow} \mathcal{B}. \]
%Since $\bar{\mathcal{G}}$ has finite weighted bicolimits, the functor % https://q.uiver.app/?q=WzAsMixbMCwwLCJcXG1hdGhjYWx7R31cXERvd25hcnJvdyBCIl0sWzEsMCwiXFxvdmVybGluZXtcXG1hdGhjYWx7R319XFxEb3duYXJyb3cgQiJdLFswLDFdXQ==

So it suffices to show that $B$ is indeed equivalent to the $ \sigma$-filtered $\sigma$-bicolimit $\sigma_\Cart\bicolim \pi_B$ for the class of cartesian morphisms. In order to see this, consider the diagram,

% https://q.uiver.app/?q=WzAsMixbMiwwLCJcXENhdF57XFxiYXJ7XFxtYXRoY2Fse0d9fV5cXGNpcmN9Il0sWzAsMCwiXFxtYXRoY2Fse0J9Il0sWzEsMCwiXFxudShqKSIsMV1d
\[\begin{tikzcd}
	{\mathcal{B}} & \ps[\bar{\mathcal{G}}^{\op}, \Cat]
	\arrow["{\nu_j}", from=1-1, to=1-2]
\end{tikzcd}\]

Because $j$ preserves finite weighted bicolimits, $\nu_j(B) = \mathcal{B}[j,B]$ is lex. By applying \cite[4.2.7]{descotte2018sigma}, we know that $\nu_j(\sigma_\Cart\bicolim \pi_B) \simeq \nu_j(B)$. Now, using the conservativity of $\nu_j$ and the fact that it must preserve $\sigma$-filtered bicolimits, this shows that $B$ must be the bicolimit of $\pi_B$.

%is $ \sigma$-cofinal and prove $ \overline{\mathcal{G}} \Downarrow B$ to be $\sigma$-filtered by \Cref{sigmacone in sigmafiltered} for the $\sigma$ pairs consisting of invertible oplax cells, by the same arguments as in \Cref{Discrete cones via bicoproducts}, \Cref{Insertion of 2-cells} and \Cref{Equification of parallel 2-cells}. This exhibits $B$ as the $\sigma$-filtered $ \sigma$-bicolimit of its canonical diagram relative to $\overline{G}$.
\end{proof}

\begin{corollary} \label{closureiscute1}
Let $\mathcal{B}$ be a finitely bipresentable $2$-category. Let $\mathcal{G} \hookrightarrow \mathcal{B}_\omega$ be a strong generator of compact objects. Then $\mathcal{B}_\omega$ coincides with the closure of $\mathcal{G}$ under finite weighted bicolimits. 
\end{corollary}
\begin{proof}
We only need to show that every object in $\mathcal{B}_\omega$ is a finite weighted bicolimit of objects in $\mathcal{G}$. Let $B$ be an object in $\mathcal{B}_\omega$. Using the proof of the previous theorem, we know that $B \simeq \sigma_\Cart\bicolim \; \pi_B$. In particular, we have a $1$-cell (which is an equivalence) \[B \to \sigma_\Cart\bicolim \; \pi_B.\] Since the diagram is $\sigma$-filtered and $B$ is bicompact, we obtain on the spot that $B$ must be a retract of one of the $G $ in $ \mathcal{G}$. This finishes the proof.
\end{proof}

\begin{corollary} \label{closureiscute2}
Let $\mathcal{B}$ be a finitely bipresentable $2$-category. Then $\mathcal{B}_\lambda$, the full $2$-subcategory of $\lambda$-compact objects is the closure of $\mathcal{B}_\omega$ under $\lambda$-small weighted bicolimits.
\end{corollary}
\begin{proof}
Using the $\lambda$-version of \Cref{1.11} is is evident that an finitely bipresentable is $\lambda$ bipresentable for every higher $\lambda$. Then we use the $\lambda$-version of \Cref{closureiscute1}.
\end{proof}

\begin{corollary} \label{closureiscute3}
Let $\mathcal{B}$ be a finitely bipresentable $2$-category. Then every object is $\lambda$-compact for some $\lambda$.
\end{corollary}
\begin{proof}
The canonical diagram must be $\lambda$-small for some $\lambda$. Then we use that $\lambda$-small bicolimits of bicompact objects are still bicompact.
\end{proof}

\begin{remark} \label{every object compact}
The previous Corollary is true at a higher level of generality. Indeed if $\mathcal{B}$ is finitely bi-accessible, every object is still the bicolimit of its canonical diagram, and \Cref{finite bicolimit of sigma-compact are sigma-compact} still applies.
\end{remark}

\subsection{Bi-accessibility of pseudo-arrow category}

Before relating our theory of accessibility to the theory of flatness, as will be done in the next section, we would like to describe a basic operations on finitely bipresentables categories, namely the notion of \emph{category of arrows}. First, recall that the \emph{pseudoarrow 2-category}\index{pseudoarrow 2-category} of a 2-category $\mathcal{B}$ is the 2-category of pseudofunctors $ \ps[2,\mathcal{B}]$, which is moreover equipped with a canonical 2-cell 
% https://q.uiver.app/?q=WzAsMixbMCwwLCJcXHBzWzIsXFxtYXRoY2Fse0J9XSJdLFsyLDAsIlxcbWF0aGNhbHtCfSJdLFswLDEsIlxcZG9tIiwwLHsiY3VydmUiOi0yfV0sWzAsMSwiXFxjb2QiLDIseyJjdXJ2ZSI6Mn1dLFsyLDMsIlxcbGFtYmRhIiwyLHsic2hvcnRlbiI6eyJzb3VyY2UiOjIwLCJ0YXJnZXQiOjIwfX1dXQ==
\[\begin{tikzcd}
	{\ps[2,\mathcal{B}]} && {\mathcal{B}}
	\arrow[""{name=0, anchor=center, inner sep=0}, "\dom", bend left=20, from=1-1, to=1-3, start anchor=10]
	\arrow[""{name=1, anchor=center, inner sep=0}, "\cod"', bend right=20, from=1-1, to=1-3, start anchor=-10]
	\arrow["\lambda"', shorten <=3pt, shorten >=3pt, Rightarrow, from=0, to=1]
\end{tikzcd}\]
where $ \lambda$ is a pseudonatural transformation whose component at \begin{itemize}
    \item an arrow $ f $ is the 1-cell $ \lambda_f= f : \dom (f) \rightarrow \cod(f)$
    \item a pseudosquare $ (u, u', \alpha) : f_1 \rightarrow f_2$ is the invertible 2-cell $ \lambda_{(u,u',\alpha)} = \alpha $ 
    \item a morphism of pseudosquares $ (\phi, \phi') : (u_1, u'_1, \alpha_1) \Rightarrow (u_2, u'_2, \alpha_2)$ is the equality of 2-cells $ \alpha_2 f_2*\phi = \phi'*f_1 \alpha_1$.
\end{itemize}
% We recall also that both bilimits and bicolimits in $ \ps[2, \mathcal{B}]$ are componentwise, which means that both $ \cod $ and $ \dom$ preserve either bilimits or bicolimits; we have two bi-adjunctions
% % https://q.uiver.app/?q=WzAsMixbMiwwLCJcXHBzWzIsXFxtYXRoY2Fse0J9XSJdLFswLDAsIlxcbWF0aGNhbHtCfSJdLFswLDEsIlxcY29kIiwyLHsiY3VydmUiOjN9XSxbMSwwLCIxIiwxXSxbMCwxLCJcXGRvbSIsMCx7ImN1cnZlIjotM31dLFsyLDMsIiIsMCx7ImxldmVsIjoxLCJzdHlsZSI6eyJuYW1lIjoiYWRqdW5jdGlvbiJ9fV0sWzMsNCwiIiwwLHsibGV2ZWwiOjEsInN0eWxlIjp7Im5hbWUiOiJhZGp1bmN0aW9uIn19XV0=
% \[\begin{tikzcd}
% 	{\mathcal{B}} && {\ps[2,\mathcal{B}]}
% 	\arrow[""{name=0, anchor=center, inner sep=0}, "\cod"', bend right=25, from=1-3, to=1-1, start anchor=165]
% 	\arrow[""{name=1, anchor=center, inner sep=0}, "1"{description}, from=1-1, to=1-3]
% 	\arrow[""{name=2, anchor=center, inner sep=0}, "\dom", bend left=25, from=1-3, to=1-1, start anchor=195]
% \end{tikzcd}\]

%for $ F : I \rightarrow \ps[2, \mathcal{B}]$ and $ W : I \rightarrow \Cat$ (resp. $ W : I^{\op} \rightarrow \Cat$), we have the following pairs 

\begin{proposition}\label{arrow 2-category is sigma-presentable}
If $ \mathcal{B}$ is finitely bi-accessible, then so is $ \ps[2, \mathcal{B}]$. If moreover $ \mathcal{B}$ is finitely bipresentable, then so is $ \ps[2, \mathcal{B}]$, with a biequivalence 
\[  (\ps[2, \mathcal{B}])_\omega \simeq \ps[2, \mathcal{B}_\omega]  \]

% and moreover both $ \dom$, $1$ and $ \cod$ are finitely bi-accessible.

%Moreover both the following bi-adjoints $ (\cod \dashv 1)$ and $ (1 \dashv \dom)$ are morphisms of finitely bipresentable 2-categories. 
\end{proposition}

\begin{proof}
From what was said before, $ \ps[2, \mathcal{B}]$ inherits bilimits and bicolimits existing in $\mathcal{B}$, hence in particular its filtered bicolimits. In particular, as soon as $ \ps[2, \mathcal{B}]$ is proven to be finitely bi-accessible, it will be automatically finitely bipresentable. The properties of the domain and codomain functors process from the fact they preserve both bilimits and bicolimits. We then have to prove bi-accessibility.\\

First, we prove that bicompact arrows have bicompact domains and codomains. For the codomain, let $ k : \dom(k) \rightarrow \cod(k)$ be a bicompact arrow in $  \ps[2, \mathcal{B}]$ and a bifiltered diagram $F :  I \rightarrow  \mathcal{B}$. Then for any $ a : \cod(k) \rightarrow \bicolim F$ defines an identity 2-cell
% https://q.uiver.app/?q=WzAsNCxbMCwwLCJcXGRvbShrKSJdLFswLDEsIlxcY29kKGspIl0sWzEsMCwiXFxiaWNvbGltXFwsIEYiXSxbMSwxLCJcXGJpY29saW1cXCwgRiJdLFswLDEsImsiLDJdLFsxLDMsImEiLDJdLFsyLDMsIlxcYmljb2xpbVxcLCAxX0YiLDAseyJsZXZlbCI6Miwic3R5bGUiOnsiaGVhZCI6eyJuYW1lIjoibm9uZSJ9fX1dLFswLDIsImFrIl0sWzAsMywiPSIsMSx7InN0eWxlIjp7ImJvZHkiOnsibmFtZSI6Im5vbmUifSwiaGVhZCI6eyJuYW1lIjoibm9uZSJ9fX1dXQ==
\[\begin{tikzcd}
	{\dom(k)} & {\bicolim\, F} \\
	{\cod(k)} & {\bicolim\, F}
	\arrow["k"', from=1-1, to=2-1]
	\arrow["a"', from=2-1, to=2-2]
	\arrow["{\bicolim\, 1_F}", Rightarrow, no head, from=1-2, to=2-2]
	\arrow["ak", from=1-1, to=1-2]
	\arrow["{=}"{description}, draw=none, from=1-1, to=2-2]
\end{tikzcd}\]
where $ 1_{\bicolim\, F}=\bicolim\, 1_F$. Hence for $ k$ is bicompact we have a factorization of $ (ak, a) : k \rightarrow 1_{\bicolim \, F}$ through some $ q_i : F(i) \rightarrow \bicolim\, F$
% https://q.uiver.app/?q=WzAsNixbMCwxLCJcXGRvbShrKSJdLFswLDMsIlxcY29kKGspIl0sWzIsMSwiXFxiaWNvbGltXFwsIEYiXSxbMSwwLCJGKGkpIl0sWzIsMywiXFxiaWNvbGltXFwsIEYiXSxbMSwyLCJGKGkpIl0sWzAsMSwiayIsMl0sWzAsMywiYiciXSxbMywyLCJxX2kiXSxbMSw0LCJhIiwxXSxbMiw0LCJcXGJpY29saW1cXCwgMV9GIiwwLHsibGV2ZWwiOjIsInN0eWxlIjp7ImhlYWQiOnsibmFtZSI6Im5vbmUifX19XSxbMSw1LCJiIiwxXSxbNSw0LCJxX2kiLDFdLFszLDUsIiIsMSx7ImxldmVsIjoyLCJzdHlsZSI6eyJoZWFkIjp7Im5hbWUiOiJub25lIn19fV0sWzAsNSwiXFxhbHBoYVxcYXRvcCBcXHNpbWVxIiwxLHsic3R5bGUiOnsiYm9keSI6eyJuYW1lIjoibm9uZSJ9LCJoZWFkIjp7Im5hbWUiOiJub25lIn19fV0sWzUsMiwiPSIsMSx7InN0eWxlIjp7ImJvZHkiOnsibmFtZSI6Im5vbmUifSwiaGVhZCI6eyJuYW1lIjoibm9uZSJ9fX1dLFswLDIsImFrIiwxLHsibGFiZWxfcG9zaXRpb24iOjcwfV1d
\[\begin{tikzcd}
	& {F(i)} \\
	{\dom(k)} && {\bicolim\, F} \\
	& {F(i)} \\
	{\cod(k)} && {\bicolim\, F}
	\arrow["k"', from=2-1, to=4-1]
	\arrow["{b'}", from=2-1, to=1-2]
	\arrow["{q_i}", from=1-2, to=2-3]
	\arrow["a"{description}, from=4-1, to=4-3]
	\arrow["{\bicolim\, 1_F}", Rightarrow, no head, from=2-3, to=4-3]
	\arrow["b"{description}, from=4-1, to=3-2]
	\arrow["{q_i}"{description}, from=3-2, to=4-3]
	\arrow[Rightarrow, no head, from=1-2, to=3-2]
	\arrow["{\alpha\atop \simeq}"{description}, draw=none, from=2-1, to=3-2]
	\arrow["{=}"{description}, draw=none, from=3-2, to=2-3]
	\arrow["ak"{description, pos=0.7}, from=2-1, to=2-3, crossing over]
\end{tikzcd}\]
where $ b'$ provides a desired lift. Hence $ \cod(k)$ is bicompact. \\

For the domain, consider the bifiltered colimit $ \dom (k) \simeq \bicolim \; \mathcal{B}_\omega \downarrow \dom(k)$, which exhibits $ k$ as the induced arrow $ k = \langle ka \rangle_{a \in \mathcal{B}_\omega\downarrow k}$. Hence we have a bifiltered colimit in $ \ps[2, \mathcal{B}]$ 
\[ k \simeq \underset{a \in \mathcal{B}_\omega\downarrow k}{\bicolim} \; ka \]
so that we have a factorization for some $ a : K \rightarrow \dom(k)$
% https://q.uiver.app/?q=WzAsNixbMCwxLCJcXGRvbShrKSJdLFswLDMsIlxcY29kKGspIl0sWzIsMSwiXFxkb20oaykiXSxbMSwwLCJLIl0sWzIsMywiXFxjb2QoaykiXSxbMSwyLCJcXGNvZChrKSJdLFswLDEsImsiLDJdLFswLDMsImInIl0sWzMsMiwiYSJdLFsxLDQsIiIsMSx7ImxldmVsIjoyLCJzdHlsZSI6eyJoZWFkIjp7Im5hbWUiOiJub25lIn19fV0sWzIsNCwiayIsMCx7ImxldmVsIjoyLCJzdHlsZSI6eyJoZWFkIjp7Im5hbWUiOiJub25lIn19fV0sWzEsNSwiIiwxLHsibGV2ZWwiOjIsInN0eWxlIjp7ImhlYWQiOnsibmFtZSI6Im5vbmUifX19XSxbNSw0LCIiLDEseyJsZXZlbCI6Miwic3R5bGUiOnsiaGVhZCI6eyJuYW1lIjoibm9uZSJ9fX1dLFszLDUsImthIiwxLHsibGFiZWxfcG9zaXRpb24iOjMwfV0sWzAsNSwiXFxhbHBoYVxcYXRvcCBcXHNpbWVxIiwxLHsic3R5bGUiOnsiYm9keSI6eyJuYW1lIjoibm9uZSJ9LCJoZWFkIjp7Im5hbWUiOiJub25lIn19fV0sWzUsMiwiPSIsMSx7InN0eWxlIjp7ImJvZHkiOnsibmFtZSI6Im5vbmUifSwiaGVhZCI6eyJuYW1lIjoibm9uZSJ9fX1dLFswLDIsIiIsMSx7ImxldmVsIjoyLCJzdHlsZSI6eyJoZWFkIjp7Im5hbWUiOiJub25lIn19fV1d
\[\begin{tikzcd}
	& K \\
	{\dom(k)} && {\dom(k)} \\
	& {\cod(k)} \\
	{\cod(k)} && {\cod(k)}
	\arrow["k"', from=2-1, to=4-1]
	\arrow["{b'}", from=2-1, to=1-2]
	\arrow["a", from=1-2, to=2-3]
	\arrow[Rightarrow, no head, from=4-1, to=4-3]
	\arrow["k", Rightarrow, no head, from=2-3, to=4-3]
	\arrow[Rightarrow, no head, from=4-1, to=3-2]
	\arrow[Rightarrow, no head, from=3-2, to=4-3]
	\arrow["ka"{description, pos=0.3}, from=1-2, to=3-2]
	\arrow["{\alpha\atop \simeq}"{description}, draw=none, from=2-1, to=3-2]
	\arrow["{=}"{description}, draw=none, from=3-2, to=2-3]
	\arrow[Rightarrow, no head, from=2-1, to=2-3, crossing over]
\end{tikzcd}\]
exhibiting $ \dom(k)$ as a pseudoretract of a bicompact, and hence as a bicompact object itself. \\

Now we prove that any arrow between bicompact object is bicompact in $ \ps[2, \mathcal{B}]$. Take $ k : K \rightarrow K'$ with $ K, \, K'$ bicompact, and $ F : I \rightarrow \ps[2, \mathcal{B}]$ a bifiltered diagram. Then for any square as below 
% https://q.uiver.app/?q=WzAsNCxbMCwwLCJLIl0sWzAsMSwiSyciXSxbMSwwLCJcXGJpY29saW0gXFwsIFxcZG9tXFwsIEYiXSxbMSwxLCJcXGJpY29saW0gXFwsIFxcY29kXFwsIEYiXSxbMCwyLCJhIl0sWzEsMywiYSciLDJdLFsyLDMsIlxcYmljb2xpbSBcXCwgIEYiXSxbMCwxLCJrIiwyXSxbMCwzLCJcXGFscGhhIFxcYXRvcCBcXHNpbWVxIiwxLHsic3R5bGUiOnsiYm9keSI6eyJuYW1lIjoibm9uZSJ9LCJoZWFkIjp7Im5hbWUiOiJub25lIn19fV1d
\[\begin{tikzcd}
	K & {\bicolim \, \dom\, F} \\
	{K'} & {\bicolim \, \cod\, F}
	\arrow["a", from=1-1, to=1-2]
	\arrow["{a'}"', from=2-1, to=2-2]
	\arrow["{\bicolim \,  F}", from=1-2, to=2-2]
	\arrow["k"', from=1-1, to=2-1]
	\arrow["{\alpha \atop \simeq}"{description}, draw=none, from=1-1, to=2-2]
\end{tikzcd}\]
bicompactness of $ K$ and $ K'$ provides us respectively with a upper and lower lift 
% https://q.uiver.app/?q=WzAsNixbMCwxLCJLIl0sWzAsMywiSyciXSxbMiwxLCJcXGJpY29saW0gXFwsIFxcZG9tXFwsIEYiXSxbMiwzLCJcXGJpY29saW0gXFwsIFxcY29kXFwsIEYiXSxbMSwwLCJcXGRvbVxcLCBGKGkpIl0sWzEsMiwiXFxjb2RcXCwgRihpJykiXSxbMCwyLCJhIiwxXSxbMSwzLCJhJyIsMV0sWzIsMywiXFxiaWNvbGltIFxcLCAgRiJdLFswLDEsImsiLDJdLFswLDQsImIiXSxbNCwyLCJcXGRvbSBcXCwgcV9pIl0sWzEsNSwiYiciXSxbNSwzLCJcXGNvZCBcXCwgcV97aSd9IiwyXSxbNCw2LCJcXGJldGEgXFxhdG9wIFxcc2ltZXEiLDEseyJzaG9ydGVuIjp7InRhcmdldCI6MjB9LCJzdHlsZSI6eyJib2R5Ijp7Im5hbWUiOiJub25lIn0sImhlYWQiOnsibmFtZSI6Im5vbmUifX19XSxbNSw3LCJcXGJldGEnIFxcYXRvcCBcXHNpbWVxIiwxLHsic2hvcnRlbiI6eyJ0YXJnZXQiOjIwfSwic3R5bGUiOnsiYm9keSI6eyJuYW1lIjoibm9uZSJ9LCJoZWFkIjp7Im5hbWUiOiJub25lIn19fV1d
\[\begin{tikzcd}
	& {\dom\, F(i)} \\
	K && {\bicolim \, \dom\, F} \\
	& {\cod\, F(i')} \\
	{K'} && {\bicolim \, \cod\, F}
	\arrow[""{name=0, anchor=center, inner sep=0}, "a"{description}, from=2-1, to=2-3]
	\arrow[""{name=1, anchor=center, inner sep=0}, "{a'}"{description}, from=4-1, to=4-3]
	\arrow["{\bicolim \,  F}", from=2-3, to=4-3]
	\arrow["k"', from=2-1, to=4-1]
	\arrow["b", from=2-1, to=1-2]
	\arrow["{\dom \, q_i}", from=1-2, to=2-3]
	\arrow["{b'}", from=4-1, to=3-2]
	\arrow["{\cod \, q_{i'}}", from=3-2, to=4-3]
	\arrow["{\beta \atop \simeq}"{description}, Rightarrow, draw=none, from=1-2, to=0]
	\arrow["{\beta' \atop \simeq}"{description}, Rightarrow, draw=none, from=3-2, to=1]
\end{tikzcd}\]
and we want to construct further lifts with a vertical arrow between them. But precomposing respectively the upper lift with $ F(i)$ and the lower lift with $ k$ provides two distincts lifts of a same arrow $ \bicolim \, F a$ as seen below
% https://q.uiver.app/?q=WzAsNCxbMCwxLCJLIl0sWzIsMSwiXFxiaWNvbGltIFxcLCBcXGNvZFxcLCBGIl0sWzEsMCwiXFxkb21cXCwgRihpKSJdLFsxLDIsIlxcY29kXFwsIEYoaScpIl0sWzAsMiwiYiJdLFszLDEsIlxcY29kIFxcLCBxX3tpJ30iLDJdLFsyLDEsIlxcYmljb2xpbSBcXCwgIEZcXGRvbSBcXCwgcV9pIl0sWzAsMSwiXFxiaWNvbGltIFxcLCAgRiBhIiwxXSxbMCwzLCJiJ2siLDJdLFsyLDcsIlxcYmljb2xpbSBcXCwgIEYqXFxiZXRhIFxcYXRvcCBcXHNpbWVxIiwxLHsic2hvcnRlbiI6eyJ0YXJnZXQiOjIwfSwic3R5bGUiOnsiYm9keSI6eyJuYW1lIjoibm9uZSJ9LCJoZWFkIjp7Im5hbWUiOiJub25lIn19fV0sWzcsMywiXFxiZXRhJyprIFxcYXRvcCBcXHNpbWVxIiwxLHsic2hvcnRlbiI6eyJzb3VyY2UiOjIwfSwic3R5bGUiOnsiYm9keSI6eyJuYW1lIjoibm9uZSJ9LCJoZWFkIjp7Im5hbWUiOiJub25lIn19fV1d
\[\begin{tikzcd}
	& {\cod\, F(i)} \\
	K && {\bicolim \, \cod\, F} \\
	& {\cod\, F(i')}
	\arrow["F(i)b", from=2-1, to=1-2]
	\arrow["{\cod \, q_{i'}}"', from=3-2, to=2-3]
	\arrow["{\cod \, q_i}", from=1-2, to=2-3]
	\arrow[""{name=0, anchor=center, inner sep=0}, "{\bicolim \,  F a}"{description}, from=2-1, to=2-3]
	\arrow["{b'k}"', from=2-1, to=3-2]
	\arrow["{\bicolim \,  F*\beta \atop \simeq}"{description}, Rightarrow, draw=none, from=1-2, to=0]
	\arrow["{\beta'*k \atop \simeq}"{description}, Rightarrow, draw=none, from=0, to=3-2]
\end{tikzcd}\]
Now for $ K$ is bicompact, those two lifts admit a common refinement for some $ d : i \rightarrow j$, $ d' : i' \rightarrow j$ equipped with an invertible 2-cell $ \gamma$, and in the diagram below
\[\begin{tikzcd}
	& {\dom\, F(i)} \\
	&& {\dom\, F(j)} && {\bicolim \, \dom\, F} \\
	& {\cod\, F(i)} &&& \bullet \\
	K && {\cod\, F(j)} && {\bicolim \, \cod\, F} \\
	& {\cod\, F(i')}
	\arrow["{F(i)b}"{description}, from=4-1, to=3-2]
	\arrow[""{name=0, anchor=center, inner sep=0}, "{\cod \, q_{i'}}"', curve={height=12pt}, from=5-2, to=4-5]
	\arrow[""{name=1, anchor=center, inner sep=0}, "{\cod \, q_{i}}"{description, pos=0.6}, curve={height=-12pt}, from=3-2, to=4-5]
	\arrow["{b'k}"', from=4-1, to=5-2]
	\arrow["{\cod\, F(d)}"{description}, from=3-2, to=4-3]
	\arrow["{\cod\, F(d')}"{description}, from=5-2, to=4-3]
	\arrow["{\cod \,q_j}"{description}, from=4-3, to=4-5]
	\arrow["{\gamma \atop \simeq}"{description}, draw=none, from=3-2, to=5-2]
	\arrow["{F(i)}"{description}, from=1-2, to=3-2]
	\arrow[""{name=2, anchor=center, inner sep=0}, "b", curve={height=-12pt}, from=4-1, to=1-2]
	\arrow["{\dom\, F(d)}"{description}, from=1-2, to=2-3]
	\arrow["{\dom \,q_j}"{description}, from=2-3, to=2-5]
	\arrow["{\bicolim \,  F}"{description}, from=2-5, to=4-5]
	\arrow[""{name=3, anchor=center, inner sep=0}, "{\dom\, q_i}"{description}, curve={height=-12pt}, from=1-2, to=2-5]
	\arrow["{q_j \atop \simeq}"{description}, draw=none, from=2-3, to=3-5]
	\arrow["{F(d) \atop \simeq}"{description}, draw=none, from=1-2, to=4-3]
	\arrow["{F(j)}"{description, pos=0.3}, from=2-3, to=4-3, crossing over]
	\arrow["{\cod \,\theta_d \atop \simeq}"{description}, shift left=4, Rightarrow, draw=none, from=1, to=4-3]
	\arrow["{\theta_d \atop \simeq}"{description}, Rightarrow, draw=none, from=4-3, to=0]
	\arrow["{=}"{description}, Rightarrow, draw=none, from=2, to=3-2]
	\arrow["{\dom \, \theta_d \atop \simeq}"{description}, Rightarrow, draw=none, from=3, to=2-3]
\end{tikzcd}\]
we can extract the following lift 
% https://q.uiver.app/?q=WzAsNCxbMCwwLCJLIl0sWzAsMSwiSyciXSxbMSwwLCJcXGRvbSBcXCwgRihqKSJdLFsxLDEsIlxcY29kIFxcLCBGKGopIl0sWzAsMSwiayIsMl0sWzAsMiwiXFxkb21cXCwgRihkKSBiIl0sWzIsMywiRihqKSJdLFsxLDMsIlxcY29kXFwsRihkJyliJyIsMl0sWzAsMywiXFxnYW1tYSBiKkYoZCkgXFxjb2RcXCwgRihkKSpcXGJldGEgXFxhdG9wIFxcc2ltZXEiLDEseyJzdHlsZSI6eyJib2R5Ijp7Im5hbWUiOiJub25lIn0sImhlYWQiOnsibmFtZSI6Im5vbmUifX19XV0=
\[\begin{tikzcd}[column sep=huge]
	K & {\dom \, F(j)} \\
	{K'} & {\cod \, F(j)}
	\arrow["k"', from=1-1, to=2-1]
	\arrow["{\dom\, F(d) b}", from=1-1, to=1-2]
	\arrow["{F(j)}", from=1-2, to=2-2]
	\arrow["{\cod\,F(d')b'}"', from=2-1, to=2-2]
	\arrow["{\gamma b*F(d) \cod\, F(d)*\beta \atop \simeq}"{description}, draw=none, from=1-1, to=2-2]
\end{tikzcd}\]
Now for the two-dimensional condition, consider a morphism of pseudosquares
% https://q.uiver.app/?q=WzAsNCxbMCwwLCJLIl0sWzAsMSwiSyciXSxbMiwwLCJcXGJpY29saW0gXFwsIFxcZG9tXFwsIEYiXSxbMiwxLCJcXGJpY29saW0gXFwsIFxcY29kXFwsIEYiXSxbMCwyLCJhXzIiLDFdLFsxLDMsImFfMiciLDJdLFsyLDMsIlxcYmljb2xpbSBcXCwgIEYiXSxbMCwxLCJrIiwyXSxbMCwzLCJcXGFscGhhXzIgXFxhdG9wIFxcc2ltZXEiLDEseyJzdHlsZSI6eyJib2R5Ijp7Im5hbWUiOiJub25lIn0sImhlYWQiOnsibmFtZSI6Im5vbmUifX19XSxbMCwyLCJhXzEiLDAseyJjdXJ2ZSI6LTR9XSxbOSw0LCJcXHBoaSIsMCx7InNob3J0ZW4iOnsic291cmNlIjoyMCwidGFyZ2V0IjoyMH19XV0=
\[\begin{tikzcd}
	K && {\bicolim \, \dom\, F} \\
	{K'} && {\bicolim \, \cod\, F}
	\arrow[""{name=0, anchor=center, inner sep=0}, "{a_1'}"{description}, from=2-1, to=2-3]
	\arrow["{\bicolim \,  F}", from=1-3, to=2-3]
	\arrow["k"', from=1-1, to=2-1]
	\arrow["{\alpha_1 \atop \simeq}"{description}, draw=none, from=1-1, to=2-3]
	\arrow["{a_1}", from=1-1, to=1-3]
	\arrow[""{name=1, anchor=center, inner sep=0}, "{a_2'}"', bend right=30, end anchor=190, from=2-1, to=2-3]
	\arrow["{\phi'}"{pos=-0.2}, shorten <=3pt, shorten >=3pt, Rightarrow, from=0, to=1]
\end{tikzcd}\; =\; \begin{tikzcd}
	K && {\bicolim \, \dom\, F} \\
	{K'} && {\bicolim \, \cod\, F}
	\arrow[""{name=0, anchor=center, inner sep=0}, "{a_2}"{description}, from=1-1, to=1-3]
	\arrow["{a_2'}"', from=2-1, to=2-3]
	\arrow["{\bicolim \,  F}", from=1-3, to=2-3]
	\arrow["k"', from=1-1, to=2-1]
	\arrow["{\alpha_2 \atop \simeq}"{description}, draw=none, from=1-1, to=2-3]
	\arrow[""{name=1, anchor=center, inner sep=0}, "{a_1}", bend left=30, end anchor=170, from=1-1, to=1-3]
	\arrow["\phi", shorten <=4pt, shorten >=4pt, Rightarrow, from=1, to=0]
\end{tikzcd}\]
Then we have both an upper and lower morphism of lifts as below
\[\begin{tikzcd}[column sep=large]
	& {\dom\, F(i_1)} \\
	K && {\dom\,F(j)} & {\bicolim \, \dom\, F} \\
	& {\dom\, F(i_2)} \\
	& {\cod\, F(i_1')} \\
	{K'} && {\cod\,F(j')} & {\bicolim \, \cod\, F} \\
	& {\cod\, F(i_2')}
	\arrow["k"', from=2-1, to=5-1]
	\arrow["{b_1}", from=2-1, to=1-2]
	\arrow["{\dom\, F(d_1)}"{description}, from=1-2, to=2-3]
	\arrow["{b_1'}"{description}, from=5-1, to=4-2]
	\arrow["{\cod\, F(d_1')}"{description}, from=4-2, to=5-3]
	\arrow["{b_2}"{description}, from=2-1, to=3-2]
	\arrow["{\dom\, F(d_2)}"{description}, from=3-2, to=2-3]
	\arrow["{b_2'}"', from=5-1, to=6-2]
	\arrow["{\cod\, F(d_2')}"{description}, from=6-2, to=5-3]
	\arrow["{\bicolim \,  F}", from=2-4, to=5-4]
	\arrow["{\dom\,q_j}"{description}, from=2-3, to=2-4]
	\arrow["{\cod\,q_{j'}}"{description}, from=5-3, to=5-4]
	\arrow["\psi", shorten <=10pt, shorten >=10pt, Rightarrow, from=1-2, to=3-2]
	\arrow["{\psi'}", shorten <=10pt, shorten >=10pt, Rightarrow, from=4-2, to=6-2]
	\arrow[""{name=0, anchor=center, inner sep=0}, "{\dom\, q_{i_1}}"{description}, curve={height=-18pt}, from=1-2, to=2-4]
	\arrow[""{name=1, anchor=center, inner sep=0}, "{\dom\, q_{i_2}}"{description}, curve={height=18pt}, from=3-2, to=2-4]
	\arrow[""{name=2, anchor=center, inner sep=0}, "{\cod\, q_{i_1'}}"{description}, curve={height=-18pt}, from=4-2, to=5-4]
	\arrow[""{name=3, anchor=center, inner sep=0}, "{\cod\, q_{i_2'}}"{description}, curve={height=18pt}, from=6-2, to=5-4]
	\arrow["{\dom\, \theta_{d_1} \atop \simeq}"{description}, Rightarrow, draw=none, from=0, to=2-3]
	\arrow["{\dom\, \theta_{d_2} \atop \simeq}"{description}, Rightarrow, draw=none, from=2-3, to=1]
	\arrow["{\cod\, \theta_{d_1'} \atop \simeq}"{description}, Rightarrow, draw=none, from=2, to=5-3]
	\arrow["{\cod\, \theta_{d_2'} \atop \simeq}"{description}, Rightarrow, draw=none, from=3, to=5-3]
\end{tikzcd}\]
and by the same argument as before, we can exhibit a further refinement $ g: j \rightarrow l $, $ g': j' \rightarrow l$ defining a morphism of lifts in $ \ps[2, \mathcal{B}]$:
% https://q.uiver.app/?q=WzAsNCxbMCwwLCJLIl0sWzAsMiwiSyciXSxbMiwwLCJcXGRvbSBcXCwgRihsKSJdLFsyLDIsIlxcY29kIFxcLCBGKGwpIl0sWzAsMSwiayIsMl0sWzAsMiwiXFxkb21cXCwgRihnZF8xKSBiXzEiLDAseyJjdXJ2ZSI6LTN9XSxbMiwzLCJGKGwpIl0sWzEsMywiXFxjb2RcXCxGKGcnZF8yJyliJ18yIiwyXSxbMCwzLCJcXGdhbW1hIGIqRihnZF8yKSBcXGNvZFxcLCBGKGRfMikqXFxiZXRhXzIgXFxhdG9wIFxcc2ltZXEiLDEseyJzdHlsZSI6eyJib2R5Ijp7Im5hbWUiOiJub25lIn0sImhlYWQiOnsibmFtZSI6Im5vbmUifX19XSxbMCwyLCJcXGRvbVxcLCBGKGdkXzIpIGJfMiIsMl0sWzUsOSwiXFxkb21cXCwgRihnKSogXFxwc2kiLDAseyJzaG9ydGVuIjp7InNvdXJjZSI6MjAsInRhcmdldCI6MjB9fV1d
\[ \begin{tikzcd}[column sep=huge]
	K && {\dom \, F(l)} \\
	\\
	{K'} && {\cod \, F(l)}
	\arrow["k"', from=1-1, to=3-1]
	\arrow["{\dom\, F(gd_1) b_1}", from=1-1, to=1-3]
	\arrow["{F(l)}", from=1-3, to=3-3]
	\arrow[""{name=0, anchor=center, inner sep=0}, "{\cod\,F(g'd_2')b'_2}"', bend right=45, end anchor=190, from=3-1, to=3-3]
	\arrow["{\gamma b*F(gd_1) \cod\, F(d_1)*\beta_1 \atop \simeq}"{description}, draw=none, from=1-1, to=3-3]
	\arrow[""{name=1, anchor=center, inner sep=0}, "{\cod\,F(g'd_2')b'_2}", from=3-1, to=3-3]
	\arrow["{\cod\,F(g')*\psi'}"{description}, shorten <=2pt, shorten >=2pt, Rightarrow, from=1, to=0]
\end{tikzcd} \; = \;  \begin{tikzcd}[column sep=huge]
	K && {\dom \, F(l)} \\
	\\
	{K'} && {\cod \, F(l)}
	\arrow["k"', from=1-1, to=3-1]
	\arrow[""{name=0, anchor=center, inner sep=0}, "{\dom\, F(gd_1) b_1}", bend left=45, end anchor=170, from=1-1, to=1-3]
	\arrow["{F(l)}", from=1-3, to=3-3]
	\arrow["{\cod\,F(g'd_2')b'_2}"', from=3-1, to=3-3]
	\arrow["{\gamma b*F(gd_2) \cod\, F(d_2)*\beta_2 \atop \simeq}"{description}, draw=none, from=1-1, to=3-3]
	\arrow[""{name=1, anchor=center, inner sep=0}, "{\dom\, F(gd_2) b_2}"', from=1-1, to=1-3]
	\arrow["{\dom\, F(g)* \psi}"{description}, shorten <=3pt, shorten >=3pt, Rightarrow, from=0, to=1]
\end{tikzcd}\]
This achieves to prove that bicompacts of $ \ps[2, \mathcal{B}]$ are exactly arrows between bicompacts. \\

Finally we have to prove that any arrow is a bifiltered bicolimit of bicompact arrows. For any $ f $ we have both that $ \dom (f) = \bicolim \, \mathcal{B}_\omega \downarrow \dom (f)$ and $ \cod (f) = \bicolim \, \mathcal{B}_\omega \downarrow \cod (f)$, and moreover $f$ is induced as $ f = \langle fa \rangle_{ \mathcal{B}_\omega \downarrow \dom (f)}$. But now, as $ \mathcal{B}_\omega \downarrow \cod (f)$ for each $a : K \rightarrow \dom (f)$ we can pick a lift 
% https://q.uiver.app/?q=WzAsNCxbMCwwLCJLIl0sWzEsMCwiXFxkb20gXFwsIGYiXSxbMSwxLCJcXGNvZCBcXCwgZiJdLFswLDEsIksnIl0sWzEsMiwiZiJdLFswLDEsImEiXSxbMywyLCJhJyIsMl0sWzAsMywiYiIsMix7InN0eWxlIjp7ImJvZHkiOnsibmFtZSI6ImRhc2hlZCJ9fX1dLFswLDIsIlxcYmV0YSBcXGF0b3AgXFxzaW1lcSIsMSx7InN0eWxlIjp7ImJvZHkiOnsibmFtZSI6Im5vbmUifSwiaGVhZCI6eyJuYW1lIjoibm9uZSJ9fX1dXQ==
\[\begin{tikzcd}
	K & {\dom \, f} \\
	{K'} & {\cod \, f}
	\arrow["f", from=1-2, to=2-2]
	\arrow["a", from=1-1, to=1-2]
	\arrow["{a'}"', from=2-1, to=2-2]
	\arrow["b"', dashed, from=1-1, to=2-1]
	\arrow["{\beta \atop \simeq}"{description}, draw=none, from=1-1, to=2-2]
\end{tikzcd}\]
and $f $ is exhibited as the bicolimit of the subcategory of $ \ps[2, \mathcal{B}_\omega] \downarrow f$ consisting of all those lifts $ (a, b, \beta)$ for $ a : K \rightarrow \dom(f)$ and $ (b,\beta) $ a lift of $fa$. Now checking that this subcategory is bifiltered is a straitforward utilisation of the bifilteredness of the canonical cones of the domain and codomain. 
\end{proof}

\section{2-Categories of flat pseudofunctors} \label{sectionbiflat}

We deduced several properties of the bi-accessible and bipresentable 2-categories from analysing their binerve pseudofunctors, which allowed to see them as 2-categories of $\Cat$-valued pseudofunctors. Here we describe the precise class of pseudofunctors obtained through this process, the analog of the ordinary \emph{flat} functors. They were defined in \cite{descotte2018sigma}, from which we give the following definitions and elementary property. As in the 1-dimensional case, those are pseudofunctors that virtually preserves finitely weighted bilimits whenever they exist (which amounts to testing real preservation at the level of the left biKan extension); it was also remarked that this amounted to requiring their category of elements to be $\sigma$-cofiltered relatively to their opcartesian morphisms. We give a further simplification of this latter property into a condition of bifilteredness thanks to our key observation, which harmonizes this result with our definitions of bi-accessibility. We then prove the 2-categories of flat pseudofunctors to be themselves bi-accessible and bipresentable whenever their domain 2-category have finite weighted bilimits. 

\subsection{Extension of flat pseudofunctors}

\begin{definition}[Bilex $2$-categories]
A \emph{bilex}\index{bilex!2-category}\index{bilex!pseudofunctor} 2-category is a 2-category with all finite weighted bilimits as defined in \Cref{finiteweibili}, and a pseudofunctor is said to be \emph{bilex} if it preserves them (up to equivalence). For a small bilex 2-category $ \mathcal{C}$ and a 2-category $\mathcal{D}$ we write $ \biLex[\mathcal{C}, \mathcal{D}]$ the 2-category of bilex pseudofunctors from $\mathcal{C}$ to $\mathcal{D}$. 
\end{definition}

This very brief subsection sets the stage for the later discussions, we will study extensions of pseudofunctors along the Yoneda 2-embedding into the 2-category of pseudofunctors into $\Cat$ (which is biequivalent to the 2-category of strict 2-functors and pseudonatural transformations)
\[ \begin{tikzcd}
    {\mathcal{C}} \arrow["\hirayo", r,hook] &  {[\mathcal{C}^{\op}, \Cat]_p \simeq \ps[\mathcal{C}^{\op}, \Cat]}
\end{tikzcd}  \]
We will use pseudofunctors into $\Cat$ to keep the correct level of strictness; again, for $ \Cat$ is a strict 2-category, pseudofunctors always form a strict 2-category, as stated for instance at \cite{johnson20212}[Corollary 4.4.13]. %The subsection is not original and follows the work of \cite{descotte2018sigma}.

\begin{definition}[Flat pseudofunctor {\cite[Def. 4.1.11]{descotte2018sigma}}]
A pseudofunctor $F : \mathcal{C} \rightarrow \Cat$ is \emph{flat}\index{flat pseudofunctor} if its left biKan extension $ \biLan_{\hirayo} F$ is bilex. In particular, a flat pseudofunctor preserves any finitely weighted bilimit already existing in $\mathcal{C}$. 
\end{definition}

\begin{proposition}[{\cite[Prop 4.1.14]{descotte2018sigma}}]
Corepresentable 2-functors $ \katayo_C : \mathcal{C} \rightarrow \Cat$ are flat 2-functors. Their biKan extension can be chosen to be the evaluation functor at the corresponding object $ C$.
\end{proposition}

In \cite{descotte2018sigma}, we find this crucial theorem which provides the 2-dimensional Diaconescu theorem of extension of flat (pseudo)functors. Our formulation puts together two results in \cite{descotte2018sigma}, which account on the possible levels of strinctness of the result.

\begin{theorem}[{\cite[Prop 4.2.7 and A.6]{descotte2018sigma}}]\label{Dubuc extension of flat pseudofunctor}
Let $\mathcal{C}$ be a small 2-category. Then for a 2-functor (resp. pseudofunctor) $ F :  \mathcal{C} \rightarrow \Cat$ we have the following equivalences \begin{itemize}
    \item $F$ is flat, that is, $ \biLan_{\hirayo} F$ is bilex
    \item $( \int F)^{\op}$ is $\sigma$-filtered relatively to the class of opcartesian arrows
    \item $F$ is a $\sigma$-filtered pseudocolimit (resp. bicolimit) of representables in $ [\mathcal{C}, \Cat]_p$ (resp. in $ \ps[\mathcal{C}, \Cat]$).
\end{itemize}
\end{theorem}

\begin{comment}
The theorem says much more than Diaconescu: not only a flat functor extends to a 2-geometric morphism, but it also says that such a functor is a filtered colimit of representables. In the case where $ \mathcal{C}$ is bilex, it proves that points form a finitely bipresentable 2-category. 
\end{comment}

But in the regard of \Cref{key lemma}, it appears that one can complete this theorem with a last item simplifying the $\sigma$-filtered decomposition into a bifiltered one:

\begin{lemma}[Flatness as a bifilteredness condition]\label{flatness as bifilteredness condition}
A pseudofunctor $ F : \mathcal{C} \rightarrow \Cat$ is flat if and only if it decomposes as a bifiltered bicolimit of representables.
\end{lemma}

\begin{proof}
In \Cref{Dubuc extension of flat pseudofunctor}, the last item is obtained by combining the general fact that any pseudofunctor $F $ is the $ \sigma$-bicolimit of the composite 
% https://q.uiver.app/?q=WzAsMyxbMCwwLCIoXFxkaXNwbGF5c3R5bGVcXGludCBGKV57XFxvcH0iXSxbMSwwLCJcXG1hdGhjYWx7Q31ee1xcb3B9Il0sWzIsMCwiXFxwc1tcXG1hdGhjYWx7Q30sIFxcQ2F0XSJdLFswLDEsIlxccGlfRiJdLFsxLDIsIlxca2F0YXlvIl1d
\[\begin{tikzcd}
	{(\displaystyle\int F)^{\op}} & {\mathcal{C}^{\op}} & {\ps[\mathcal{C}, \Cat]}
	\arrow["{\pi_F}", from=1-1, to=1-2]
	\arrow["\katayo", from=1-2, to=1-3]
\end{tikzcd}\]
for the class of opposites of opcartesian morphisms. Now $F$ is flat if and only if $ ((\int F)^{\op}, \textrm{op}\Cart^{\op})$ is a $\sigma$-filtered pair, which amounts by \Cref{key lemma} to saying that the full on 0-cells and 2-cells subcategory $ \iota_F : \textrm{op}\Cart^{\op} \hookrightarrow (\int F)^{\op}$ is bifiltered and $\sigma$-cofinal in $ (\int F)^{\op}$ relatively to itself. But then by \Cref{cofinal functors and bicolimits}, those two observations yield altogether that we have a bifiltered bicolimit 
\[  F \simeq \underset{\textrm{op}\Cart^{\op} }{\bicolim} \katayo \pi_F \iota_F  \]
over the restriction of its 2-category of elements to opcartesian morphisms. 
\end{proof}

\begin{remark}
This result just is dual to our observations that the pseudoslice restricted to bicompact objects provides a convenient notion of canonical diagram in a bi-accessible category. We are precisely going to see why in the next section. 
\end{remark}

Now it appears that we already encountered flat pseudofunctors when examining bi-accessible and bipresentable 2-categories: their binerve pseudofunctor identified them with 2-categories of flat pseudofunctors over their generator of bicompact object:

\begin{theorem}[Canonical representation of bi-accessible $2$-categories]\label{sigma-accessible 2-cat are categories of flat functors}
For a finitely bi-accessible 2-category $\mathcal{B}$, we have a biequivalence 
\[ \mathcal{B} \simeq \Flat_\ps[(\mathcal{B}_\omega)^{\op}, \Cat] \]
\end{theorem}

\begin{proof}
By \Cref{biacc as category of flat pseudofunctors}, $ \mathcal{B}$ is equivalent to a full on 1-cells and 2-cells sub-2-category of $\ps[(\mathcal{B}_\omega)^{\op}, \Cat] $ which is moreover closed under bifiltered bicolimits by \Cref{binerve preserve bifiltered bicolimits}. On the other hand, we know from \Cref{the pseudococone is bifiltered} that for each $ B$ the canonical pseudococone $ \mathcal{B}_\omega\downarrow B$ is bifiltered, or equivalently that the oplax cocone $ \mathcal{B}_\omega \Downarrow B$ is $\sigma$-filtered relatively to $\Cart_B$: but we saw at \Cref{the oplaxcocone is the category of element of the binerve} that the oplax cocone is exactly the category of elements of the image $ \nu_B$ of $B$ along the binerve; this exactly means that $ \nu_B$ is flat, so that $ \nu$ factorizes through the inclusion $ \Flat_\ps[(\mathcal{B}_\omega)^{\op}, \Cat]  \hookrightarrow \ps[(\mathcal{B}_\omega)^{\op}, \Cat] $, exhibiting $\mathcal{B} $ as consisting of flat pseudofunctors. 

For the converse, observe that representable 2-functors of the form $ \hirayo_K : (\mathcal{B}_\omega)^{\op} \rightarrow \Cat$ are in particular equivalent to the image of the underlying object along the binerve, that is $ \hirayo_K \simeq \nu_K $, so they are at the same time flat and in the range of $ \nu$. But by \Cref{Dubuc extension of flat pseudofunctor} any flat pseudofunctor is a $\sigma$-filtered bicolimit of representables; for $ \nu$ preserves $\sigma$-filtered bicolimits, this forces any flat functor to be equivalent to a functor of the form $ \nu_B$ for $B$ in $\mathcal{B}$. 
\end{proof}

In the presence of finite weighted bilimits, flatness simplifies as follows:

\begin{proposition}[\cite{descotte2018sigma} Proposition 4.2.9]\label{flat is bilex}
If $ \mathcal{C}$ has finite weighted bilimits, then flat pseudofunctors $ \mathcal{C} \rightarrow \Cat$ are exactly the bilex ones, that is that we have a biequivalence
\[  \Flat_\ps[\mathcal{C},\Cat] \simeq \biLex[\mathcal{C}, \Cat] \]
\end{proposition}

In particular, in the case of a finitely bipresentable 2-category, where the generator of bicompact object is closed under finite weighted bicolimits, \Cref{sigma-accessible 2-cat are categories of flat functors} reduces to the following:

\begin{theorem}[Representation theorem for finitely bipresentable $2$-categories] \label{A is flat functors}
Let $\mathcal{B}$ be a finitely bipresentable $2$-category. Then the binerve pseudofunctor induces a biequivalence of $2$-categories with bilex pseudofunctors \[ \mathcal{B} \simeq \textup{\biLex}[(\mathcal{B}_\omega)^\op, \Cat]\] 
\end{theorem}

\subsection{2-category of flat pseudofunctors are finitely bi-accessibles}

In this subsection, we want to exhibit the relation between a small bilex category and the bicompact objects of the associated category of pseudofunctors. The results of this part are actually quite similar to the strategy in the 1-categorical context: \begin{itemize}
    \item[\ref{creation of sigmafiltered colimits of flat functors}] we first prove that 2-categories of flat pseudofunctors have $\sigma$-filtered colimits which are computed in the 2-categories of pseudofunctors,
    \item[\ref{corepresentable are compact}] then we prove that bicorepresentable are always bicompact in 2-categories of flat pseudofunctors,
   % \item[] then we describe the canonical cone of a given functor 
    \item[\ref{compact are retract of representable}] finally, we prove that bicompact objects are pseudoretracts of birepresentables in $\Flat_{\ps}[\mathcal{C}, \Cat]$.
\end{itemize}

First of all, observe that for $\mathcal{C}$ a small 2-category, birepresentables are in $\Flat_{\ps}[\mathcal{C}, \Cat]$. This is an immediate consequence of preservation of any bilimit by corepresentables. Hence the Yoneda embedding $ \katayo$ factorizes through the 2-category of flat pseudofunctors.

\begin{lemma}\label{corepresentable are tiny}
For $\mathcal{C}$ a small 2-category, birepresentables are $\sigma$-tiny in $ \ps[\mathcal{C}, \Cat]$.
\end{lemma}

\begin{proof}
Let $ F : \mathcal{C} \rightarrow \Cat $ be a flat pseudofunctor equipped with a pseudonatural equivalence $ F \simeq \katayo_A$
% https://q.uiver.app/?q=WzAsNCxbMCwwLCJGIl0sWzAsMSwiRiJdLFsxLDAsIlxcaGlyYXlvXipfQSJdLFsxLDEsIlxcaGlyYXlvXipfQSJdLFswLDEsIiIsMCx7ImxldmVsIjoyLCJzdHlsZSI6eyJoZWFkIjp7Im5hbWUiOiJub25lIn19fV0sWzAsMiwiciJdLFsyLDMsIiIsMix7ImxldmVsIjoyLCJzdHlsZSI6eyJoZWFkIjp7Im5hbWUiOiJub25lIn19fV0sWzEsMywiciIsMl0sWzIsMSwicyIsMV0sWzAsOCwiXFxhbHBoYSBcXGF0b3AgXFxzaW1lcSIsMSx7InNob3J0ZW4iOnsidGFyZ2V0IjoyMH0sInN0eWxlIjp7ImJvZHkiOnsibmFtZSI6Im5vbmUifSwiaGVhZCI6eyJuYW1lIjoibm9uZSJ9fX1dLFs4LDMsIlxcYmV0YSBcXGF0b3AgXFxzaW1lcSIsMSx7InNob3J0ZW4iOnsic291cmNlIjoyMH0sInN0eWxlIjp7ImJvZHkiOnsibmFtZSI6Im5vbmUifSwiaGVhZCI6eyJuYW1lIjoibm9uZSJ9fX1dXQ==
and $ G : I \rightarrow \ps[\mathcal{C}, \Cat] $ with $ I$ a small 2-category equipped with a class $\Sigma$, then we have a sequence of equivalences
\begin{align*}
    \ps[\mathcal{C}, \Cat] \big[ F, {\Sigma \bicolim} \, G \big] &\simeq \ps[\mathcal{C}, \Cat] \big[ \katayo_A, {\Sigma\bicolim} \, G \big] \\ 
    &\simeq ({\Sigma\bicolim} \, G)(A) \\ 
    &\simeq \underset{i \in I}{\Sigma\bicolim} \, G(i)(A)\\
    &\simeq \underset{i \in I}{\Sigma\bicolim} \, \ps[\mathcal{C}, \Cat] \big[ \katayo_A, G(i) \big] \\
    &\simeq \underset{i \in I}{\Sigma\bicolim}\, \ps[\mathcal{C}, \Cat] \big[ F, G(i) \big]
\end{align*}
where the third equivalence comes from the fact that $\sigma$-bicolimits are pointwise in pseudofunctors categories. 
\end{proof}

\begin{proposition}\label{creation of sigmafiltered colimits of flat functors}
If $ \mathcal{C}$ is a small 2-category, then the 2-category $ \Flat_{\ps}[\mathcal{C}, \Cat]$ has $\sigma$-filtered bicolimits. Moreover they are created by the pseudo-fully faithful inclusion 
\[ \Flat_{\ps}[\mathcal{C}, \Cat] \stackrel{i_C}{\hookrightarrow }  \ps[\mathcal{C}, \Cat] \]
\end{proposition}

\begin{proof}
See \cite{descotte2020theory}[2.3 and 2.3.5]. 
\end{proof}

\begin{corollary}
As a consequence, $ \Flat_{\ps}[\mathcal{C}, \Cat]$ also has bifiltered bicolimits created by the inclusion. 
\end{corollary}

\begin{corollary}\label{corepresentable are compact}
For any small 2-category $\mathcal{C}$, bicorepresentables are bicompact in $  \Flat_{\ps}[\mathcal{C}, \Cat]$
\end{corollary}

\begin{proof}
From \Cref{corepresentable are tiny}, corepresentables are $\sigma$-tiny in $\ps[\mathcal{C}, \Cat]$, where they are hence bicompact; and as $\Flat_{\ps}[\mathcal{C}, \Cat]$ is closed in $\ps[\mathcal{C}, \Cat]$ under bifiltered bicolimits, we are done.  
\end{proof}

In the general case, there are more bicompacts than bicorepresentables in 2-categories of flat functors; we would like to characterize those bicompact objects.

\begin{proposition}\label{compact are retract of representable}
Let $ \mathcal{C}$ be an arbitrary 2-category. Then any bicompact object in $\Flat_{\ps}[\mathcal{C}, \Cat]$ is a pseudoretract of a bicorepresentable. 
\end{proposition}

\begin{proof}
Let $ K : \mathcal{C} \rightarrow \Cat$ be a flat pseudofunctor which is bicompact in $ \Flat_{\ps}[\mathcal{C}, \Cat]$ and $ r_K, s_K, \alpha, \beta, \theta$ as above. $K$ being flat, it decomposes as a bifiltered bicolimit by \Cref{flatness as bifilteredness condition} and there exists a pseudonatural equivalence 
% https://q.uiver.app/?q=WzAsNCxbMCwwLCJLIl0sWzAsMSwiSyJdLFsxLDAsIlxcdW5kZXJzZXR7KEEsYSkgXFxpbiBcXGludCBLfXtcXGJpY29saW19IFxcLCBcXGhpcmF5b14qX0EiXSxbMSwxLCJcXHVuZGVyc2V0eyhBLGEpIFxcaW4gXFxpbnQgS317XFxiaWNvbGltfSBcXCwgXFxoaXJheW9eKl9BIl0sWzIsMywiIiwwLHsibGV2ZWwiOjIsInN0eWxlIjp7ImhlYWQiOnsibmFtZSI6Im5vbmUifX19XSxbMCwxLCIiLDAseyJsZXZlbCI6Miwic3R5bGUiOnsiaGVhZCI6eyJuYW1lIjoibm9uZSJ9fX1dLFswLDIsInJfSyJdLFsxLDMsInJfSyIsMl0sWzMsMCwic19LIiwxXSxbOCwxLCJcXGFscGhhIFxcYXRvcCBcXHNpbWVxIiwxLHsic2hvcnRlbiI6eyJzb3VyY2UiOjIwfSwic3R5bGUiOnsiYm9keSI6eyJuYW1lIjoibm9uZSJ9LCJoZWFkIjp7Im5hbWUiOiJub25lIn19fV0sWzQsOCwiXFxiZXRhIFxcYXRvcCBcXHNpbWVxIiwxLHsic2hvcnRlbiI6eyJ0YXJnZXQiOjIwfSwic3R5bGUiOnsiYm9keSI6eyJuYW1lIjoibm9uZSJ9LCJoZWFkIjp7Im5hbWUiOiJub25lIn19fV1d
\[\begin{tikzcd}
	K & {\underset{I_K}{\bicolim} \, \katayo_A} \\
	K & {\underset{I_K}{\bicolim} \, \katayo_A}
	\arrow[""{name=0, anchor=center, inner sep=0}, Rightarrow, no head, from=1-2, to=2-2]
	\arrow[Rightarrow, no head, from=1-1, to=2-1]
	\arrow["{r_K}", from=1-1, to=1-2]
	\arrow["{r_K}"', from=2-1, to=2-2]
	\arrow[""{name=1, anchor=center, inner sep=0}, "{s_K}"{description}, from=2-2, to=1-1]
	\arrow["{\alpha \atop \simeq}"{description}, Rightarrow, draw=none, from=1, to=2-1]
	\arrow["{\beta \atop \simeq}"{description, pos=0.4}, Rightarrow, draw=none, from=0, to=1]
\end{tikzcd}\]
But hypothesis that $ K$ is bicompact, there is for some $ (A,a)$ an invertible 2-cell
% https://q.uiver.app/?q=WzAsMyxbMCwwLCJLIl0sWzIsMCwiXFx1bmRlcnNldHsoQSxhKSBcXGluIChcXGludCBLKV57XFxvcH19e1xcYmljb2xpbX0gXFwsIFxcaGlyYXlvXipfQSJdLFsxLDEsIkEiXSxbMCwyLCJ4IiwyXSxbMiwxLCJxX3soQSxhKX0iLDJdLFswLDEsInJfSyJdLFsyLDUsIlxceGkgXFxhdG9wIFxcc2ltZXEiLDEseyJzaG9ydGVuIjp7InRhcmdldCI6MjB9LCJzdHlsZSI6eyJib2R5Ijp7Im5hbWUiOiJub25lIn0sImhlYWQiOnsibmFtZSI6Im5vbmUifX19XV0=
\[\begin{tikzcd}
	K && {\underset{I_K}{\bicolim} \, \katayo_A} \\
	& \katayo_A
	\arrow["x"', from=1-1, to=2-2]
	\arrow["{q_{(A,a)}}"', from=2-2, to=1-3]
	\arrow[""{name=0, anchor=center, inner sep=0}, "{r_K}", from=1-1, to=1-3]
	\arrow["{\xi \atop \simeq}"{description}, Rightarrow, draw=none, from=2-2, to=0]
\end{tikzcd}\]
we can paste with $ \alpha$ to exhibit $ x$ as a pseudosection of $ s_K q_{(A,a)}$:
% https://q.uiver.app/?q=WzAsNCxbMCwxLCJLIl0sWzIsMSwiXFx1bmRlcnNldHsoQSxhKSBcXGluIChcXGludCBLKV57XFxvcH19e1xcYmljb2xpbX0gXFwsIFxcaGlyYXlvXipfQSJdLFsxLDIsIkEiXSxbMCwwLCJLIl0sWzAsMiwieCIsMl0sWzIsMSwicV97KEEsYSl9IiwyXSxbMCwxLCJyX0siLDFdLFszLDAsIiIsMSx7ImxldmVsIjoyLCJzdHlsZSI6eyJoZWFkIjp7Im5hbWUiOiJub25lIn19fV0sWzEsMywic19LIiwyXSxbMiw2LCJcXHhpIFxcYXRvcCBcXHNpbWVxIiwxLHsic2hvcnRlbiI6eyJ0YXJnZXQiOjIwfSwic3R5bGUiOnsiYm9keSI6eyJuYW1lIjoibm9uZSJ9LCJoZWFkIjp7Im5hbWUiOiJub25lIn19fV0sWzgsMCwiXFxhbHBoYSBcXGF0b3AgXFxzaW1lcSIsMSx7InNob3J0ZW4iOnsic291cmNlIjoyMH0sInN0eWxlIjp7ImJvZHkiOnsibmFtZSI6Im5vbmUifSwiaGVhZCI6eyJuYW1lIjoibm9uZSJ9fX1dXQ==
\[\begin{tikzcd}
	K \\
	K && {\underset{I_K}{\bicolim} \, \katayo_A} \\
	& \katayo_A
	\arrow["x"', from=2-1, to=3-2]
	\arrow["{q_{(A,a)}}"', from=3-2, to=2-3]
	\arrow[""{name=0, anchor=center, inner sep=0}, "{r_K}"{description}, from=2-1, to=2-3]
	\arrow[Rightarrow, no head, from=1-1, to=2-1]
	\arrow[""{name=1, anchor=center, inner sep=0}, "{s_K}"', from=2-3, to=1-1]
	\arrow["{\xi \atop \simeq}"{description}, Rightarrow, draw=none, from=3-2, to=0]
	\arrow["{\alpha \atop \simeq}"{description}, Rightarrow, draw=none, from=1, to=2-1]
\end{tikzcd}\]
and pasting this 2-cell with by pasting the following 2-cell with $ \theta_{(A,a)}$ exhibits $ x$ as a pseudosection of $a$. This exhibits $ K$ as a pseudoretract of a representable. 
\end{proof}

\begin{corollary}\label{2-cat of flat pseudofunctor are accessible}
For any small 2-category $\mathcal{C}$, $\Flat_{\ps}[\mathcal{C}, \Cat]$ is finitely bi-accessible. 
\end{corollary}

\begin{proof}
We saw that flat pseudofunctors inherit $\sigma$-filtered and bifiltered bicolimits of $\ps[\mathcal{C}, \Cat]$; moreover from \Cref{Dubuc extension of flat pseudofunctor} we know that any flat pseudofunctor is a $\sigma$-filtered bicolimit of its canonical cocone, so the representable form a generator. Finally, for bicompact objects are retracts of corepresentables, they form an essentially small subcategory for $\mathcal{C} $ is small, hence has a small set of pseudo-idempotent. 
\end{proof}

\begin{comment}
In the case of a pseudocompact object, the $ \xi$ induced from the property above is actually an equality. In the case of a strict 2-functor, $ \alpha$ and $ \beta$ can be chosen as equalities, as well as the $\theta$ because we actually take a pseudocolimit. Hence a pseudocompact strict 2-functor is a strict retract of a representable. 
\end{comment}

\subsection{Finitely bipresentable 2-categories of flat pseudofunctors}
Now we want to refine this result in in the case of a bilex 2-category: here we can replace everywhere the condition of being flat by the condition of being bilex thanks to \Cref{flat is bilex}. We are going to prove that for any small bilex 2-category, the corresponding 2-category of flat (aka bilex) pseudofunctors is finitely bipresentable. The strategy is the following:\begin{itemize}
    \item[\ref{preserveandcreatecolimits}] we prove that the category of flat pseudofunctors is bicomplete - in fact, its bilimits are computed in the 2-category of pseudofunctors,
    \item[\ref{bilex cat have bisplitting of pseudoidempotent}] then we prove that bilex 2-categories have bisplitting of pseudoidempotent,
    \item[\ref{GUhalf1}] then we deduce that bicompact are exactly the birepresentable in the bilex context,
    \item[\ref{Flat pseudofunctors on bilex are bipres}] then combining those result with admissibility ensures the desired result. 
\end{itemize}

We need however the following general observation before anything:

\begin{lemma}
For a small bilex 2-category $ \mathcal{C}$, the Yoneda embedding turns finite bilimits into finite bicolimits.   \[  \mathcal{C}^{\op} \stackrel{\katayo }{\hookrightarrow} \biLex[\mathcal{C}, \Cat]  \]
 
\end{lemma}

\begin{proof}
Let $ G : I \rightarrow \mathcal{C}$ be with $ I$ a finite 2-category and $ W : I \rightarrow \Cat$ a finite weight. Then for any flat pseudofunctor $ F : \mathcal{C} \rightarrow \Cat $ we have an equivalence
\begin{align*}
   \biLex[\mathcal{C}, \Cat]\big[ \katayo_{\bilim^{W} G}, F  \big] &\simeq F(\bilim^{W} G) \\ 
   &\simeq \underset{i \in I}{\bilim}^{W} F(G(i)) \\
   &\simeq \underset{i \in I}{\bilim}^{W} \biLex[\mathcal{C}, \Cat]\big[ \katayo_{G(i)}, F  \big] \\ 
   &\simeq \biLex[\mathcal{C}, \Cat]\big[ \underset{i \in I}{\bicolim}^{W^{\op}} \!\katayo_{G(i)}, F  \big]
 \end{align*}
\end{proof}

% We first observe the following generality ensuring before anything the bicompleteness of their corresponding 2-categories of flat functors: 

\begin{proposition}\label{preserveandcreatecolimits}
If $ \mathcal{C}$ is a small bilex 2-category, then the 2-category $ \biLex[\mathcal{C}, \Cat]$ has small bilimits, which are created by the pseudofully faithful inclusion 
\[ \biLex[\mathcal{C}, \Cat] \stackrel{i_C}{\hookrightarrow }  \ps[\mathcal{C}, \Cat] \]
Moreover, finite bilimits in $\biLex[\mathcal{C}, \Cat] $ commute with bifiltered bicolimits.
\end{proposition}

\begin{proof}
Bilimits exist and are pointwise in $\ps[\mathcal{C}, \Cat]$: for any $F : I \rightarrow \ps[\mathcal{C}, \Cat]$, any weight $ W : I \rightarrow \Cat$ and any $ C$ in $\mathcal{C}$ we have 
\[  (\bilim^{W}\, F)(C) \simeq \underset{i \in I}{\bilim}^{W} \, F_i(C) \]
Let us prove that this pseudofunctor is flat. As $\mathcal{C}$ is supposed to be bilex, this amounts to check that $ \bilim^{W} \, F$ is itself bilex. But this is a consequence of commutations of bilimits. Indeed, let $ G : J \rightarrow \mathcal{C}$ be with $ J$ a finite and $ V : J \rightarrow \Cat$ a finite weight, so that $ \bilim^{V}G$ exists in $\mathcal{C}$. Then we have 
\begin{align*}
    (\bilim^{W}F)(\bilim^{V}G) &\simeq \ps[\mathcal{C}, \Cat][ \katayo_{\bilim^{V}G}, \bilim^{W}F] \\ 
    &\simeq \ps[I, \Cat]\big[ W, \biLex[\mathcal{C},\Cat][ \katayo_{\bilim^{V}G}, F(-)]  \big]
\end{align*}
But observe that in each $i \in I$ we have an equivalence 
\begin{align*}
     \biLex[\mathcal{C}, \Cat][ \katayo_{\bilim^{V}G}, F(i)] &\simeq F(i)(\bilim^{V}G) \\ 
     &\simeq \underset{j \in J}{\bilim}^{\!V}F(i)(G(j))
\end{align*}
and this provides us with a pseudonatural equivalence 
\[ \biLex[\mathcal{C},\Cat][ \katayo_{\bilim^{V}G}, F(-)] \simeq \underset{j \in J}{\bilim}^{\!V}F(-)(G(j)) \]
Now, for we can extract a bilimit on the right of a homcategory, this gives the following chain of equivalences
\begin{align*}
    (\bilim^{W}F)(\bilim^{V}G) &\simeq \ps[I, \Cat]\big[ W, \underset{j \in J}{\bilim}^{\!V}F(-)(G(j))] \\
    &\simeq  \underset{j \in J}{\bilim}^{\!V} [I, \Cat]\big[ W, \biLex[\mathcal{C},\Cat][ \katayo_{G(j)}, F(-)]  \big]\\
    &\simeq \underset{j \in J}{\bilim}^{\!V} \biLex[\mathcal{C},\Cat]\big[\katayo_{G(j)}, \bilim^{W}F \big]\\
    &\simeq \underset{j \in J}{\bilim}^{\!V} (\bilim^{W}F)(G(j))
\end{align*}
Hence the pseudofunctor $ (\bilim^{W}F)$ is bilex, hence is in $\biLex[\mathcal{C}, \Cat]$. Now for we have a pseudo-fully faithful inclusion, we know that this is already a bilimit there. 
\end{proof}

In the general case, we saw that bicorepresentables were bicompact, but one had also to consider pseudoretracts of bicorepresentables to have all the bicompacts. In the bilex case, we prove that this simplifies, as the one dimensional case. To this end we introduce the following 2-dimensional analog of idempotents:

\begin{definition}
In a 2-category $\mathcal{C}$, a \emph{pseudoidempotent}\index{pseudoidempotent} is a 1-cell $ e: A \rightarrow A$ equipped with an invertible 2-cell
% https://q.uiver.app/?q=WzAsMyxbMCwwLCJBIl0sWzEsMCwiQSJdLFsyLDAsIkEiXSxbMCwxLCJlIl0sWzEsMiwiZSJdLFswLDIsImUiLDIseyJjdXJ2ZSI6M31dLFsxLDUsIlxcdXBzaWxvbiBcXGF0b3AgXFxzaW1lcSIsMSx7ImxhYmVsX3Bvc2l0aW9uIjozMCwic2hvcnRlbiI6eyJ0YXJnZXQiOjIwfSwic3R5bGUiOnsiYm9keSI6eyJuYW1lIjoibm9uZSJ9LCJoZWFkIjp7Im5hbWUiOiJub25lIn19fV1d
\[\begin{tikzcd}
	A & A & A
	\arrow["e", from=1-1, to=1-2]
	\arrow["e", from=1-2, to=1-3]
	\arrow[""{name=0, anchor=center, inner sep=0}, "e"', curve={height=20pt}, from=1-1, to=1-3]
	\arrow["{\upsilon \atop \simeq}"{description, pos=0.3}, Rightarrow, draw=none, from=1-2, to=0]
\end{tikzcd}\]
Now a \emph{bisplitting}\index{bisplitting} of a pseudoidempotent is a pair of invertible 2-cell as below
% https://q.uiver.app/?q=WzAsNCxbMCwwLCJBIl0sWzIsMCwiQSJdLFswLDEsIkIiXSxbMiwxLCJCIl0sWzAsMSwiZSJdLFswLDIsInIiLDJdLFsyLDEsInMiLDFdLFsxLDMsInIiXSxbMiwzLCIiLDIseyJsZXZlbCI6Miwic3R5bGUiOnsiaGVhZCI6eyJuYW1lIjoibm9uZSJ9fX1dLFswLDYsIlxcYWxwaGEgXFxhdG9wIFxcc2ltZXEiLDEseyJzaG9ydGVuIjp7InRhcmdldCI6MjB9LCJzdHlsZSI6eyJib2R5Ijp7Im5hbWUiOiJub25lIn0sImhlYWQiOnsibmFtZSI6Im5vbmUifX19XSxbNiwzLCJcXGJldGEgXFxhdG9wIFxcc2ltZXEiLDEseyJzaG9ydGVuIjp7InNvdXJjZSI6MjB9LCJzdHlsZSI6eyJib2R5Ijp7Im5hbWUiOiJub25lIn0sImhlYWQiOnsibmFtZSI6Im5vbmUifX19XV0=
\[\begin{tikzcd}
	A && A \\
	B && B
	\arrow["e", from=1-1, to=1-3]
	\arrow["r"', from=1-1, to=2-1]
	\arrow[""{name=0, anchor=center, inner sep=0}, "s"{description}, from=2-1, to=1-3]
	\arrow["r", from=1-3, to=2-3]
	\arrow[Rightarrow, no head, from=2-1, to=2-3]
	\arrow["{\alpha \atop \simeq}"{description}, Rightarrow, draw=none, from=1-1, to=0]
	\arrow["{\beta \atop \simeq}"{description}, Rightarrow, draw=none, from=0, to=2-3]
\end{tikzcd}\]
\end{definition}

\begin{lemma}\label{bilex cat have bisplitting of pseudoidempotent}
A bilex 2-category is closed under splitting of pseudoidempotents. 
\end{lemma}

\begin{proof}
Consider the following bi-iso-inserter
% https://q.uiver.app/?q=WzAsNSxbMSwyLCJBIl0sWzMsMSwiQSJdLFs0LDFdLFswLDEsIlxcYmlJc29JbnMoZSwxX0EpIl0sWzEsMCwiQSJdLFswLDEsIjFfQSIsMix7Im9mZnNldCI6MSwibGV2ZWwiOjIsInN0eWxlIjp7ImhlYWQiOnsibmFtZSI6Im5vbmUifX19XSxbMywwLCJpX3soZSwxX0EpfSIsMl0sWzQsMSwiZSIsMCx7Im9mZnNldCI6LTF9XSxbMyw0LCJpX3soZSwxX0EpfSJdLFs0LDAsIlxcbXVfeyhlLDFfQSl9IFxcYXRvcCBcXHNpbWVxIiwxLHsic2hvcnRlbiI6eyJzb3VyY2UiOjIwLCJ0YXJnZXQiOjIwfSwibGV2ZWwiOjIsInN0eWxlIjp7ImJvZHkiOnsibmFtZSI6Im5vbmUifSwiaGVhZCI6eyJuYW1lIjoibm9uZSJ9fX1dXQ==
\[\begin{tikzcd}
	& A \\
	{\biIsoIns(e,1_A)} &&& A & {} \\
	& A
	\arrow["{1_A}"', shift right=1, Rightarrow, no head, from=3-2, to=2-4]
	\arrow["{i_{(e,1_A)}}"', from=2-1, to=3-2]
	\arrow["e", shift left=1, from=1-2, to=2-4]
	\arrow["{i_{(e,1_A)}}", from=2-1, to=1-2]
	\arrow["{\mu_{(e,1_A)} \atop \simeq}"{description}, draw=none, from=1-2, to=3-2]
\end{tikzcd}\]
Then observe that precomposing the parallel pair $(e,1_A)$ with $ e$ also insert an invertible 2-cell $ \upsilon$ so the universal property of the bilimit provides us with a canonical map $ r$ and a pair of invertible 2-cells $\alpha, \beta$ as below
\[\begin{tikzcd}
	&&& A \\
	A && {\biIsoIns(e,1_A)} &&& A & {} \\
	&&& A
	\arrow["e", shift left=1, from=1-4, to=2-6]
	\arrow["{i_{(e,1_A)}}"{description}, from=2-3, to=1-4]
	\arrow[""{name=0, anchor=center, inner sep=0}, "e", curve={height=-18pt}, from=2-1, to=1-4]
	\arrow["r"{description}, dashed, from=2-1, to=2-3]
	\arrow["{1_A}"', shift right=1, Rightarrow, no head, from=3-4, to=2-6]
	\arrow["{i_{(e,1_A)}}"{description}, from=2-3, to=3-4]
	\arrow[""{name=1, anchor=center, inner sep=0}, "{e^2}"', curve={height=18pt}, from=2-1, to=3-4]
	\arrow["{\mu_{(e, 1_A)} \atop \simeq}"{description}, draw=none, from=1-4, to=3-4]
	\arrow["{\alpha \atop \simeq}"{description}, draw=none, from=0, to=2-3]
	\arrow["{\beta \atop \simeq}"{description}, draw=none, from=2-3, to=1]
\end{tikzcd}\]
which form altogether with the composite $ i_{(e,1_A)}i_{(\mu_{(e,1_A)}}$ a biplitting of the pseudoidempotent $e$. 
\end{proof}

\begin{theorem} \label{GUhalf1}
For any bilex 2-category $ \mathcal{C}$, we have a biequivalence \[ \mathcal{C}^{\op} \simeq (\biLex[\mathcal{C}, \Cat])_\omega \]
In other words, bicompact objects are bicorepresentables and bicorepresentables are bicompact in  $\biLex[\mathcal{C}, \Cat]$.
\end{theorem}

\begin{proof}
At this point the theorem appears as a corollary of the previous lemmas. We saw that birepresentable are bicompact; conversely we saw that any bicompact object is a pseudoretract of a representable
% https://q.uiver.app/?q=WzAsMyxbMCwwLCJLIl0sWzIsMCwiSyJdLFsxLDEsIlxcaGlyYXlvXipfQSJdLFswLDEsIiIsMCx7ImxldmVsIjoyLCJzdHlsZSI6eyJoZWFkIjp7Im5hbWUiOiJub25lIn19fV0sWzAsMiwicyIsMl0sWzIsMSwiciIsMl0sWzMsMiwiXFxhbHBoYSBcXGF0b3BcXHNpbWVxIiwxLHsic2hvcnRlbiI6eyJzb3VyY2UiOjIwfSwic3R5bGUiOnsiYm9keSI6eyJuYW1lIjoibm9uZSJ9LCJoZWFkIjp7Im5hbWUiOiJub25lIn19fV1d
\[\begin{tikzcd}
	K && K \\
	& {\katayo_A}
	\arrow[""{name=0, anchor=center, inner sep=0}, Rightarrow, no head, from=1-1, to=1-3]
	\arrow["s"', from=1-1, to=2-2]
	\arrow["r"', from=2-2, to=1-3]
	\arrow["{\alpha \atop\simeq}"{description}, Rightarrow, draw=none, from=0, to=2-2]
\end{tikzcd}\]
But then we have a pseudo-idempotent in $\biLex[\mathcal{C}, \Cat]$
% https://q.uiver.app/?q=WzAsMyxbMSwwLCJLIl0sWzAsMSwiXFxoaXJheW9eKl9BIl0sWzIsMSwiXFxoaXJheW9eKl9BIl0sWzEsMCwiciJdLFswLDIsInMiXSxbMSwyLCJzciIsMl0sWzAsNSwiPSIsMSx7InNob3J0ZW4iOnsidGFyZ2V0IjoyMH0sInN0eWxlIjp7ImJvZHkiOnsibmFtZSI6Im5vbmUifSwiaGVhZCI6eyJuYW1lIjoibm9uZSJ9fX1dXQ==
\[\begin{tikzcd}
	& K \\
	{\katayo_A} && {\katayo_A}
	\arrow["r", from=2-1, to=1-2]
	\arrow["s", from=1-2, to=2-3]
	\arrow[""{name=0, anchor=center, inner sep=0}, "sr"', from=2-1, to=2-3]
	\arrow["{=}"{description}, Rightarrow, draw=none, from=1-2, to=0]
\end{tikzcd}\]
which comes uniquely by full faithfulness of $\katayo$ from a pseudo-idempotent $ e : A \rightarrow A$ in $\mathcal{C}$. But now as $\mathcal{C}$ is bilex, this pseudo-idempotent has a bisplitting in $\mathcal{C}$,
% https://q.uiver.app/?q=WzAsMyxbMCwxLCJBIl0sWzIsMSwiQSJdLFsxLDAsIkIiXSxbMCwxLCJlIiwyXSxbMCwyLCJyJyJdLFsyLDEsInMnIl0sWzMsMiwiXFxhbHBoYScgXFxhdG9wIFxcc2ltZXEiLDEseyJzaG9ydGVuIjp7InNvdXJjZSI6MjB9LCJzdHlsZSI6eyJib2R5Ijp7Im5hbWUiOiJub25lIn0sImhlYWQiOnsibmFtZSI6Im5vbmUifX19XV0=
\[\begin{tikzcd}
	& B \\
	A && A
	\arrow[""{name=0, anchor=center, inner sep=0}, "e"', from=2-1, to=2-3]
	\arrow["{r'}", from=2-1, to=1-2]
	\arrow["{s'}", from=1-2, to=2-3]
	\arrow["{\alpha' \atop \simeq}"{description}, Rightarrow, draw=none, from=0, to=1-2]
\end{tikzcd}\]
which is preserved by $\katayo$. Now by uniqueness up to equivalence of bisplitting, we must have a pseudonatural equivalence $ K \simeq \katayo_B$. 
\end{proof}

%Now putting altogether the previous results, we get the following:

\begin{theorem}\label{Flat pseudofunctors on bilex are bipres}
Let $\mathcal{C}$ be a small bilex 2-category. Then $ \biLex[\mathcal{C}, \Cat]$ is finitely bipresentable, and $ \mathcal{C}$ can be chosen as a generator of bicompact objects. 
\end{theorem}

\section{$2$-dimensional Gabriel-Ulmer duality}

\subsection{Bi-accessible pseudofunctors}

\begin{definition}[Bi-accessible pseudofunctors]
A pseudo-functor between finitely bi-accessible $2$-categories is finitely bi-accesssible if it preserves bifiltered bicolimits. Similarly we say that it is $\lambda$-bi-accessible if it preserves $\lambda$-bifiltered bicolimits.
\end{definition}

\begin{proposition}
Let $\mathcal{A},\mathcal{B}$ be finitely bi-accessible categories. There is a $1$-to-$1$ correspondence between finitely bi-accessible pseudofunctors $\mathcal{A} \to \mathcal{B}$ and pseudofunctors $\mathcal{A}_\omega \to B$, which is induced by biKan extension.

% https://q.uiver.app/?q=WzAsMyxbMCwwLCJcXG1hdGhjYWx7QX1fXFxvbWVnYSJdLFswLDEsIlxcbWF0aGNhbHtBfSJdLFsxLDEsIlxcbWF0aGNhbHtCfSJdLFswLDEsImkiLDIseyJzdHlsZSI6eyJ0YWlsIjp7Im5hbWUiOiJob29rIiwic2lkZSI6InRvcCJ9fX1dLFsxLDIsIlxcYmlsYW5faSBmIiwyXSxbMCwyLCJmIl0sWzUsMSwiXFxzaW1lcSIsMSx7InNob3J0ZW4iOnsic291cmNlIjoyMH0sInN0eWxlIjp7ImJvZHkiOnsibmFtZSI6Im5vbmUifSwiaGVhZCI6eyJuYW1lIjoibm9uZSJ9fX1dXQ==
\[\begin{tikzcd}[column sep=large]
	{\mathcal{A}_\omega} \\
	{\mathcal{A}} & {\mathcal{B}}
	\arrow["i"', hook, from=1-1, to=2-1]
	\arrow["{\biLan_i f}"', from=2-1, to=2-2]
	\arrow[""{name=0, anchor=center, inner sep=0}, "f", from=1-1, to=2-2]
	\arrow["\simeq"{description}, Rightarrow, draw=none, from=0, to=2-1]
\end{tikzcd}\]
\end{proposition}
\begin{proof}
Several things need to be shown. We start from showing that given a pseudofunctor $f: \mathcal{A}_\omega \to \mathcal{B}$, its left biKan extension preserves bifiltered bicolimits. This follows from the fact that we have an explicit way to compute the biKan extension. First observe that, combining the cancellation rule observed at \Cref{cancellation rule} together with \Cref{bikan ext along ff} knowning that here both $\iota$ and $\nu_\mathcal{A}$ are pseudofully-faithful, all the 2-cells in the following diagram happen to be natural equivalences:
% https://q.uiver.app/?q=WzAsNCxbMCwwLCJcXG1hdGhjYWx7QX1fXFxvbWVnYSJdLFsxLDEsIlxcbWF0aGNhbHtBfSJdLFsyLDAsIlxcbWF0aGNhbHtCfSJdLFsxLDIsIlxccHNbKFxcbWF0aGNhbHtBfV9cXG9tZWdhKV57XFxvcH0sIFxcQ2F0XSJdLFswLDEsIlxcaW90YSIsMSx7InN0eWxlIjp7InRhaWwiOnsibmFtZSI6Imhvb2siLCJzaWRlIjoidG9wIn19fV0sWzEsMiwiXFxiaUxhbl9pIGYiLDFdLFswLDIsImYiXSxbMCwzLCJcXGhpcmF5byIsMix7ImN1cnZlIjo0fV0sWzEsMywiXFxudV9cXG1hdGhjYWx7QX0iLDEseyJjb2xvdXIiOlswLDYwLDYwXX0sWzAsNjAsNjAsMV1dLFszLDIsIlxcYmlMYW5fXFxoaXJheW8gRiIsMSx7ImN1cnZlIjo0fV0sWzcsMSwiXFxzaW1lcSIsMSx7InNob3J0ZW4iOnsic291cmNlIjoyMH0sInN0eWxlIjp7ImJvZHkiOnsibmFtZSI6Im5vbmUifSwiaGVhZCI6eyJuYW1lIjoibm9uZSJ9fX1dLFs2LDEsIlxcc2ltZXEiLDEseyJzaG9ydGVuIjp7InNvdXJjZSI6MjB9LCJzdHlsZSI6eyJib2R5Ijp7Im5hbWUiOiJub25lIn0sImhlYWQiOnsibmFtZSI6Im5vbmUifX19XSxbMSw5LCJcXHNpbWVxIiwxLHsic2hvcnRlbiI6eyJ0YXJnZXQiOjIwfSwic3R5bGUiOnsiYm9keSI6eyJuYW1lIjoibm9uZSJ9LCJoZWFkIjp7Im5hbWUiOiJub25lIn19fV1d
\[\begin{tikzcd}
	{\mathcal{A}_\omega} && {\mathcal{B}} \\
	& {\mathcal{A}} \\
	& {\ps[(\mathcal{A}_\omega)^{\op}, \Cat]}
	\arrow["\iota"{description}, hook, from=1-1, to=2-2]
	\arrow["{\biLan_i f}"{description}, from=2-2, to=1-3]
	\arrow[""{name=0, anchor=center, inner sep=0}, "f", from=1-1, to=1-3]
	\arrow[""{name=1, anchor=center, inner sep=0}, "\hirayo"', curve={height=24pt}, from=1-1, to=3-2]
	\arrow["{\nu_\mathcal{A}}", hook, from=2-2, to=3-2]
	\arrow[""{name=2, anchor=center, inner sep=0}, "{\biLan_\hirayo f}"', curve={height=24pt}, from=3-2, to=1-3]
	\arrow["\simeq"{description}, Rightarrow, draw=none, from=1, to=2-2]
	\arrow["\simeq"{description}, Rightarrow, draw=none, from=0, to=2-2]
	\arrow["\simeq"{description}, Rightarrow, draw=none, from=2-2, to=2]
\end{tikzcd}\]
%Consider the diagram below,

% https://q.uiver.app/?q=WzAsNSxbMCwwLCJcXG1hdGhjYWx7QX1fXFxvbWVnYSJdLFswLDIsIlxcbWF0aGNhbHtBfSJdLFsyLDIsIlxcbWF0aGNhbHtCfSJdLFswLDQsIlxccHNbKFxcbWF0aGNhbHtBfV9cXG9tZWdhKV57XFxvcH0sIFxcQ2F0XSJdLFsyLDQsIlxccHNbKFxcbWF0aGNhbHtCfV9cXG9tZWdhKV57XFxvcH0sIFxcQ2F0XSJdLFswLDEsImkiLDJdLFsxLDIsIlxcYmlMYW5faSBmIiwxLHsic3R5bGUiOnsiYm9keSI6eyJuYW1lIjoiZGFzaGVkIn19fV0sWzAsMiwiZiJdLFswLDMsIlxcaGlyYXlvIiwxLHsiY3VydmUiOjV9XSxbMSwzLCJcXG51X1xcbWF0aGNhbHtBfSIsMSx7ImNvbG91ciI6WzAsNjAsNjBdfSxbMCw2MCw2MCwxXV0sWzIsNCwiXFxudV9cXG1hdGhjYWx7Qn0iLDJdLFszLDQsIlxcYmlMYW5fXFxoaXJheW8gXFxudV97XFxtYXRoY2Fse0J9fWYiLDIseyJjb2xvdXIiOlswLDYwLDYwXX0sWzAsNjAsNjAsMV1dXQ==
\begin{comment}
\[\begin{tikzcd}
	{\mathcal{A}_\omega} \\
	\\
	{\mathcal{A}} && {\mathcal{B}} \\
	\\
	{\ps[(\mathcal{A}_\omega)^{\op}, \Cat]} && {\ps[(\mathcal{B}_\omega)^{\op}, \Cat]}
	\arrow["i"', from=1-1, to=3-1]
	\arrow["{\biLan_i f}"{description}, dashed, from=3-1, to=3-3]
	\arrow["f", from=1-1, to=3-3]
	\arrow["\hirayo"{description}, curve={height=30pt}, from=1-1, to=5-1]
	\arrow["{\nu_\mathcal{A}}"{description}, from=3-1, to=5-1]
	\arrow["{\nu_\mathcal{B}}"', from=3-3, to=5-3]
	\arrow["{\biLan_\hirayo \nu_{\mathcal{B}}f}"', from=5-1, to=5-3]
\end{tikzcd}\]
\end{comment}
Now, both binerve pseudofunctors $ \nu_\mathcal{A} $ and $\nu_\mathcal{B}$ preserve bifiltered bicolimits by \Cref{creation of sigmafiltered colimits of flat functors} and \Cref{A is flat functors}, the biKan extension $\biLan_\hirayo \nu_\mathcal{B} f$ preserves all bicolimits by its universal property. Thus, the composition $\biLan_\hirayo(\nu_\mathcal{B}f)\nu_{\mathcal{A}}$ preserves bifiltered bicolimits. Moreover, it lands by construction in the closure under bifiltered bicolimits of the corepresentables, and thus lifts against $\nu_\mathcal{B}$, and this lift is exhibited as being the biKan extension $ \biLan_\iota f$

%\textcolor{red}{This diagram should be part of the previous one. Also, what is this lifting argument ? Isnt this rather an exploitation of the universal property of biLan together with the fact that the nerve and inclusion are pseudofully faithful ?} {\color{blue} Yes, which is precisely what I mean by "lift against"...}

% https://q.uiver.app/?q=WzAsNCxbMCwwLCJcXG1hdGhjYWx7QX0iXSxbMiwwLCJcXG1hdGhjYWx7Qn0iXSxbMCwxLCJcXHBzWyhcXG1hdGhjYWx7QX1fXFxvbWVnYSlee1xcb3B9LCBcXENhdF0iXSxbMiwxLCJcXHBzWyhcXG1hdGhjYWx7Qn1fXFxvbWVnYSlee1xcb3B9LCBcXENhdF0iXSxbMCwxLCJGIl0sWzAsMiwiXFxudV9cXG1hdGhjYWx7QX0iLDIseyJzdHlsZSI6eyJ0YWlsIjp7Im5hbWUiOiJob29rIiwic2lkZSI6InRvcCJ9fX1dLFsxLDMsIlxcbnVfXFxtYXRoY2Fse0J9IiwyLHsic3R5bGUiOnsidGFpbCI6eyJuYW1lIjoiaG9vayIsInNpZGUiOiJ0b3AifX19XSxbMiwzLCJcXGJpTGFuX1xcaGlyYXlvIFxcbnVfe1xcbWF0aGNhbHtCfX1GIiwyXSxbNCw3LCJcXHNpbWVxIiwxLHsic2hvcnRlbiI6eyJzb3VyY2UiOjIwLCJ0YXJnZXQiOjIwfSwic3R5bGUiOnsiYm9keSI6eyJuYW1lIjoibm9uZSJ9LCJoZWFkIjp7Im5hbWUiOiJub25lIn19fV1d
\[\begin{tikzcd}
	{\mathcal{A}} && {\mathcal{B}} \\
	{\ps[(\mathcal{A}_\omega)^{\op}, \Cat]} && {\ps[(\mathcal{B}_\omega)^{\op}, \Cat]}
	\arrow[""{name=0, anchor=center, inner sep=0}, "\biLan_\iota f", from=1-1, to=1-3]
	\arrow["{\nu_\mathcal{A}}"', hook, from=1-1, to=2-1]
	\arrow["{\nu_\mathcal{B}}"', hook, from=1-3, to=2-3]
	\arrow[""{name=1, anchor=center, inner sep=0}, "{\biLan_\hirayo \nu_{\mathcal{B}}f}"', from=2-1, to=2-3]
	\arrow["\simeq"{description}, Rightarrow, draw=none, from=0, to=1]
\end{tikzcd}\]

\textcolor{red}{}
\end{proof}

\begin{theorem} \label{laKan}
Let $F: \mathcal{A} \to \mathcal{B}$ be a $\lambda$-biaccessible pseudofunctor between $\lambda$-bipresentable $2$-categories. Then in the diagram below

% https://q.uiver.app/?q=WzAsMyxbMCwxLCJcXG1hdGhjYWx7QX0iXSxbMSwxLCJcXG1hdGhjYWx7Qn0iXSxbMCwwLCJcXG1hdGhjYWx7QX1fXFxsYW1iZGEiXSxbMCwxLCJGIiwxXSxbMiwwLCJcXGlvdGEiLDFdLFsyLDEsIkZcXGlvdGEiLDFdXQ==
\[\begin{tikzcd}
	{\mathcal{A}_\lambda} \\
	{\mathcal{A}} & {\mathcal{B}}
	\arrow["F"', from=2-1, to=2-2]
	\arrow["\iota"', from=1-1, to=2-1]
	\arrow["F\iota", from=1-1, to=2-2]
\end{tikzcd}\]

$F$ coincides up to a natural equivalence with the with the biKan extension $\text{biLan}_\iota F\iota$.
\end{theorem}
\begin{proof}
It follows from the density of $\mathcal{A}_\lambda$, the fact that biKan extensions are pointwise.
\end{proof}

% \begin{remark} \label{smallcolimito}
% In the theorem above, when $\mathcal{B}$ is $\Cat$, the theorem is telling use that $F$ is a small weighted bicolimit of corepresentables, as it is the bicolimit \textcolor{red}{what is that ?} $\{F\iota, \mathcal{A}(\iota-,-)\}$. 
% \end{remark}

\subsection{Adjoint Functor Theorems} \label{AFT}

In this section we will prove two versions of the adjoint functor theorem. In both cases, we reduce the existence of the left/right adjoint to the computation of a biKan extension, this is possible thanks to a classical result that we state without proof below. While the proof does not explicitely appear in the literature, it is a $2$-dimensional analog of \cite[3.7.2]{borceux1994handbook} and the proof carries to our context without any effort.

\begin{lemma}
Let $L: \mathcal{A} \to \mathcal{B}$ be a pseudofunctor between $2$-categories. Then $L$ has a right adjoint if and only if the following two conditions hold.
\begin{enumerate}
    \item[(a)] $\biLan_F(1_\mathcal{A})$ exists,
    \item[(b)] $F$ preserves it.
\end{enumerate}
\end{lemma}

Of course, the dual version holds for left adjoints and right biKan extensions.

\begin{theorem}\label{RAFT}
Let $L: \mathcal{A} \to \mathcal{B}$ be a pseudofunctor preserving all weighted bicolimits between finitely bipresentable $2$-categories. Then it has a right biadjoint.
\end{theorem}
\begin{proof}
We reduced to show that: \begin{enumerate}
    \item[(a)] $\biLan_F(1_\mathcal{A})$ exists,
    \item[(b)] $F$ preserves it.
\end{enumerate}

Indeed, in this case $\biLan_F(1_\mathcal{A})$ provides a right biadjoint for $F$. (a) Because $\mathcal{A}$ is a large $2$-category, we cannot apply on the spot any bi-analog of \cite[3.7.2]{borceux1994handbook}, thus we need to massage the biKan extension. Consider the following diagram.

% https://q.uiver.app/?q=WzAsNCxbMCwwLCJcXG1hdGhjYWx7QX1fXFxvbWVnYSJdLFsyLDAsIlxcbWF0aGNhbHtBfSJdLFsyLDIsIlxcbWF0aGNhbHtCfSJdLFs0LDAsIlxcbWF0aGNhbHtBfSJdLFswLDEsImkiLDFdLFsxLDIsIkwiLDFdLFswLDIsIkxpIiwxXSxbMSwzLCIxX1xcbWF0aGNhbHtBfSIsMV0sWzIsMywiXFxiaUxhbl97TGl9IGkiLDEseyJzdHlsZSI6eyJib2R5Ijp7Im5hbWUiOiJkYXNoZWQifX19XV0=
% https://q.uiver.app/?q=WzAsNCxbMCwwLCJcXG1hdGhjYWx7QX1fXFxvbWVnYSJdLFsxLDAsIlxcbWF0aGNhbHtBfSJdLFsxLDEsIlxcbWF0aGNhbHtCfSJdLFsyLDAsIlxcbWF0aGNhbHtBfSJdLFswLDEsImkiXSxbMSwyLCJMIiwxXSxbMCwyLCJMaSIsMl0sWzEsMywiMV9cXG1hdGhjYWx7QX0iXSxbMiwzLCJcXGJpTGFuX3tMaX0gaSIsMl0sWzEsNiwiXFxzaW1lcSIsMSx7InNob3J0ZW4iOnsidGFyZ2V0IjoyMH0sInN0eWxlIjp7ImJvZHkiOnsibmFtZSI6Im5vbmUifSwiaGVhZCI6eyJuYW1lIjoibm9uZSJ9fX1dLFsxLDgsIlxcbGFtYmRhIiwwLHsic2hvcnRlbiI6eyJ0YXJnZXQiOjIwfX1dXQ==
\[\begin{tikzcd}
	{\mathcal{A}_\omega} & {\mathcal{A}} & {\mathcal{A}} \\
	& {\mathcal{B}}
	\arrow["\iota", hook, from=1-1, to=1-2]
	\arrow["L"{description}, from=1-2, to=2-2]
	\arrow[""{name=0, anchor=center, inner sep=0}, "L\iota"', from=1-1, to=2-2]
	\arrow["{1_\mathcal{A}}", from=1-2, to=1-3]
	\arrow[""{name=1, anchor=center, inner sep=0}, "{\biLan_{L\iota} \iota}"', from=2-2, to=1-3]
	\arrow["\eta", shorten >=2pt, Rightarrow, from=1-2, to=1]
\end{tikzcd}\]

We claim that, if it exists, $\biLan_{L\iota} \iota$ is pseudonaturally isomorphic to $\biLan_{L}1_{\mathcal{A}}$ by application of the cancellation rule \Cref{cancellation rule}: \[ \biLan_{L\iota} \iota \simeq \biLan_{L}\biLan_\iota \iota \simeq \biLan_{L}1_{\mathcal{A}}\]
Now applying \Cref{cancellation rule}, we know that in the the following diagram the leftmost 2-cell is invertible:
% https://q.uiver.app/?q=WzAsNCxbMCwwLCJcXG1hdGhjYWx7QX1fXFxvbWVnYSJdLFsxLDEsIlxcbWF0aGNhbHtCfSJdLFsyLDAsIlxcbWF0aGNhbHtBfSJdLFsxLDMsIlxcQ2F0XlxcbWF0aGNhbHtBXlxcY2lyY30iXSxbMCwxLCJMaSIsMV0sWzAsMiwiaSJdLFsxLDMsIlxcbWF0aGNhbHtCfShMaSwoLSkpIiwxXSxbMCwzLCJcXHlvbiIsMix7ImN1cnZlIjozfV0sWzMsMiwiXFxsYW5fXFx5b24gaSIsMix7ImN1cnZlIjozfV0sWzEsMiwiIiwxLHsic3R5bGUiOnsiYm9keSI6eyJuYW1lIjoiZGFzaGVkIn19fV1d
% https://q.uiver.app/?q=WzAsNCxbMCwwLCJcXG1hdGhjYWx7QX1fXFxvbWVnYSJdLFsxLDEsIlxcbWF0aGNhbHtCfSJdLFsyLDAsIlxcbWF0aGNhbHtBfSJdLFsxLDIsIlxccHNbKFxcbWF0aGNhbHtBfV9cXG9tZWdhKV57XFxvcH0sIFxcQ2F0XSJdLFswLDEsIkxcXGlvdGEiLDFdLFswLDIsIlxcaW90YSIsMCx7InN0eWxlIjp7InRhaWwiOnsibmFtZSI6Imhvb2siLCJzaWRlIjoidG9wIn19fV0sWzEsMywiXFxtYXRoY2Fse0J9W0xcXGlvdGEsLV0iXSxbMCwzLCJcXGhpcmF5byIsMix7ImN1cnZlIjozfV0sWzMsMiwiXFxiaUxhbl9cXGhpcmF5byBcXGlvdGEiLDIseyJjdXJ2ZSI6M31dLFsxLDIsIlxcYmlMYW5fe0xcXGlvdGF9IFxcaW90YSIsMSx7InN0eWxlIjp7ImJvZHkiOnsibmFtZSI6ImRhc2hlZCJ9fX1dLFs3LDEsIkxcXGlvdGEqIiwwLHsic2hvcnRlbiI6eyJzb3VyY2UiOjIwfX1dLFs1LDEsIlxcbGFtYmRhIiwyLHsic2hvcnRlbiI6eyJzb3VyY2UiOjIwfX1dLFsxLDgsIlxcc2ltZXEiLDEseyJzaG9ydGVuIjp7InRhcmdldCI6MjB9LCJzdHlsZSI6eyJib2R5Ijp7Im5hbWUiOiJub25lIn0sImhlYWQiOnsibmFtZSI6Im5vbmUifX19XV0=
\[\begin{tikzcd}
	{\mathcal{A}_\omega} && {\mathcal{A}} \\
	& {\mathcal{B}} \\
	& {\ps[(\mathcal{A}_\omega)^{\op}, \Cat]}
	\arrow["L\iota"{description}, from=1-1, to=2-2]
	\arrow[""{name=0, anchor=center, inner sep=0}, "\iota", hook, from=1-1, to=1-3]
	\arrow["{\mathcal{B}[L\iota,-]}", from=2-2, to=3-2]
	\arrow[""{name=1, anchor=center, inner sep=0}, "\hirayo"', curve={height=18pt}, from=1-1, to=3-2]
	\arrow[""{name=2, anchor=center, inner sep=0}, "{\biLan_\hirayo \iota}"', curve={height=18pt}, from=3-2, to=1-3]
	\arrow["{\biLan_{L\iota} \iota}"{description}, dashed, from=2-2, to=1-3]
	\arrow["{L\iota*}", shorten <=8pt, Rightarrow, from=1, to=2-2]
	\arrow["\eta"', shorten <=3pt, Rightarrow, from=0, to=2-2]
	\arrow["\simeq"{description}, Rightarrow, draw=none, from=2-2, to=2]
\end{tikzcd}\]
exhibiting $ \biLan_{L\iota} \iota$ as the composite $ \biLan_\hirayo \iota \mathcal{B}[L\iota, -]$, which both exist, ensuring it itself exists, and so does the desired $\biLan_L 1_\mathcal{A}$.
%\textcolor{red}{Who are the 2-cells in this diagram ? none is invertible} \textcolor{blue}{I do not understand, I never claimed that there are (even though there are) or that they should be invertible. In fact, their existence is not relevant in the proof and displaying them would just be incredibly confusing because the diagram would need a complete and very counter-intuitive re-design. Maybe we should discuss this on zoom, because while I do understand your critique, I do not see a good solution.}
%Indeed \cite[4.1.5]{descotte2018sigma} \textcolor{red}{2 things : in the rest of the paper citation style include the type of citation (propositon, theorem, def...) ; but here if this is something about general technics about left bikan we should rather have it stated explicitely in the short section about left bikan in the prerequisites} \textcolor{blue}{But I added it, check out 1.3.2 and 1.3.3} is telling us that the composition above is computing the desired (dashed) biKan extension, in complete analogy with \cite[A.6 and A.7]{di2020codensity}.

(b) Using \Cref{Weighted colimit expression of pointwise biLan} (which is the results of \cite[Sec. 4]{descotte2018sigma}), we know that if $\biLan_F1_\mathcal{A}$ exists, it is computed via a weighted bicolimits, in complete analogy to the $1$-dimensional theory of Kan extensions. Thus, because $F$ is bicocontinuous, there is no doubt that it would preserve $\biLan_F1_\mathcal{A}$.
\end{proof}

\begin{theorem}\label{LAFT}
Let $R: \mathcal{A} \to \mathcal{B}$ be a pseudofunctor preserving all weighted bilimits and bifiltered colimits between finitely bi-accessible $2$-categories with weighted bilimits. Then it has a left biadjoint.
\end{theorem}
\begin{proof}
As in the proof of \Cref{RAFT}, and again by the bi-analog of \cite[3.7.6]{borceux1994handbook}, we are reduced only to show that the right biKan extension in the diagram below exists.

% https://q.uiver.app/?q=WzAsMyxbMSwwLCJcXG1hdGhjYWx7QX0iXSxbMCwwLCJcXG1hdGhjYWx7QX0iXSxbMCwxLCJcXG1hdGhjYWx7Qn0iXSxbMSwyLCJSIiwyXSxbMSwwLCIxX1xcbWF0aGNhbHtBfSJdLFsyLDAsIlxcYmlSYW5fUigxKSIsMl0sWzEsNSwiXFxlcHNpbG9uIiwyLHsic2hvcnRlbiI6eyJ0YXJnZXQiOjIwfX1dXQ==
\[\begin{tikzcd}
	{\mathcal{A}} & {\mathcal{A}} \\
	{\mathcal{B}}
	\arrow["R"', from=1-1, to=2-1]
	\arrow["{1_\mathcal{A}}", from=1-1, to=1-2]
	\arrow[""{name=0, anchor=center, inner sep=0.5}, "{\biRan_R(1_\mathcal{A})}"', from=2-1, to=1-2]
	\arrow["\epsilon"', shorten >=2pt, Rightarrow, from=0, to=1-1]
\end{tikzcd}\]

This time we cannot use the previous proof strategy to finish the proof, because we would need a codense (as opposed to dense) sub-$2$-category of $\mathcal{A}$. So we need to be more subtle. Call $\ps[\mathcal{A}, \Cat]_s \hookrightarrow \ps[\mathcal{A}, \Cat]$ the full $2$-subcategory spanned by those pseudofunctors that are \textit{small} weighted bilimits of corepresentables and consider the following diagrams.
% https://q.uiver.app/?q=WzAsNCxbMiwwLCJcXG1hdGhjYWx7QX0iXSxbMCwwLCJcXG1hdGhjYWx7QX0iXSxbMSwxLCJcXG1hdGhjYWx7Qn0iXSxbMSwzLCJcXHBzW1xcbWF0aGNhbHtBfSwgXFxDYXRdXlxcY2lyY19zIl0sWzEsMiwiUiIsMV0sWzEsMCwiMV9cXG1hdGhjYWx7QX0iXSxbMiwwLCJcXHRleHR7YmlSYW59X1IoMV9cXG1hdGhjYWx7QX0pIiwxXSxbMSwzLCJcXGthdGF5byIsMix7ImN1cnZlIjo1fV0sWzIsMywiXFxtYXRoY2Fse0J9KC0sUikiLDFdLFszLDAsIlxcdGV4dHtiaVJhbn1fe1xca2F0YXlvfSAxX1xcbWF0aGNhbHtBfSIsMix7ImN1cnZlIjo1fV0sWzUsMiwiXFxzaW1lcSIsMSx7InNob3J0ZW4iOnsic291cmNlIjoyMH0sInN0eWxlIjp7ImJvZHkiOnsibmFtZSI6Im5vbmUifSwiaGVhZCI6eyJuYW1lIjoibm9uZSJ9fX1dLFsyLDksIlxcc2ltZXEiLDEseyJzaG9ydGVuIjp7InRhcmdldCI6MjB9LCJzdHlsZSI6eyJib2R5Ijp7Im5hbWUiOiJub25lIn0sImhlYWQiOnsibmFtZSI6Im5vbmUifX19XSxbMiw3LCIiLDEseyJzaG9ydGVuIjp7InRhcmdldCI6MjB9LCJzdHlsZSI6eyJib2R5Ijp7Im5hbWUiOiJub25lIn0sImhlYWQiOnsibmFtZSI6Im5vbmUifX19XV0=
% https://q.uiver.app/?q=WzAsNCxbMiwwLCJcXG1hdGhjYWx7QX0iXSxbMCwwLCJcXG1hdGhjYWx7QX0iXSxbMSwxLCJcXG1hdGhjYWx7Qn0iXSxbMSwyLCJcXHBzW1xcbWF0aGNhbHtBfSwgXFxDYXRdXlxcb3AiXSxbMSwyLCJSIiwxXSxbMSwwLCIxX1xcbWF0aGNhbHtBfSJdLFsyLDAsIlxcdGV4dHtiaVJhbn1fUigxX1xcbWF0aGNhbHtBfSkiLDFdLFsxLDMsIlxca2F0YXlvIiwyLHsiY3VydmUiOjR9XSxbMiwzLCJcXG1hdGhjYWx7Qn1bLSxSXSJdLFszLDAsIlxcdGV4dHtiaVJhbn1fe1xca2F0YXlvfSAxX1xcbWF0aGNhbHtBfSIsMix7ImN1cnZlIjo0fV0sWzIsNSwiXFxlcHNpbG9uIiwwLHsic2hvcnRlbiI6eyJzb3VyY2UiOjIwfX1dLFsyLDksIlxcc2ltZXEiLDEseyJzaG9ydGVuIjp7InRhcmdldCI6MjB9LCJzdHlsZSI6eyJib2R5Ijp7Im5hbWUiOiJub25lIn0sImhlYWQiOnsibmFtZSI6Im5vbmUifX19XSxbNywyLCJSKiIsMCx7InNob3J0ZW4iOnsic291cmNlIjoyMCwidGFyZ2V0IjoyMH19XV0=
\[\begin{tikzcd}
	{\mathcal{A}} && {\mathcal{A}} \\
	& {\mathcal{B}} \\
	& {\ps[\mathcal{A}, \Cat]_s^\op}
	\arrow["R"{description}, from=1-1, to=2-2]
	\arrow[""{name=0, anchor=center, inner sep=0}, "{1_\mathcal{A}}", from=1-1, to=1-3]
	\arrow["{\text{biRan}_R1_\mathcal{A}}"{description}, from=2-2, to=1-3]
	\arrow[""{name=1, anchor=center, inner sep=0}, "\katayo^\op"', curve={height=24pt}, from=1-1, to=3-2]
	\arrow["{\mathcal{B}[-,R]}", from=2-2, to=3-2]
	\arrow[""{name=2, anchor=center, inner sep=0}, "{\text{biRan}_{\katayo^\op} 1_\mathcal{A}}"', curve={height=24pt}, from=3-2, to=1-3]
	\arrow["\epsilon", shorten <=3pt, Rightarrow, from=2-2, to=0]
	\arrow["\simeq"{description}, Rightarrow, draw=none, from=2-2, to=2]
	\arrow["{R*}"', shorten <=8pt, shorten >=8pt, Rightarrow, from=2-2, to=1]
\end{tikzcd}\]
% https://q.uiver.app/?q=WzAsNCxbMiwwLCJcXG1hdGhjYWx7QX0iXSxbMCwwLCJcXG1hdGhjYWx7QX0iXSxbMCwyLCJcXG1hdGhjYWx7Qn0iXSxbMCw0LCJcXHVuZGVybGluZXsoXFxDYXRee0F9KV5cXGNpcmN9Il0sWzEsMiwiUiIsMV0sWzEsMCwiMV9cXG1hdGhjYWx7QX0iLDFdLFsyLDAsIlxcdGV4dHtiaVJhbn1fUigxKSIsMV0sWzEsMywiXFxoaXJheW9eXFxmbGF0IiwxLHsiY3VydmUiOjV9XSxbMiwzLCJcXG1hdGhjYWx7Qn0oPSxSKC0pKSIsMSx7ImNvbG91ciI6WzMwMCw2MCw2MF19LFszMDAsNjAsNjAsMV1dLFszLDAsIlxcdGV4dHtiaVJhbn1fe1xcaGlyYXlvXlxcZmxhdH0gMV9cXG1hdGhjYWx7QX0iLDEseyJjdXJ2ZSI6NSwiY29sb3VyIjpbMCw2MCw2MF19LFswLDYwLDYwLDFdXV0=
The outer biKan extension exists because $\mathcal{A}$ has small weighted bilimits (in fact the two conditions are equivalent by the dual version \cite[Sec. 4]{descotte2018sigma}). If we prove that $\mathcal{B}[-,R]$ is well defined, then we finish as in the previous theorem following \Cref{cancellation rule for biRan}. Putting everything together, we need to show that for every $ B $ in $\mathcal{B}$, the functor $\mathcal{B}[B, R] : \mathcal{A} \to \Cat$ is a small weighted bicolimit of corepresentables in $\ps[\mathcal{A}, \Cat]_s$. Now, recall that, by \Cref{closureiscute3} and \Cref{every object compact}, $B$ must be $\lambda$-compact for some $\lambda$. Thus the functor $\mathcal{B}(B, R) : \mathcal{A} \to \Cat$ preserves $\lambda$-directed colimits for some $\lambda$, thus the theorem follow by \Cref{cr} and \Cref{laKan}.
\end{proof}

\subsection{2-dimensional Gabriel-Ulmer Duality} \label{GU}

We shall now present a Gabriel-Ulmer duality between bilex $2$-categories and locally finitely bipresentable $2$-categories. While the technical material that will deliver the duality is already disseminated in the previous sections of the paper (\Cref{sigma-accessible 2-cat are categories of flat functors} and \Cref{GUhalf1}), we must dedicate some paragraphs to clarify the $3$-dimensional setting in which the duality takes place. To do so, we will need to acknowledge that bilex $2$-categories and locally finitely bipresentable $2$-categories can be both organized in tricategories. As a general reference for the theory of tricategories we refer to \cite[Chap. 11]{johnson20212}. 

\begin{remark}
Before diving in the relevant definitions for this section, we may recall some delicate points related to the usual problem of coherence in higher dimensional category theory. Most of \cite[Chap. 11]{johnson20212} is dedicated to show that the collection of bicategories, pseudofunctors, pseudonatural transformations and modifications form a tricategory, and this amounts to quite a long and tedious proof. Of course, they cannot have a more strict structure, say that of a $3$-category, because none of the compositions can be defined on the nose due to the pseudo-ness. Despite restricting our attention to strict $2$-categories, our choice of morphisms puts up in a similar situation to that of \cite[Chap. 11]{johnson20212}. In \cite{shulman2012not}, Shulman observes that the tricategory $2\hy\Cat_{\ps}$ of $2$-categories, pseudofunctors, pseudonatural transformations and modifications is \textit{iconic}. Iconic tricategories are more general than Gray categories.
\end{remark}

\begin{definition}[The tricategory $\mathbf{biLex}$]
The tricategory $\mathbf{biLex}$ has objects small $2$-categories with weighed finite bilimits in the sense of \Cref{finiteweibili}. $1$-cells are pseudofunctors preserving finite bilimits, $2$-cells are pseudonatural transformations and $3$-cells are modifications. The structure of tricategory is inherited by that of $2\hy\Cat_{\ps}$.
\end{definition}

\begin{definition}[The tricategory $\mathbf{biP}_\omega$]
The tricategory $\mathbf{biP}_\omega$ has objects finitely bipresentable $2$-categories. $1$-cells are right biadjoints preserving bifiltered bicolimits, $2$-cells are pseudonatural transformations and $3$-cells are modifications. The structure of tricategory is inherited by that of $2\hy\Cat_{\ps}$.
\end{definition}

\begin{theorem}[$2$-dimensional Gabriel-Ulmer duality]\label{GUth}
There is a tri-equivalence of tricategories % https://q.uiver.app/?q=WzAsMixbMCwwLCJcXG1hdGhiZntMZXh9XntcXG9wfSJdLFsyLDAsIlxcbWF0aGJme0ZQfV9cXHNpZ21hIl0sWzAsMSwiXFxtYXRoYmZ7TW9kfSAiLDAseyJjdXJ2ZSI6LTN9XSxbMSwwLCJcXG1hdGhiZntUaH0iLDAseyJjdXJ2ZSI6LTN9XSxbMiwzLCIiLDAseyJsZXZlbCI6MSwic3R5bGUiOnsibmFtZSI6ImFkanVuY3Rpb24ifX1dXQ==
\[\begin{tikzcd}
	{\mathbf{biLex}^{\op}} && {\mathbf{biP}_\omega}
	\arrow[""{name=0, anchor=center, inner sep=0}, "{\mathbf{Mod} }", curve={height=-18pt}, from=1-1, to=1-3]
	\arrow[""{name=1, anchor=center, inner sep=0}, "{\mathbf{Th}}", curve={height=-18pt}, from=1-3, to=1-1]
	\arrow["\dashv"{anchor=center, rotate=-90}, draw=none, from=0, to=1]
\end{tikzcd}\]
\end{theorem}

\begin{division}[$\mathbf{Mod}$]
The trifunctor $\mathbf{Mod} : \mathbf{biLex}^{\op} \to \mathbf{biP}_\omega$ can be described by, \[\mathbf{Mod}(\mathcal{C}) = \textbf{biLex}[\mathcal{C}, \Cat].\]
To be more precise, this correspondence is well-defined at the level of objects because $\textbf{biLex}[\mathcal{C}, \Cat]$ is indeed a finitely bipresentable $2$-category by \Cref{Flat pseudofunctors on bilex are bipres}. At the level of morphisms, let $F: \mathcal{C} \to \mathcal{D}$ be a pseudofunctor preserving finite bilimits, then the precomposition 
% https://q.uiver.app/?q=WzAsMixbMCwwLCJcXHRleHRiZntiaUxleH1bXFxtYXRoY2Fse0R9LCBcXENhdF0iXSxbMSwwLCJcXHRleHRiZntiaUxleH1bXFxtYXRoY2Fse0N9LCBcXENhdF0iXSxbMCwxLCJGXiogIl1d
\[\begin{tikzcd}
	{\textbf{biLex}[\mathcal{D}, \Cat]} & {\textbf{biLex}[\mathcal{C}, \Cat]}
	\arrow["{F^* }", from=1-1, to=1-2]
\end{tikzcd}\]
is clearly well defined, indeed for every lex functor $X : \mathcal{D} \to \Cat$, $XF$ is bilex too, because both $X$ and $F$ are. Now, looking at the commutative diagram below we notice that,
% https://q.uiver.app/?q=WzAsNCxbMCwwLCJcXHRleHRiZntiaUxleH1bXFxtYXRoY2Fse0R9LCBcXENhdF0iXSxbMCwxLCJcXHBzW1xcbWF0aGNhbHtEfSwgXFxDYXRdIl0sWzEsMCwiXFx0ZXh0YmZ7YmlMZXh9W1xcbWF0aGNhbHtDfSwgXFxDYXRdIl0sWzEsMSwiXFxwc1tcXG1hdGhjYWx7Q30sIFxcQ2F0XSJdLFswLDEsIlxcaW90YV9EIiwyLHsic3R5bGUiOnsidGFpbCI6eyJuYW1lIjoiaG9vayIsInNpZGUiOiJib3R0b20ifX19XSxbMiwzLCJcXGlvdGFfQyIsMCx7InN0eWxlIjp7InRhaWwiOnsibmFtZSI6Imhvb2siLCJzaWRlIjoiYm90dG9tIn19fV0sWzAsMiwiRl4qIl0sWzEsMywiRl4qIiwyXSxbMCwzLCI9IiwxLHsic3R5bGUiOnsiYm9keSI6eyJuYW1lIjoibm9uZSJ9LCJoZWFkIjp7Im5hbWUiOiJub25lIn19fV1d
\[\begin{tikzcd}
	{\textbf{biLex}[\mathcal{D}, \Cat]} & {\textbf{biLex}[\mathcal{C}, \Cat]} \\
	{\ps[\mathcal{D}, \Cat]} & {\ps[\mathcal{C}, \Cat]}
	\arrow["{\iota_D}"', hook', from=1-1, to=2-1]
	\arrow["{\iota_C}", hook', from=1-2, to=2-2]
	\arrow["{F^*}", from=1-1, to=1-2]
	\arrow["{F^*}"', from=2-1, to=2-2]
	\arrow["{=}"{description}, draw=none, from=1-1, to=2-2]
\end{tikzcd}\]
\begin{itemize}
    \item $F^*$ preserves all weighted bilimits. This is because weighted bilimits are created by $\iota_\mathcal{D}$ (\Cref{preserveandcreatecolimits}) and $\iota_\mathcal{C}$ and the precomposition functor on the bottom of the diagram preserves all weighted bilimits and bicolimits.
    \item Similarly,  $F^*$ preserves all bifiltered bicolimits, again by applying \Cref{preserveandcreatecolimits}.
\end{itemize}
By our version of the adjoint functor theorem (\Cref{LAFT}), $F^*$ is a right biadjoint preserving bifiltered bicolimits. The action of $\mathbf{Mod}$ on $2$ and $3$-cells is relatively straightforward and does not require much justification, we describe it for the case of $2$-cells. Consider a pseudonatural transformation as below,

% https://q.uiver.app/?q=WzAsMixbMCwwLCJDIl0sWzIsMCwiRCJdLFswLDEsIkYiLDAseyJjdXJ2ZSI6LTN9XSxbMCwxLCJHIiwyLHsiY3VydmUiOjN9XSxbMiwzLCJcXGFscGhhIiwwLHsic2hvcnRlbiI6eyJzb3VyY2UiOjIwLCJ0YXJnZXQiOjIwfX1dXQ==
\[\begin{tikzcd}
	C && D
	\arrow[""{name=0, anchor=center, inner sep=0}, "F", curve={height=-18pt}, from=1-1, to=1-3]
	\arrow[""{name=1, anchor=center, inner sep=0}, "G"', curve={height=18pt}, from=1-1, to=1-3]
	\arrow["\alpha", shorten <=5pt, shorten >=5pt, Rightarrow, from=0, to=1]
\end{tikzcd}\]

Then of course we get a pseudonatural transformation between the precompositons functors $F^* \Rightarrow G^*$.

\end{division}

\begin{division}[$\mathbf{Th}$]
For the trifunctor $\mathbf{Th}: \mathbf{biP}_\omega \to \mathbf{biLex}^{\op}$ we proceed as follows,
\begin{itemize}
    \item at the level of objects, it maps a finitely bipresentable $2$-category $\mathcal{B}$ to the opposite of its full sub $2$-category of bicompact objects $\mathcal{B} \mapsto \mathcal{B}_\omega^{\op}$. Via \Cref{finite bicolimit of sigma-compact are sigma-compact}, $\mathcal{B}_\omega$ is closed under finite weighted bicolimits and thus $ \mathcal{B}_\omega^{\op}$ is in $\mathbf{biLex}$.
    \item Given a right adjoint preserving bifiltered colimits $R: \mathcal{A} \to \mathcal{B}$, we know that its left adjoint $L$ must map bicompact objects to bicompact objects via \Cref{left biadjoint of finitary 2-functor preserve sigma-compact}. So we define the opposite of its restriction % https://q.uiver.app/?q=WzAsMixbMCwwLCIoXFxtYXRoY2Fse0J9X1xcb21lZ2EpXntcXG9wfSJdLFsxLDAsIihcXG1hdGhjYWx7QX1fXFxvbWVnYSlee1xcb3B9Il0sWzAsMSwie0xee1xcb3B9fSJdXQ==
\[\begin{tikzcd}
	{(\mathcal{B}_\omega)^{\op}} & {(\mathcal{A}_\omega)^{\op}}
	\arrow["{{L^{\op}}}", from=1-1, to=1-2]
\end{tikzcd}\] to be the image of $R$ under the action of $\mathbf{Th}$.
\item For the behavior of $\mathbf{Th}$ at the level of $2$-cells and $3$-cells, recall that the $2$-category of right biadjoints, pseudonatural transformations and modifications is biequivalent to the opposite of left biadjoints, pseudonatural transformations and modifications $RAdj(\mathcal{A}, \mathcal{B}) \simeq LAdj(\mathcal{B}, \mathcal{A})^\op$ and thus we can easily define $\mathbf{Th}$ on $2$ and $3$-cells as we did for $\mathbf{Mod}$.
\end{itemize}
\end{division}

\begin{proof}[Proof of \Cref{GUth}]

As discussed in \Cref{biacc as category of flat pseudofunctors} and in the first lines of \Cref{sigma-accessible 2-cat are categories of flat functors}, we have a pseudofunctor 

% https://q.uiver.app/?q=WzAsMixbMCwwLCJcXG1hdGhjYWx7Qn0iXSxbMSwwLCJcXEZsYXRfXFxwc1soXFxtYXRoY2Fse0J9X1xcb21lZ2EpXntcXG9wfSwgXFxDYXRdIl0sWzAsMSwiXFxudV9cXG1hdGhjYWx7Qn0gIl1d
\[\begin{tikzcd}
	{\mathcal{B}} & {\Flat_\ps[(\mathcal{B}_\omega)^{\op}, \Cat]}
	\arrow["{\nu_\mathcal{B} }", from=1-1, to=1-2]
\end{tikzcd}\]
which is easy to acknowledge as a morphism of locally finitely bipresentable $2$-categories by \Cref{binerve preserve bifiltered bicolimits}. The collection of all the $\nu$'s gives us a pseudonatural transformation  $1 \Rightarrow \mathbf{Mod}\circ \mathbf{Th}$. \Cref{sigma-accessible 2-cat are categories of flat functors} proves that such transformation is a pointwise biequivalence of $2$-categories.
Similarly to the previous discussion, the Yoneda embedding discussed in \Cref{GUhalf1}, gives a biequivalence $\mathcal{C}^{\op} \simeq (\biLex[\mathcal{C}, \Cat])_\omega$ in $\mathbf{biLex}$ which yields a pseudonatural transformation $\mathbf{Th}\circ \mathbf{Mod} \Rightarrow 1$, which -- again -- is shown to be a pointwise biequivalence of $2$-categories.
\end{proof}

\section{Examples}

\subsection{Cat}

The following ur-example is key to the next results. Recall that in $\Cat$, finite categories are bicompact, which we already proved at \Cref{finite cat are sigma-compact}.

% \begin{lemma}
% In $\Cat$, finite categories are bicompact. 
% \end{lemma}

% \begin{remark}
% See also \cite{street1976limits}[Proposition 4] in the case where $ \alpha$ is $\omega$. Beware again, as observed in \cref{Bicompact categories might not be finite}, that the converse is not true. 
% \end{remark}

\begin{theorem}
$\Cat$ is finitely bipresentable.
\end{theorem}

\begin{proof}
1 and 2 are strong generators, and they are bicompacts: hence $\Cat$ admits a strong generator of bicompact objects, which ensures it to be finitely bipresentable from \Cref{1.11}.
\end{proof}

\subsection{2-categories of pseudo-algebras of bifinitary pseudomonads}

There is a well known theory of pseudomonads and their algebras (see for example \cite{MarmolejoBeck}), for which we dispense us of definition. 

%Inheritance of pseudolimits in 2-category of strict algebras and pseudomorphisms is established in \cite{blackwell1989two}[Theorem 2.6] and also is similar to the more specific \cite{szyld2018lifting}:

\begin{comment}
Let $ T$ be a strict pseudomonad on $\mathcal{C}$: then if $\mathcal{C}$ has weighted bilimits, so has $ T\hy\Alg$ and they are preserved by the forgetful 2-functor $ U_T^p : T\hy\Alg \rightarrow \Cat$.
\end{comment}

%If $ T$ is a strict pseudomonad on $\mathcal{C}$, then if $ \mathcal{C}$ has weighted bicolimits, so has $ T\hy\Alg$, and they are computed as the free algebra on the underlying bicolimit in $\mathcal{C}$
%\[ \underset{i \in I}{\bicolim}\, (A_i, a_i) \simeq ( T(\underset{i \in I}{\bicolim}\, A_i), \mu \]

\begin{comment}
Here $T\hy\Alg $ denotes strict algebras and pseudomorphisms. Observe this result is an instance of inheritance of a class of limits in an adjunction where the right adjoint is not required to be full - nor even faithful in fact.  Moreover, observe that, as pseudolimits are in particular bilimits, $T\hy\Alg$ can be thought hence as having all bilimits. 
\end{comment}

\begin{definition}
A pseudomonad is said to be \emph{bifinitary} if it preserves bifiltered bicolimits.
\end{definition}

It is well known since \cite{blackwell1989two}[Theorem 5.8] that 2-categories of \emph{strict} algebras and pseudomorphisms for finitary 2-monads are bicocomplete. However, for our purposes, in particular in the context of $\Phi$-exactness we are investigating in the last subsection, we need a corresponding statement concerning the 2-category of \emph{pseudo-algebras}, furthermore in the case of a \emph{pseudomonad}. This is the content of \cite{osmondcodescent}, from which we use the following result:

\begin{proposition}{\cite{osmondcodescent}[Theorem 4.19]}\label{pseudoalg are bicocomplete}
Let $(T, \eta, \mu, (\xi, \zeta, \rho))$ be a bifinitary pseudomonad on a bicomplete and bicocomplete 2-category $\mathcal{C}$. Then the 2-category of pseudo-algebras and pseudomorphisms $ T\hy\psAlg$ is bicocomplete. 
\end{proposition}

% \begin{proposition}
% If $ T$ is a finitary pseudomonad on $\Cat$, then $ T\hy\Alg$ has weighted bicolimits. 
% \end{proposition}

% \begin{remark}
% Beware that those bicolimits may not be actually pseudocolimits; moreover they need not being preserved by the forgetfull functors - except the filtered ones, and in particular the bifiltered ones. In the following we need the slightly stronger condition of preservations of bifiltered pseudocolimits.
% \end{remark}

The theorem below is the 2-categorical analog of the famous result of \cite{gabriel2006lokal}: %{\color{red}Or a reference by Kelly ? Did the second paper ever appear ?}:

\begin{theorem}\label{Algebras of a finitary pseudomonad are sigma-presentable}
Let  $\mathcal{B}$ be a finitely bipresentable 2-category and $ T$ a bifinitary pseudomonad on $\mathcal{B}$. Then $T\hy\ps\Alg$ is also finitely bipresentable, and the forgetful 2-functor $ U_T : T\hy\ps\Alg \rightarrow \mathcal{B}$ is finitely bi-accessible. 
\end{theorem}

\begin{proof}
%We know that $T\hy\Alg$ will inherit bilimits from $ \mathcal{B}$. Moreover, we know also that $ T\hy\Alg$ has bicolimits because $ T$ is finitary and $ \mathcal{B}$ is bicocomplete  by \cite[5.8 (and 3.8)]{blackwell1989two}. \\

%and as preservation of bifiltered bicolimits implies to be finitary, we also have existence of pseudocolimits. 

From \cref{pseudoalg are bicocomplete}, we know that $ T\hy\psAlg$ is bicocomplete for $ T$ is bifinitary and $\mathcal{B}$ is bicomplete (see \cref{bipres are bicomplete}) and bicocomplete as a finitely bipresentable 2-category. Though arbitrary bicolimits in $ T\hy\ps\Alg$ need not be preserved by the forgetful functor, we are going to prove that bifiltered bicolimits are. We claim that they are computed as follows: for $F : I \rightarrow  T\hy\ps\Alg$ a 2-functor with $I$ bifiltered, with $F(i) = (A_i, a_i)$, we have 
\[  T(\underset{i \in I }{\bicolim}\, A_i ) \simeq  \underset{i \in I }{\bicolim} \,TA_i \]
which, together with the morphism induced by the universal property of the bicolimit at the composites $q_ia_i$, provides us with a structure of pseudomorphism of pseudo-algebras for the bicolimit inclusions
% https://q.uiver.app/?q=WzAsNSxbMiwwLCJUKFxcdW5kZXJzZXR7aSBcXGluIEkgfXtcXGJpY29saW19XFwsIEFfaSApIl0sWzEsMCwiXFx1bmRlcnNldHtpIFxcaW4gSSB9e1xcYmljb2xpbX0gXFwsVEFfaSJdLFswLDAsIlRBX2kiXSxbMSwxLCJcXHVuZGVyc2V0e2kgXFxpbiBJIH17XFxiaWNvbGltfSBcXCxBX2kiXSxbMCwxLCJBX2kiXSxbMiwxLCJUcV9pIl0sWzIsNCwiYV9pIiwyXSxbMSwzLCJcXGxhbmdsZSBxX2lhX2kgXFxyYW5nbGVfe2kgXFxpbiBJfSJdLFs0LDMsInFfaSIsMl0sWzEsMCwiXFxzaW1lcSIsMSx7InN0eWxlIjp7ImJvZHkiOnsibmFtZSI6Im5vbmUifSwiaGVhZCI6eyJuYW1lIjoibm9uZSJ9fX1dLFsyLDMsIlxcc2ltZXEiLDEseyJzdHlsZSI6eyJib2R5Ijp7Im5hbWUiOiJub25lIn0sImhlYWQiOnsibmFtZSI6Im5vbmUifX19XV0=
% https://q.uiver.app/?q=WzAsNCxbMSwwLCJcXHVuZGVyc2V0e2kgXFxpbiBJIH17XFxiaWNvbGltfSBcXCxUQV9pIl0sWzAsMCwiVEFfaSJdLFsxLDEsIlxcdW5kZXJzZXR7aSBcXGluIEkgfXtcXGJpY29saW19IFxcLEFfaSJdLFswLDEsIkFfaSJdLFsxLDAsIlRxX2kiXSxbMSwzLCJhX2kiLDJdLFswLDIsIlxcbGFuZ2xlIHFfaWFfaSBcXHJhbmdsZV97aSBcXGluIEl9Il0sWzMsMiwicV9pIiwyXSxbMSwyLCJcXHNpbWVxIiwxLHsic3R5bGUiOnsiYm9keSI6eyJuYW1lIjoibm9uZSJ9LCJoZWFkIjp7Im5hbWUiOiJub25lIn19fV1d
\[\begin{tikzcd}
	{TA_i} & {\underset{i \in I }{\bicolim} \,TA_i} \\
	{A_i} & {\underset{i \in I }{\bicolim} \,A_i}
	\arrow["{Tq_i}", from=1-1, to=1-2]
	\arrow["{a_i}"', from=1-1, to=2-1]
	\arrow["{\langle q_ia_i \rangle_{i \in I}}", from=1-2, to=2-2]
	\arrow["{q_i}"', from=2-1, to=2-2]
	\arrow["\simeq"{description}, draw=none, from=1-1, to=2-2]
\end{tikzcd}\]
so the structure of $T$-pseudo-algebra is induced as the universal map 
\[\begin{tikzcd}[sep=large]
	& {A_i} & TAi \\
	{\underset{i \in I }{\bicolim}\, A_i } & {\underset{i \in I }{\bicolim}\, TA_i } & {A_i} \\
	& {\underset{i \in I }{\bicolim}\, A_i }
	\arrow["{q_i}"', from=1-2, to=2-1]
	\arrow["{q_i}", from=2-3, to=3-2]
	\arrow["{\eta_{\underset{i \in I }{\bicolim}\, A_i }}", pos=0.6, from=2-1, to=2-2]
	\arrow[""{name=0, anchor=center, inner sep=0}, Rightarrow, no head, from=2-1, to=3-2]
	\arrow[""{name=1, anchor=center, inner sep=0}, Rightarrow, no head, from=1-2, to=2-3]
	\arrow[""{name=2, anchor=center, inner sep=0}, "{\eta_{A_i}}", from=1-2, to=1-3]
	\arrow["{a_i}", from=1-3, to=2-3]
	\arrow["{\langle q_ia_i \rangle_{i \in I}}"{description}, from=2-2, to=3-2]
	\arrow["\simeq"{description}, draw=none, from=2-2, to=2-3]
	\arrow["\simeq"{description}, draw=none, from=1-2, to=2-2]
	\arrow[crossing over, from=1-3, to=2-2]
	\arrow["{\langle\alpha_i\rangle_{i \in I} \atop \simeq}"{description}, pos=0.3, Rightarrow, draw=none, from=2-2, to=0]
	\arrow["{\alpha_i \atop \simeq}"{description, pos=0.4}, Rightarrow, draw=none, from=2, to=1]
\end{tikzcd}\]
The top and left squares exhibit the bicolimit inclusions as pseudomorphisms of $T$-algebras. Hence $ T\hy\ps\Alg$ has bifiltered bicolimits. We left the verification that those data satisfy the coherence condition of pseudoalgebras and pseudomorphisms to the careful reader.

We must prove that the free algebras on bicompacts form a strong generator of bicompact objects for $T\hy\ps\Alg$. Consider the 2-adjunction
% https://q.uiver.app/?q=WzAsMixbMCwwLCJUXFxoeVxcQWxnX3AiXSxbMiwwLCJcXG1hdGhjYWx7Qn0iXSxbMCwxLCJVX1QiLDIseyJjdXJ2ZSI6Mn1dLFsxLDAsIkZfVCIsMix7ImN1cnZlIjoyfV0sWzMsMiwiIiwwLHsibGV2ZWwiOjEsInN0eWxlIjp7Im5hbWUiOiJhZGp1bmN0aW9uIn19XV0=
\[\begin{tikzcd}
	{T\hy\ps\Alg} && {\mathcal{B}}
	\arrow[""{name=0, anchor=center, inner sep=0}, "{U_T}"', start anchor=-12, bend right=20, from=1-1, to=1-3]
	\arrow[""{name=1, anchor=center, inner sep=0}, "{F_T}"', end anchor=12, bend right=20, from=1-3, to=1-1]
	\arrow["\dashv"{anchor=center, rotate=-90}, draw=none, from=1, to=0]
\end{tikzcd}\]

From \Cref{left biadjoint of finitary 2-functor preserve sigma-compact}, we know that free algebras on bicompacts are bicompact as the free algebra functor is left 2-adjoint to the forgetful functor which preserves bifiltered pseudocolimits. Moreover, as bicompact objects form a dense generator, the 2-functors $ \mathcal{B}[K,-]$ jointly reflects equivalences, as well as their restriction to objects that bear a structure of algebra; but by 2-adjunction we have at each $(A,a)$ of $T\hy\ps\Alg$ natural isomorphisms of categories $ \mathcal{B}[K, U_T(A,a)] \simeq T\hy\ps\Alg[F_T(K), (A,a)]$, which provides a natural equivalence of functors 
\[ \mathcal{B}[K, U_T] \simeq T\hy\ps\Alg[F_T(K), -] \]
Hence the representable $T\hy\ps\Alg[F_T(K), -]$ jointly reflect equivalences in $T\hy\ps\Alg$: hence they form a strong generator. \Cref{1.11} then ensures that $T\hy\ps\Alg$ is finitely bipresentable.
\end{proof}

\begin{comment}\label{pseudoalg are bipres}
Let  $\mathcal{B}$ be a finitely bipresentable 2-category and $ T$ a bifinitary pseudomonad on $\mathcal{B}$. Then $T\hy\\Alg$ is also finitely bipresentable, and the forgetful 2-functor $ U_p : T\hy\\Alg \rightarrow \mathcal{B}$ is finitely bi-accessible. 
\end{comment}

\subsection{$\Lex$}

Recall that $ \Lex$ is the 2-category of small lex categories and lex functors - where lex functors preserve finite limits only up to isomorphism. We will prove that $\Lex$ is finitely bipresentable. It is well known (see for instance \cite{blackwell1989two}) that $ \Lex$ is the 2-category of pseudo-algebras and pseudomorphisms for a finitary KZ-monad on $\Cat$. For the sake of completeness, we will prove again the finitary part of the result, as the rank of accessibility of $\Lex$ is really crucial here.  

\begin{division}
Let us give a few words on why $\Lex$ is KZ-monadic on $ \Cat$. Consider the free completion under finite limits 
% https://q.uiver.app/?q=WzAsMixbMCwwLCJcXENhdCJdLFsxLDAsIlxcQ2F0Il0sWzAsMSwiXFxMZXgiXV0=
% https://q.uiver.app/?q=WzAsMixbMCwwLCJcXENhdCJdLFsxLDAsIlxcQ2F0Il0sWzAsMSwiXFxMZXhbXFxoeV0iXV0=
\[\begin{tikzcd}
	\Cat & \Cat
	\arrow["{\Lex[\hy]}", from=1-1, to=1-2]
\end{tikzcd}\]
sending a small category to its free completion under finite limits - which is still small; this functor defines a pseudomonad on $\Cat$. This monad is well known to be KZ, and  we have a biequivalence exhibiting $\Lex$ as the 2-category of pseudo-algebras for this KZ-monad
\[ \Lex \simeq \Lex[-]\hy\ps\Alg \]

Hence $\Lex$ is ensured to have bicolimits as well as pseudolimits - and then bilimits; the later are also preserved by the pseudo-faithful 2-functor
% https://q.uiver.app/?q=WzAsMixbMCwwLCJcXExleCJdLFsxLDAsIlxcQ2F0Il0sWzAsMSwiVSIsMCx7InN0eWxlIjp7InRhaWwiOnsibmFtZSI6Imhvb2siLCJzaWRlIjoidG9wIn19fV1d
\[\begin{tikzcd}
	\Lex & \Cat
	\arrow["U", hook, from=1-1, to=1-2]
\end{tikzcd}\]

Beware that this functor is not full, which prevents us to use directly reflection theorems bipresentability as they uses fullness. 
\end{division}

\begin{lemma}\label{Lex is finitary}
$\Lex$ is closed in $\Cat$ under bifiltered bicolimits.
\end{lemma}

\begin{proof}
Recall that one can always use a pseudocolimit as a bicolimit; but in $\Cat$, pseudocolimit are obtained as localization of oplaxcolimits. Let $ \mathcal{C}_{(-)} : I \rightarrow \Lex $ be a bifiltered diagram of lex categories and lex functors. Then the can consider the Grothendieck construction (which is an opfibration on $I$) $ \oplaxcolim_{i \in I} \mathcal{C}_i$, and obtain the pseudocolimit in $\Cat$ as (the underlying category of) its localization at opcartesian morphisms 
\[ \underset{i \in I}{\pscolim} \; \mathcal{C}_i \simeq \underset{i \in I}{\oplaxcolim} \; \mathcal{C}_i[\textbf{opCart}^{-1}] \]

We claim that $ \pscolim_{i \in I} \; \mathcal{C}_i$ already is lex. It is clear it is still small as $I$ and each $\mathcal{C}_i$ are. Moreover, its finite limits are computed as follows. For each finite diagram $ G : J \rightarrow \pscolim_{i \in I} \; \mathcal{C}_i$ with $J$ a finite category, we can pick for each $ j$ a representing object $ (i_j, c_j)$ for $ G(j) $; then for $ J$ is finite, there is by \Cref{sigmacone in sigmafiltered} a pseudocone $ (d_j : i_j \rightarrow i_J)_{j \in J}$ in $I$, which produces then a diagram $ (f_{d_j}(c_j))_{j \in J}$ in $\mathcal{C}_{i_J}$, which admits hence a limit in $\mathcal{C}_{i_J}$. Then $ (i_J, \lim_{j \in J}f_{d_j}(c_j))$ is a representant for a limit in $ \pscolim_{i \in I} \;\mathcal{C}_i$. 

If now one has a pseudocone $( f_i : \mathcal{C}_i \rightarrow \mathcal{C})_{i \in I}$ in $\Lex$, then we have in particular a two-steps strict factorization in $\Cat$
% https://q.uiver.app/?q=WzAsNCxbMCwwLCJcXG1hdGhjYWx7Q31faSJdLFsyLDAsIlxcbWF0aGNhbHtDfSJdLFswLDEsIlxcdW5kZXJzZXR7aSBcXGluIEl9e1xcb3BsYXhjb2xpbX1cXDtcXG1hdGhjYWx7Q31faSJdLFswLDIsIlxcdW5kZXJzZXR7aSBcXGluIEl9e1xccHNjb2xpbX0gXFwsXFxtYXRoY2Fse0N9X2kgIl0sWzAsMiwicV9pIiwyXSxbMiwzLCJxIiwyLHsic3R5bGUiOnsiaGVhZCI6eyJuYW1lIjoiZXBpIn19fV0sWzAsMSwiZl9pIl0sWzIsMSwiXFxsYW5nbGUgZl9pIFxccmFuZ2xlX3tpIFxcaW4gSX0iLDFdLFszLDEsIlxcb3ZlcmxpbmV7XFxsYW5nbGUgZl9pIFxccmFuZ2xlX3tpIFxcaW4gSX19IiwyLHsiY3VydmUiOjJ9XV0=
\[\begin{tikzcd}
	{\mathcal{C}_i} && {\mathcal{C}} \\
	{\underset{i \in I}{\oplaxcolim}\;\mathcal{C}_i} \\
	{\underset{i \in I}{\pscolim} \,\mathcal{C}_i }
	\arrow["{q_i}"', from=1-1, to=2-1]
	\arrow["q"', two heads, from=2-1, to=3-1]
	\arrow["{f_i}", from=1-1, to=1-3]
	\arrow["{\langle f_i \rangle_{i \in I}}"{description}, from=2-1, to=1-3]
	\arrow["{\overline{\langle f_i \rangle_{i \in I}}}"', curve={height=12pt}, from=3-1, to=1-3]
\end{tikzcd}\]
Hence the induced functor $\overline{\langle f_i \rangle_{i \in I}} $ is lex as one has for any finite diagram $G : J \rightarrow \pscolim_{i \in I} \mathcal{C}_i$
\begin{align*}
    \overline{\langle f_i \rangle_{i \in I}}( [(i_J, \lim_{j \in J}\,f_{d_j}(c_j)) ]_{\sim_I} ) 
    &\simeq \langle f_i \rangle_{i \in I} (i_J, \lim_{j \in J}\,f_{d_j}(c_j)) \\
    &\simeq f_{i_J}(\lim_{j \in J}\,f_{d_j}(c_j)) \\
    &\simeq \lim_{j \in J}\, f_{i_J}f_{d_j}(c_j) \\
    &\simeq \lim_{j \in J}\, f_{i_j}(c_j) \\
    &\simeq \lim_{j \in J}\, \overline{\langle f_i \rangle_{i \in I}}( [(i_j, c_j) ]_{\sim_I} ) 
\end{align*}

Hence $ \pscolim_{i \in I} \; \mathcal{C}_i$, though computed in $\Cat$, is a small lex category and provides a bicolimit in $ \Lex$. Hence $ \Lex$ is closed in $ \Cat$ under bifiltered bicolimits. 
\end{proof}

\begin{theorem}
$\Lex$ is finitely bipresentable.
\end{theorem}

\begin{proof}
From \Cref{left biadjoint of finitary 2-functor preserve sigma-compact} we know that free lex categories on finite categories are bicompact in $\Lex$. Now $ \Lex$ is the category of strict algebras and pseudomorphisms of the $\Lex[-]$-monad on $\Cat$, which is finitary by \Cref{Lex is finitary}. Hence by \Cref{Algebras of a finitary pseudomonad are sigma-presentable}, it is finitely bipresentable and its forgetful functor is finitely accessible. 
\end{proof}

Also, we end here with a short lemma to ensure that as expected finitely generated lex categories are bicompact:

\begin{lemma}\label{reviewer stupidity}
    A finitely generated lex category -- that is, the free lex category over a finite category -- is bicompact in $\Lex$
\end{lemma}

\begin{proof}
Trivial from the bicompactness of finite categories in $\Cat$, using the universal property of the free completion under finite limits.
\end{proof}

\begin{remark}
It was expected that $\Lex$ should be finitely bipresentable. Morally, this is because $\Lex$ should be thought of as a 2-category of models of some 2-limit theory with all the finite diagrams as arities. We should give a remark here about 2-dimensional limits and colimits in $\Lex$. It is known that $\Lex$ inherit bilimits from $\Cat$ as a category of pseudo-algebras, and bilimits of $\Cat$ being actually pseudolimits, so are they in $\Lex$. Moreover, $\cite{bourke2020accessible}$ tells us they are actually \emph{flexible}. However, it has not all strict 2-limits, see \cite[7.3]{betti1988complete}. Moreover, it only has bicolimits (some of them behave in a surprizing way, being also pseudolimits of some diagrams, see \cite{Cole}) though bifiltered ones are pseudocolimits as being computed in $\Cat$).  However, having only bicolimits rather than pseudo-ones or strict ones, $\cite{kelly1982structures}$ does not apply for it requires enriched colimits. Concerning \cite{bourke2020accessible}, we believe that it does not totally cover the 2-dimensional structure involved in bipresentability for it only considers 1-dimensional colimits. %But all those finite diagrams are sigma-compact in $\Cat$ are sigma-compact in $\Lex$. 
\end{remark}

\begin{division}[Monoidal categories] \label{monoidalcats}
Another example is provided by the 2-category $ \Mon_\textup{s}$ of \emph{monoidal categories}\footnote{Notice that by monoidal category we mean a monoidal structure whose coherences are only up-to-isomorphism.} and \emph{strong monoidal functors}. As explained in \cite{lack20092}[4.1], one can construct a 2-monad on $\Cat$ whose underlying functor maps \[A \mapsto  \coprod_{n \in \mathbb{N}} A^n,\] where $ A^n$ is the $n^{th}$ power of $A$ : this is the free monoidal category on the category $A$. Then one can show that monoidal categories are the pseudo-algebras of this 2-monad, while strong monoidal functors are the pseudomorphisms. Similarly, \emph{lax monoidal functors} correspond to the \emph{lax morphisms of algebras}.    
\end{division}

\begin{proposition}
    The 2-category $ \Mon_\textup{s}$ is locally finitely bipresentable.
\end{proposition}

\begin{proof}
Following \Cref{Algebras of a finitary pseudomonad are sigma-presentable} we have to prove the 2-monad above to be finitary. But this comes from the fact that $\Cat$ is itself locally finitely bipresentable, so that finitely weighted bilimits commute with bifiltered bicolimits here; but for each $ n \in \mathbb{N}$ the power with $n$ is a finite weighted bilimit, and combined with the commutation of bicolimits with coproducts, this ensure the 2-monad for free monoidal category to be finitary. 
\end{proof}

\begin{remark}
    One could also ask whether the category $\Mon_\textup{lax}$ of monoidal categories and lax monoidal functors is locally finitely bipresentable: but we know this cannot be the case, for this 2-category lacks bicolimits. This could be expected for our \Cref{Algebras of a finitary pseudomonad are sigma-presentable} does not apply to 2-categories of pseudo-algebra with lax morphisms. However, we do not know whether it is biaccessible or not. 
\end{remark}

\subsection{$ \textbf{Reg, Ex, Coh, Ext, Adh, Pretop}$}

Here we capture a large class of examples thanks to \cite{GarnerLacklexcolimits} amongst the different flavours of exact categories: regular, exact, extensive, coherent categories and (finitary) pretopoi. Those were unified under the formalism of \emph{$\Phi$-exactness}, which we will recall briefly before proving that the $\Phi$-exact categories it studies are instances of finitely bipresentable 2-categories. %\textcolor{red}{A word on the enriched nature of this construction} 

\begin{division}[$\Phi$-exactness à la Garner and Lack]\label{Phi-exact as pseudoalgebras}

In the following, $\Phi$ denotes a class of weights $ W : I^{\op} \rightarrow \Set$, where each $I$ is finitely complete. We will have to suppose them to be finite in the sense of \Cref{finiteweibili}. For such a $\Phi$ and a category $\mathcal{C}$, we can consider the category $ \Phi_l(\mathcal{C})$ as the full subcategory of the presheaf category $ \widehat{\mathcal{C}}$ consisting is the closure of the representables in $\widehat{\mathcal{C}}$ under finite limits and $\Phi$-lex-colimits (see \cite[Sec. 3, especially 3.1]{GarnerLacklexcolimits}). A small category $ \mathcal{C}$ is \emph{$\Phi$-lex-cocomplete} if it is lex, and for any weight $ W : I^{\op} \rightarrow \Set$ in $\Phi$ and any \emph{lex} functor $ F: I \rightarrow \mathcal{C}$ in $\Lex$, the colimit $ \colim^W_I \, F $ exists already in $\mathcal{C}$ -- beware that the functor we compute the colimit of has to be lex, as the indexing category. This amounts to requiring the existence of a left adjoint 
% https://q.uiver.app/?q=WzAsMixbMiwwLCJcXFBoaV9sKFxcbWF0aGNhbHtDfSkiXSxbMCwwLCJcXG1hdGhjYWx7Q30iXSxbMSwwLCJcXGlvdGFfXFxtYXRoY2Fse0N9IiwyLHsiY3VydmUiOjIsInN0eWxlIjp7InRhaWwiOnsibmFtZSI6Imhvb2siLCJzaWRlIjoidG9wIn19fV0sWzAsMSwiTF9cXG1hdGhjYWx7Q30iLDIseyJjdXJ2ZSI6Mn1dLFszLDIsIiIsMCx7ImxldmVsIjoxLCJzdHlsZSI6eyJuYW1lIjoiYWRqdW5jdGlvbiJ9fV1d
\[\begin{tikzcd}
	{\mathcal{C}} && {\Phi_l(\mathcal{C})}
	\arrow[""{name=0, anchor=center, inner sep=0}, "{\iota_\mathcal{C}}"', bend right=25, start anchor=-30, end anchor=200, hook, from=1-1, to=1-3]
	\arrow[""{name=1, anchor=center, inner sep=0}, "{L_\mathcal{C}}"', bend right=25, start anchor=160, end anchor=30, from=1-3, to=1-1]
	\arrow["\dashv"{anchor=center, rotate=-90}, draw=none, from=1, to=0]
\end{tikzcd}\]
Now a $\Phi$-lex-cocomplete category is said to be \emph{$\Phi$-exact} if this left adjoint is lex, which amounts to saying that $ (\mathcal{C}, L_\mathcal{C})$ bears a structure of pseudo-algebra for the pseudomonad $ \Phi_l$ on $\Lex$. Our proof technique will be based on an analysis of the forgetful functor % https://q.uiver.app/?q=WzAsMixbMCwwLCJcXFBoaV9sXFxoeVxcQWxnIl0sWzEsMCwiXFxMZXgiXSxbMCwxLCJVX1xcUGhpIl1d
\[\begin{tikzcd}
	{\Phi_l\hy\ps\Alg} & \Lex
	\arrow["{U_\Phi}", from=1-1, to=1-2]
\end{tikzcd}\]
Of course, by construction, the underlying category of $ U_\Phi\Phi_l(\mathcal{C})$ has as objects pairs $ (W,F)$ with $ W : I^{\op} \rightarrow \Set$ a weight in $ \Phi$ and $ F : I \rightarrow \mathcal{C}$ a functor, this observation will be relevant in the next Lemma.
\end{division}

We are going to prove that the 2-categories of pseudo-algebras and pseudomorphisms for the pseudomonad $ \Phi_l$ on $\Lex$ are finitely bipresentable if $\Phi$ consists of weight indexed by finitely generated lex categories: this will be done by showing the forgetfull functor $ U_\Phi : \Phi_l\hy\ps\Alg \rightarrow \Lex$ to be finitary. We need first the following lemma ensuring that $ U_\Phi$ preserves bifiltered colimits of free pseudo-algebras:

\begin{lemma}\label{colimits of free algebras}
Suppose that $\Phi$ consists of weights indexed by finitely generated lex categories and $I$ is bifiltered; then for any 2-functor $ F : I \rightarrow \Lex$, one has 
\[  U_\Phi \Phi_l (\underset{i \in I}{\bicolim} \, F(i)) \simeq  \underset{i \in I}{\bicolim} \,  U_\Phi\Phi_l F(i) \]
\end{lemma}

\begin{proof}
Take a weight $ W : J \rightarrow \Set$ in $ \Phi $ and a lex functor $ D : J \rightarrow  {\bicolim}_{i \in I} \, F(i) $: this is an object of the underlying category of $ U_\Phi\Phi_l( \bicolim_I F)$. %The bifiltered bicolimit above is computed in $\Cat$ from \Cref{Lex is finitary}. 
Now for $J $ is finitely generated, it is bicompact in $\Lex $ as observed in \Cref{reviewer stupidity} so the lex functor $D$ factorizes through some $ F(i)$ in $\Lex$
% https://q.uiver.app/?q=WzAsMyxbMCwxLCJKIl0sWzEsMSwiXFxzaWdtYV9cXFNpZ21hXFx1bmRlcnNldHtJfVxcYmljb2xpbSBcXCwgRiJdLFsxLDAsIkYoaSkiXSxbMiwxLCJxX2kiXSxbMCwxLCJEIiwyXSxbMCwyLCJEJyIsMCx7InN0eWxlIjp7ImJvZHkiOnsibmFtZSI6ImRhc2hlZCJ9fX1dXQ==
\[\begin{tikzcd}
	& {F(i)} \\
	J & {\underset{I}\bicolim \, F}
	\arrow["{q_i}", from=1-2, to=2-2]
	\arrow["D"', from=2-1, to=2-2]
	\arrow["{D'}", dashed, from=2-1, to=1-2]
\end{tikzcd}\]
But this latter factorization $ D'$ defines an object of $ \Phi_lF(i)$. It is routine to check such lifts are functorial and induce the desired equivalence. 
\end{proof}

\begin{lemma}[Gluing pointswise adjunctions into global adjunctions]\label{colimits and adjunctions}
Let $I$ be a small 2-category and $ F,G : I \rightrightarrows \mathcal{C}$ parallel 2-functors admitting both a bicolimit in $\mathcal{C}$, together with a pair of pseudonatural transformations $ R : F \Rightarrow G$ and $ L : G \Rightarrow F$ forming an adjunction $ L \dashv R$ in $ [I, \mathcal{C}]_p$. Then the induced functors $ \langle q'_iR_i \rangle_{i \in I} $ and $ \langle q_iL_i \rangle_{i \in I}$ between the respective colimits form an adjunction. 
\end{lemma}

\begin{proof}
Let $q : F \Rightarrow \Delta_{\bicolim_I F}  $ be and $ q' : G \Rightarrow \Delta_{\bicolim_I G}$ the corresponding bicolimiting cocones; the natural unit $ 1_G \Rightarrow RL $ and counit $ LR \Rightarrow 1_F$ induce the following natural modifications in $ [I, \mathcal{C}]$
% https://q.uiver.app/?q=WzAsMixbMCwwLCJHIl0sWzIsMCwiXFxEZWx0YV97XFxiaWNvbGltX0kgR30iXSxbMCwxLCJxIiwwLHsiY3VydmUiOi0zfV0sWzAsMSwicVJMIiwyLHsiY3VydmUiOjN9XSxbMiwzLCJxKlxcZXRhIiwxLHsic2hvcnRlbiI6eyJzb3VyY2UiOjIwLCJ0YXJnZXQiOjIwfX1dXQ==
\[\begin{tikzcd}
	G && {\Delta_{\bicolim_I G}}
	\arrow[""{name=0, anchor=center, inner sep=0}, "q'", bend left=25, start anchor=40, from=1-1, to=1-3]
	\arrow[""{name=1, anchor=center, inner sep=0, pos=0.53}, "q'RL"', bend right=25, start anchor=-40, from=1-1, to=1-3]
	\arrow["{q'*\eta}"', shorten <=5pt, shorten >=5pt, Rightarrow, from=0, to=1]
\end{tikzcd} \hskip1cm
% https://q.uiver.app/?q=WzAsMixbMCwwLCJGIl0sWzIsMCwiXFxEZWx0YV97XFxiaWNvbGltX0kgR30iXSxbMCwxLCJxIiwwLHsiY3VydmUiOi0zfV0sWzAsMSwicUxSIiwyLHsiY3VydmUiOjN9XSxbMywyLCJxKlxcZXBzaWxvbiIsMSx7InNob3J0ZW4iOnsic291cmNlIjoyMCwidGFyZ2V0IjoyMH19XV0=
\begin{tikzcd}
	F && {\Delta_{\bicolim_I F}}
	\arrow[""{name=0, anchor=center, inner sep=0}, "q", bend left=25, start anchor=40, from=1-1, to=1-3]
	\arrow[""{name=1, anchor=center, inner sep=0, pos=0.53}, "qLR"',bend right=25, start anchor=-40, from=1-1, to=1-3]
	\arrow["{q*\epsilon}", shorten <=5pt, shorten >=5pt, Rightarrow, from=1, to=0]
\end{tikzcd}\]

But now functoriality of the universal property of the bicolimits defines two 2-cells in $\mathcal{C}$
% https://q.uiver.app/?q=WzAsMixbMSwwLCJcXGxhbmdsZSBxJ19pUl9pIFxccmFuZ2xlX3tpIFxcaW4gSX0gXFxsYW5nbGUgcV9pTF9pIFxccmFuZ2xlX3tpIFxcaW4gSX0gIl0sWzAsMCwiMV97XFxiaWNvbGltX0kgR30iXSxbMSwwLCJcXGxhbmdsZSBxJypcXGV0YSBcXHJhbmdsZSIsMCx7ImxldmVsIjoyfV1d
\[\begin{tikzcd}
	{1_{\bicolim_I G}} & {\langle q'_iR_i \rangle_{i \in I} \langle q_iL_i \rangle_{i \in I} }
	\arrow["{\langle q'*\eta \rangle}", Rightarrow, from=1-1, to=1-2]
\end{tikzcd} \hskip-0.5cm
% https://q.uiver.app/?q=WzAsMyxbMSwwLCJcXGxhbmdsZSBxX2lMX2kgXFxyYW5nbGVfe2kgXFxpbiBJfSBcXGxhbmdsZSBxJ19pUl9pIFxccmFuZ2xlX3tpIFxcaW4gSX0gIl0sWzAsMF0sWzIsMCwiMV97XFxiaWNvbGltX0kgR30iXSxbMCwyLCJcXGxhbmdsZSBxKlxcZXBzaWxvbiBcXHJhbmdsZSIsMCx7ImxldmVsIjoyfV1d
\begin{tikzcd}
	{} & {\langle q_iL_i \rangle_{i \in I} \langle q'_iR_i \rangle_{i \in I} } & {1_{\bicolim_I G}}
	\arrow["{\langle q*\epsilon \rangle}", Rightarrow, from=1-2, to=1-3]
\end{tikzcd}\]
We must prove those two 2-cells to satisfy the triangles identities of adjunctions: but this is just a consequence of the functoriality of the equivalences between homcategories, combined to the fact that $\eta$, $ \epsilon$ satisfy already those identities. 
%Adjunctions being absolute, they are pointwise so that in each $i$ i $I$ we have an adjoint pair $ L_i \dashv R_i$ with $ \eta_i, \epsilon_i$ the associated unit and counit. 
\end{proof}

\begin{lemma}[Bifiltered bicolimits of $\Phi$-exact categories are computed in Cat]\label{The phiexactmonad is finitary}
For $ \Phi$ a class of finitely generated weights, $ U_\Phi$ is bifinitary. 
\end{lemma}

\begin{proof}
If $I$ is bifiltered and $ F : I \rightarrow \Phi_l\hy\Alg$ is a 2-functor, then for each $i$ in $I$ we have an adjunction with $ L_i : \Phi_l F(i) \rightarrow F(i)$ lex:
% https://q.uiver.app/?q=WzAsMixbMCwwLCJGKGkpIl0sWzIsMCwiVV9cXFBoaVxcUGhpX2xGKGkpIl0sWzAsMSwiXFxpb3RhX2kiLDIseyJjdXJ2ZSI6Mn1dLFsxLDAsIkxfaSIsMix7ImN1cnZlIjoyfV0sWzMsMiwiIiwwLHsibGV2ZWwiOjEsInN0eWxlIjp7Im5hbWUiOiJhZGp1bmN0aW9uIn19XV0=
\[\begin{tikzcd}
	{F(i)} && {U_\Phi\Phi_lF(i)}
	\arrow[""{name=0, anchor=center, inner sep=0}, "{\iota_i}"', bend right=25, start anchor=-45, from=1-1, to=1-3]
	\arrow[""{name=1, anchor=center, inner sep=0}, "{L_i}"', bend right=25, end anchor=45, from=1-3, to=1-1]
	\arrow["\dashv"{anchor=center, rotate=-90}, draw=none, from=1, to=0]
\end{tikzcd}\] 
Then we can compose each left adjoint $ L_i$ with the corresponding colimit inclusion $ q_i : F(i) \rightarrow  {\bicolim}_{i \in I} \, F(i) $ to get a cocone in $\Cat $ as below, which lifts to a functor $ \langle q_i L_i \rangle_{i \in I}$ as below from \Cref{colimits of free algebras}: 
% https://q.uiver.app/?q=WzAsNSxbMCwyLCJGKGkpIl0sWzAsMCwiXFxzaWdtYV9cXFNpZ21hIFxcdW5kZXJzZXR7aSBcXGluIEl9e1xcYmljb2xpbX0gXFwsIEYoaSkgIl0sWzIsMCwiVV9cXFBoaSBcXFBoaV9sIChcXHNpZ21hX1xcU2lnbWEgXFx1bmRlcnNldHtpIFxcaW4gSX17XFxiaWNvbGltfSBcXCwgRihpKSkiXSxbMiwyLCJcXFBoaV9sIEYoaSkgIl0sWzIsMSwiIFxcc2lnbWFfXFxTaWdtYSBcXHVuZGVyc2V0e2kgXFxpbiBJfXtcXGJpY29saW19IFxcLCAgVV9cXFBoaVxcUGhpX2wgRihpKSAiXSxbMywwXSxbMCwxLCJxX2kiXSxbMyw0XSxbNCwyLCJcXHNpbWVxIiwzLHsic3R5bGUiOnsiYm9keSI6eyJuYW1lIjoibm9uZSJ9LCJoZWFkIjp7Im5hbWUiOiJub25lIn19fV0sWzIsMSwiIiwzLHsic3R5bGUiOnsiYm9keSI6eyJuYW1lIjoiZGFzaGVkIn19fV1d
\[\begin{tikzcd}
	{\underset{i \in I}{\bicolim} \, F(i) } && {U_\Phi \Phi_l (\underset{i \in I}{\bicolim} \, F(i))} \\
	&& { \underset{i \in I}{\bicolim} \,  U_\Phi\Phi_l F(i) } \\
	{F(i)} && {U_\Phi\Phi_l F(i) }
	\arrow["L_i",from=3-3, to=3-1]
	\arrow["{q_i}", from=3-1, to=1-1]
	\arrow["q'_i"', from=3-3, to=2-3]
	\arrow["\simeq"{marking}, draw=none, from=2-3, to=1-3]
	\arrow["\langle q_iL_i \rangle_{i \in I}"', dashed, from=1-3, to=1-1]
\end{tikzcd}\]
Moreover we know this functor $ \langle q_i L_i \rangle_{i \in I}$ to be lex for the bifiltered colimits above are created by the forgetful functor $ \Lex \rightarrow \Cat$. We must prove that this functor is left adjoint to the unit of $ {\bicolim}_{i \in I} \, F(i)$. The unit $ \eta : 1_\Lex \Rightarrow \Phi_l$ is natural, while at each $ d : i \rightarrow j$ we took a pseudomorphism of algebras $(F(d), \alpha_d)$ as below
% https://q.uiver.app/?q=WzAsNCxbMCwwLCJcXFBoaV9sRihpKSJdLFswLDEsIkYoaSkiXSxbMSwwLCJcXFBoaV9sRihqKSJdLFsxLDEsIkYoaikiXSxbMCwxLCJMX2kiLDJdLFswLDIsIlxcUGhpX2xGKGQpIl0sWzIsMywiTF9qIl0sWzEsMywiRihkKSIsMl0sWzAsMywiXFxhbHBoYV9kIFxcYXRvcCBcXHNpbWVxIiwxLHsic3R5bGUiOnsiYm9keSI6eyJuYW1lIjoibm9uZSJ9LCJoZWFkIjp7Im5hbWUiOiJub25lIn19fV1d
\[\begin{tikzcd}
	{\Phi_lF(i)} & {\Phi_lF(j)} \\
	{F(i)} & {F(j)}
	\arrow["{L_i}"', from=1-1, to=2-1]
	\arrow["{\Phi_lF(d)}", from=1-1, to=1-2]
	\arrow["{L_j}", from=1-2, to=2-2]
	\arrow["{F(d)}"', from=2-1, to=2-2]
	\arrow["{\alpha_d \atop \simeq}"{description}, draw=none, from=1-1, to=2-2]
\end{tikzcd}\]
so that the data of all $ L_i$ defines a pseudonatural transformation $ U_\Phi \Phi_l \Rightarrow F$ we denote as $L$. Then, from \Cref{colimits and adjunctions}, we know that the induced $ \langle q_iL_i \rangle_{i \in I}$ and $ \langle q_i' \iota_i \rangle_{i \in I}$ form an adjunction in $ \Cat$. This achieves to prove that the $ $filtered bicolimit is canonically equiped with a structure of pseudo-algebra which is sent to the underlying bifiltered colimit of categories by $U_\Phi$, which is hence finitary. 
\end{proof}

\begin{theorem}[2-categories of $\Phi$-exact categories are finitely bipresentable]
If $ \Phi$ is a class of finitely generated weights, then $ \Phi_l\hy\ps\Alg$ is finitely bipresentable, and moreover its bifiltered colimits are computed in $\Cat$.  
\end{theorem}

\begin{proof}
Under those assumptions, the pseudomonad $ \Phi_l$ is finitary by \Cref{The phiexactmonad is finitary}, as a composite of $ U_\Phi\Phi_l$ which both preserve bifiltered bicolimits. Hence by \Cref{Algebras of a finitary pseudomonad are sigma-presentable} we know $ \Phi_l\hy\Alg$ to be finitely bipresentable.
\end{proof}

\begin{corollary}
The following 2-categories are finitely bipresentable:\begin{itemize}
    \item $ \textbf{Reg}$, the 2-category of small regular categories and regular functors;
    \item $ \textbf{Ex}$, the 2-category of small (Barr)-exact categories and exact functors;
    \item $ \textbf{Coh}$, the 2-category of small coherent categories and coherent functors;
    \item $ \textbf{Ext}_\omega$, the 2-category of small finitely-extensive categories and functors preserving finite coproducts;
    \item $ \textbf{Adh}$, the 2-category of small adhesive categories and adhesive functors;
    \item $ \textbf{Pretop}_\omega$, the 2-category of small finitary pretopoi and coherent functors.
\end{itemize}
\end{corollary}
\begin{proof}
By \cite[Sec. 5]{GarnerLacklexcolimits} these are all in the hypotheses of the theorem above, as in each case, the weights are indexed by the free lex category over some finite category. Indeed, $\textbf{Reg}$ is \cite[Sec. 5.1]{GarnerLacklexcolimits}, $ \textbf{Ex}$ is \cite[Sec. 5.2]{GarnerLacklexcolimits}, $ \textbf{Coh}$ is \cite[Sec. 5.6]{GarnerLacklexcolimits}, $ \textbf{Ext}_\omega$ is \cite[Sec. 5.3]{GarnerLacklexcolimits},   $ \textbf{Adh}$ is \cite[Sec. 5.7]{GarnerLacklexcolimits}. Pretopoi are just exact and extensive categories.
\end{proof}

\begin{remark}
    Notice that for classes of weights that are not indexed by finitely generated lex categories, the 2-category of $\Phi$-exact categories is not expected to be finitely bipresentable: for instance the 2-category of categories with filtered colimits where filtered colimits commutes with finite limits discussed in \cite[Sec. 5.9]{GarnerLacklexcolimits}. This would not be surprising, as those correspond to categories with unbounded operations, in analogy with the fact that \emph{preframes} as defined in \cite{elephant}[Part C, Remark C1.1.2] are not finitely presentable. 
\end{remark}

\section*{Acknowledgements}

The first author was supported by the Swedish Research Council (SRC, Vetenskapsrådet) under Grant No. 2019-04545. The research has received funding from Knut and Alice Wallenbergs Foundation through the Foundation’s program for mathematics. He is also grateful to the IRIF for their invitation in December 2021, when the the main proof of the last Section was found.
Both the authors are grateful to Paul-André Melliès for his support towards this project. We are also grateful to Eduardo Dubuc (whose previous work \cite{descotte2018sigma} was an important inspiration to us) for remarks on the first version of this paper. Similarly, we are grateful to John Bourke for a careful reading of the first version of this paper as well as important remarks. 
The authors are very grateful to the anonymous referee for their comments, which led to an improvement in the exposition of the paper and to the addition of \Cref{monoidalcats}.

\bibliography{Bib}

\newcommand{\etalchar}[1]{$^{#1}$}
\begin{thebibliography}{BKPS89}

\bibitem[AR94]{adamek1994locally}
Jiri Adamek and Jiri Rosicky.
\newblock {\em Locally presentable and accessible categories}, volume 189.
\newblock Cambridge University Press, 1994.

\bibitem[BG88]{betti1988complete}
Renato Betti and Marco Grandis.
\newblock Complete theories in $2 $-categories.
\newblock {\em Cahiers de topologie et g{\'e}om{\'e}trie diff{\'e}rentielle
  cat{\'e}goriques}, 29(1):9--57, 1988.

\bibitem[BKP89]{blackwell1989two}
Robert Blackwell, Gregory~M Kelly, and A~John Power.
\newblock Two-dimensional monad theory.
\newblock {\em Journal of pure and applied algebra}, 59(1):1--41, 1989.

\bibitem[BKPS89]{bird1989flexible}
Gregory Bird, Gregory~M Kelly, A~John Power, and RH~Street.
\newblock Flexible limits for 2-categories.
\newblock {\em Journal of Pure and Applied Algebra}, 61(1):1--27, 1989.

\bibitem[Bor94]{borceux1994handbook}
Francis Borceux.
\newblock {\em Handbook of categorical algebra: volume 1, Basic category
  theory}, volume~1.
\newblock Cambridge University Press, 1994.

\bibitem[Bou20]{bourke2020accessible}
John Bourke.
\newblock Accessible aspects of 2-category theory, 2020.

\bibitem[Col16]{Cole}
Julian Cole.
\newblock The bicategory of topoi and spectra.
\newblock pages 1--16, 2016.

\bibitem[DDS16]{descotte2016exactness}
Maria~Emilia Descotte, Eduardo~Julio Dubuc, and Martin Szyld.
\newblock A construction of certain weak colimits and an exactness property of
  the 2-category of categories, 2016.

\bibitem[DDS18]{descotte2018sigma}
Maria~Emilia Descotte, Eduardo~Julio Dubuc, and Martin Szyld.
\newblock Sigma limits in 2-categories and flat pseudofunctors.
\newblock {\em Advances in Mathematics}, 333:266--313, 2018.

\bibitem[Des20]{descotte2020theory}
Maria~Emilia Descotte.
\newblock A theory of 2-pro-objects, a theory of 2-model 2-categories and the
  2-model structure for 2-pro (c).
\newblock {\em arXiv preprint arXiv:2010.10636}, 2020.

\bibitem[GHL21]{gagna2021bilimits}
Andrea Gagna, Yonatan Harpaz, and Edoardo Lanari.
\newblock Bilimits are bifinal objects.
\newblock {\em arXiv preprint arXiv:2103.16394}, 2021.

\bibitem[GL12]{GarnerLacklexcolimits}
Richard Garner and Stephen Lack.
\newblock Lex colimits.
\newblock {\em Journal of Pure and Applied Algebra}, 216(6):1372–1396, Jun
  2012.

\bibitem[GU06]{gabriel2006lokal}
Peter Gabriel and Friedrich Ulmer.
\newblock {\em Lokal pr{\"a}sentierbare kategorien}, volume 221.
\newblock Springer-Verlag, 2006.

\bibitem[J{\etalchar{+}}02]{elephant}
Peter~T Johnstone et~al.
\newblock {\em Sketches of an Elephant: A Topos Theory Compendium}.
\newblock Oxford University Press, 2002.

\bibitem[JY21]{johnson20212}
Niles Johnson and Donald Yau.
\newblock {\em 2-dimensional categories}.
\newblock Oxford University Press, USA, 2021.

\bibitem[Kel82]{kelly1982structures}
Gregory~M Kelly.
\newblock Structures defined by finite limits in the enriched context, i.
\newblock {\em Cahiers de topologie et g{\'e}om{\'e}trie diff{\'e}rentielle
  cat{\'e}goriques}, 23(1):3--42, 1982.

\bibitem[Ken92]{kennison1992fundamental}
John~F Kennison.
\newblock The fundamental localic groupoid of a topos.
\newblock {\em Journal of pure and applied algebra}, 77(1):67--86, 1992.

\bibitem[Lac09]{lack20092}
Stephen Lack.
\newblock A 2-categories companion.
\newblock In {\em Towards higher categories}, pages 105--191. Springer, 2009.

\bibitem[LMV02]{MarmolejoBeck}
I.J. {Le Creurer}, F.~Marmolejo, and E.M. Vitale.
\newblock Beck's theorem for pseudo-monads.
\newblock {\em Journal of Pure and Applied Algebra}, 173(3):293--313, 2002.

\bibitem[Mak95]{makkai1995gabbay}
Michael Makkai.
\newblock On gabbay's proof of the craig interpolation theorem for
  intuitionistic predicate logic.
\newblock {\em Notre Dame Journal of Formal Logic}, 36(3):364--381, 1995.

\bibitem[MP89]{makkai1989accessible}
Michael Makkai and Robert Par{\'e}.
\newblock {\em Accessible Categories: The Foundations of Categorical Model
  Theory: The Foundations of Categorical Model Theory}, volume 104.
\newblock American Mathematical Soc., 1989.

\bibitem[Nun16]{nunes2016biadjoint}
Fernando~Lucatelli Nunes.
\newblock On biadjoint triangles.
\newblock {\em arXiv preprint arXiv:1606.05009}, 2016.

\bibitem[Osm24]{osmondcodescent}
Axel Osmond.
\newblock Codescent and bicolimits of pseudo-algebras.
\newblock {\em Applied Categorical Structures}, 32(2):1--56, 2024.

\bibitem[Pow89]{power1989general}
A~John Power.
\newblock A general coherence result.
\newblock {\em Journal of Pure and Applied Algebra}, 57(2):165--173, 1989.

\bibitem[Shu12]{shulman2012not}
Michael~A Shulman.
\newblock Not every pseudoalgebra is equivalent to a strict one.
\newblock {\em Advances in Mathematics}, 229(3):2024--2041, 2012.

\bibitem[Str76]{street1976limits}
Ross Street.
\newblock Limits indexed by category-valued 2-functors.
\newblock {\em Journal of Pure and Applied Algebra}, 8(2):149--181, 1976.

\bibitem[Str82a]{street1982characterization}
Ross Street.
\newblock Characterization of bicategories of stacks.
\newblock In {\em Category Theory}, pages 282--291. Springer, 1982.

\bibitem[Str82b]{street1982two}
Ross Street.
\newblock Two-dimensional sheaf theory.
\newblock {\em Journal of Pure and Applied Algebra}, 23(3):251--270, 1982.

\end{thebibliography}
\bibliographystyle{alpha}

\vfill

 \begin{minipage}{0.49\textwidth}
Ivan Di Liberti\newline
Department of Mathematics\newline
Stockholm University\newline
Stockholm, Sweden\newline
\href{mailto:diliberti.math@gmail.com}{\sf diliberti.math@gmail.com}
  \end{minipage}
 \begin{minipage}{0.49\textwidth}
Axel Osmond\newline
Istituto Grothendieck \newline
Paris, France\newline
\href{mailto:axelosmond@orange.fr}
{\sf axelosmond@orange.fr}
  \end{minipage}

%\printbibliography

\end{document}